\documentclass{article}
\usepackage[utf8]{inputenc}

\usepackage[left=2cm, right=2cm, top=2cm]{geometry}
\usepackage{amsfonts}
\usepackage{amsmath}
\usepackage{amsthm}
\usepackage{amssymb}
\usepackage{cancel}
\usepackage{bbold}
\usepackage{xcolor}
\usepackage{esint}
\usepackage{hyperref}
\usepackage{xifthen}
\usepackage{verbatim}
\usepackage{dsfont}
\usepackage{mathtools}
\usepackage{bm}
\usepackage{tikzsymbols}
\usepackage{mathtools}
\usepackage{mathrsfs}
\usepackage[title]{appendix}
\mathtoolsset{showonlyrefs}
\usepackage{tikz}
\usetikzlibrary{arrows.meta}
\usepackage{float}

\newtheorem{thm}{Theorem}[section]
\newtheorem{lem}[thm]{Lemma}
\newtheorem{prop}[thm]{Proposition}
\newtheorem{cor}[thm]{Corollary}

\theoremstyle{remark}
\newtheorem{rmk}[thm]{Remark}
\newtheorem{df}[thm]{Definition}
\newtheorem{note}[thm]{Notation}
\newtheorem{ex}[thm]{Example}
\newtheorem{qu}{Question}

\newcommand{\NN}{\mathbb{N}}

\newcommand{\R}[1]{\mathbb{R}^{#1}}

\newcommand{\weak}{\rightharpoonup}

\newcommand{\La}{\langle}
\newcommand{\Ra}{\rangle}

\newcommand{\Hess}{\mathrm{Hess}\,}

\newcommand{\eps}{\varepsilon}
\newcommand{\Ric}{\mathrm{Ric}}

\newcommand{\Vol}{\text{Vol}}
\renewcommand\div{\mathrm{div}}

\newcommand{\Tr}{\text{tr}}

\newcommand{\ptau}{\partial_\tau}
\newcommand{\Lip}{\mathrm{Lip}}
\newcommand{\tr}{\text{tr}}
\renewcommand{\Tilde}{\widetilde}

\usepackage[colorinlistoftodos,prependcaption]{todonotes}

\newcommand{\cD}{\mathcal{D}}
\newcommand{\cE}{\mathcal{E}}
\newcommand{\cF}{\mathcal{F}}

\newcommand{\cL}{\mathcal{L}}

\newcommand{\cP}{\mathcal{P}}

\newcommand{\cS}{\mathcal{S}}

\newcommand{\cW}{\mathcal{W}}

\usepackage{graphicx} 

\AtEndDocument{%
  \par
  \medskip
  \begin{tabular}{@{}l@{}}%
    \textbf{Institut für Angewandte Mathematik, Universität Bonn, Bonn, Germany}\\
    \textit{E-mail addresses}: \text{\href{mailto:flaim@iam.uni-bonn.de}{flaim@iam.uni-bonn.de}, \ \href{mailto:hupp@iam.uni-bonn.de}{hupp@iam.uni-bonn.de}}
  \end{tabular}}

\title{On a parabolic curvature lower bound generalizing Ricci flows}
\author{Marco Flaim and Erik Hupp}
\date{}

\begin{document}

\maketitle
\begin{abstract}
Optimal transport plays a major role in the study of manifolds with Ricci curvature bounded below. A number of results in this setting have been extended to super Ricci
flows, revealing a unified approach to analysis on Ricci nonnegative manifolds and Ricci
flows. However we observe that the monotonicity of Perelman's functionals ($\cF$, $\cW$, reduced volume), which hold true for Ricci flows and Ricci nonnegative manifolds, cannot be strictly generalized to super Ricci flows. In 2010 Buzano introduced a condition which still generalizes Ricci flows and Ricci nonnegative manifolds, and on which Perelman's monotonicities do hold. We provide characterizations of this condition using optimal transport methods and understand it heuristically as Ricci non-negativity of the space-time. This interpretation is consistent with its equivalence to Ricci non-negativity on Perelman's infinite dimensional manifold.

More precisely, we prove that for smooth evolutions of Riemannian manifolds, this condition is equivalent to a Bochner inequality (resembling Perelman's Harnack inequality but for the forward heat flow), a gradient estimate for the heat flow, a Wasserstein contraction along the adjoint heat flow, the convexity of a modified entropy along Wasserstein geodesics, and an Evolutionary Variational Inequality (EVI). The optimal transport statements use Perelman's $L$ distance as cost, as first studied on Ricci flows by Topping and by Lott. As in the setting of manifolds with Ricci curvature bounded from below, we also consider dimensionally improved and weighted versions of these conditions. The dimensional Bochner inequality and all gradient estimates for the forward heat equation, along with the EVIs, appear to be new even for general solutions to the Ricci flow, and are related to the Hamiltonian perspective on the $L$ distance. Most of our proofs do not use tensor calculus or Jacobi fields, suggesting the possibility of future extensions to more singular settings.
\end{abstract}

\setcounter{tocdepth}{2}
\tableofcontents

\section{Introduction}
Ricci flow, introduced by Hamilton in \cite{Hamilton82}, has had a central role in the development of geometric analysis in the last 30 years. Part of its strength relies on its parabolic nature, serving as a kind of intrinsic heat flow for Riemannian metrics. It is interesting then to understand which properties are only valid for Ricci flows (and even characterize Ricci flows), and which ones can be extended to a more general class of evolutions. In particular, a number of functional inequalities that hold on Ricci flow resemble inequalities that hold on manifolds of nonnegative Ricci curvature, and in fact are equivalent to this curvature lower bound. One might therefore aim to prove the equivalence of the functional inequalities on the Ricci flow side and identify the underlying ``curvature lower bound''.

As a first generalization one might consider the super Ricci flow condition, and indeed when one studies the (adjoint) heat equation on super Ricci flows one recovers several properties that hold for the heat equation on Ricci nonnegative manifolds. A stronger inequality that generalizes Ricci flow was introduced by Buzano via an inhomogeneous tensor inequality dubbed the ``$\mathcal{D}$-condition'', which recovers many of the powerful monotone quantities that have been exploited in the study of Ricci flows. The aim of this work is to characterize this condition by a number of coarse metric-measure inequalities that have strict analogs in the static, Ricci nonnegative setting.

\subsection{Super Ricci flows}
We start recalling the properties of super Ricci flows (SRFs). This setting was first considered by McCann and Topping \cite{mccanntopp} and was then studied in a non-smooth setting by Sturm and Kopfer--Sturm \cite{sturmsrf, KopStu18}. It is also central to Bamler's theory of $\mathbb{F}$-convergence, where it forms a precompact family of metric flows \cite{BamlerEntropy,BamlerCompactness}. Consider a time dependent Riemannian manifold $(M,g_t)_{t\in I}$ for some interval $I \subseteq \R{}$, we consider on it:
\begin{itemize}
    \item The heat equation (forward heat equation for functions) with associated semigroup $P_{t,s}$, $s<t$, acting on functions on $M$:
\begin{align}
    \partial_t P_{t,s}u =\Delta_t P_{t,s}u.
\end{align}

    \item The conjugate heat equation (backward heat equation for measures) with semigroup $\hat P_{t,s}$ acting on measures on $M$
    \begin{align}
    \partial_{s}\hat P_{t,s}\mu = -\Delta_{s} \hat P_{t,s}\mu.
\end{align}
    In terms of densities, if $\hat P_{t,s}\mu=u_s\,dV_{g_s} =: (P^*_{t,s} u_t)\,dV_s$, then
\begin{align}\label{eq:adj-heat-eq}
    \partial_t u_t = -\Delta_t u_t + S_t u_t
\end{align}
where $S=-\frac{1}{2}\tr(\partial_tg)$ represents (minus) the derivative of the volume form.
\end{itemize} 

\begin{thm}[\cite{mccanntopp,sturmsrf}]\label{mccanntopping}
    Consider a smooth, closed time-dependent Riemannian manifold $(M,g_t)_{t\in I}$, then the following are equivalent
    \begin{enumerate}
        \item $(M,g_t)$ is a super Ricci flow: 
        \begin{align}\label{eq:SRFcond}
            \partial_t g_t \geq -2\Ric(g_t).
        \end{align}
        \item The Bochner formula holds: for any heat flow
        \begin{align}
            (\partial_t-\Delta_t)|P_{t,s}u|^2\leq 0\,.
        \end{align}
        \item The following gradient estimate holds for any heat flow
        \begin{align}
            |\nabla P_{t,s}u|^2_t \leq P_{t,s}|\nabla u|_s^2\,.
        \end{align}
        \item The following Wasserstein contraction holds for any two conjugate heat flows
        \begin{align}
            W_s(\hat{P}_{t,s}\mu, \hat{P}_{t,s}\nu) \leq W_t(\mu,\nu),
        \end{align}
        where $W_t$ is the 2-Wasserstein distance induced by the distance induced by $g_t$.
        \item The following dynamic convexity of the entropy holds: For any $t$, for any $W_t$ geodesic $(\mu_a)_{a\in[0,1]}$ in the domain of the entropy $\cE_t$, 
        \begin{align}\label{eq:SRFEntconv}
            \partial_a^+\cE_t(\mu_a)\big|_{a=1^-} - \partial_a^-\cE_t(\mu_a)\big|_{a=0^+} \geq -\frac{1}{2}\partial_t^-W_{t}^2(\mu_0,\mu_1),
        \end{align}
        where $\cE_t(\mu)=\int\rho\log\rho \,dV_t $ when $\mu=\rho \,dV_t$, and $+\infty$ otherwise.
        \end{enumerate}
\end{thm}
When $\partial_t g=0$, it recovers the celebrated result for manifolds with nonnegative Ricci, see \cite{sturmvonrenesse}, later extended in the theory of $\mathsf{RCD}(0,\infty)$ spaces, see \cite{AmbGigSav15}. Indeed, while the differential inequality \eqref{eq:SRFcond} uses regularity of the metric tensor, the other conditions (in particular \eqref{eq:SRFEntconv}) can be formulated on metric measure spaces, see \cite{MR2237206,MR2237207,MR2480619,AGS1,AGS2,AmbGigSav15} for the static setting and \cite{sturmsrf,KopStu18} for the dynamic setting. In \cite{SturmUpper}, \cite{erbar2025synthetic} the notion of Ricci upper bounds and sub Ricci flows for metric measure spaces is studied by means of optimal transport, while in \cite{naber2013characterizations,HaslhoferNaber18Bochner,HaslhoferNaber18RF,HKN22,Kennedy23BochnerRF} by functional inequalities on the path space.

\subsection{Perelman's monotone functionals}
In \cite{perelman2002entropy}, Perelman showed that a number of functionals (such as the $\mathcal{F}$ energy, the $\mathcal{W}$ entropy, and the reduced volume) are monotone along any Ricci flow coupled with the conjugate heat equation, which was a crucial step in the resolution of Poincaré conjecture. For the sake of brevity, we recall here only the $\mathcal{W}$-entropy of $\mu=u\,dV_{g_\tau}$ at time $\tau$. Let $R_\tau$ denote the scalar curvature of $g_\tau$, then

\begin{align}
    \mathcal{W}(\mu):=\int_M \tau \left[(|\nabla \log u|^2_\tau + R_\tau ) -\log(u) \right] d\mu + \frac{n}{2}\log(4\pi \tau) - n.
\end{align}

This monotonicity is closely related to functional inequalities involving the $L_-$ distance: for a curve $\gamma:[\sigma,\tau] \to M$, parametrized by a positive backward time variable $ T -t =:\eta \in [\sigma,\tau]\subseteq (T - I) \cap [0,\infty)$ for some choice of $T\in\R{}$, define
\begin{align}
    \mathcal{L}_-(\gamma) = \int_\sigma^\tau \sqrt{\eta} \left(|\dot \gamma|_\eta^2 + R_\eta(\gamma_\eta)\right) d\eta\,,
\end{align}
and $L_-^{\sigma,\tau}(x,y)$ by classical minimization among paths connecting $x$ and $y$.

Using this cost to define a Wasserstein distance between probability measures, Topping \cite{Top09} proved the following result.
\begin{thm}
    On a time-reversed Ricci flow $(M,g_\tau)_{\tau\in[0,T]}$, two conjugate heat flows satisfy the following
        \begin{equation}
         \sqrt{\alpha}W_{L_-^{\alpha \sigma,\alpha \tau}}(\hat{P}_{\sigma,\alpha \sigma}\mu,\hat{P}_{\tau,\alpha \tau}\nu) \leq W_{L_-^{\sigma,\tau}}(\mu,\nu) + \frac{n}{2}(\alpha-1)\left(\sqrt{\tau} - \sqrt{\sigma}\right)
        \end{equation}
    for all $0<\sigma<\tau < T$, $1 < \alpha\leq T/\tau$. Here $W_{L^{\sigma,\tau}_-}(\cdot,\cdot)$ is the Wasserstein distance induced by $L_-^{\sigma,\tau}$.
\end{thm}
Moreover, such Wasserstein monotonicity implies the monotonicity of the $\mathcal{W}$-entropy, also by \cite{Top09}. Similar results were obtained by Lott \cite{Lot09} using the $L_0$ distance and $\mathcal{F}$ energy, with different techniques (Eulerian point of view).

Shortly after Perelman's work, Ni \cite{Ni04} noticed that similar monotonicity formulas hold also on static manifolds with $\Ric\geq0$. Moreover, the result of Topping is closely related (see Remark \ref{rmk:Kuwada-contraction}) to a Wasserstein contraction statement that holds on $\Ric\geq0$ manifolds for different times:
\begin{align}
    W^2(\hat {P}_s\mu,\hat{P}_t\nu) \leq W^2(\mu,\nu) + 2n (\sqrt{t}-\sqrt{s})^2,
\end{align}
where $W$ is the 2-Wasserstein distance induced by $d_g$, see for example \cite{kuwadaspace}. In the context of $\mathsf{RCD}$ spaces the monotonicity of $\cW$ was obtained in \cite{JianZhang16, KuwadaLi, brena2025perelman}.

Summing up, such monotonicity results hold on Ricci flows and on manifolds with nonnegative Ricci curvature: hence it might be natural to guess that they hold on super Ricci flows. This is however not true in general, see \eqref{eq:F-round-spheres} and the surrounding discussion.

\subsection{The \texorpdfstring{$\cD$}{D}-condition}\label{ss:D-condition}
In \cite{mullermonot} Buzano introduced the following condition. Let $(M,g_\tau)_{\tau \in I}$ be a time dependent Riemannian manifold evolving smoothly with
\begin{equation}
    \frac{\partial}{\partial \tau} g_\tau = 2 \mathcal{S}_\tau
\end{equation}
for some $\mathcal{S}_\tau$ symmetric bilinear form with trace $\text{tr}_\tau(\mathcal{S}_\tau)=S_\tau$. We define
\begin{equation}\label{Mullercond}
    \mathcal{D}(X):=-\partial_\tau S - \Delta S - 2|\mathcal{S}|^2 + 4(\div \mathcal{S})(X) - 2 \langle\nabla S, X\rangle + 2 \Ric(X,X)- 2\mathcal{S}(X,X)
\end{equation}
for any vector field $X$ on $M$, at some time $\tau$. The key assumption is that 
\begin{equation} \mathcal{D}(X)\geq0 \ \ \ \text{for any vector field $X$ and any time $\tau$},
\end{equation}
and we will denote it by ``$\mathcal{D}$-condition'' or $\mathcal{D}\geq0$. Note that when we consider the flow in the forward direction, we define $\partial_t g=-2\mathcal{S}$ and 
\begin{align}
    \mathcal{D}(X)=\partial_t S - \Delta S - 2|\mathcal{S}|^2 + 4(\div \mathcal{S})(X) - 2 \langle\nabla S, X\rangle + 2 \Ric(X,X)- 2\mathcal{S}(X,X)\,.
\end{align}

In \cite{mullermonot} Buzano showed that the condition $\mathcal{D}\geq0$ is sufficient for the monotonicity of reduced volume to hold for a generalized $L_-$ distance. Later Huang \cite{Huang} noticed that Topping's $L_-$-Wasserstein monotonicity, and hence also the monotonicity of the $\mathcal{W}$ entropy, holds under $\mathcal{D}\geq0$.

\paragraph{Flows satisfying $\mathcal{D}\geq0$.}
On static manifolds the tensor becomes $\mathcal{D} = 2\Ric$, hence the $\mathcal{D}$ condition is equivalent to having nonnegative Ricci curvature. 

If $(M,g_\tau)$ is a (time-reversed) Ricci flow we have $\mathcal{D}\equiv0$. We will provide evidence that $\cD \geq 0$ might serve as a more accurate uniformization of Ricci flows and Ricci nonnegative manifolds than the super Ricci flow condition, at least in the smooth setting.

Other, more exotic examples of flows satisfying $\mathcal{D}\geq0$ are provided by Buzano \cite{mullermonot}. Note that in some of these examples the linear terms $4(\div \mathcal{S})(X) - 2 \langle\nabla S, X\rangle$ do not vanish.

\paragraph{$\mathcal{D}\geq0$ implies being a SRF.}
In order to understand it better, one can view the quantity $\mathcal{D}$ as the sum of three addenda: the first three terms are related to the evolution of the scalar curvature along the Ricci flow, the fourth and the fifth are a sort of Bianchi relation (they cancel if $\mathcal{S}=\Ric$ by the twice contracted Bianchi identity) and the last two terms correspond to the super Ricci flow:
\begin{equation}\label{eq:Dcondition}
    -\partial_\tau S - \Delta S - 2|\mathcal{S}|^2 \qquad \qquad 4(\div \mathcal{S})(X) - 2 \langle\nabla S, X\rangle \qquad \qquad 2 \Ric(X,X)- 2\mathcal{S}(X,X).
\end{equation}

Note that the three addenda have different homogeneity with respect to scaling the vector field $X$. Scaling $X$ is also equivalent to parabolically rescaling $\Tilde{g}(\tau,x)=\lambda g(\frac{\tau}{\lambda},x)$. This allows the conditions to be decoupled and implies that \textit{the first and the last addendum have to be nonnegative by themselves}: in particular, every flow satisfying this condition must be a super Ricci flow. 

\paragraph{Examples of SRF not satisfying $\mathcal{D}\geq0$.}\label{ss:super-Ric-flow-bad}
Note that this is a second-order condition in time, since we are taking the derivative of $S$.

We have said that \[\{\{\text{Ricci flows}\} \cup \{\text{static with }\Ric\geq 0\}\}\subseteq \{\cD\geq0\}\subseteq\{\text{super Ricci flows}\}.\]
Now we note that the inclusions above are strict. Indeed, consider a flow of round spheres $(\mathbb{S}^n,r^2(t)g_{\mathbb{S}^n})_{t \in I}$. Then it is a: \begin{itemize}
        \item Ricci flow if $\partial_t r^2(t) = -2(n-1)$;
        \item super Ricci flow if $\partial_t r^2(t) \geq -2(n-1)$;
        \item flow satisfying $\mathcal{D}\geq0$ if \[\begin{cases}
            \partial_t r^2(t) \geq -2(n-1) \ \ \ &(\text{SRF term} )\\
            \partial^2_tr^2(t) \leq 0 \ \ \ &(\text{from } \partial_t S - \Delta S - 2|\mathcal{S}|^2\geq0)
        \end{cases}\,.\]
    \end{itemize}

Staying with this example, one can also check that the monotonicity of the $\cF$ and $\cW$ functionals along the adjoint heat flow is a property of smooth families of Riemannian manifolds $(M^n,g_t)_{t \in I}$ satisfying $\cD \geq 0$ (originally due to \cite{Huang}, see also Section \ref{s:FandW}), but not of super Ricci flows. For the definition of these functionals, we direct the reader to Section \ref{s:FandW}. 

Indeed, on such a flow, $\mu_t\equiv \mu_0 := dV_{g_{t_0}}/\Vol(\mathbb{S}^n, g_{t_0})$ is a solution of the adjoint heat equation, and we have
\begin{align}\label{eq:F-round-spheres}
    \partial_t\mathcal{F}(\mu_t) = \frac{n}{2r^2(t)}\left(-\partial_t^2 r^2(t) + \frac{(\partial_t r^2(t))^2}{r^2(t)}\right)\,,
\end{align}
hence the super Ricci flow condition does not suffice for its monotonicity. In this case the homogeneity $0$ term of the $\mathcal{D}$ condition is in fact sufficient for the monotonicity of $\mathcal{F}$, but in the general case the non-negativity of this term alone is not sufficient (on a static space this term is always zero, but $\cF$ might not be monotone).

It is worth noting that in \cite{XDLi12, LiLi15} it is shown that $\cW$ and $\cF$ have a monotonicity property on $(M,g_t,\mathfrak{m}_t=e^{-f_t}dV_t)$ satisfying the weighted super-Ricci flow condition and with $f_t$ evolving so that $\mathfrak{m}_t$ is fixed. Note that in their case, $\cF$ and $\cW$ are non-decreasing along a weighted {\it forward} heat flow (which preserves probability measures by time-independence of $\mathfrak{m}_t$).

\subsection{Main results}
In this work we show geometric and functional characterizations of the $\mathcal{D}$ condition, in the spirit of the characterization of Ricci lower curvature bounds \cite{sturmvonrenesse} and of super Ricci flows, \cite{mccanntopp,KopStu18}. In fact there are a few possible characterizations with varying sensitivity to the dimension $\dim M$, formulated either in terms of the $L_0$ distance or in terms of the $L_-$ distance (also in terms of $L_+$, see Subsection \ref{ss:L+}). The $L_0$ distance is a superficially simpler version of $L_-$ and is defined by minimizing the action 
\begin{align}
    \mathcal{L}_0(\gamma) = \int_s^t \left( |\dot\gamma|_r^2 + S_r(\gamma_r)\right) dr\,.
\end{align}
Its properties are related to the monotonicity of the $\mathcal{F}$-functional (as the $L_-$ is related to the $\mathcal{W}$-entropy). Moreover, while $L_-$ is used to detect properties of shrinking solitons (hence more suitable for the study of singularities), the $L_0$ is designed for steady solitons, and $L_+$ for expanding ones. We report here the $L_0$ and $L_-$-based characterizations. In the first case the results and the proofs are notationally lighter and it is easier to grasp the general ideas, but the second case might be more relevant for applications (as it is associated to the monotonicity of the parabolic scale-invariant functional $\cW$).

\begin{note}
All nonstandard notation in the following two Theorems (\ref{thm:L0charact}, \ref{thm:L-charact}) is defined in the preliminary Section \ref{s:prelims}.
\end{note}
\begin{thm}\label{thm:L0charact}
    Given $(M,g_t)_{t\in I}$ a smooth time-dependent Riemannian manifold, the following are equivalent:
    \begin{enumerate}
        \item $\mathcal{D}\geq 0$;
        \item \textbf{Bochner inequality}: for every $v:M\to\mathbb{\R{}}$ smooth and times $[s,t] \subseteq I$,
        \begin{align} \label{eq:BochnerL0}
            (\partial_t - \Delta_t) \left(|\nabla P_{t,s}v|_t^2 - 2\Delta_t P_{t,s}v  -S_t\right) \leq 0\,;\tag{B\({}_{L_0}\)}
        \end{align}
        \item \textbf{Gradient estimate}: for every $v$ smooth and times $[s,t]\subseteq I$,
        \begin{align} \label{eq:GEL0}
            |\nabla P_{t,s}v|_t^2 - 2\Delta_t P_{t,s}v - S_t \leq P_{t,s}(|\nabla v|_s^2 - 2\Delta_s v - S_s)\,; \tag{GE\({}_{L_0}\)}
        \end{align}
        \item \textbf{Wasserstein contraction}: for every probability measures $\mu,\nu\in \cP(M)$, times $[s, t] \subseteq I$, and every flow time $0 < h < \sup I - t$
        \begin{align} \label{eq:WCL0}
            W_{L_0^{s,t}}(\hat{P}_{s+h,s}\mu,\hat{P}_{t+h,t}\nu) \leq W_{L_0^{s+h,t+h}}(\mu,\nu)\,, \tag{WC\({}_{L_0}\)}
        \end{align}
        where $W_{L_0^{s,t}}$ is the Wasserstein distance with respect to the cost induced by $\mathcal{L}_0$;
        \item \textbf{Entropy convexity}: for every $\mu_s,\mu_t\in P(M)$ and times $[s,t] \subseteq I$, there is a $W_{L_0^{s,t}}$-geodesic $(\mu_r)_{r\in[s,t]}$ (equivalently, for all such curves) such that \noeqref{eq:ECL0}
        \begin{align}\label{eq:ECL0}
            r\mapsto \mathcal{E}_{r} (\mu_r) - \mathcal{E}_{s} (\mu_s) - W_{L_0^{s,r}}(\mu_s,\mu_r) \ \ \text{is convex;}\tag{EC\({}_{L_0}\)}
        \end{align}
        \item \textbf{$\mathrm{EVI}$ characterization}: for every $\mu,\nu \in \cP(M)$ and times $[s,t]\subseteq I$, the heat flow $\hat P_{t,s}\mu$ satisfies the $\mathrm{EVI}$s
        \begin{align}
            -\partial_{a^-}^- W_{L_0^{s,a}}(\mu,\hat{P}_{t,a}\nu) & \leq \frac{1}{a-s}\left(\cE(\mu) - \cE(\hat{P}_{t,a}\nu) + W_{L_0^{s,a}}(\mu, \hat{P}_{t,a}\nu)\right) \,, \qquad \forall a \in (s, t]\,.\\
            -\partial_{b^-}^-W_{L_0^{b,t}}( \hat{P}_{s,b}\mu, \nu) & \leq \frac{1}{t - b}\left(\cE(\nu) - \cE(\hat{P}_{s,b}\mu) - W_{L_0^{b, t}}( \hat{P}_{s,b}\mu, \nu)\right) \,, \qquad\forall b \in (\inf I,s]\,.\label{eq:EVIL0}\tag{EVI\({}_{L_0}\)}
        \end{align}
        where $\partial_{t^-}^-$ is the Dini derivative, see \ref{Dini-def}. 
    \end{enumerate}
\end{thm}
In the case of the $L_-$ characterization it is more natural to use a backwards time variable.
\begin{thm}\label{thm:L-charact}
    Given $(M^n,g_\tau)_{\tau\in I}$ a smooth time-dependent Riemannian manifold of dimension $\dim M = n$ and $I \subseteq (0,\infty)$, the following are equivalent:
    \begin{enumerate}
        \item $\mathcal{D}\geq 0$;
        \item \textbf{Bochner inequality}: for every $v:M\to\mathbb{\R{}}$ smooth and times $[\sigma,\tau] \subseteq I$,
        \begin{align}\label{eq:BochnerL-}
            (-\partial_\sigma - \Delta)\left( |\nabla P_{\sigma,\tau}v|_\sigma ^2 - 2\sigma\Delta_\sigma P_{\sigma,\tau}v - \sigma^2 S_\sigma + \frac{n}{2}\sigma \right) \leq 0\,;\tag{B\({}_{L_-}\)}
        \end{align}
        \item \textbf{Gradient estimate}: for every $v$ smooth and times $[\sigma,\tau] \subseteq I$,
        \begin{align}\label{eq:GEL-}
            |\nabla P_{\sigma,\tau}v|_\sigma^2-2\sigma \Delta_\sigma P_{\sigma,\tau}v-\sigma^2S_{\sigma} \leq P_{\sigma,\tau}\left(|\nabla v|_\tau ^2 - 2\tau\Delta_\tau v - \tau ^2S_{\tau}\right) + \frac{n}{2}(\tau-\sigma)\,; \tag{GE\({}_{L_-}\)}
        \end{align}
        \item \textbf{Wasserstein contraction}: for every probability measures $\mu,\nu \in \cP(M)$, times $[\sigma,\tau] \subseteq I$, and flow factor $1 < \alpha < \sup I/\tau$,
        \begin{align} \label{eq:WCL-}
            W_{\tilde L_-^{\alpha\sigma,\alpha\tau}}(\hat P_{\sigma,\alpha\sigma}\mu, \hat P_{\tau,\alpha\tau}\nu) \leq W_{\tilde L_-^{\sigma,\tau}}(\mu,\nu) + (\sqrt{\tau}-\sqrt{\sigma})^2\frac{n}{2}(\alpha-1)\,, \tag{WC\({}_{L_-}\)}
        \end{align}
        where $W_{\tilde L_-^{\sigma,\tau}}$ is the Wasserstein distance with respect to the cost induced by $(\sqrt\tau-\sqrt\sigma)\cL_-$;
        \item \textbf{Entropy convexity}: 
        for every $\mu_\sigma,\mu_\tau\in \cP(M)$ and times $[\sigma,\tau] \subseteq I$, there is a $W_{L_-^{\sigma,\tau}}$-geodesic $(\mu_\eta)_{\eta\in[\sigma,\tau]}$ (equivalently, for all such curves) such that 
        \begin{align}
            \eta\mapsto \cE_\eta(\mu_\eta) - \cE_\sigma(\mu_\sigma) + \frac{1}{\sqrt{\eta}}W_{L_-^{\sigma,\eta}}(\mu_{\sigma},\mu_{\eta}) + \frac{n}{2}\ln(\eta/\sigma)\ \ \text{is convex in the variable $\eta^{-1/2}$} \tag{EC\({}_{L_-}\) }\label{eq:ECL-}
        \end{align}    
        in the sense of \eqref{eq:L-minus-entropy-convexity};
        \item \textbf{$\mathrm{EVI}$ characterization}: for every $\mu,\nu \in \cP(M)$ and times $[\sigma,\tau]\subseteq I$, the heat flow $\hat P_{\sigma,\tau}\mu$ satisfies the $\mathrm{EVI}$s
        \begin{align}
            \eta\cdot\partial_{\eta^+}^+ W_{L_-^{\eta,\tau}}(\hat{P}_{\sigma,\eta}[\nu], \mu) &\leq \frac{1}{2(\eta^{-1/2} - \tau^{-1/2})} \left[\cE(\mu) -\cE(\hat{P}_{\sigma,\eta}[\nu]) + \frac{W_{L_-^{\eta,\tau}}(\hat{P}_{\sigma,\eta}[\nu], \mu)}{\sqrt{\tau}} \right] \\
            &\qquad+\frac{n\ln(\tau/\eta)}{4(\eta^{-1/2} - \tau^{-1/2})} - \frac{n\sqrt{\eta}}{2}\,, \qquad \forall\eta \in [\sigma,\tau)\,.\\
            \varsigma\cdot\partial_{\varsigma^+}^+ W_{L_-^{\sigma,\varsigma}}(\nu,\hat{P}_{\tau,\varsigma}[\mu]) &\leq \frac{1}{2(\sigma^{-1/2} - \varsigma^{-1/2})} \left[\cE(\nu) - \cE(\hat{P}_{\tau,\varsigma}[\mu])  - \frac{W_{L_-^{\sigma,\varsigma}}(\nu,\hat{P}_{\tau,\varsigma}[\mu])}{\sqrt{\sigma}}\right]\\
            &\qquad+ \frac{n\ln(\sigma/\varsigma)}{4(\sigma^{-1/2} - \varsigma^{-1/2})} + \frac{n\sqrt{\varsigma}}{2}\,, \qquad \forall \varsigma \in [\tau,\sup I)\,.\label{eq:EVIL-}\tag{EVI\({}_{L_-}\)}\\
    \end{align}
    \end{enumerate}
\end{thm}

We now briefly comment the characterization.
\begin{enumerate}
    \item The $\mathcal{D}$ condition was already discussed in Subsection \ref{ss:D-condition}.
    \item The Bochner formula \eqref{eq:BochnerL0} was already stated in \cite[Lem.~6.88]{rf1}, $\alpha=0$. The $L_-$-adapted Bochner formula \eqref{eq:BochnerL-} appears to be new even for solutions to the Ricci flow, and is an equality for shrinking solitons for an appropriate choice of $v$. They are also reminiscent of Perelman's Harnack inequality, \cite[Eq.~9.1]{perelman2002entropy}. However, the latter is for solutions of the \textit{adjoint} heat flow, while the one we consider is for the {\it forward} heat flow. One could use Perelman's for the characterization, but we are interested in \eqref{eq:BochnerL0} and \eqref{eq:BochnerL-} because they are related to gradient estimates for the forward heat equation and hence to Wasserstein contraction. It seems plausible that the Bochner inequalities for the forward and adjoint heat flows are equivalent by some formal (without using tensor calculus) identity exploiting duality, but we have not explored that direction here.

    \item The gradient estimates seem to be new, even for the Ricci flow. Note that the Laplacian term is linear (can be substituted with $-v$) and in particular we can estimate the commutator of the laplacian and the semigroup in terms of the commutator of the Hamiltonian and the semigroup. For example \eqref{eq:GEL0} can be restated as
    \[|P_{t,s}\Delta v - \Delta P_{t,s}v| \leq P_{t,s}(|\nabla v|^2 -S_s)  - (|\nabla P_{t,s}v|^2 - S_t). \]
    Taking $\lambda v$, $\lambda \gg1$, we recover the super Ricci flow gradient estimate. Taking $v=0$ instead we get that $P_{t,s}S_s \leq S_t$ and by maximum principle $\inf S_t \geq \inf P_{t,s}S_s \geq  \inf S_s$, which also follows directly from the maximum principle applied to the evolution inequality for $S$, which is the first term of the $\mathcal{D}$ condition.

    Similarly, the $L_-$ version gives $\tau^2P_{\sigma,\tau}S_\tau \leq \sigma^2 S_\sigma + \frac{n}{2}(\tau-\sigma)$, which yields blow-up in finite time when we have $\inf S_\tau>0$, see Remark \ref{GEL-and-max-princ}.

    In Section \ref{s:FandW} we will see that \eqref{eq:GEL0} directly implies monotonicity of $\cF$, and \eqref{eq:GEL-} of $\cW$.

    Finally, recall that (or see \cite{mccanntopp}) super Ricci flows can also be characterized by saying that the Lipschitz constant is preserved along the heat flow. We have a similar characterization, phrased as the preservation of a Hamilton-Jacobi inequality, see Lemmas \ref{lem:heat-flow-preserves-HJ}, \ref{prop:heat-flow-preserves-L_-HJ}.
    \item The Wasserstein contractions \eqref{eq:WCL0}, \eqref{eq:WCL-} are already present in the literature in \cite{Lot09,Top09} respectively, where they are proven in a Ricci flow background (see also \cite{kuwada-philipowski}, \cite{Cheng15} for stochastics proofs using coupling). Such approaches also only use $\mathcal{D}\geq0$, but the reverse implication seems a bit delicate. Our proof is entirely different, and is obtained from the gradient estimate, in the spirit of ``Kuwada duality'' \cite{kuwadaduality}.
    
    Taking an appropriate limit $s\to t$ or $\sigma\to \tau$ (i.e. using the usual distance in place of $\mathcal{L}_0$) we recover the contraction for super Ricci flow from Theorem \ref{mccanntopping}. The estimate \eqref{eq:WCL-} also recovers the space-time dimensional contraction in a static manifold with nonnegative Ricci, see Remark \ref{rmk:Kuwada-contraction}. Finally, note that we are able to compare pairs of different time slices, but the pairs must be somehow related (additively for \eqref{eq:WCL0}, multiplicatively for \eqref{eq:WCL-}). For statements about unrelated times, see Corollary \ref{cor:WC-with-time-translations} and Theorem \ref{thm:dimL0charact}.
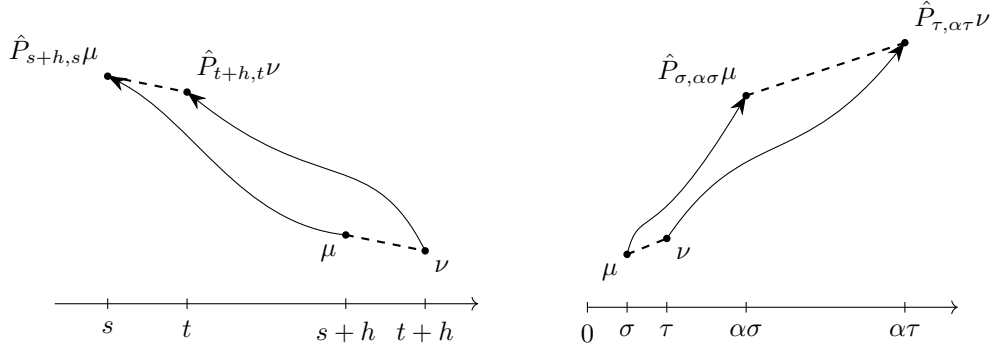
\begin{figure}[H]
\centering
\begin{tikzpicture}[scale=0.7]

  \coordinate (Pmu) at (0,3.3);
  \coordinate (Pnu) at (1.5,3);
  \coordinate (mu) at (4.5,0.3);
  \coordinate (nu) at (6,0);

  \fill (mu) circle (2pt) node[below left] {$\mu$};
  \fill (nu) circle (2pt) node[below right] {$\nu$};
  \fill (Pmu) circle (2pt) node[above left] {$\hat P_{s+h,s}\mu$};
  \fill (Pnu) circle (2pt) node[above right] {$\hat P_{t+h,t}\nu$};

  \draw[thick, dashed] (Pmu) -- (Pnu);

  \draw[<-, >={Stealth[scale=1.5]}] (Pnu) .. controls (4,1) and (5,2) .. (nu);
  \draw[<-, >={Stealth[scale=1.5]}] (Pmu) .. controls (1.6,2.5) and (2.5,0.5) .. (mu);
  \draw[thick, dashed] (mu) -- (nu);
  
  \draw[->] (-1, -1) -- (7, -1);
  \foreach \x/\xtext in {0/{$s$}, 1.5/{$t$}, 4.5/{$s+h$}, 6/{$t+h$}}
    \draw (\x, -1.1) -- (\x, -0.9) node[below=5pt] {\xtext};

\end{tikzpicture}
\hspace{1cm}
\begin{tikzpicture}[scale=0.7]
  \coordinate (Pnu) at (6,4);
  \coordinate (Pmu) at (3,3);
  \coordinate (mu) at (0.75,0);
  \coordinate (nu) at (1.5,0.3);

  \fill (mu) circle (2pt) node[below left] {$\mu$};
  \fill (nu) circle (2pt) node[below right] {$\nu$};
  \fill (Pmu) circle (2pt) node[above left] {$\hat P_{\sigma,\alpha\sigma}\mu$};
  \fill (Pnu) circle (2pt) node[above right] {$\hat P_{\tau,\alpha\tau}\nu$};
  
  \draw[thick, dashed] (Pmu) -- (Pnu);

  \draw[->, >={Stealth[scale=1.5]}] (nu) .. controls (3,2.5) and (4,1.5) .. (Pnu);
  \draw[->, >={Stealth[scale=1.5]}] (mu) .. controls (1,1) and (1.2,0) .. (Pmu);
  \draw[thick, dashed] (mu) -- (nu);
  
  \draw[->] (0, -1) -- (7, -1);
  \foreach \x/\xtext in {0/{$0$}, 0.75/{$\sigma$}, 1.5/{$\tau$}, 3/{$\alpha\sigma$}, 6/{$\alpha\tau$}}
  \draw (\x, -1.1) -- (\x, -0.9) node[below=5pt] {\xtext};
\end{tikzpicture}
\caption{The \ref{eq:WCL0} and \ref{eq:WCL-} contractions}
\end{figure}

    \item The entropy convexities were stated in \cite[Prop.~11,~16]{Lot09} in a Ricci flow background. Our proof is different and completely avoids $L_{0/\pm}$-geodesics and Jacobi fields. Instead, we follow the pattern of \cite{AmbGigSav15,DanSav08}, with some mild complication stemming from our time-dependent setting. In the static case, entropy convexity plays a central role in the definition of Ricci lower bounds on non smooth spaces. Although we restrict our attention to the smooth setting, the methods we use to prove entropy convexity were originally designed for $\mathsf{RCD}$ spaces, and we hope to explore similar extensions to a suitable class of singular metric flows in future work.

    \item The $\mathrm{EVI}$s \eqref{eq:EVIL0} and \eqref{eq:EVIL-} are new also for Ricci flows and are motivated by gradient flow theory. In the theory of $\mathsf{RCD}$ spaces they were studied in \cite{AGS2,AmbGigSav15}, and in a SRF setting in \cite[Sec.~6]{KopStu18}. These estimates are very strong, firstly because they imply entropy convexity (Prop \ref{prop:L0-entropy-cvx}) in the spirit of \cite{AGS2,AmbGigSav15,DanSav08}; secondly they are natural one-sided versions of the Wasserstein contractions, and summing the $\mathrm{EVI}$s one recovers the Wasserstein contractions (Rmk \ref{rmk:EVItoWC}). On the other hand they can be obtained from the gradient estimates, but in this case the proof is more involved. Note that in the static theory the existence of $\mathrm{EVI}$ gradient flows of the entropy is stronger (it can be used as definition of $\mathsf{RCD}$ space) than the entropy convexity (which defines $\mathsf{CD}$).

    In a related direction, in \cite{KopStu18,Kop18} the adjoint heat flow is identified as a time-dependent gradient flow of the entropy functional.

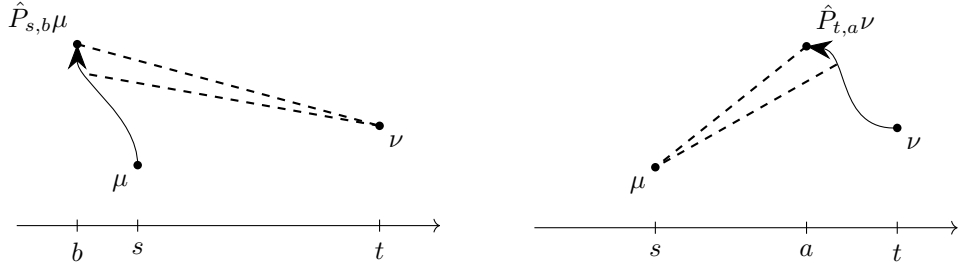
\begin{figure}[H]
\centering
\begin{tikzpicture}[scale=0.8]

  \coordinate (P) at (3,2);
  \coordinate (mu) at (0,0);
  \coordinate (nu) at (4,0.65);
  \coordinate (Pmu-s) at (-1,2);
    
  \draw[->, >={Stealth[scale=2]}] (mu) .. controls (0,0.8) and (-1,1.4) .. (Pmu-s);
  \draw[thick, dashed] (Pmu-s) -- (nu);
  \draw[thick, dashed] (-0.8,1.5) -- (nu);

  \fill (mu) circle (2pt) node[below left] {$\mu$};
  \fill (nu) circle (2pt) node[below right] {$\nu$};
  \fill (Pmu-s) circle (2pt) node[above left] {$\hat{P}_{s,b}\mu$};

  \draw[->] (-2, -1) -- (5, -1);
  \foreach \x/\xtext in {-1/{$b$}, 0/{$s$}, 4/{$t$}}
    \draw (\x, -1.1) -- (\x, -0.9) node[below=5pt] {\xtext};

\end{tikzpicture}
\hspace{1cm}
\begin{tikzpicture}[scale=0.8]

  \coordinate (mu) at (0,0);
  \coordinate (nu) at (4,0.65);
  \coordinate (Pnu-a) at (2.5,2);
    
  \draw[->, >={Stealth[scale=2]}] (nu) .. controls (3,0.6) and (3.2,1.9) .. (Pnu-a);
  \draw[thick, dashed] (Pnu-a) -- (mu);
  \draw[thick, dashed] (3,1.7) -- (mu);

  \fill (mu) circle (2pt) node[below left] {$\mu$};
  \fill (nu) circle (2pt) node[below right] {$\nu$};
  \fill (Pnu-a) circle (2pt) node[above right] {$\hat{P}_{t,a}\nu$};

  \draw[->] (-2, -1) -- (5, -1);
  \foreach \x/\xtext in {0/{$s$}, 2.5/{$a$}, 4/{$t$}}
    \draw (\x, -1.1) -- (\x, -0.9) node[below=5pt] {\xtext};
\end{tikzpicture}
\caption{The backward and the forward \eqref{eq:EVIL0}s.}
\end{figure}

\end{enumerate}

Apart from the two main Theorems, there are a handful of other results that are of comparable interest but have not been fully spelled out in this section for expository reasons. We informally list and outline them here; careful statements can be found in the referenced Theorems:

\begin{itemize}
\item {\bf Curvature-dimension type conditions}: In Theorem \ref{thm:L-charact} the dimension plays a role, but is fixed a priori, while it would be preferable (as in the synthetic Ricci lower bound $\mathsf{CD}(K,N)$ theory) to have conditions that characterize $\{\cD\geq0, \ \dim M\leq N\}$. Such a combination is necessary in the $\mathsf{CD}(K,N)$ setting for it to be preserved under an appropriate convergence of metric-measure spaces, where the topological dimension of the space may drop in the limit. Such convergence issues are beyond the scope of this paper, but in Theorem \ref{thm:dimL0charact} we do provide a list of conditions equivalent to $\{\cD\geq0, \ \dim M\leq N\}$, parallel to those of Theorems \ref{thm:L0charact}, \ref{thm:L-charact}. 
\item {\bf Conditions for the $L_+$ distance}: There is a strict analog of Theorem \ref{thm:L-charact} for the $L_+$ distance in Section \ref{ss:L+}, Theorem \ref{thm:L+charact}. We recall here that this notion of distance is induced by the action functional for curves
\begin{equation}
    \cL_+(\gamma) = \int_s^t \sqrt{r}\big(|\dot\gamma|^2_r + S_r(\gamma_r)\big)\,dr\,,
\end{equation}
and was introduced in \cite{Feldman-Ilmanen-Ni} for the purpose of studying expanding solitons. Correspondingly, it is related to parabolic rescaling with the time-origin chosen {\it previous} to the time of consideration.
\item {\bf Monotonicity of Perelman's $\cF$ and $\cW$ functionals}: In \cite{perelman2002entropy} Perelman discovered the monotonicity of the functional $\cF$, and perhaps more importantly of the scale-invariant version $\cW$, for solutions to the adjoint heat equation along a Ricci flow. These properties were a crucial analytic tool in the resolution of the Poincar\'e conjecture. More recently, the monotonicity of these functionals has been exploited in the study of partial regularity and singular structure of Ricci flows \cite{BamlerEntropy,BamlerStructure}. Note that some of these results were in fact extended to flows satisfying $\cD\geq0$ in \cite{kunikawa2023almost}.
The monotonicity of $\cW$ was re-derived in \cite{Top09} using \eqref{eq:WCL-}. We provide a new and quite direct proof of these monotonicities in Theorem \ref{thm:F-W-monotonicity}, using the gradient estimates \eqref{eq:GEL0} resp. \eqref{eq:GEL-} along with duality of forward/backward heat flows for functions/measures.
\item {\bf Weighted Riemannian manifolds}: In Theorems \ref{thm:L0charact} and \ref{thm:L-charact}, the input can be seen as a (smooth) time-dependent metric space $(M,d_{g_t})_{t\in I}$. However if one is to believe that these conditions should mirror aspects of (static) $\mathsf{CD}(K,N)$ theory in the dynamic setting, we should allow for the additional data of a measure other than $dV_{g_t}$. In Sections \ref{subsec:infinity-weighted} and \ref{subsec:N-weighted} we introduce the $\cD_{U,\infty}$ and the $\cD_{U,N}$ conditions --corresponding roughly to $\mathsf{CD}(0,\infty)$ resp. $\mathsf{CD}(0,N)$-- and provide analogs of theorems stated in the unweighted case. The $\cD_{U,\infty}$ condition is satisfied by the weighted Ricci flow, introduced in \cite{Lot09}, see Example \ref{ex:wRF} and more generally by the generalized Ricci flow (gauge-changed with a dilation flow)\footnote{This particular evolution was brought to the authors' attention by Jeff Streets soon after the first version of this paper was posted.}, considered in \cite{Streets23, KopferStreetsOTGRF}, see Example \ref{ex:gRF}. The $\cD_{U,N}$ condition is satisfied by the dimensional weighted Ricci flow, Example \ref{ex:NwRF}.
\end{itemize}

\subsection{Heuristic motivation: Perelman's infinite dimensional manifold}
In this subsection we venture an interpretation regarding $\mathcal{D}\geq0$. We will argue that while the super Ricci flow condition can be seen as the nonnegativity of the Ricci curvature of the space-time in the space-like directions, $\mathcal{D}\geq0$  corresponds to the nonnegative Ricci curvature of the space-time in every direction. This justifies that while in the SRF case we can only study optimal transport on a fixed time-slice, under $\cD\geq0$ we are able to transport mass between different time slices.

In \cite[Ch.~6]{perelman2002entropy}, Perelman gives a formal proof of the monotonicity of the reduced volume, considering the space-time manifold given by a Ricci flow and applying classical Riemannian geometry to it.

More precisely, given an $n$-dimensional time-dependent Riemannian manifold $(M,g_\tau)_{\tau\in(\tau_1,\tau_2)}$, $0\leq\tau_1<\tau_2<\infty$, he considered the following $n+N+1$-dimensional manifold:
\begin{align}
    &\Tilde{M} = M\times(\tau_1,\tau_2)\times \mathbb{S}^N \\
    &\Tilde{g}=(g_\tau)_{ij}dx^idx^j + \left(S+\frac{N}{2\tau}\right)d\tau^2 + \tau (g_{\mathbb{S}^N})_{\alpha\beta} dy^\alpha dy^\beta
\end{align}
where ${x_i}$ is a coordinate system on $M$ and ${y_\alpha}$ on $\mathbb{S}^N$, the latter endowed with the metric $g_{\mathbb{S}^N}$ of constant sectional curvature $\frac{1}{2N}$.

Perelman argues that as $N\to\infty$, if $(M,g_\tau)$ is a time-reversed Ricci flow, the Ricci curvature of $(\Tilde{M},\Tilde{g})$ vanishes. Moreover the geodesics on $\Tilde{M}$ behave asymptotically as the $L_-$ geodesics and the reduced volume $V$ can be seen as the volume of the balls in $(\Tilde{M},\Tilde{g})$. Therefore, the monotonicity of $V$ can be explained by the Bishop--Gromov comparison theorem. In a similar direction, in \cite{bustamante2025deriving} it is shown that Colding's monotonicity \cite{ColdingMonoton} applied to $\widetilde M$ recovers $\cW$ monotonicity; and in \cite{CabRivHasl2020} path space analysis on $\widetilde M$ is carried out. However, for the Bishop--Gromov theorem the lower bound on $\Tilde{\Ric}$ is enough, and we observe that this lower bound is equivalent to $\mathcal{D}\geq0$.

\begin{prop}\label{infinitedimmfd}
    Given $(M,g_\tau)$ as above, the following are equivalent:
    \begin{enumerate}
        \item $(M,g_\tau)$ satisfies $\mathcal{D}\geq0$;
        \item The space-time manifold $(\Tilde{M},\Tilde{g})$ has $\Ric \geq O\left(1/N\right)$ for $N\to\infty$.
    \end{enumerate}
\end{prop}

\begin{proof}
    Computations show that the curvature of $(\Tilde{M},\Tilde{g})$ is
    \begin{align}
        &\Tilde{R}_{ijkl} = R_{ijkl} + O(1/N) \\
        &\Tilde{R}_{ijk\tau} = -\left(\nabla_{\partial_i}\mathcal{S}_{jk} - \nabla_{\partial_j}\mathcal{S}_{ik}\right) + O(1/N) \\
        &\Tilde{R}_{i\tau\tau j} = -\frac{1}{2}(\nabla\nabla S)_{ij} -\partial_\tau\mathcal{S}_{ij}-\mathcal{S}_{ik}\mathcal{S}_{kj} -\frac{\cS_{ij}}{2\tau} +O(1/N) \\
        &\Tilde{R}_{\alpha\beta\gamma\delta} = O(1/N^2) \\
        &\Tilde{R}_{i\alpha \beta j} = -\frac{1}{2\tau(S+\frac{N}{2\tau})}\cS_{ij}\tilde g_{\alpha\beta} \\
        &\Tilde{R}_{i\alpha \beta \tau} = \frac{1}{4\tau(S+\frac{N}{2\tau})}\nabla_i S\tilde g_{\alpha\beta} \\
        &\Tilde{R}_{\tau \alpha \beta \tau} = \frac{1}{4\tau\left(\mathrm{S}+\frac{N}{2\tau}\right)}\left(\partial_\tau\mathrm{S}+\frac{\mathrm{S}}{\tau}\right)\tilde g_{\alpha\beta}
        \end{align}
    and tracing we get
    \begin{align}
        &\Tilde{\Ric}_{ij} = \Ric_{ij}-\mathcal{S}_{ij} + O(1/N) \\
        &\Tilde{\Ric}_{i\tau} = \nabla^l\mathcal{S}_{il}-\frac{\nabla_iS}{2}+O(1/N) \\
        &\Tilde{\Ric}_{\tau\tau} = - \frac{\Delta S}{2}-\frac{\partial_\tau S}{2} - |\mathcal{S}|^2 + O(1/N) \\
        &\Tilde{\Ric}_{i\alpha} = \Tilde{\Ric}_{\tau\alpha} = 0 \\
        &\Tilde{\Ric}_{\alpha\beta} = O(1/N)
    \end{align}

    For $V= X + \lambda\partial_\tau + \theta$, $\lambda$ a constant and $X,\theta$ tangent to $M$ and $\mathbb{S}^N$, we have
    \begin{align}
        \Tilde{\Ric}(V,V) &= \Tilde{\Ric}(X,X) + 2\lambda \Tilde{\Ric}(X,\partial_\tau) + \lambda^2\Tilde{\Ric}(\partial_\tau,\partial_\tau) + O(1/N) \\
        &=\Ric(X,X)-\mathcal{S}(X,X) + \lambda \left(2\div\mathcal{S}(X) - \nabla_XS)\right) \\
        & \ \ \ + \lambda^2\left(-\frac{\partial_\tau S}{2} - \frac{\Delta S}{2} - |\mathcal{S}|^2\right) + O(1/N) \\
        &= \frac{\lambda^2}{2}\cD(X/\lambda) + O(1/N)\,.
    \end{align}
\end{proof}

It is interesting to note that the Ricci curvature of $\tilde M$ in each direction has different scaling, splitting $\mathcal{D}$ as in \eqref{eq:Dcondition}.

\begin{rmk}
    A similar phenomenon is observed in Cabezas-Rivas--Topping's canonical shrinking soliton \cite{CabezasRivas-Topping}. In this case one can see that the space-time $M\times(a,b)$ they construct has positive weighted Ricci curvature as $N\to \infty$ if and only if $\mathcal{D}\geq0$.
\end{rmk}

\subsection{Organization of the paper}

The following is a brief roadmap for the structure and logical dependencies of this paper.

\begin{itemize}
\item In Section \ref{s:prelims} we state basic definitions and preliminaries about evolutions of smooth Riemmanian manifolds and associated optimal transport problems, with the more unwieldy proofs delayed to the appendices.
\item In Section \ref{s:L0} and \ref{s:L-} we give the proof of the main theorems, proving the following implications:
\begin{itemize}
    \item $\cD\geq 0 \iff (\mathrm{Bochner})$ in Sec \ref{ss:L0DBoch} (for $L_0$), Sec \ref{ss:L-DBoch} (for $L_-$).
    \item $(\mathrm{Bochner}) \iff (\mathrm{GE})$ in Sec \ref{ss:L0BochGE} (for $L_0$), Sec \ref{ss:L-BochGE} (for $L_-$).
    \item $(\mathrm{GE}) \iff (\mathrm{WC})$ in Sec \ref{ss:GEWCdualityL0} (for $L_0$), Sec \ref{ss:GEWCdualityL-} (for $L_-$).
    \item $(\cE\,\mathrm{convexity})\implies (\mathrm{WC})$ in Sec \ref{sss1:entropyEVIL0} (for $L_0$), Sec \ref{sss1:entropyEVIL-} (for $L_-$).
    \item $(\mathrm{GE}),(\mathrm{WC}) \implies (\mathrm{EVI})\implies (\cE\,\mathrm{convexity})$ in Sec \ref{sss2:entropyEVIL0} (for $L_0$), Sec \ref{sss2:entropyEVIL-} (for $L_-$).
\end{itemize}
\item In Section \ref{s:FandW} we prove the monotonicity of $\cF$ and $\cW$ functionals for solutions to the adjoint heat equation on smooth families of Riemannian manifolds satisfying $(\mathrm{GE})$.
\item In Section \ref{s:dimL0} we first state and prove the equivalence of a list of conditions which characterize $\cD \geq 0$ and $\dim M \leq N$ simultaneously (Sec \ref{ss:nL0Bochner} - \ref{ss:GEWCdualitydimL0}). We also state a dimensionally sensitive $(\cE \text{ convexity})$ and $(\mathrm{EVI})$ written in terms of $L_0$, then prove that these are equivalent to (any of) the $L_-$ conditions of Section \ref{s:L-} (Sec \ref{ss:nL0EC/EVI}).
\item In Section \ref{s:weight} we explain how to extend all dimensionless (Sec \ref{subsec:infinity-weighted}) and dimensionally sensitive (Sec \ref{subsec:N-weighted}) results to smooth evolutions of Riemannian manifolds equipped with a smooth, time-dependent weight.
\end{itemize}

\subsection*{Acknowledgments}
The authors would like to thank Theo Sturm for suggesting the $\cD \geq 0$ condition as a promising direction for further investigation, and Eva Kopfer and Lorenzo Portinale for fruitful discussions. We also thank Jeff Streets for suggesting the generalized Ricci flow as a potential example for the settings considered in this paper, and generously sharing related insights.

M.F. acknowledges funding by the Deutsche Forschungsgemeinschaft (DFG, German Research Foundation) under Germany's Excellence Strategy – EXC-2047/1 – 390685813. E.H. acknowledges support by the Deutsche Forschungsgemeinschaft (DFG) within the CRC 1060, at University of Bonn project number 211504053. A large part of this work was carried out during the HIM Trimester Program: Metric Analysis funded by the Deutsche Forschungsgemeinschaft (DFG) under Germany's Excellence Strategy - EXC-2047/1 - 390685813.

\section{Preliminaries}\label{s:prelims}

\subsection{Time dependent manifolds}\label{subsec:t-dep-mflds}
In this paper we always work with a \textit{closed} (compact without boundary), smooth, time dependent Riemannian manifold $(M,g_t)_{t\in I}$, where $I$ is some interval. Often one can just take $I = [0,T]$, but in some instances translations will play a role. We denote with $\cS$ the time dependent, symmetric 2-tensor and $S$ the (scalar-valued) function given by 
\begin{equation}
\cS_t := -\frac{1}{2}\partial_tg\,, \qquad S_t :=\text{tr}_t \cS_t  = -\frac{\partial_t dV_t}{dV_t}\,.
\end{equation}
Here $dV_t$ is the Riemannian volume form associated to $g_t$. In the case of Ricci flow, $\cS=\Ric$ and $S=R$. In many situations the compactness could be replaced by a uniform bound $S\geq -C$ (e.g. for the finiteness of the $L$-distance). Note that, as for the Ricci flow, if we assume $\mathcal{D}\geq0$, it is enough to assume this bound on $S$ at the initial time.

Sometimes it is convenient to work with a backward time variable, which is always denoted with greek letters, e.g. $\partial_\tau g= 2\mathcal{S}_\tau$. When switching between backwards and forwards time for a given flow, we need to make a choice of reference time $T \in \R{}$ to write 

\begin{equation}
\tau := T - t \in T - I =: \tilde I\,,\qquad (M,g_\tau) := (M,g_t)\,.
\end{equation}
Using both backwards and forwards time is purely a presentation choice, and it would also be possible to work exclusively with one or the other.

A main role is played by the (adjoint) heat semigroup on such time-dependent manifolds. We follow the conventions of \cite[Ch.~3]{KopStu18}. Let us start considering the natural heat equation starting at some $u_s\in C^0(M)$
\begin{align}
    \begin{cases}
        \partial_t u =\Delta_t u \\
        u|_{t=s}=u_s\,.
    \end{cases}
\end{align}
Given $t>s$ there exists a unique solution on $[s,t]\times M$ (in the sense that it solves the equation on $(s,t)\times M$ and converges uniformly to $u_s$ as $t\to s$) and it can be represented by a time dependent semigroup $P_{t,s}u_s:=u_t$, which is defined for all $t\geq s$ (``brings from $s$ to $t$''). In the case $v_s\in C^2(M)$, the equation is solved up to the initial time $t=s$ (in the sense of taking the one-sided right derivative). 

The adjoint heat equation is defined by
\begin{align}
    \begin{cases}
        \partial_s v = -\Delta_s v + S_s v\\
        v|_{s=t}=v_t
    \end{cases}\,.
\end{align}
Its solution (well defined for a given final datum, the equation being backward parabolic) is represented by $P^*_{t,s}v_t:=v_s$ for all $s\leq t$ (``brings from $t$ to $s$''). Note that, as with the heat equation on static spaces, this equation preserves the total mass. Indeed it can be interpreted as the backward heat equation for measures: let $\mu_s :=v_s \,dV_s$, then 
\begin{align}
    \partial_s \mu = -\Delta_s\mu\,,
\end{align}
where $\Delta_s$ is the distributional Laplacian at time $s$. We set $\hat P_{t,s}\mu_t:=\mu_s=P_{t,s}^*v_t \,dV_s=v_s\,dV_s$ and extend to the space of all Borel probability measures $\cP(M)$ by density of a.c. smooth measures in the weak topology. The adjoint heat equation is so-named because the operator $\partial_t+\Delta_t-S_t$ is the space-time adjoint of $\partial_t-\Delta_t$. As a consequence we have the important relation
\begin{align}
    \int_M P_{t,s}u\, d\mu = \int_M u\, d\hat P_{t,s}\mu\,,
\end{align}
and thus the formula for the time evolution in the other variable (with $\Delta$ interpreted distributionally if necessary)
\begin{equation}
    \partial_s P_{t,s}u = -P_{t,s}\Delta_s u\,,\qquad \partial_t \hat{P}_{t,s}\mu = \hat{P}_{t,s} \Delta_t \mu\,.
\end{equation}
When considering backward time $(M,g_\tau)$, we will have the semigroup $P_{\sigma,\tau}$, $\sigma\leq\tau$ ("brings from $\tau$ to $\sigma$") solving
\begin{align}
    \partial_\sigma P_{\sigma,\tau}u_\tau = -\Delta_\sigma P_{\sigma,\tau}u_\tau,
\end{align} 
and $\hat P_{\sigma,\tau}$ $\sigma\leq\tau$ ("brings from $\sigma$ to $\tau$") solving
\begin{align}
    \partial_\tau \hat P_{\sigma,\tau}\mu_\sigma = \Delta_\tau \hat P_{\sigma,\tau}\mu_\sigma.
\end{align} 
The heat equation is closely related to the (Boltzmann-Shannon) entropy functional, which plays a central role in certain parts of this paper. Given a probability measure $\mu$ at a time $t\in I$, we define
\begin{align}
 \cE(\mu)=\begin{cases}
     \int \rho\log\rho \,dV_t &\text{if $\mu=\rho \,dV_t$}\,, \\
     +\infty &\text{otherwise.}
 \end{cases}   
\end{align}

Lastly, we will occasionally want to make statements about limiting difference quotients in time other than standard differentiation. For this we use the following notation for one-sided Dini derivatives.

\begin{equation}\label{Dini-def}
\partial_{t^+}^+ := \limsup_{u\downarrow t}\frac{1}{u - t}\,,\qquad \partial_{t^-}^+ := \limsup_{u\uparrow t}\frac{1}{t - u}\,,\qquad  \partial_{t^-}^- := \liminf_{u\uparrow t}\frac{1}{t - u}\,,\qquad \partial_{t^+}^- := \liminf_{u\downarrow t}\frac{1}{u - t}\,.
\end{equation}
The superscript is dropped to denote a one-sided limit, while the subscript is dropped to denote an honest $\limsup$ or $\liminf$.

\subsection{\texorpdfstring{$L$}{L} distance}
We define now the $L_{m}$ distance, for $m=0,+,-$. For an extensive treatment of the $L_-$ distance, see \cite[Ch.~7-8]{rf1}. We slightly detach from the usual literature, as we add a $1/2$ factor in front of our definition of $L_m$ distance.

\begin{df}\label{df:L0}
    Let $(M,g_t)_{t\in I}$ be a smooth time-dependent Riemannian manifold. For $\gamma:[s,t]\to M$ a smooth curve on $[s,t]\subseteq I$ we define the $\mathcal{L}_0$ energy
    \begin{align}\label{eq:L0action}
        \mathcal{L}_0(\gamma) = \frac{1}{2}\int_s^t |\dot\gamma|_r^2+S_r(\gamma_r) \,dr\,.
    \end{align}
    Now for $(s,x)$, $(t,y)$ we let \begin{align}
        L_{0}^{s,t}(x,y) = \inf\{\mathcal{L}_0(\gamma) \ : \ \text{$\gamma$ with $\gamma_s=x$, $\gamma_t=y$}\}\,.
    \end{align}
    It is useful to define $\tilde L_0^{s,t}(x,y)= (t-s)L_0^{s,t}(x,y)$, so that $\tilde L_0^{s,t}(x,y)\to \frac{1}{2}d^2_s(x,y)$ as $t\to s$, just by adapting \cite[Lem.~7.13]{rf1}.
\end{df}

\begin{df}\label{df:L-}
    Let $(M,g_\tau)_{\tau\in I}$ be a smooth time-dependent Riemannian manifold. For $\gamma:[\sigma,\tau]\to M$ a smooth curve on $[\sigma,\tau]\subseteq I \subseteq [0,\infty)$ we define the $\mathcal{L}_-$ energy
    \begin{align}\label{eq:L-action}
        \mathcal{L}_-(\gamma) = \frac{1}{2}\int_\sigma^\tau \sqrt{\eta}(|\dot\gamma|_\eta^2+S_\eta(\gamma_\eta)) \,d\eta.
    \end{align}
    Now for $(\sigma,x)$, $(\tau,y)$ we let \begin{align}
        L_-^{\sigma,\tau}(x,y) = \inf\{\mathcal{L}_-(\gamma) \ : \ \text{$\gamma$ with $\gamma_\sigma=x$, $\gamma_\tau=y$}.\}
    \end{align}
    We let $\tilde L_-^{\sigma,\tau}(x,y)=(\sqrt \tau-\sqrt \sigma)L_-^{\sigma,\tau}(x,y)$, so that $\tilde L_-^{\sigma,\tau}(x,y)\to \frac{1}{4}d^2_\sigma(x,y)$ as $\tau\to \sigma$, see \cite[Lem.~7.13]{rf1}.
\end{df}

\begin{df}\label{df:L+}
    On $(M,g_t)_{t\in I}$ as in Definition \ref{df:L0}, for $\gamma:[s,t]\to M$ a smooth curve on $[s,t]\subseteq I\subseteq [0,\infty)$ we define
    \begin{align}
        \mathcal{L}_+(\gamma) = \frac{1}{2}\int_s^t \sqrt r ( |\dot\gamma|_r^2+S_r(\gamma_r))\, dr.
    \end{align}
    Now for $(s,x)$, $(t,y)$ we let \begin{align}
        L_{+}^{s,t}(x,y) = \inf\{\mathcal{L}_+(\gamma) \ : \ \text{$\gamma$ with $\gamma_s=x$, $\gamma_t=y$}.\}
    \end{align}
    Again, for $\tilde L_+^{s,t}(x,y)=(\sqrt t-\sqrt s)L_+^{s,t}(x,y)$, it holds $\tilde L_+^{s,t}(x,y)\to \frac{1}{4}d^2_s(x,y)$ as $t\to s$.
\end{df}

In \cite[Ch.~7]{rf1} one can find for the case of $L_-$, the existence of minimizing geodesics, Euler Lagrange equation, and the Lipschitz continuity of $L_-$ (and much more). We will see these features in subsection \ref{subsec:LagrangianProperties} for general Lagrangians. 

\begin{rmk}
    The $L_-$ distance was first considered on a Ricci flow background by Perelman \cite{perelman2002entropy}, $L_0$ appears in \cite{Lot09}. Similar expressions were used in \cite{LiYau86} for the study of the parabolic Schr\"odinger equation, and even before in \cite{Agmonlectures,CarmonaSimon} for the elliptic Schr\"odinger equation.
\end{rmk}

\subsection{Optimal transport}
\subsubsection{Basics on optimal transports}
Here we recall some facts on optimal transport, for further reference see for example \cite{AmbBruSem21}, \cite{Vil09}. We work with $M$ a compact manifold without boundary, and consider $c:M\times M\to [-C,\infty)$ a cost function (in our case it will be $L_m^{s,t}$ for some fixed times $s,t$). Given two probability measures we are interested in the Kantorovich problem
\begin{align}
    K_c(\mu,\nu)=\inf \left \{ \int_{M\times M} c(x,y)\,d\pi(x,y) \ : \ \pi \in \mathrm{Cpl}(\mu,\nu)\right \}\,,
\end{align}
where $\mathrm{Cpl}(\mu,\nu)=\{\pi\in \cP(M\times M) : (p_1)_\#\pi=\mu, (p_2)_\#\pi=\nu\}$. Just assuming $c$ lower semicontinuous, there exists an optimal plan $\pi$. We will often make use of the dual problem:
\begin{align}
    K_c(\mu,\nu) 
    &= \sup\left\{\int_M\phi \,d\mu + \int \psi \,d\nu \ : \ \phi\in L^1(d\mu),\, \psi\in L^1(d\nu),\, \phi\oplus\psi \leq c \right\} \label{eq:dualpbm1} \\
    &= \sup\left\{\int_M\phi \,d\mu + \int \phi^c\,d\nu \ : \ \phi\in L^1(d\mu) \right\} \label{eq:dualpbm2}
    \\
    &= \sup\left\{\int_M\psi^c \,d\mu + \int \psi \,d\nu \ : \ \psi\in L^1(d\nu) \right\} \label{eq:dualpbm3}
\end{align}
where $\phi^c(y)=\inf_x\{c(x,y)-\phi(x)\}$, $\psi^c(x)=\inf_y\{c(x,y)-\psi(y)\}$ are the $c$-transforms. When $c:M\times M\to\R{}$ is Lipschitz, then the supremum is attained by a pair $\phi,\phi^c\in \mathrm{Lip}(M)$, see \cite[Cor.~3.19]{AmbBruSem21}. The functions $(-\phi,\phi^c)$ will be referred to as Kantorovich potentials for the problem \eqref{eq:dualpbm3}--note the unusual sign convention here which was chosen to make some later statements cleaner. Note that in our case $c = L_m^{s,t}$ is Lipschitz on $M\times M$. 

\subsubsection{Regularity of $L$ and Hopf--Lax semigroup}\label{subsec:LagrangianProperties}
We consider the case where $c$ is induced by a Lagrangian action. More precisely, for fixed $a\leq s<t \leq b$, we assume that the cost is given by 
\begin{align}
    c(x,y)=L^{s,t}(x,y)=\inf\{\mathcal{L}(\gamma) \ : \ \gamma_s=x,\gamma_t=y\}, \ \ \ \text{where $\mathcal{L}(\gamma)=\int_s^t L(\dot\gamma_r,\gamma_r,r)dr$.}
\end{align}
Here $L:TM\times [a,b]\to\mathbb{R}$ is the Lagrangian action, which satisfies certain properties stated below. We work in a manifold $M$ endowed with a smooth metric $g$: this serves as support for some assumptions and regularity results (in the time-dependent case, one can choose the distance induced by the metric $g_t$, for any of the $t\in I$).

An important role in this work will be played by the Hamiltonian $H:T^*M\times [a,b]\to\mathbb{R}$, defined as 
\begin{align}\label{eq:DefHamiltonian}
    H(w,x,t)=L^*(w,x,t)=\sup_{v\in T_xM}\left\{ \La v, w\Ra - L(v,x,t) \right\}
\end{align} 
i.e. as the Legendre transform of the Lagrangian $L$ in $v$. 

The following assumptions on $L$ are from \cite[Thm.~6.2.5]{CanSin04}, see also \cite[Ch.~7]{Vil09} for related discussion.
\paragraph{Assumptions:}\label{par:L-assumptions}
\begin{itemize}
    \item $L(v,p,t)$ is $C^1$ in all variables.
    \item $L$ is strictly convex in $v$.
    \item $L$ is superlinear in $v$ in the following uniform sense: There exist $c_0\geq0$, $\theta:\mathbb{R} ^+\to\R{}$ with $\lim_{\lambda\to\infty}\frac{\theta(\lambda)}{\lambda}=\infty$ and
    \begin{align}
        L(v,x,t)\geq \theta(|v|)-c_0.
    \end{align}
    \item There exists $C>0$ s.t. 
    \begin{align}
        |d_x L(v,x,t)|+|\partial_t L(v,x,t)|\leq C\theta(|v|).
    \end{align}
    \item For all $M>0$ there exists $K_M>0$ with $\theta(q+m)\leq K_M(1+\theta(q))$ for all $m\in[0,M]$, $q\geq0$.
\end{itemize}

\begin{rmk}\label{rmk:LmaregoodLagr}
    One can see that $L_m$, $m=0,-,+$ satisfy the assumptions above, provided that in the case $m=-,+$ we work on a time interval away from $\sigma=0$, resp. $s = 0$. Therefore in these cases we will work on $I\subset (0,\infty)$; some statements can be extended up to $\sigma=0$ or $s= 0$ by continuity.
\end{rmk}

Under these assumptions, we have the following properties, which we will restate more precisely and prove, or give the reference, in Appendix \ref{Appe:Lagrangian}.
\paragraph{Properties of $L^{s,t}$:}
\begin{itemize}
    \item For any $(x,s)$, $(t,y)$ there exists a $C^1$-minimizer, solving the Euler-Lagrange equation. 
    \item As $t\to s$, $L^{s,t}(x,y)\to+\infty$ if $x\neq y$, and $L^{s,t}(x,x)\to0$.
    \item $(y,t)\mapsto L^{s,t}(x,y)$ is locally Lipschitz.
    \item The Legendre transform of $L$ is well-defined. Moreover the map $T_pM\ni v\mapsto w \in T_p^*M$ associating to $v$ its maximizer in \eqref{eq:DefHamiltonian} is a homeomorphism, with $w = d_v L(v,x,t)$.
\end{itemize}

\begin{df}
    We define the (forward) Hopf--Lax semigroup: 
\begin{align}
    Q^{s,t}\phi(y) = \inf_x\left\{ \phi(x) + L^{s,t}(x,y)\right\}.
\end{align}
\end{df}

\begin{rmk}\label{rmk:Kantorovich-via-HL}
Note that it is essentially the same as the $c$-transform: $\phi^c=Q^{s,t}(-\phi)$, hence we can rewrite \eqref{eq:dualpbm2} as
\begin{align}\label{eq:dualpbmHL}
    K &=\sup\left\{\int_M Q^{s,t}(\varphi) d\nu - \int \varphi d\mu \ : \ \varphi\in L^1(d\mu) \right\}. 
\end{align}
More generally, one can see that if $\phi$ achieves the above supremum, $(\phi,Q^{s,r}\phi)$ are Kantorovich potentials between $\mu$ and $\mu_r$, where $\mu_r$ is an $r$-midpoint between $\mu$ and $\nu$: $W_{L^{s,r}}(\mu,\mu_r) + W_{L^{r,t}}(\mu_r,\nu) = W_{L^{s,t}}(\mu,\nu)$.
\end{rmk}

The main property of $Q^{s,t}$ is that it satisfies the Hamilton-Jacobi equation
\begin{align}
    \partial_tQ^{s,t}\phi + H(dQ^{s,t}\phi,Q^{s,t}\phi,t)=0
\end{align}
in appropriate sense, see \cite[pg.~157]{Vil09}, Appendix \ref{Appe:HJ}, and Corollary \ref{cor:dynamic-potentials}. For the strongest statement we assume also that $L$ is centered quadratic (which is the case for $L_m$, $m=0,-,+$), i.e. that for every $(p,t)$ fixed
\begin{align}
    L(v,p,t) =A_{p,t}(v,v)-C_{p,t}, \ \ \ \text{where $A_{p,t}$ is a positive definite quadratic form.}
\end{align}

\subsubsection{Optimal transport with Lagrangian cost}
We now get back to the optimal transport problem with cost $c=L^{s,t}$. In this case we write $K_c=W_{L^{s,t}}$, as we are mimicking Wasserstein distances.

In the case of classical $p$-Wasserstein distance on a compact metric space, $W_{d^p}$ metrizes weak convergence of Borel measures. One should not expect such a clean statement in the current level of generality, where $W_{L^{s,t}}$ need not be nonnegative and the function $ v\mapsto L(v, x, t)$ on $T_xM$ need not be homogeneous. For our purposes, the following much weaker statement will be sufficient

\begin{prop}\label{prop:W-implies-weak-conv}
For $M$ compact and $L$ satisfying the Assumptions of \ref{par:L-assumptions}, let $(\mu_{t_i})_{i \in \NN} \subseteq \cP(M)$ be a sequence and $\mu_t \in \cP(M)$ a measure such that $t_i \uparrow t$ and $\limsup_{i\to \infty} W_{L^{t_i,t}}(\mu_{t_i},\mu_t) < \infty$ (or $t_i\downarrow t$ and $\limsup_{i\to \infty} W_{L^{t,t_i}}(\mu_t,\mu_{t_i}) < \infty$). Then also weak convergence holds: $\mu_{t_i}\weak \mu_t \in \cP(M)$ as $t_i \to t$. 
\end{prop}

\begin{note}
In the proof below and later, we may use the notation $\Psi(a_1,\ldots, a_k\mid b_1,\ldots, b_\ell)$ to denote any quantity depending on real parameters $a_1,\ldots, a_k,b_1,\ldots, b_\ell \in (0,\infty)$ that converges to $0$ when $a_1,\ldots, a_k \to 0$ simultaneously while $b_1,\ldots, b_\ell$ are kept fixed. For instance, $a/b = \Psi(a\mid b)$ but $\neq\Psi(a,b)$, for $a,b \in (0,\infty)$.
\end{note}

\begin{proof}
We will argue in the case that $t_i\downarrow t$; the other case is analogous. Fix a number $\delta > 0$. By the assumption that $\limsup_{i\to \infty} W_{L^{t_i,t}}(\mu_{t_i},\mu_t) < \infty$ and Lemma \ref{lem:costfromlagrangian}, the transport plans must concentrate near the diagonal. More precisely, there are a sequence of couplings (for instance the $W_{L^{t_i,t}}$-optimal ones) $\pi^{t_i,t}$ between $\mu_{t_i}$ and $\mu_t$ such that  $\pi^{t_i,t}\big(M\times M \setminus B_\delta(\Delta)\big) =\Psi(t - t_i\mid\delta)$, where $\Delta := \{(x,x)\mid x \in M\}$ is the diagonal. Thus we have 

\begin{align}
\int_M f(x)\,d\mu_{t_i}(x) = \int_{M \times M} f(x)\,d\pi^{t_i,t}(x,y) &= \left(\int_{M \times M \setminus B_\delta(\Delta)} + \int_{B_\delta(\Delta)}\right) f(x)\,d\pi^{t_i,t}(x,y)\\ 
&= \Psi(t - t_i\mid \delta)\|f\|_\infty + \sup_{d(x,y) < \delta}|f(x) - f(y)|+\int_{M\times M}f(y)\,d\pi^{t_i,t}(x,y) \\
&= \Psi(t - t_i\mid \delta)\|f\|_\infty + \Psi(\delta\mid f) + \int_M f(y)\,d\mu_{t}(y)\,,\qquad \forall f \in C^0(M)\,.
\end{align}

If we send $0 < t - t_i \ll \delta \to 0$ for arbitrary but fixed $f \in C^0(M)$, we see that the LHS converges to the third term of the RHS, which is weak convergence.
\end{proof}

Despite the fact that $W_{L^{s,t}}$ may take on negative values, it still satisfies a triangle inequality for intermediate times, see \cite[Thm.~7.21,~Step~1.]{Vil09}.

\begin{equation}\label{eq:triangle-inequality}
W_{L^{s,t}}(\mu_s,\mu_t) \leq W_{L^{s,r}}(\mu_s,\mu_r) + W_{L^{r,t}}(\mu_r,\mu_t)\,,\qquad \forall r \in [s,t]\,;\,\mu_s,\,\mu_r,\,\mu_t \in \cP(M)\,.
\end{equation}

It is thus possible to formulate a notion of optimal curve(s) between two measures as one that saturates the above inequality for all intermediate times.

\begin{df}[\cite{Vil09}, Theorem 7.21]\label{df:WL-geodesic}
    We say that a curve of probability measures $(\mu_r)_{r\in[s,t]}$ is a $W_{L^{s,t}}$-geodesic (called displacement interpolation in \cite{Vil09}) if 
    \begin{align}
        W_{L^{s,r}}(\mu_s,\mu_r)+W_{L^{r,t}}(\mu_r,\mu_t) = W_{L^{s,t}}(\mu_s,\mu_t)
    \end{align}
    for all $r\in [s,t]$.
\end{df}

Given $\mu_s,\mu_t\in P(M)$, there always exists such a $(\mu_r)$, see \cite[Cor.~7.22]{Vil09}. If we solve the dual problem between $\mu_s$ and $\mu_r$ we obtain a family of potentials $(\phi_r)$, that (see Remark \ref{rmk:Kantorovich-via-HL}, Proposition \ref{prop:HLsemigroup}, and Proposition \ref{prop:fiber-quadratic-HJ}) will solve the Hamilton-Jacobi equation. For ease of future reference we give the statement.
\begin{cor}\label{cor:dynamic-potentials}
For a pair of measures $(\mu_s,\mu_t) \in \cP(M) \times \cP(M)$, there is a space-time Lipschitz function $\phi_r: [s,t]\times M \to \R{}$ for which

\begin{enumerate}
\item $\phi_r$ is given by the Hopf-Lax semigroup and solves the Hamilton-Jacobi equation
\begin{align}
    \phi_r := Q^{s,r}\phi\,,\qquad
        \begin{cases}
            \partial_rQ^{s,r}\phi(x) + H(d Q^{s,r}\phi,x,t) = 0 \label{eq:general-HJequation}\\
            Q^{s,r}\phi \to \phi
        \end{cases}\qquad \text{ a.e. }(r,x)\in [s,t]\times M\,.
    \end{align}

    If additionally $ v\mapsto L(v,x,r) = A_{x,r}(v,v) + C_{x,r}$ is centered quadratic on $T_xM$ and $C^2$ in all variables, then the above holds for every $r \in [s,t]$, a.e. $x \in M$ (with appropriate one-sided time derivatives at the endpoints  $r = s,t$).
\item For any $W_{L}$ geodesic $(\mu_r)_{r\in[s,t]} \subseteq \cP(M)$ with endpoints $(\mu_s,\mu_t)$, the $\phi_r$ are Kantorovich potentials for the dual problem between any two measures along $(\mu_r)_r$. More precisely
\begin{equation}
W_{L^{\bar s,\bar t}}(\mu_{\bar s},\mu_{\bar t}) = \int_M \phi_{\bar t}\,d\mu_{\bar t} - \int_M \phi_{\bar s}\,d\mu_{\bar s}\,,\qquad \forall s \leq \bar s < \bar t \leq t\,.
\end{equation}
\end{enumerate}
\end{cor}

\begin{rmk}
The statement that $\phi$ solves \eqref{eq:general-HJequation} in (only) an a.e. sense  would in general be too weak to be useful, but the Lipschitzianity of $\phi$ automatically improves \eqref{eq:general-HJequation} to an identity in the distributional sense for $\partial_r Q^{s,r}\phi$ on $[s,t]\times M$, and in the FTC sense for $\partial_r\int_M Q^{s,r}\phi \,d\mu$ on $[s,t]$ for any smooth measure $\mu \in \cP(M)$. It is also possible to obtain a pointwise statement, for instance by replacing \eqref{eq:general-HJequation} with an inequality where the RHS involves the slope, or interpreting \eqref{eq:general-HJequation} in the viscosity sense. This will be unnecessary for our purposes, since \eqref{eq:general-HJequation} will always be integrated against a smooth function, usually a heat kernel.
\end{rmk}

\subsubsection{Dynamic transport plans}
Here we recall the formulation of optimal transport that uses dynamic transport plans, which are measures on curves $C([s,t],M)$, giving an alternative formulation of the Kantorovich problem. The proofs can be found in the Appendix \ref{Appe:Dynplans}.

\begin{prop}\label{prop:Lisini-for-geodesics}
    For every $s,t$ fixed, for given $\mu,\nu\in \cP(M)$,
    \begin{align}
        K_c (\mu,\nu) = \inf \left\{ \int_{C([s,t],M)} \mathcal{L}(\gamma)\, d\eta(\gamma) \ : \ \eta\in \cP(C([s,t],M)), \ (e_s)_\#\eta=\mu, \ (e_t)_\#\eta=\nu \right\}\,,
    \end{align}
    where $e_r:C([s,t],M)\to M$, $e_r(\gamma)=\gamma_r$ is the evaluation map. Moreover there exists a minimizing plan $\eta$, which is supported on $L^{s,t}$ geodesics, and $\mu_r=(e_r)_\#\eta$ is a $W_{L^{s,t}}$ geodesic with
    \begin{align}
        (\dot{\mu}_r) _{L^r} = \int_C L(\dot\gamma_r,\gamma_r,r)\,d\eta(\gamma) \ \ \ \text{for every $r\in[s,t]$}\,,
    \end{align}
    where 
    \begin{align}
    (\dot{\mu}_r)_{L^r} := \lim_{h\to 0}\frac{W_{L^{r,r+h}}(\mu_r,\mu_{r+h})}{h}\,.
    \end{align}
Moreover, if $\mu_s$ is absolutely continuous with respect to $dV$, it holds
\begin{align}
    (\dot{\mu}_r)_{L^r} = \int_M L(-(d\phi_r(x))^*,x,r)d\mu_r \ \ \ \text{for every $r\in[s,t]$}\,,
\end{align}
    where $(\phi_r)_{r\in[s,t]}$ is the family of Kantorovich potentials from Corollary \ref{cor:dynamic-potentials}, and $(w)*\in TM$ is the optimizer in the Legendre duality.
\end{prop}

We now want a similar representation for any (regular enough) curve of measures, not necessarily geodesic (in particular, we will want to apply it to measures evolving under the conjugate heat flow). Let us recall that given a curve of measures $(\mu_r)_{r\in[s,t]}$, it is said to solve the continuity equation for some velocity field $v_r\in \Gamma(TM)$ if
\begin{align}\label{eq:continuityeqn}
    \partial_r\mu_r = -\div(\mu_r v_r)
\end{align}
in the distributional sense.

Note that when $\mu_r=\rho_r\,dV_r$, $\rho\in C^\infty(M\times[s,t],(0,\infty))$, setting $v_r=\nabla\phi_r$ it becomes an elliptic equation and hence we find a unique $\nabla\phi_r$: for our scope, it will be enough to prove the statement for such curves.

\begin{lem}\label{lem:dynamiclifting} Assume that the Lagrangian $L$ is centered quadratic, i.e. $L(v,x,r) = A_{x,r}(v,v) + C_{x,r}$ with $A_{x,r}$ a positive definite quadratic form, $C_{x,r} \in \R{}$. Then for any curve of probability measures $(\mu_r)_{r\in [s,t]}$ with space-time smooth and positive densities there exists $\eta\in P(C([s,t]\to M))$ supported on smooth curves such that $(e_r)_\#\eta = \mu_r$ and 
\begin{align}
        (\dot{\mu}_r)_{L^r} = \int_C L(\dot\gamma_r,\gamma_r,r)d\eta(\gamma) \ \ \ \text{for every $r\in(0,1)$}.
    \end{align}
\end{lem}

\section{Inequalities for the \texorpdfstring{$L_0$}{L0} distance}\label{s:L0}
In this section we give the proof of Theorem \ref{thm:L0charact}, spread across several subsections. We fix an arbitrary smooth family of closed Riemannian manifolds $(M, g_t)_{t \in I}$ defined for some interval $I \subseteq \R{}$. We will use this time-dependent manifold, along with the associated gadgets ($\cS_t$, $S_t$, etc.) defined in Subsection \ref{subsec:t-dep-mflds}, without further comment for the rest of the section.

\subsection{The \texorpdfstring{$\cD$}{D}-condition and the Bochner formula}\label{ss:L0DBoch}
The Bochner inequality of Theorem \ref{thm:L0charact} follows from a Bochner identity, where the role of the curvature term is played by the $\mathcal{D}$ quantity.
\begin{prop}\label{prop:DcondiffBochner}
    The following identity holds (all quantities are $t$-dependent):
    \begin{align}\label{eq:BochnerL0identity}
        (\partial_t-\Delta)(|\nabla v|^2 + 2\Delta v - S) = -2|\mathcal{S}-\Hess v|^2 - \mathcal{D}(-\nabla v)  + 2 \Delta(\partial_t-\Delta) v + 2\langle\nabla(\partial_t-\Delta) v,\nabla v\rangle
    \end{align}
    for all smooth $v:M\times I \to\R{}$. In particular, $\mathcal{D}\geq 0$ iff \begin{align}\label{eq:BochnerL0-id-to-ineq}
        (\partial_t-\Delta)(|\nabla v|^2 + 2\Delta v - S) \leq 0 \tag{B\({}_{L_0}\)}
    \end{align}
    for any smooth $v$ solution of the heat equation at $t$.
\end{prop}

\begin{rmk}\label{rmk:Bochner-as-commutation}
    We can reformulate the middle term of the Bochner inequality \eqref{eq:BochnerL0-id-to-ineq} as the commutation of time derivative and Laplacian, and up to changing the sign of $v$ we get
    \begin{align}
        |\partial_t\Delta v - \Delta \partial_t v| \leq -(\partial_t-\Delta)\frac{1}{2}(|\nabla v|^2 - S).
    \end{align}
    Note also that in the RHS the Hamiltonian associated to $L_0$ appears.
\end{rmk}

\begin{proof}
    The first identity is just a computation with the resulting terms grouped appropriately. Note that in the time derivative of the first term the derivative of the metric appears,
    \begin{align}
        (\partial_t-\Delta)|\nabla v|^2 = - 2|\Hess v|^2 +  2\cS (\nabla v,\nabla v) -2\Ric(\nabla v,\nabla v) + 2\La \nabla (\partial_t-\Delta)v,v\Ra,
    \end{align}
    in the second the Bianchi-term
    \begin{align}
        (\partial_t-\Delta)(\Delta v)=2\La \cS,\Hess v\Ra + \La2\div\cS-\nabla S,\nabla v\Ra + \Delta(\partial_t-\Delta)v.
    \end{align}
    and in the last one the evolution of $S$, with $2|\cS|^2$ added and subtracted to complete the square. It is clear that $\mathcal{D}\geq 0$ implies the inequality \eqref{eq:BochnerL0-id-to-ineq}. 
    
    Vice versa, assume that we have that every $v$ heat flow defined on any subinterval $[s,t]\subset I$ satisfies the inequality \eqref{eq:BochnerL0-id-to-ineq}. Then fix $(r,p)\in I\times M$, $V \in T_pM$ and choose $v_r:M\to\R{}$ such that $\nabla v_r= V$ and $\Hess v_r(p)=\mathcal{S}_r(p)$. This is always possible at a point, for example using normal coordinates centered at $p$, and choosing in charts $v_r(x)=a_{ij}x^ix^j+b_kx^k$ with $b_k=V_k$ and
    $a_{ij}=a_{ji}=\frac{1}{2}(\cS_r(p))_{ij}$. At this point, combining the assumed \eqref{eq:BochnerL0-id-to-ineq} with \eqref{eq:BochnerL0identity} and evaluating at time $(r,p)$ we get
    \begin{align}
        0 \geq (\partial_t-\Delta)(|\nabla v|^2 + 2\Delta v - S) =- \mathcal{D}(-\nabla v)\,,
    \end{align}
    which proves $\mathcal{D}_{r,p}(V) \geq 0$ for arbitrary $(r,p,V)\in(I\times TM)$.
\end{proof}

\subsection{Bochner formula and gradient estimates}\label{ss:L0BochGE}

\begin{prop} \label{prop:L0gradestimate}
    The Bochner inequality 
    \begin{align}\label{eq:BochnerL0-Bochner-to-GE}
        (\partial_t-\Delta)(|\nabla P_{t,s}v|_t^2 + 2\Delta_t P_{t,s}v - S_t) \leq 0 \,,\qquad\forall [s,t] \subseteq I \tag{B\({}_{L_0}\)}
    \end{align}
    is equivalent to the gradient estimate
    \begin{align}\label{eq:GEL0-Bochner-to-GE}
        |\nabla P_{t,s}v|_t^2 + 2\Delta_t P_{t,s}v - S_t \leq P_{t,s}(|\nabla v|^2_s +2\Delta_s v - S_s)\,,\qquad \forall [s,t]\subseteq I \tag{GE\({}_{L_0}\)}
    \end{align}
    for any smooth $v:M\to\R{}$. For $v$ Lipschitz the gradient estimate is still implied if we define $P_{t,s}\Delta_s v(x)= -\int \langle \nabla_{y} \rho_{t,s}(x,y),\nabla v(y)\rangle_{s} \,dV_s(y)$, where $\rho$ is the heat kernel.
\end{prop}

\begin{rmk}
    Just as observed in Remark \ref{rmk:Bochner-as-commutation}, we can change the sign of $v$, and obtain an estimate for the commutator of the Laplacian and the semigroup:
    \[|P_{t,s}\Delta v - \Delta P_{t,s}v| \leq \frac{1}{2}(P_{t,s}|\nabla v|^2_s - |\nabla P_{t,s}v|^2_t) + \frac{1}{2}(S_t - P_{t,s}S_s). \]
    Note that both terms of the RHS are nonnegative thanks to the super Ricci flow inequality and the evolution inequality for $S$, both implied by $\mathcal{D}\geq0$.
\end{rmk}

\begin{proof}
    $\implies$) Define for $s\leq r\leq t$ the function $\phi_r= |\nabla P_{r,s}v|^2_r + 2\Delta_r P_{r,s}v - S_r$. Then one can compute that
    \begin{align} \label{eq:derivativeofsemigroup}
        \partial_r P_{t,r}(\phi_r) = P_{t,r}((\partial_r-\Delta)\phi_r) \leq 0
    \end{align}
    using the Bochner \eqref{eq:BochnerL0-Bochner-to-GE} for the inequality. Integrating the above inequality between $s$ and $t$ gives
    \begin{align}
        \phi_t \leq P_{t,s}\phi_s\,.
    \end{align}
    The statement for Lipschitz functions follows by approximating $v$ in $W^{1,2}$ with smooth functions, and using the smoothness of the heat kernel.
    
    $\impliedby$) Since \eqref{eq:GEL0-Bochner-to-GE} is an equality at $t=s$, we have 
    \begin{align}
        \partial_t|_{t=s} \phi_t \leq \partial_t|_{t=s} P_{t,s}\phi_s = \Delta P_{t,s}\phi_t |_{t=s} = \Delta \phi_s
    \end{align}
    i.e. $(\partial_t-\Delta) \phi_t |_{t=s} \leq 0 $. Varying $s$ we recover the Bochner \eqref{eq:BochnerL0-Bochner-to-GE} at any time.
\end{proof}

\subsection{Duality of gradient estimates and \texorpdfstring{$L_0$}{L0} Wasserstein contraction}\label{ss:GEWCdualityL0}
In this section we prove the equivalence of the gradient estimate and Wasserstein contraction. The approach is inspired by the works of Kuwada \cite{kuwadaduality,kuwadaspace}, with the additional difficulty that we are not working with a positive distance.
\begin{rmk}
    Note that (as done in \cite{kuwadaduality}) the duality between the gradient estimate and Wasserstein contraction requires much less structure than the one we are working with: it is enough to have a space, a Lagrangian and a time-dependent semigroup with its adjoint. In particular, the link between the Lagrangian $|\dot\gamma|^2+S(\gamma)$ and $S$ being the derivative of the measure (or part of the adjoint heat equation) is never used. 
\end{rmk}

\begin{rmk}
    The following arguments extend Kuwada duality to the case of a time-dependent semigroup, showing that this duality happens also at the level of Legendre duality. More precisely we show that for a time-dependent semigroup $P_{t,s}$ on $L^2(M,dV_t)$ with generators $(G_t)_t$ and adjoint $\hat P_{t,s}$ on $\cP(M)$, and a (sufficiently regular) Lagrangian $L$ with associated Hamiltonian $H$, the following are equivalent
    \begin{align}
        &|G_t P_{t,s}v - P_{t,s}G_s v| \leq P_{t,s}(H(d v,v,s))-H(dP_{t,s}v,P_{t,s}v,t) \qquad \text{for all $v$ smooth, $x\in M$, $t>s>0$}; \\
        &W_{L^{s,t}}(\hat P_{s+h,s}\mu, \hat P_{t+h,t}\nu)\leq W_{L^{s + h,t + h}}(\mu,\nu) \qquad \text{for all $\mu,\nu\in\cP(M)$, $t\geq s$, $h>0$}.
    \end{align}
\end{rmk}

\subsubsection{From GE to WC}
\begin{prop}\label{prop:GEtoWCL0}
    Assume that for any Lipschitz $v: M \to \R{}$, any $[s,t]\subseteq I$, the gradient estimate \begin{align}\label{eq:GEL0-GE-to-WC}\tag{GE\({}_{L_0}\)}
        |\nabla P_{t,s}v|_t^2 + 2\Delta_t P_{t,s}v - S_t \leq P_{t,s}(|\nabla v|^2_s +2\Delta_s v - S_s)
    \end{align} holds. Then, for all probability measures $\mu,\nu \in \cP(M)$, all times $ [s,t] \in I$, and all flow times $h < \sup I - t$, we have the Wasserstein contraction
    \begin{align}\label{eq:WCL0-GE-to-WC}
        W_{L_0^{s,t}}(\hat{P}_{s+h,s}\mu,\hat{P}_{t+h,t}\nu) \leq W_{L_0^{s+h,t+h}}(\mu,\nu)\,.\tag{WC\({}_{L_0}\)}
    \end{align}
\end{prop}

We break down the proof in two parts: first we show that \eqref{eq:GEL0-GE-to-WC} implies that the HJ inequality is preserved under the heat flow, and then, exploiting the Kantorovich duality, we get \eqref{eq:WCL0-GE-to-WC}.

Note that the Hopf--Lax semigroup $Q^{s,t}$
associated to $L_0$ solves (recall \ref{subsec:LagrangianProperties}) 
\begin{equation}\label{eq:L0-HJ}
\partial_t Q^{s,t}\phi (x,t) = -\frac{1}{2}\left(|\nabla Q^{s,t}\phi|_t^2(x,t)-S(x,t)\right)
\end{equation}
as the right-hand side is (minus) the Legendre transform of the $L_0$ Lagrangian.

\begin{lem} \label{lem:heat-flow-preserves-HJ}
Suppose that \eqref{eq:GEL0-GE-to-WC} holds, $f: [t_1,t_2] \times M\to \R{}$, $[t_1,t_2]\subseteq I$ is Lipschitz, and satisfies the Hamilton-Jacobi inequality
\begin{equation}
\partial_t f(t,x) \leq -\frac{1}{2}|\nabla f(t,x)|^2 + \frac{1}{2}S_{(t,x)}\,,\qquad \text{for a.e. }(t,x)\in I\times M\,.
\end{equation}

Then the Lipschitz function $P_h[f]: [t_1 + h, t_2 + h] \times M \to \R{}$, defined for $0 < h \leq \sup I - t_2$ by $P_h[f](t+h,x) := P_{t+h, t}f_t(x)$, itself satisfies the Hamilton-Jacobi inequality
\begin{equation}
\partial_t P_h[f](t+h,x) \leq -\frac{1}{2}|\nabla P_h[f](t+h,x)|^2 + \frac{1}{2}S_{(t+h,x)}\,,\qquad \text{for a.e. } t\in  [t_1,t_2], \text{ for all } x\in M\,,
\end{equation}

and is dominated by the $L_0$-distance:
\begin{equation}\label{eq:L0-dominated-preserved}
P_h[f](t+h,y) - P_h[f](s+h,x)  \leq L^{s+h,t+h}_0(x,y)\,,\qquad \text{for all} \ x,y \in M\,;\, s < t\in [t_1, t_2]\,. 
\end{equation}
\end{lem}

\begin{proof}
    Note first that $P_hf_t$ is differentiable in space everywhere, by regularization of the heat flow, and in time for a.e. $t\in I$, \textbf{for all} $x\in M$.
    Indeed, because $f$ is Lipschitz, there is a full measure subset $I_0 \subset I$ such that for $t\in I_0$, $f_t$ is differentiable at a.e. $x\in M$. Then, for $t\in I_0$ we also have that for every $x\in M$, using a Leibniz rule argument,
    \begin{align}
        \lim_{\eps\to0} \frac{P_h[f](t+h+\eps,x)-P_h[f](t+h,x)}{\eps} &= \lim_{\eps\to0} \frac{P_{t+h+\eps,t+\eps}f_{t+\eps}(x)-P_{t+h,t}f_t(x)}{\eps} \\
        &=\Delta P_{t+h,t}f_t(x) - P_{t+h,t}\Delta f_t(x) \\
        &\qquad + \lim_{\eps\to0} \frac{1}{\eps}\int \rho_{t+h,t}(y,x)(f_{t+\eps}(y)-f_t(y))dV_t(y) \\
        &=\Delta P_{t+h,t}f_t(x) - P_{t+h,t}\Delta f_t(x) + P_{t+h,t}(\partial_tf_t)(x)
    \end{align}
    where in the last line we used that $f$ is Lipschitz in order to bring the limit inside $P_{t+h,t}$, and $P_{t+h,t}(\partial_tf_t)(x)$ is well defined for every $x\in M$. Now, for $t\in I_0$, using the Hamilton-Jacobi inequality for $f$ together with the maximum principle for $P_{t+h,t}$, and then \eqref{eq:GEL0-GE-to-WC} we have
    \begin{align}
        \partial_t P_{t+h,t}f_t &= \Delta P_{t+h,t}f_t - P_{t+h,t}\Delta f_t + P_{t+h,t}\partial_t f_t \\
        &\leq \Delta P_{t+h,t}f_t - P_{t+h,t}\Delta f_t - \frac{1}{2}P_{t+h,t}|\nabla f_t|^2 + \frac{1}{2}P_{t+h,t}S_t \\
        &\leq -\frac{1}{2}|\nabla P_{t+h,t}f_t|^2 + \frac{1}{2}S_{t+h}\,.
    \end{align}

    Integrating this Hamilton-Jacobi inequality we get domination by the $L_0$-distance: considering $\gamma$ an $L_0$-geodesic between $(s+h,x)$ and $(t+h,y)$ we have
    \begin{align}
        P_h[f](t+h,y) - P_h[f](s+h,x) &= \int_{s+h}^{t+h} \partial_r P_h[f](r,\gamma_{r}) + \langle \nabla P_h[f](r,\gamma_{r}),\dot\gamma_{r} \rangle dr \\
        &\leq \int_{s+h}^{t+h} -\frac{1}{2}|\nabla P_h[f](r,\gamma_{r})|^2 + \frac{1}{2}S_{r,\gamma_{r}} + \langle \nabla P_h[f](r,\gamma_{r}),\dot\gamma_{r} \rangle dr \\
        &\leq \int_{s+h}^{t+h} \frac{1}{2}|\dot\gamma_{r}|^2 + \frac{1} {2}S_{r,\gamma_{r}} dr  \\
        &= L^{s+h,t+h}_0(y,x)
    \end{align}
    using Young's inequality (or the Legendre duality). Note that the quantities in the integrals are well defined for a.e. time along the curve thanks to the improved regularity of $P_h[f]$.
\end{proof}

\begin{proof}[Proof of Prop \ref{prop:GEtoWCL0}]
    Using the dual problem \eqref{eq:dualpbm2} we get
    \begin{align}
        W_{L_0^{s,t}}(\hat{P}_{s+h,s}\mu,\hat{P}_{t+h,t}\nu) &= \sup \left\{ \int Q^{s,t}\phi \  d\hat{P}_{t+h,t} \nu - \int \phi \ d\hat{P}_{s + h,s} \mu \ : \ \phi\in L^1 \right\} \\
        &= \sup \left\{ \int P_{t+h,t}[Q^{s,t}\phi](y)d\nu - \int P_{s + h,s}\phi(x) d\mu \ : \ \phi\in L^1 \right\} \\
        &\leq \sup \left\{ \int \psi(y)d\nu - \int \phi(x) d\mu \ : \ \phi\oplus \psi \leq L^{s+h,t+h} \right\} \\
        &= W_{L_0^{s+h,t+h}}(\mu,\nu)\,,
    \end{align}
    where for the inequality we used \eqref{eq:L0-dominated-preserved}.
\end{proof}

\subsubsection{From WC to GE}
\begin{prop}\label{prop:L0-contraction-to-GE}
    Suppose that for every $\mu,\nu\in \cP(M)$, times $[s,t]\subseteq I$, and flow time $h < \sup I - t$ the Wasserstein contraction holds \noeqref{eq:WCL0-WC-to-GE}
    \begin{align}
        W_{L_0^{s,t}}(\hat{P}_{s+h,s}\mu,\hat{P}_{t+h,t}\nu) \leq W_{L_0^{s+h,t+h}}(\mu,\nu)\,.\label{eq:WCL0-WC-to-GE}\tag{WC\({}_{L_0}\)}
    \end{align} 
    Then we recover the gradient estimate \noeqref{eq:GEL0-WC-to-GE}
    \begin{align}
        |\nabla P_{t,s}v|^2 + 2\Delta P_{t,s}v - S_t \leq P_{t,s}(|\nabla v|^2 +2\Delta v - S_s)\,,\label{eq:GEL0-WC-to-GE}\tag{GE\({}_{L_0}\)}
    \end{align}
    for all $[s,t] \subseteq I$, all $v$ smooth.
\end{prop}

\begin{proof}
    For given $x,y$, $s,t$, $h>0$, let $\bar\gamma:[0,t-s]\to M$ be a curve from $x$ to $y$, then
    \begin{align}
        P_{t+h,t}v(y)-P_{s+h,s}v(x) &= \int v d\hat{P}_{t+h,t}\delta_{\bar\gamma_{t-s}} - \int v d\hat{P}_{s+h,s}\delta_{\bar\gamma_0}.
    \end{align}
    Consider now the curve $[0,t-s]\ni r \mapsto \hat{P}_{s+h+r,s+r}\delta_{\bar\gamma_r} =: \mu_r$ which is a curve of probability measures from $\hat{P}_{s+h,s}\delta_{\bar\gamma_0}$ to $\hat{P}_{t+h,t}\delta_{\bar\gamma_{t-s}}$. 
    Using Lemma \ref{lem:dynamiclifting} we can find $\eta \in \mathcal{P}(C([0,t-s]\to M))$ with $(e_r)_\#\eta = \mu_r$ and
    \begin{align} \label{eq:liftingwithsameKE}
        \frac{1}{2}\int_{C([0,t-s]\to M)} |\dot{\gamma}_r|_{s+r}^2 + S_{s+r}(\gamma_r) d\eta(\gamma) = (\dot{\mu}_r)_{L_0^{s+r}}.
    \end{align}
    Hence
    \begin{align}
        \int v d\hat{P}_{t+h,t}\delta_{\bar\gamma_{t-s}} - \int v d\hat{P}_{s+h,s}\delta_{\bar\gamma_0} &= \int_{C([0,t-s]\to M)} \int_0^{t-s} \langle dv(\gamma_r) \dot{\gamma}_r\rangle dr d\eta(\gamma) \\
        & \leq \int_{C([0,t-s]\to M)} \int_0^{t-s} \left(\frac{1}{2} |\nabla v(\gamma_r)|_{s+r}^2 + \frac{1}{2}|\dot{\gamma}_r|_{s+r}^2\right) dr d\eta(\gamma) \\
        &= \frac{1}{2} \int_{C} \int_0^{t-s} |\nabla v(\gamma_r)|_{s+r}^2 dr d\eta(\gamma) + \frac{1}{2}\int_C\int_0^{t-s}|\dot{\gamma}_r|_{s+r}^2 drd\eta(\gamma) \\
        &= \frac{1}{2} \int_0^{t-s} \int_M |\nabla v|_{s+r}^2 d\mu_r dr + \frac{1}{2}\int_C\int_0^{t-s}\left(|\dot{\gamma}_r|_{s+r}^2 + S_{s+r}(\gamma_r)\right) drd\eta \\
        & \ \ \ \ - \frac{1}{2}\int_0^{t-s}\int_M S_{s+r}d\mu_r dr \\
        &= \frac{1}{2}\int_0^{t-s} P_{s+h+r,s+r}(|\nabla v|_{s+r}^2)(\bar\gamma_r) dr + \int_0^{t-s}(\dot{\mu}_r)_{L_0^{s+r}} dr \\
        & \ \ \ \ - \frac{1}{2}\int_0^{t-s}P_{s+h+r,s+r}S_{s+r}(\bar\gamma_r)dr. \label{eq:L0-Lisini-application}
    \end{align}
    We used Young's inequality in the inequality, repeatedly the fact that $\mu_r=(e_r)_\#\eta$, and \eqref{eq:liftingwithsameKE} in the last equality.

    Now, the Wasserstein contraction estimate implies
    \begin{align}
        W_{L_0^{s+r,s+r+\epsilon}}(\mu_r,\mu_{r+\epsilon}) \leq W_{L_0^{s+r+h,s+r+\epsilon+h}}(\delta_{\bar\gamma_r},\delta_{\bar\gamma_{r+\epsilon}})\,,
    \end{align}
    and dividing by $\epsilon$ and for $\epsilon\to0$ we get
    \begin{align}
        (\dot{\mu}_r)_{L_0^{s+r}} \leq (\dot{\delta}_{\bar\gamma_r})_{L_0^{s+h+r}} = \frac{1}{2}\left(|\dot{\bar\gamma}_r|_{s+h+r}^2 + S_{s+h+r}(\bar\gamma_r)\right).
    \end{align}
    Substituting in the previous estimate,
    \begin{align}
        P_{t+h,t}v(y)-P_{s+h,s}v(x) &\leq \frac{1}{2}\int_0^{t-s} P_{s+h+r,s+r}(|\nabla v|_{s+r}^2)(\bar\gamma_r) dr + \int_0^{t-s}(\dot{\mu}_r)_{L_0^{s+r}} dr \\
        & \ \ \ \ - \frac{1}{2}\int_0^{t-s}P_{s+h+r,s+r}S_{s+r}(\bar\gamma_r) dr \\
        &\leq \frac{1}{2}\int_0^{t-s} P_{s+h+r,s+r}(|\nabla v|_{s+r}^2)(\bar\gamma_r) dr + \frac{1}{2}\int_0^{t-s}\left(|\dot{\bar\gamma}_r|_{s+h+r}^2 + S_{s+h+r}(\bar\gamma_r)\right) dr \\
        & \ \ \ \ - \frac{1}{2}\int_0^{t-s}P_{s+h+r,s+r}S_{s+r}(\bar\gamma_r) dr.
    \end{align}
    In the estimate we have the freedom to choose $y$ and $\bar\gamma$ accordingly, and hence we pick $\bar\gamma$ with $\dot{\bar\gamma}_0 = \nabla P_{s + h,s}v(x)$. We now divide everything by $t-s$ and take the limit as $t\to s$. The left-hand side becomes 
    \begin{align}
        \partial_r|_{r=0} P_{s+h+r,s+r}v(\bar\gamma_r) &= \Delta P_{s+h,s}v(x) - P_{s+h,s}\Delta v(x) + \langle dP_{s+h,s}v(x),\dot{\bar\gamma}(0) \rangle \\
        &= \Delta P_{s+h,s}v(x) - P_{s+h,s}\Delta v(x) + |\nabla P_{s+h,s} v|^2_{s+h}(x).
    \end{align}
    We then get
    \begin{align}
        \Delta P_{s+h,s}v(x) - P_{s+h,s}\Delta v(x) &+ |\nabla P_{s+h,s} v|^2_{s+h}(x) \\ &\leq 
        \frac{1}{2} P_{s+h,s}|\nabla v|_{s}^2(x) + \frac{1}{2}\left(|\nabla P_{s+h,s} v|_{s+h}^2(x) + S_{s+h}(x)\right) - \frac{1}{2}P_{s+h,s}S_{s}(x)
    \end{align}
    and conclude
    \begin{align}
        \Delta P_{s+h,s}v - P_{s+h,s}\Delta v &+ \frac{1}{2} |\nabla P_{s+h,s} v|^2_{s+h} \leq 
        \frac{1}{2} P_{s+h,s}|\nabla v|_{s}^2 + \frac{1}{2}S_{s+h} - \frac{1}{2}P_{s+h,s}S_{s}.
    \end{align}
\end{proof}

\subsection{Entropy convexity and Evolution Variational Inequality (\texorpdfstring{$\mathrm{EVI}$}{EVI})}\label{ss:entropyEVIL0}

In this section we show the equivalence of the $\cD$ condition and the convexity of a quantity related to entropy. This convexity statement has already been obtained in \cite{Top09} in the $L_-$ case via a Jacobi field approach, and in \cite{Lot09} for the $L_0$ case using a formal Otto Calculus argument. They both worked in the setting of Ricci flows, but Huang \cite{Huang} later noticed that the $\cD$-condition is sufficient. We will show that the two are in fact equivalent. 

The proof of $W_{L_0}$ contraction from entropy convexity proceeds largely as in \cite{Top09} where the semiconcavity of the $L_0$ cost, and therefore the smooth structure of the underlying space, plays a role in estimating the one-sided derivatives of Entropy along $W_{L_0}$ geodesics. In the time-independent case, these derivative estimates can be obtained in non-smooth settings--see e.g. \cite[Prop.~5.10]{Gig15}--but we leave such considerations to future work.

For the backward direction, one could go all the way back to the $\cD$ condition, then rely on \cite{Top09,Huang}. However, we give a new proof following the approach of \cite{AmbGigSav15} for $\mathsf{RCD}$ spaces, who use the gradient estimate to obtain an $\mathrm{EVI}$ (Evolution Variational Inequality). This inequality can then be used to recover entropy convexity by following an argument from \cite{DanSav08}. One motivation for this alternative proof is that it does not rely on tensor calculus or Jacobi fields, and may therefore still be valid in lower-regularity settings.

\subsubsection{From entropy convexity to $W_{L_0}$ contraction}\label{sss1:entropyEVIL0}

Here we show that the convexity of an appropriate modification of entropy along $W_{L_0}$ geodesics implies $W_{L_0}$ contraction. In order to perform computations with the assumed entropy convexity, we wish to rewrite it as a monotonicity of first derivatives. We may not be able to compute these first derivatives explicitly, but we can bound them from the (for the purposes of this implication) correct side.

\begin{lem}[{\cite[l.~(3.18), (3.24-25)]{Top09}}]\label{lem:Top-1-sided-est}
Consider times $[s, t] \subseteq I$ and absolutely contiuous probability measures $\mu_{s} = \rho_{s}\,dV_{s}, \mu_{t} = \rho_{t}\,dV_{t} \in \cP(M)$ with smooth, positive densities $0 < \rho_{s}, \rho_t \in C^\infty(M)$. Let $(\mu_r)_{r \in [s,t]}$ be a $W_{L_0}$-geodesic connecting them, and $(\phi_r)_{r \in [s,t]} :M \to \R{}$ the Kantorovich potentials of Corollary \ref{cor:dynamic-potentials}. Then the following one-sided derivatives exist and satisfy the estimate:

\begin{align}
\left.\partial_r\right|_{r = s^+}\cE(\mu_r) &\geq \int_M (\La \nabla \ln(\rho_{s}),\nabla \phi_{s}\Ra + S_{s} )\,d\mu_{s}\,, \\
\left.\partial_r\right|_{r = t^-} \cE(\mu_r) &\leq \int_M (\La \nabla \ln(\rho_{t}),\nabla \phi_{t}\Ra + S_{t} )\,d\mu_{t}\,.
\end{align}
\end{lem}

\begin{rmk}
Topping uses a different sign convention for the Kantorovich potentials, and writes the statement in terms of backwards time. Moreover, his proof is for the $L_-$ cost instead of $L_0$, and at first glance seems to use the Ricci flow equation to obtain a finite lower bound on the second derivative of volume density along $L_-$-geodesics. However, only space-time regularity and diameter boundedness of $(M,g_t)$ are actually being used (e.g. in the form $\cD(X) \geq -C(|X|^2 + 1) > -\infty$ for some $C \in \R{}$, see \cite[Lem.~2.4]{Huang}), along with the fact that the Lagrangian inducing the cost is smooth and centered-quadratic on each fiber.
\end{rmk}

\begin{rmk}
The inequalities of Lemma \ref{lem:Top-1-sided-est} cannot be upgraded to equalities in general. This loss occurs exactly in inequality \cite[l.~(3.22)]{Top09}, which uses the semiconvexity of $\phi_s$ and semiconcavity of $\phi_{t}$ to argue that the singular parts of $\Delta \phi_{s}$ and $\Delta \phi_t$ are signed. However, the presence of Lipschitz peaks cannot be ruled out, e.g. by \cite[Cor.~12.21]{Vil09}. One might attempt to solve this issue by taking a limit of the computation in the interior of the geodesic where both inequalities of Lemma \ref{lem:Top-1-sided-est}, and therefore also equality, might hold. But for $r \in (s,t)$ there is no guarantee that $\rho_r$ will be smooth, even if $0 < \rho_{s},\rho_t \in C^\infty(M)$ (see e.g. \cite[Thm.~12.44]{Vil09}), and one has insufficient regularity to make sense of $\int_M \La\nabla \phi_r, \nabla \rho_r\,dV_r\Ra$ as the pairing of an a.e. defined function with a distribution.
\end{rmk}

We may then proceed to one direction of the equivalence between entropy (semi)convexity and $W_{L_0}$-contraction along the heat flow, using an argument that can essentially be found in \cite[Sec.~4]{Top09}.

\begin{prop}\label{prop:L0-cvxty-implies-contr/EVI}
Suppose that for any times $[s, t] \subseteq I$ and measures $\mu_{s},\mu_t \in \cP(M)$, there exists a $W_{L_0}$-geodesic $(\mu_{r})_{r \in [s,t]}$ such that
\begin{align}
    r\mapsto \cE(\mu_r) - \cE(\mu_{s}) - W_{L_0^{s,r}}(\mu_{s}, \mu_r) \ \ \text{ is convex.}\label{eq:ECL0-EC-to-WC}\tag{EC\({}_{L_0}\)}
\end{align} Then for any $\mu_{s}, \mu_{t} \in \cP(M)$ and $h > s - \inf I$ the following $W_{L_0}$ contraction holds \noeqref{eq:WCL0-EC-to-WC}
\begin{equation}
W_{L_0^{s - h,t - h}}(\hat{P}_{s,t - h}\mu_{s},\hat{P}_{t, t - h}\mu_{t}) \leq W_{L_0^{s, t}}(\mu_{s},\mu_{t})\,.\label{eq:WCL0-EC-to-WC}\tag{WC\({}_{L_0}\)}
\end{equation}
\end{prop}

\begin{proof}
Let us first apply the heat flow for some small time; suppose that $\mu_{s - \eps} = \hat{P}_{
s,s - \eps}\nu_{s}$ and $\mu_{t - \eps} = \hat{P}_{
t,t - \eps}\nu_{t}$ for some $\eps > 0$ and $\nu_{s}, \nu_t \in \cP(M)$. Then $\mu_{s - \eps},\mu_{t - \eps}$ are smooth with positive density (by the strong maximum principle). Let $(\phi_r)_{r \in [s -\eps,t-\eps]} :M \to \R{}$ be the space-time Lipschitz Kantorovich potentials of Corollary \ref{cor:dynamic-potentials} associated to the pair $(\mu_{s - \eps}, \mu_{t - \eps})$. By the convexity assumption, we have

\begin{equation}
0 \leq \left.\frac{d}{dr}\left[\cE(\mu_r) - W_{L_0^{s - \eps,r}}(\mu_{s - \eps},\mu_r)\right]\right|_{s - \eps^+}^{t - \eps^-}\,.
\end{equation}

Note that the existence of the one-sided derivative of the first term by itself is provided and estimated by Lemma \ref{lem:Top-1-sided-est}, while the one-sided derivative of the second term is provided by Proposition \ref{prop:Lisini-for-geodesics} applied to the $L_0$ Lagrangian, which yield the further inequality

\begin{equation}
\cdots \leq \left.\int_M (\La \nabla \ln( \rho_r),\nabla \phi_r \Ra + S_r)\,d\mu_r - \frac{1}{2}\int_M (|\nabla \phi_r|^2 + S_r)\,d\mu_r \right|_{s - \eps}^{t - \eps}\,.
\end{equation}

Integrating by parts ($\rho_{s - \eps}$ and $\rho_{t - \eps}$ are smooth) and then using the heat equation and the Hamilton-Jacobi equation gives

\begin{align}
\cdots &= \left.\int_M -\phi_r\Delta \rho_r \,dV_r - \frac{1}{2}\int_M (|\nabla \phi_r|^2 - S_r)\,d\mu_r\right|_{s - \eps}^{t - \eps} \\
&=\left.\frac{d}{dh}\right|_{h = 0^+} \left[\int_{M} \phi_{t - \eps + h } \,d\hat{P}_{t,t - \eps + h}\nu_{t} - \int_{M} \phi_{s - \eps + h} \,d\hat{P}_{s,s - \eps + h}\nu_{s}\right]\,.
\end{align}

It is worth noting that $\phi_r$ is a priori only defined for $r \in [s - \eps, t - \eps]$, but can be freely extended to later times using the Hopf-Lax semigroup and solves the Hamilton-Jacobi equation at times $s - \eps$, $t - \eps$, and at a.e. point $x \in M$ by the semiconvexity of $\phi_{s - \eps}$ (use that the Kantorovich potential $\phi_{s - \eps}$ is itself given by a backwards $L_0^{s - \eps,t -\eps}$-transform, then apply \cite[Thm.~6.4.3]{CanSin04}) and Proposition \ref{prop:fiber-quadratic-HJ}. We now use the fact that $(\phi_{s - \eps + h}, \phi_{t - \eps + h})$ are related by the $L_0^{s - \eps +h, t - \eps + h}$-transform, and thus are candidates for the dual problem as stated in \eqref{eq:dualpbm1} associated to $(\hat{P}_{ s,s - \eps + h}\nu_{s}, \hat{P}_{ t,t - \eps + h}\nu_{t})$, to estimate the first and third term of the following difference quotient from above

\begin{align}
&\phantom{\geq}\frac{1}{h}\left[\int_{M} (\phi_{t - \eps + h} \,d\hat{P}_{t,t -\eps + h}\nu_{t} - \phi_{t - \eps}\,d\mu_{t - \eps}) - \int_{M} (\phi_{s - \eps + h} \,d\hat{P}_{s,s - \eps + h}\nu_{s} - \phi_{s - \eps}\,d\mu_{s - \eps})\right]\\
&\leq  \frac{1}{h} \left[W_{L_0^{s - \eps + h, t -\eps + h}}(\hat{P}_{s,s - \eps + h}\nu_{s}, \hat{P}_{t,t - \eps + h}\nu_{t}) - W_{L_0^{s - \eps, t - \eps}}(\mu_{s - \eps}, \mu_{t - \eps})\right]\,,\qquad h > 0\,.
\end{align}

Taking the $\liminf$ as $h \downarrow 0$ and chaining together the previous inequalities shows that \begin{equation}
\left.\partial_h^-\right|_{h = 0^+}W_{L_0^{s - \eps + h, t - \eps + h}}(\hat{P}_{s,s - \eps + h}\nu_{s}, \hat{P}_{t,t - \eps + h}\nu_{t}) \geq 0\,,
\end{equation}
i.e. $W_{L_0}$ is non-increasing in the backwards time direction along the adjoint heat flow, except perhaps at the initial (backwards) time. 

To extend this statement to $\eps = 0$, note that the heat flow for measures is weakly continuous ($\hat P_{t,t - \eps}\nu_{t} \weak \nu_{t}, \,\hat P_{s,s - \eps}\nu_{s} \weak \nu_{s} \in \cP(M)$) and the costs are uniformly converging ($c_\eps:= L_0^{s - \eps,t - \eps} \to c_0:=L_0^{s, t} \in C^0(M\times M)$) as $\eps \downarrow 0$. Then, using \cite[Thm.~5.20]{Vil09}, any optimal transport plans $\pi_\eps$ for $(\hat P_{s,s - \eps}\nu_{s},\hat P_{t,t - \eps}\nu_{t})$ with respect to $L_0^{s - \eps,t - \eps} $ weakly converge (up to subsequence) to an optimal transport plan $\pi_0$ for $(\nu_{s},\nu_{t})$ with respect to the cost $L_0^{s, t}$. As the pairing of a uniformly convergent sequence of functions with a weakly convergent sequence of measures, we have $\lim_{\eps \to 0}\int_{M\times M} c_\eps \,d\pi_\eps = \int_{M\times M} c_0\,d\pi_0 = W_{L_0^{s,t}}(\nu_{s},\nu_{t})$, the second equality holding by optimality of $\pi_0$. This convergence, along with the contraction property for $\eps > 0$, gives 

\begin{equation}
W_{L_0^{s - \eps, t - \eps}}(\hat{P}_{s,s - \eps}\nu_{s}, \hat{P}_{t,t - \eps}\nu_{t}) \downarrow W_{L_0^{s , t}}(\nu_{s}, \nu_{t}) \qquad \text{as }\eps \downarrow 0\,.
\end{equation}
\end{proof}

\subsubsection{From gradient estimates to entropy convexity (via $\mathrm{EVI}$).}\label{sss2:entropyEVIL0}
We now show how to prove the convexity statement assuming the gradient estimate. The approach is roughly based on \cite[Sec.~4]{AmbGigSav15}, albeit with a few complications due to the time-dependence and many simplifications due to the smoothness of the background space-time. See also \cite[Sec.~6]{KopStu18} for a related approach in the setting of singular super Ricci flows with a quantitative regularity assumption in time.

The general idea is to first approximate $W_{L_0}$-Wasserstein geodesics $(\mu_{r})_{r \in [s,t]}$ in an appropriate topology by curves of smooth measures $(\mu_{r}^\eps)_{r \in [s - \eps,t-\eps]}$ (at the cost of losing the geodesic property) using Lemma \ref{lem:wass-geodesic-smoothing}. Then in Lemma \ref{lem:L0-action-entropy-est} we obtain an action-entropy estimate for the approximating curve, which becomes the $\mathrm{EVI}$ (Proposition \ref{prop:L0GE-to-EVI}) after sending the smoothing parameter $\eps \to 0$, then taking a limiting difference quotient. Finally, we use the $\mathrm{EVI}$ to obtain the entropy convexity in Proposition \ref{prop:L0-entropy-cvx}. 

In the spirit of \cite[Thm.~4.16]{AmbGigSav15} we can show:
\begin{lem}\label{lem:L0-action-entropy-est}
Let $(\nu_r)_{r \in [s,t]} = (\rho_r dV_r)_{r \in [s,t]}$, $\rho \in C^\infty\big([s,t] \times M; (0,\infty)\big)$, $[s,t] \subseteq I$, a curve of smooth measures with positive densities. For $t - s, s - \inf I > h > 0$, fix constants $\lambda_I := h/(t - s + h)$ and $\lambda_{II}:= h/(t-h - s)$. Then if \eqref{eq:GEL0-GE-to-EVI} holds, we have the action-entropy estimates

\begin{align}
\lambda_I \cE(\hat{P}_{s, s - h}[\nu_{s}]) + (1-\lambda_I^2)W_{L_0^{s - h,t}}(\hat{P}_{s, s - h}[\nu_{s}], \nu_{t}) &\leq \lambda_I \cE(\nu_{t}) + (1 - \lambda_I)\int_{s}^{t}(\dot{\nu}_{r})_{L_0^{r}}\,dr \\
&- \frac{\lambda_I^2}{2}\int_{s - h}^{t}\int_M P_{r + \lambda_I(t - r),r}[S_r]\,d\nu_{r + \lambda_I(t - r)}\,dr\,, \label{eq:L0-preEVI-I}
\end{align}

and 

\begin{align}
-\lambda_{II}\cE(\nu_{s}) + (1-\lambda_{II}^2)W_{L_0^{s,t - h}}(\nu_{s}, \hat{P}_{t,t -h}[\nu_{t}]) &\leq -\lambda_{II} \cE(\hat{P}_{t,t -h}[\nu_{t}]) + (1 + \lambda_{II})\int_{s}^{t}(\dot{\nu}_{r})_{L_0^{r}}\,dr \\
&- \frac{\lambda^2_{II}}{2}\int_{s}^{t - h}\int_M P_{r + \lambda_{II}(r - s),r}[S_{r}]\,d\nu_{r + \lambda_{II}(r - s)}\,dr\,. \label{eq:L0-preEVI-II}
\end{align}
\end{lem}

\begin{rmk}
The final terms on the RHS's above are bounded above by $O(t^2)$, and in particular will disappear upon taking the limit $\limsup_{t \downarrow 0}\frac{1}{t}$. It may be possible to remove these errors pre-limit by making more careful choices throughout the proof, but it is not necessary here.
\end{rmk}

\begin{proof}
Consider the increasing (bijective) affine map of intervals $\xi_h: [s - h,t]\to [s,t]$, and write the constant $\lambda := 1 - \xi_h'  \equiv h/(t - s + h)$. We now define a curve of smooth measures $\nu_{r, h}:= \hat{P}_{\xi_h(r),r}[\nu_{\xi_h(r)}]$ for $r \in [s - h,t]$, whose endpoints $\nu_{s,h} = \hat P_{s,s - h}[\nu_{s}]$ and $\nu_{t,h} = \nu_{t}$ are those appearing in \eqref{eq:L0-preEVI-I}. Write $(\phi_r)_{r \in [s - h,t]}$ the Kantorovich potentials for the pair $(\nu_{s - h,h}, \nu_{t,h})$ given by Corollary \ref{cor:dynamic-potentials}. We now compute time derivatives of entropy and the pairing with $\phi_r$ along the curve $(\nu_{r,h})_r$. Let us point out here that there is no need to regularize the integrand of the entropy, since $\nu_{r,h}$ has densities $\rho_{r,h} := P^*_{\xi_h(r), r}[\rho_{\xi_h(r)}]$, which are bounded away from $0$ by the strong maximum principle.

\begin{align}
\partial_r \cE(\nu_{r, h}) &= \partial_r \int_M \rho_{r,h}\ln(\rho_{r,h})\,dV_r \\
&= \int_M [(1 + \ln(\rho_{r,h}))\partial_r\rho_{r,h} -  \rho_{r,h}\ln(\rho_{r,h})S_r]\,dV_r \\
&= - \int_M (1 + \ln(\rho_{r,h}))\Delta_r \,d\nu_{r,h} + \int_M S_r \,d\nu_{r,h} \\
& \qquad +(1 - \lambda)\int_M (1 + \ln(\rho_{r,h}))\,d\hat{P}_{\xi_h(r), r}[\left.(\partial_u + \Delta)\right|_{u = \xi_h(r)}\nu_u]\\
&= (1 - \lambda)\int_M P_{\xi_h(r),r}\left[\frac{\lambda}{1 - \lambda}|\nabla \ln (\rho_{r,h})|^2-\Delta_r \ln(\rho_{r,h})\right]\,d\nu_{\xi_h(r)}\\
&\qquad + \int_M S_r \,d\nu_{r,h}+(1 - \lambda)\int_M(\Delta_{\xi_h(r)} P_{\xi_h(r),r}[\ln(\rho_{r,h})] )\,d\nu_{\xi_h(r)}  \\
& \qquad + (1 - \lambda) \int_M P_{\xi_h(r),r}[1 + \ln(\rho_{r,h})]\left.\partial_u\right|_{u = \xi_h(r)} \,d\nu_{u}\,.
\end{align}

In the last equality we used duality of the heat flow and integration by parts, in particular rewriting the first term in a strange way (note that $1=(1-\lambda)(\frac{\lambda}{1-\lambda}+1)$). We also used that the forward heat flow preserves constant functions to simplify $1 + \ln(\rho_{r,h})$ whenever it is differentiated, which we will continue to do without further comment. We similarly compute for a.e. time, differentiating under the integral sign by time-Lipshitzianity of $\phi_r$ and using the equation \eqref{eq:L0-HJ} solved by $\phi_r$ for a.e. $x \in M$

\begin{align}
\partial_r \int_M \phi_r\, d\nu_{r,h} &= \int_M\left(-\frac{1}{2}|\nabla \phi_r|^2 + \frac{1}{2}S_r\right) \, d\nu_{r,h} + \int_M \phi_r (\partial_r \rho_{r,h} - S_r\rho_{r,h})\,dV_r \\
&= (1 - \lambda)\int_M P_{\xi_h(r), r}\left[\frac{\lambda}{1 - \lambda}\La\nabla \phi_r, \nabla \ln (\rho_{r,h})\Ra - \Delta_r \phi_r +\frac{1}{2} S_{r} -\frac{1}{2}|\nabla \phi_r|^2\right]\,d\nu_{\xi_h(r)} \\
&\qquad + (1 - \lambda)\int_M(\Delta_{\xi_h(r)} P_{\xi_h(r), r}[\phi_r] )\,d\nu_{\xi_h(r)} + \lambda\int_M \frac{1}{2}S_r - \frac{1}{2}|\nabla \phi_r|^2\,d\nu_{r,h} \\
&\qquad + (1 - \lambda) \int_M P_{\xi_h(r),r}[\phi_r]\left.\partial_u\right|_{u = \xi_h(r)} \,d\nu_{u}\,.
\end{align}

Amongst the terms of the previous two RHS's, after multiplying the first equation by $-\frac{\lambda}{1-\lambda}$, we notice the following perfect square:

\begin{align}
\frac{1}{2}\left|\nabla \left(\phi_r - \frac{\lambda\ln(\rho_{r,h})}{1 - \lambda}\right)\right|^2 &= \frac{\lambda^2}{2(1 - \lambda)^2}|\nabla \ln(\rho_{r,h})|^2 - \frac{\lambda}{1 - \lambda}\La \nabla \phi_r,\nabla \ln(\rho_{r,h})\Ra + \frac{1}{2} |\nabla\phi_r|^2 \\
&\left(\leq \frac{\lambda^2}{(1 - \lambda)^2}|\nabla \ln(\rho_{r,h})|^2 - 
\frac{\lambda}{1 - \lambda}\La \nabla \phi_r,\nabla \ln(\rho_{r,h})\Ra + \frac{1}{2} |\nabla\phi_r|^2\right)\,.
\end{align}

In order to lighten notation, we introduce the auxiliary function $f_r := \phi_r - \lambda\ln(\rho_{r,h})/(1 - \lambda)$. We may now integrate in time and add the first two inequalities (forming the perfect square previously observed) to obtain:

\begin{align}\label{eq:EVI-before-grad-est}
\left.\int_M \phi_r\, d\nu_{r,h} - \frac{\lambda}{1 - \lambda}\cE(\nu_{r,h})\right|_{s - h}^{t} &\leq \int_{s - h}^{t}\left((1 - \lambda)\int_M P_{\xi_h(r), r}\left[-\Delta_r f_r - \frac{1}{2}|\nabla f_r|^2 + \frac{1}{2}S_r\right]\,d\nu_{\xi_h(r)} \right.\\
&\qquad +(1 - \lambda) \int_M(\Delta_{\xi_h(r)} P_{\xi_h(r),r}[f_r] )\,d\nu_{\xi_h(r)} \\
&\qquad - \lambda\int_M \frac{1}{2}S_r + \frac{1}{2}|\nabla \phi_r|^2\,d\nu_{r,h} - \frac{\lambda^2}{1 - \lambda}\int_M S_r \,d\nu_{r,h} \\
&\qquad + \left.(1 - \lambda) \int_M P_{\xi_h(r),r}[f_r]\,\left.\partial_u\right|_{u = \xi_h(r)} \,d\nu_u \right)\,dr\,.
\end{align}

We further estimate the the final term of the integrand above by using the computation of \eqref{eq:L0-Lisini-application}, which holds for a general curve of smooth measures, to bound

\begin{align}
&(1 - \lambda) \int_M P_{\xi_h(r),r}[f_r]\,\left.\partial_u\right|_{u = \xi_h(r)} \,d\nu_u \\
&\leq \frac{1 - \lambda}{2} \int_M \Big(|\nabla P_{\xi_h(r),r}[f_r]|^2 - S_{\xi_h(r)}\Big)\,d\nu_{\xi_h(r)} + (1 - \lambda) (\dot{\nu}_{\xi_h(r)})_{L_0^{\xi_h(r)}}\,.
\end{align}

After implementing the above inequality in our estimate for $\left.\int_M \phi_r \,d\nu_{r,h} - \lambda\cE(\nu_{r,h})/(1 - \lambda) \right|_{s - h}^{t}$, we may use the assumed $L_0$-type gradient estimate \eqref{eq:GEL0-GE-to-EVI} to bound the terms

\begin{align}
&\phantom{-}(1 - \lambda)\int_M P_{\xi_h(r), r}\left[-\Delta_r f_r - \frac{1}{2}|\nabla f_r|^2 + \frac{1}{2}S_r\right]\,d\nu_{\xi_h(r)} \\
&- (1 - \lambda)\int_M\left(-\Delta_{\xi_h(r)} P_{\xi_h(r), r}[f_r] -\frac{1}{2} |\nabla P_{\xi_h(r),r}[f_r]|^2 + \frac{1}{2}S_{\xi_h(r)} \right)\,d\nu_{\xi_h(r)} \leq 0\,.
\end{align}

When the dust settles, we obtain 
\begin{align}
\left.\int_M \phi_r \,d\nu_{r,h} - \frac{\lambda}{1 - \lambda}\cE(\nu_{r,h}) \right|_{s - h}^{t} &\leq -\frac{\lambda}{2}\int_{s - h}^{t}\int_M S_r + |\nabla \phi_r|^2\,d\nu_{r,h} \,dr + (1 - \lambda)\int_{s - h}^{t}(\dot{\nu}_{\xi_h(r)})_{L_0^{\xi_h(r)}}\,dr \\
&\qquad - \frac{\lambda^2}{1 - \lambda}\int_{s - h}^{t}\int_M S_r \,d\nu_{r,h}\,dr\,.
\end{align}

Recalling now that the $\phi_r$ solve \eqref{eq:L0-HJ} and are $W_{L_0}$ Kantorovich potentials for any pair of measures on the curve $\nu_{r,h}$, we can rewrite the above as

\begin{align}
\left. -\frac{\lambda}{1 - \lambda}\cE(\nu_{r,h}) \right|_{s - h}^{t} + (1 + \lambda)W_{L_0^{s - h,t}}(\nu_{s - h,h}, \nu_{t,h}) &\leq (1 - \lambda)\int_{s - h}^{t}(\dot{\nu}_{\xi_h(r)})_{L_0^{\xi_h(r)}}\,dr\\
&\qquad- \frac{\lambda^2}{1 - \lambda}\int_{s - h}^{t}\int_M S_r\,d\nu_{r,h}\,dr\,.
\end{align}

After applying duality in the $O(\lambda^2)$ term and changing variables in the action term, this proves the first estimate of the lemma. The second follows in essentially the same way, after running the same argument with $\nu_r^h := \hat{P}_{\tilde\xi_h(r),r}[\nu_{\tilde\xi_h(r)}]$ for the increasing, affine, bijective map of intervals $\tilde{\xi}:[s, t - h] \to [s,t]$.
\end{proof}

We now take a limiting quotient of the action-entropy estimates obtained in the previous Lemma to obtain the pair of \eqref{eq:EVIL0-GE-to-EVI}s.

\begin{prop}\label{prop:L0GE-to-EVI} Assume that for any Lipschitz $v$, any times $[s,t] \subseteq I$, the gradient estimate
\begin{equation}
|\nabla P_{t,s}v|^2_t + 2\Delta_t P_{t,s}v - S_t \leq P_{t,s}(|\nabla v|^2_s + 2\Delta_s v - S_s)\label{eq:GEL0-GE-to-EVI}\tag{GE\({}_{L_0}\)}
\end{equation}
holds. Then, for any Borel probability measures $\mu,\nu \in \cP(M)$ and times $[s, t] \subseteq I$, the backward heat flow $\hat{P}$ satisfies the following $\mathrm{EVI}$s:

\begin{align}
-\partial_{b^-}^-W_{L_0^{b,t}}( \hat{P}_{s,b}\mu, \nu) & \leq \frac{1}{t - b}\left(\cE(\nu) - \cE(\hat{P}_{s,b}\mu) - W_{L_0^{b, t}}( \hat{P}_{s,b}\mu, \nu)\right) \,, \qquad\forall b \in (\inf I,s]\,.\\
-\partial_{a^-}^- W_{L_0^{s,a}}(\mu,\hat{P}_{t,a}\nu) & \leq \frac{1}{a - s}\left(\cE(\mu) - \cE(\hat{P}_{t,a}\nu) + W_{L_0^{s,a}}(\mu, \hat{P}_{t,a}\nu)\right) \,, \qquad\forall a \in (s, t]\,. \label{eq:EVIL0-GE-to-EVI}\tag{EVI\({}_{L_0}\)}
\end{align}
\end{prop}

\begin{proof}
We prove the first inequality; the second follows in essentially the same way. First, let $(\nu_t)_{t \in [b,t]}$ be a $W_{L_0^{b,t}}$ geodesic (see Definition \ref{df:WL-geodesic}) connecting $\hat{P}_{s,b}\mu$ and $\nu$. Next, apply Lemma \ref{lem:wass-geodesic-smoothing} for $\eps > 0$ to obtain a curve of smooth measures $(\nu_r^\eps)_{r \in [b - \eps ,t -\eps]}$. Lemma \ref{lem:L0-action-entropy-est} can now be applied to the smoothed curve for small $h > 0$, yielding

\begin{align}
\lambda \cE(\hat P_{b - \eps, b - \eps - h}[\nu_{b - \eps}^\eps]) + (1 - \lambda^2)W_{L_0^{b - \eps - h, t - \eps}}(\hat P_{b - \eps, b - \eps - h}[\nu_{b - \eps}^\eps],\nu_{t - \eps}^\eps) &\leq \lambda\cE(\nu^\eps_{t - \eps}) + (1 - \lambda)\int_{b - \eps}^{t - \eps}(\dot\nu^\eps_r)_{L_0^r}\,dr + C\lambda^2\,,
\end{align}

where $\lambda = h/(t  + h - b)$. We would like to remove the $\eps$'s in the above inequality; indeed by Lemma \ref{lem:wass-geodesic-smoothing}, the action and second entropy term above converge to the corresponding quantities for $(\nu_{r})_{r \in [b,t]}$ as $\eps \to 0$. The first entropy term converges because weak convergence $\nu_{b - \eps}^\eps \weak \nu_b$ implies smooth convergence of positive densities $P_{b - \eps,b - \eps - h}^*\rho_{b - \eps}^\eps \to P_{b,b - h}^*\rho_b$ for any fixed $h > 0$, and thus convergence of $\cE$. The $W_{L_0}$ term converges because of the weak convergence $\hat P_{b - \eps, b - \eps - h}[\nu_{b - \eps}^\eps] \weak \hat P_{b, b - h}[\nu_{b}]$ and $\nu_{t - \eps}^\eps \weak \nu_t$, followed by an application of \cite[Thm.~5.20]{Vil09}. Thus, the action-entropy inequality also holds for the $W_{L_0}$-geodesic $(\nu_r)_{r \in [b,t]}$. 

Since the action corresponds to the $W_{L_0}$ distance between endpoints for this curve, the inequality can be rewritten (after unpacking notation and rearranging terms)

\begin{align}
W_{L_0^{b - h,t}}(\hat{P}_{b,b-h}\nu_b, \nu_t) - W_{L_0^{b,t}}(\nu_b,\nu_t) \leq \frac{h}{(t + h - b)}(\cE(\nu_t) - \cE(\hat{P}_{b,b-h}\nu_b) - W_{L_0^{b,t}}(\nu_b,\nu_t)) + Ch^2\,.
\end{align}

Taking $\limsup_{h \downarrow 0}\frac{1}{h}$ on both sides gives the desired conclusion (with $\lim_{h \to 0}\cE(\hat{P}_{b,b-h}\nu_b) = \cE(\nu_b)$ provided by the proof of entropy convergence in Lemma \ref{lem:wass-geodesic-smoothing}).
\end{proof}

\begin{rmk}
    It is worth noting that in other contexts $\mathrm{EVI}$s are often stated as inequalities for classical (as opposed to Dini) time derivatives and for a.e. (as opposed to every) time. In the static setting, the function $t\mapsto W_{d^2}(\eta, \hat P_t \mu)$ is absolutely continuous and the two conditions coincide, with the absolute continuity ultimately stemming from the a.c. regularity of the curve $[0,\infty)\ni t\mapsto \hat P_t\mu \in (\cP(M), W_{d^2})$. We have made no attempt to develop a theory of a.c. curves of measures adapted to the $W_{L_{0/\pm}}$ transportation costs in this paper, though such considerations might be useful if one wished to venture beyond the smooth setting.
\end{rmk}

To show that \eqref{eq:EVIL0-GE-to-EVI} implies \eqref{eq:ECL0-EC-to-WC}, we now follow the argument of \cite[Thm.~3.2]{DanSav08}, i.e. make good on \cite[Rmk.~8]{Lot09}.

\begin{prop}\label{prop:L0-entropy-cvx}
Suppose that for any Borel probability measures $\mu,\nu \in \cP(M)$ and times $[s, t] \subseteq I$, the backward heat flow $\hat{P}$ satisfies the $\mathrm{EVI}$s \eqref{eq:EVIL0-GE-to-EVI}.

Then for times $\bar r \in [s,t]\subset I$ and a $W_{L_0}$ geodesic $(\mu_r)_{r \in [s,t]} \subseteq \cP(M)$, the following entropy convexity statement holds:

\begin{equation}\label{eq:ECL0-EVI-to-EC
}
\frac{t - \bar{r}}{t - s}\cE(\mu_{s}) + \frac{\bar{r} - s}{t - s}\cE(\mu_{t}) - \cE(\mu_{\bar{r} }) + \frac{t - \bar{r}}{t - s}W_{L_0^{s,\bar{r}}}(\mu_{s}, \mu_{\bar{r}}) - \frac{\bar{r} - s}{t - s}W_{L_0^{\bar{r},t}}(\mu_{\bar{r}}, \mu_{t}) \geq 0\,.
\end{equation}
\end{prop}

\begin{rmk}\label{rmk:L0-entropy-alt-form}
The $W_{L_0}$ terms of \eqref{eq:ECL0-EVI-to-EC
} might look a bit strange, but they are simply an alternate way of writing Lott's convex quantity \cite[Prop.~11]{Lot09} or equivalently \eqref{eq:ECL0-EC-to-WC}. Indeed, if one writes $(\phi_r)_{r \in [s,t]}$ the Kantorovich potentials for the pair of measures $(\mu_{s},\mu_{t})$ given by Corollary \ref{cor:dynamic-potentials}, then

\begin{equation}
W_{L_0^{\bar{s},\bar{t}}}(\mu_{\bar{s}},\mu_{\bar{t}}) = \int_{M}\phi_{\bar{t}}\,d\mu_{\bar{t}} - \int_M\phi_{\bar{s}}\,d\mu_{\bar{s}}\,,\qquad \forall s \leq \bar{s} \leq \bar{t} \leq t\,.
\end{equation}

One then has the following identity relating convex combinations of $ \int_M -\phi_r \,d\mu_r$ and the last two terms of \eqref{eq:ECL0-EVI-to-EC
}, by first rearranging terms and then applying the above identity:

\begin{align}
&\phantom{=}\frac{t - \bar{r}}{t - s}\int_M(-\phi_{s})\,d\mu_{s} + \frac{\bar{r} - s}{t - s}\int_M (-\phi_{t})\,d\mu_{t} - \int_M (-\phi_{\bar{r}})\,d\mu_{\bar{r}} \\
&= \frac{t - \bar{r}}{t - s}\left(\int_M \phi_{\bar{r}}\,d\mu_{\bar{r}} - \int_M\phi_{s}\,d\mu_{s}\right) + \frac{\bar{r} - s}{t - s}\left(\int_M \phi_{\bar{r}}\,d\mu_{\bar{r}} - \int_M\phi_{t}\,d\mu_{t}\right) \\
&= \frac{t - \bar{r}}{t - s}W_{L_0^{s, \bar{r}}}(\mu_{s},\mu_{\bar{r}}) - \frac{\bar{r} - s}{t - s}W_{L_0^{\bar{r},t}}(\mu_{\bar{r}},\mu_{t})\,,\qquad \bar{r}\in [s,t]\,.
\end{align}

\end{rmk}

\begin{proof}
Let $(\mu_r)_{r \in [s,t]}$ be a Wasserstein geodesic, and we may assume that $\cE(\mu_{s}),\, \cE(\mu_{t})< \infty$ (otherwise the desired conclusion follows immediately). Picking an intermediate time $\bar{r} \in [s,t]$, we take a sum of the assumed $\mathrm{EVI}$s for the pairs of measures $(\mu_{\bar{r}},\mu_{t})$ resp. $(\mu_{s},\mu_{\bar{r}})$ at the initial flow time, normalized by the factor $(t - \bar{r})(\bar{r} - s)/(t - s)$:

\begin{align}
&\phantom{=} \frac{t - \bar{r}}{t - s}\cE(\mu_{s}) + \frac{\bar{r} - s}{t - s}\cE(\mu_{t}) - \cE(\mu_{\bar{r}}) + \frac{t - \bar{r}}{t - s}W_{L_0^{s,\bar{r}}}(\mu_{s}, \mu_{\bar{r}}) - \frac{\bar{r} - s}{t - s}W_{L_0^{\bar{r},t}}(\mu_{\bar{r}}, \mu_{t})\\
&\geq \frac{(\bar{r} - s)(t - \bar{r})}{(t - s)}\Big(-\left.\partial_{r}^-\right|_{r = \bar{r}^-} W_{L_0^{s,r}}(\mu_{s},\hat{P}_{\bar{r},r}\mu_{\bar{r}}) -\left.\partial_{r}^-\right|_{r = \bar{r}^-} W_{L_0^{r,t}}(\hat{P}_{\bar{r},r}\mu_{\bar{r}},\mu_{t})\Big)\,.
\end{align}

The (negative) Dini derivatives on the RHS above can be lower bounded by (the negative of) a $\liminf$ of difference quotients for an arbitrary sequence converging to $\bar{r}$. Indeed, let $r_i \uparrow \bar{r}$, then  

\begin{align}
&-\left.\partial_{r}^-\right|_{r = \bar{r}^-} W_{L_0^{s,r}}(\mu_{s},\hat{P}_{\bar{r},r}\mu_{\bar{r}}) -\left.\partial_{r}^-\right|_{r = \bar{r}^-} W_{L_0^{r,t}}(\hat{P}_{\bar{r},r}\mu_{\bar{r}},\mu_{t}) \\
&\geq -\liminf_{i \to \infty}\frac{1}{\bar{r} - r_i}\Big((W_{L_0^{s,\bar{r}}}(\mu_{s},\mu_{\bar{r}}) - W_{L_0^{s,r_i}}(\mu_{s},\hat P_{\bar{r},r_i}\mu_{\bar{r}})) + (W_{L_0^{\bar{r},t}}(\mu_{\bar{r}},\mu_{t})-W_{L_0^{r_i,t}}(\hat P_{\bar{r},r_i}\mu_{\bar{r}},\mu_{t}))\Big)\,.
\end{align}

The RHS above is a limit of difference quotients, each of which is $\geq 0$ by applying the $W_{L_0}$ triangle inequality for the triplet of measures $(\mu_{s},\hat P_{\bar{r},r_i}\mu_{\bar{r}}, \mu_{t})$, along with the triangle {\it equality} for the triplet of points $(\mu_{s},\mu_{\bar{r}},\mu_{t})$ lying on the $W_{L_0}$ geodesic $(\mu_r)_{r \in [s,t]}$.

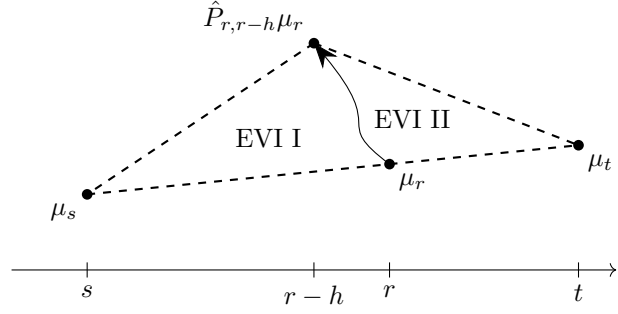
\begin{figure}[H]
\centering
\begin{tikzpicture}

  \coordinate (P) at (3,2);
  \coordinate (mu-s) at (0,0);
  \coordinate (mu-r) at (4,0.4);
  \coordinate (mu-t) at (6.5,0.65);

  \draw[<-, >={Stealth[scale=2]}] (P) .. controls (4,0.8) and (3.2,1) .. (mu-r);
  \draw[thick, dashed] (P) -- (mu-s);
  \draw[thick, dashed] (mu-s) -- (mu-r);
  \draw[thick, dashed] (mu-r) -- (mu-t);
  \draw[thick, dashed] (P) -- (mu-t);
  \fill (P) circle (2pt) node[above left] {$\hat P_{r,r - h}\mu_r$};
  \fill (mu-s) circle (2pt) node[below left] {$\mu_s$};
  \fill (mu-r) circle (2pt) node[below right] {$\mu_r$};
  \fill (mu-t) circle (2pt) node[below right] {$\mu_t$};
  
  \node at (2.4, 0.75) {$\mathrm{EVI}$ I};
  \node at (4.3, 1) {$\mathrm{EVI}$ II};

  \draw[->] (-1, -1) -- (7, -1);
  \foreach \x/\xtext in {0/{$s$}, 3/{$r - h$}, 4/{$r$}, 6.5/{$t$}}
    \draw (\x, -1.1) -- (\x, -0.9) node[below=5pt] {\xtext};

\end{tikzpicture}
\caption{Appending the two \eqref{eq:EVIL0-GE-to-EVI}s returns entropy convexity \eqref{eq:ECL0}.}
\end{figure}
\end{proof}

\begin{rmk}\label{rmk:EVItoWC}
    The \eqref{eq:EVIL0-GE-to-EVI}s can also be used to obtain \eqref{eq:WCL0} in a few lines, without needing to pass through \eqref{eq:ECL0}. Indeed they imply that the function $h\mapsto W_{L_0^{s-h,t-h}}(\hat P_{s,s-h}\mu,\hat P_{t,t-h}\nu)$ is non increasing: for all $h\geq 0$
\begin{align}
    \partial_{h^+}^+ W_{L_0^{s-h,t-h}}(\hat P_{s,s-h}\mu,\hat P_{t,t-h}\nu) 
    &= \partial_{k^+}^+ W_{L_0^{s-k,t-h}}(\hat P_{s,s-k}\mu,\hat P_{t,t-h}\nu)|_{k=h} + \partial_{k^+}^+ W_{L_0^{s-h,t-k}}(\hat P_{s,s-h}\mu,\hat P_{t,t-k}\nu)|_{k=h} \\
    &= -\partial_{r^-}^- W_{L_0^{r,t-h}}(\hat P_{s,r}\mu,\hat P_{t,t-h}\nu)|_{r=s-h} - \partial_{r^-}^- W_{L_0^{s-h,r}}(\hat P_{s,s-h}\mu,\hat P_{t,r}\nu)|_{r=t-h} \\
    &\overset{\eqref{eq:EVIL0-GE-to-EVI}}{\leq} 0 \,.
\end{align}
\end{rmk}

\section{Inequalities for the \texorpdfstring{$L_-$}{L-} distances}\label{s:L-}

In this section we give the proof of Theorem \ref{thm:L-charact}, again divided into multiple subsections. We also state an analogous theorem tailored to the $L_+$ distance. Throughout the section, we work with an arbitrary but fixed smooth family of closed Riemannian manifolds parametrized by backwards time $(M^n, g_\tau)_{\tau \in I}$ of dimension $\dim M = n$, defined for some interval $I \subseteq (0,\infty)$, with the notation of Subsection \ref{subsec:t-dep-mflds} still in force.

\subsection{The \texorpdfstring{$\cD$}{D}-condition in dimension \texorpdfstring{$n$}{n} and the \texorpdfstring{$L_-$}{L-} Bochner formula}\label{ss:L-DBoch}

\begin{prop}\label{prop:L-Bochner}
For any $\tau\in I$, any smooth solution $v$ of the heat flow at $\tau$, the following identity holds:
\begin{equation}
-(\ptau + \Delta)\left(|\nabla v|^2 - 2\tau\Delta v - \tau^2 S + \frac{n}{2}\tau\right) = -2\left|\tau\mathcal{S} - \frac{g}{2} + \Hess v\right|^2 -\tau^2\mathcal{D}(\nabla v/\tau)\,.
\end{equation}
In particular, $\cD \geq 0$ iff 

\begin{equation}\label{eq:BochnerL--id-to-ineq}
-(\ptau + \Delta)\left(|\nabla v|^2 - 2\tau\Delta v - \tau^2 S + \frac{n}{2}\tau\right) \leq 0\,, \tag{B\({}_{L_-}\)}
\end{equation}
for any $\tau \in I$, any smooth solution $v$ of the heat flow at $\tau$.
\end{prop}

\begin{proof}
Let us recall the following identities:
\begin{align}
&-(\ptau + \Delta)|\nabla v|^2 = -2|\Hess v|^2+(2\mathcal{S}-2\Ric)(\nabla v,\nabla v), \\ & -(\ptau + \Delta)\Delta v = 2\La \mathcal{S}, \Hess v\Ra + \La 2\div \mathcal{S}-\nabla S,\nabla v \Ra, \\
& -(\ptau + \Delta)S = 2|\mathcal{S}|^2 -(\ptau S+\Delta S+2|\mathcal{S}|^2) 
\end{align}
Note that we have omitted the (zero) terms that involve the heat operator applied to $v$. One may now compute
\begin{align}
-(\ptau + \Delta)\left(|\nabla v|^2 - 2\tau\Delta v - \tau^2 S +\frac{n}{2}\tau \right) &= -2|\Hess v|^2 + 2\Delta v - 4\tau \langle\mathcal{S},\Hess v\rangle + 2\tau S - 2\tau^2 |\mathcal{S}|^2  -\frac{n}{2} \\
& \qquad  + (2\mathcal{S}-2\Ric)(\nabla v,\nabla v) -2\tau \La 2\div \mathcal{S}-\nabla S,\nabla v \Ra \\
&\qquad +\tau^2(\ptau S+\Delta S+2|\mathcal{S}|^2)  \\
&= -2\left( |\Hess v|^2 +\tau^2 |\mathcal{S}|^2 + \frac{n}{4} +2\tau \langle\mathcal{S},\Hess v\rangle - \Delta v -\tau S \right) \\ 
& \qquad -\tau^2 \mathcal{D}(\nabla v/\tau) \\
&= -2\left|\tau\mathcal{S} - \frac{g}{2} +  \Hess v \right|^2 -\tau^2 \mathcal{D}(\nabla v/\tau)\,.
\end{align}

Given this formula, it is clear that $\mathcal{D}\geq 0$ implies \eqref{eq:BochnerL--id-to-ineq}. Vice versa, assume that every smooth solution $v$ to the heat flow at $\tau \in I$ satisfies \eqref{eq:BochnerL--id-to-ineq}. Then one can argue as in the proof of Proposition \ref{prop:DcondiffBochner}, picking initial data $v_{\tau}$ for a heat flow such that $\tau \cS_p - \frac{1}{2}g_p + (\Hess v)_p = 0$ and $(\nabla v)_p/\tau = V$ for arbitrary $V \in T_pM$ to conclude that $\cD_{\tau,p} \geq0$.
\end{proof}

\begin{rmk}
    If we assume, for some $N>0$, that
    \begin{equation}
-(\ptau + \Delta)\left(|\nabla v|^2 - 2\tau\Delta v - \tau^2 S + \frac{N}{2}\tau\right) \leq 0 \,,
\end{equation}
we get
    \begin{align}
        2\left|\tau\mathcal{S} - \frac{g}{2} +  \Hess v \right|^2 +\tau^2 \cD(\nabla v/\tau) + \frac{N-n}{2} \geq0\,,
    \end{align}
    which only yields (choosing $v$ as above)
    \begin{align}
        \tau^2 \cD(X/\tau) + \frac{N-n}{2} \geq0 \qquad \text{for all $(p,X,\tau)\in TM \times I$.}
    \end{align}
    In order to decouple the $\cD$ condition (in particular, its zero-order in $v$ part) and the dimension bound one may consider time-translations, see Section \ref{s:dimL0}.
\end{rmk}

\subsection{\texorpdfstring{$L_-$}{L-} Bochner formula and gradient estimates}\label{ss:L-BochGE}

\begin{prop}\label{prop:L-GradEst}
The Bochner inequality for any $\tau \in I$, any smooth solution $v$ to the heat equation at $\tau$ \noeqref{eq:BochnerL--Bochner-to-GE}
\begin{equation}\label{eq:BochnerL--Bochner-to-GE}\tag{B\({}_{L_-}\)}
-(\partial_\tau + \Delta)\left(|\nabla v|^2_\tau - 2\tau\Delta v - \tau^2 S_\tau + \frac{n}{2}\tau\right) \leq 0\,,
\end{equation}
is equivalent to the gradient estimate
\begin{align}\label{eq:GEL--Bochner-to-GE}
    |\nabla P_{\sigma,\tau}v|^2-2\sigma \Delta P_{\sigma,\tau}v-\sigma^2S_{\sigma} \leq P_{\sigma,\tau}\left(|\nabla v|^2-2\tau\Delta v - \tau^2S_{\tau}\right) + \frac{n}{2}(\tau-\sigma) \tag{GE\({}_{L_-}\)}
\end{align}
for all $[\sigma,\tau]\in I$ and all Lipschitz $v: M \to \R{}$, in which case we define $P_{\sigma,\tau} \Delta_{\tau} v (x) := -\int_M \La \nabla_{y} \rho_{\sigma,\tau}(x,y), \nabla v(y)\Ra_{\tau} \,dV_{\tau}(y)$, where $\rho$ is the heat kernel.
\end{prop}

\begin{rmk}\label{GEL-and-max-princ}
    The estimate with $v=0$, along with the maximum principle yields 
    \begin{align}
        S_{\sigma} \geq \frac{1}{\sigma^2}\left(\tau^2\underline{S}_\tau -\frac{n}{2}(\tau-\sigma)\right) \geq \frac{1}{\sigma^2}\left(\tau^2\underline{S}_\tau -\frac{n}{2}\tau\right)
    \end{align}
    where $\underline{S}_\tau=\inf_MS_\tau$. Therefore, if $\tau^2\underline{S}_\tau -\frac{n}{2}\tau >0$, then $S$ must blow up at time $\sigma=0$. In particular, if  $\underline{S}_\tau>0$, then the flow cannot exist before $\tau-\frac{n}{2\underline{S}_\tau}$, recovering a classical statement for the Ricci flow, see e.g. \cite[Cor.~3.2.4]{topping2006lectures}.
\end{rmk}

The proof of Proposition \ref{prop:L-GradEst} is almost identical to that of Proposition \ref{prop:L0gradestimate} (taking care of the sign changes when passing to reversed time). We therefore omit it. It is also convenient to record the following ``gradient estimate with parameter'', by exploiting the commutator of time-translation and parabolic rescaling. It will be used only later, in Subsection \ref{ss:GEWCdualitydimL0}.

\begin{cor}\label{cor:L_GradEst_param}
Suppose that the Gradient estimate \eqref{eq:GEL--Bochner-to-GE} holds for any shift $(M,g_{\tau - T})_{\tau \in (I + T)\cap (0,\infty)}$ of the flow $(M,g_\tau)_{\tau \in I}$. For a Lipschitz function $w: M \to \R{}$, times $[\sigma, \tau] \in I$, and a parameter $0 < \lambda < \tau/\sigma$, the following gradient estimate holds
\begin{align}
&(|\nabla P_{\sigma,\tau}[w]|^2 - 2\lambda\sigma \Delta P_{\sigma,\tau}[w] - (\lambda\sigma)^2S_{\sigma}) \\
&- (P_{\sigma,\tau}[|\nabla w|^2] - 2\tau P_{\sigma,\tau}[\Delta w] - \tau^2P_{\sigma,\tau}[S_{\tau}
]) \leq \frac{n}{2}\frac{(\tau - \lambda \sigma)^2}{\tau - \sigma}
\end{align}
\end{cor}

\begin{proof}
We begin by writing the quantity from the LHS of \eqref{eq:GEL--Bochner-to-GE} as follows, for some constant $c > 0$ to be determined later:

\begin{align}
&\phantom{=}P_{\sigma,\tau}[|\nabla w|^2] - 2\tau P_{\sigma,\tau}[\Delta w] - \tau^2P_{\sigma,\tau}[S_{\tau}
] \\
&= \frac{1}{c^2}\left[P_{\sigma,\tau}[|\nabla (cw)|^2] - 2c\tau P_{\sigma,\tau}[\Delta (cw)] - (c \tau)^2P_{\sigma,\tau}[S_{\tau}
]\right] \\
&= \frac{1}{c^2}\left[P_{\sigma,\tau}[|\nabla (cw)|^2] - 2\Big(\tau + \tau(c - 1)\Big)P_{\sigma,\tau}[\Delta (cw)] - \Big(\tau + \tau(c - 1)\Big)^2P_{\sigma,\tau}[S_{\tau}
]\right]\,.
\end{align}

We now apply the gradient estimate \eqref{eq:GEL--Bochner-to-GE} but for the shifted flow $(M,g_{\tau - T})_{\tau \in I + T}$ with $T := \tau(c - 1)$, to the bracketed expression on the RHS to obtain

\begin{align}
&\phantom{\geq}P_{\sigma,\tau}[|\nabla w|^2] - 2\tau P_{\sigma,\tau}[\Delta w] - \tau^2 P_{\sigma,\tau}[S_{\tau}
] \\
&\geq \frac{1}{c^2}\left[|\nabla P_{\sigma,\tau}[cw]|^2 - 2\Big(\sigma + \tau(c - 1)\Big)\Delta P_{\sigma,\tau}[cw] - \Big(\sigma + \tau(c - 1)\Big)^2 S_{\sigma}
 -\frac{n}{2}(\tau - \sigma)\right]
\end{align}

It is at this point that we choose $c > 0$, which we do so that $\big(\sigma + \tau(c - 1)\big)/c = \lambda\sigma$--note that this is only possible if $0 < \lambda < \tau/\sigma$ and ensures that $\sigma +T = c\lambda\sigma > 0$. In particular, the above inequality simplifies to

\begin{align}
&\phantom{\geq}P_{\sigma,\tau}[|\nabla w|^2] - 2\tau P_{\sigma,\tau}[\Delta w] - \tau^2P_{\sigma,\tau}[S_{\tau}
] \\
&\geq \left[|\nabla P_{\sigma,\tau}[w]|^2 - 2\lambda\sigma \Delta P_{\sigma,\tau}[w] - (\lambda\sigma)^2S_{\sigma}
\right] - \frac{n}{2}\left(\frac{\tau - \lambda\sigma}{\tau - \sigma}\right)^2(\tau - \sigma)\,.
\end{align}

\end{proof}

\subsection{Duality of \texorpdfstring{$L_-$}{L-} gradient estimates and \texorpdfstring{$L_-$}{L-}-Wasserstein contraction}\label{ss:GEWCdualityL-}

Before proving the equivalence, we fix some notation and perform a couple computations. Let us start by recording the gradient estimate \eqref{eq:GEL--Bochner-to-GE} written in terms of times $(\tau,\alpha\tau)$:
    \begin{align}\label{eq:alphaGE}
        |\nabla P_{\tau,\alpha\tau}v|^2-2\tau\Delta P_{\tau,\alpha\tau}v-\tau^2 S_{\tau} \leq P_{\tau,\alpha\tau}\left(|\nabla v|^2-2\alpha\tau\Delta v - \alpha^2\tau^2S_{\alpha\tau}\right) + \frac{n}{2}(\alpha\tau-\tau)\,.
    \end{align}
In this subsection we will work with the Lagrangians $L_-$, $\tilde L_-$ defined in \ref{df:L-}. We compute the Hamiltonian associated to $\tilde L_-$ given by the Legendre transform
\begin{align}
    \tilde{H}^{\sigma,\tau}(v,p,\eta) &= \left( (\sqrt{\tau}-\sqrt{\sigma})\sqrt{\eta}\left(\frac{1}{2}|v|^2+\frac{1}{2}S(p,\eta)\right) \right)^* \\
    &= \left( (\sqrt{\tau}-\sqrt{\sigma})\sqrt{\eta}\frac{1}{2}|v|^2 \right)^* - \left( (\sqrt{\tau}-\sqrt{\sigma})\sqrt{\eta}\frac{1}{2}S(p,\eta) \right) \\
    &= \frac{1}{(\sqrt{\tau} -\sqrt{\sigma})\sqrt{\eta}}\frac{1}{2}|v|^2 - (\sqrt{\tau}-\sqrt{\sigma})\sqrt{\eta}\frac{1}{2}S(p,\eta).
\end{align}

Applying now Proposition \ref{prop:HLsemigroup}, for any Lipschitz $\phi: M \to \R{}$ we have that the Hopf-Lax semigroup $Q$ associated to $\tilde L$ satisfies
\begin{align}\label{HLforRescaledL-}
    \partial_\eta Q^{\alpha\sigma,\alpha \eta}\phi 
    &= -\alpha \tilde{H}^{\alpha\sigma,\alpha \eta}(d Q^{\alpha\sigma,\alpha \eta}\phi,\phi,\alpha \eta) \\
    &= -\alpha \left(\frac{1}{(\sqrt{\alpha\tau} -\sqrt{\alpha\sigma})\sqrt{\alpha \eta}}\frac{1}{2}|\nabla Q^{\alpha\sigma,\alpha \eta} \phi|^2 - (\sqrt{\alpha\tau}-\sqrt{\alpha\sigma})\sqrt{\alpha \eta}\frac{1}{2}S_{\alpha \eta} \right) \\
    &= -\frac{1}{(\sqrt{\tau} -\sqrt{\sigma})\sqrt{\eta}}\frac{1}{2}|\nabla Q^{\alpha\sigma,\alpha \eta}\phi|^2 + \alpha^2(\sqrt{\tau} -\sqrt{\sigma}) \sqrt{\eta}\frac{1}{2}S_{\alpha \eta}\,.
\end{align}

We now proceed to the first direction of the equivalence. We give a slightly different proof from the $L_0$ case (and following \cite{kuwadaduality} more closely). In this case we do not pass through the Hamilton-Jacobi preservation (as done in the $L_0$ case, Lemma \ref{lem:heat-flow-preserves-HJ}), which we nevertheless state later for completeness (Proposition \ref{prop:heat-flow-preserves-L_-HJ}). This is purely a presentation choice, and one could exchange the approaches.

\begin{prop}\label{prop:WL-contraction}
Suppose that for a Lipschitz $v: M \to \R{}$ and times $[\sigma,\tau]\subseteq I$, the following gradient estimate holds
\begin{align} \label{eq:GEL--GE-to-WC}
    |\nabla P_{\sigma,\tau}v|_\sigma^2-2\sigma \Delta_\sigma P_{\sigma,\tau}v-\sigma^2S_{\sigma} \leq P_{\sigma,\tau}\left(|\nabla v|_\tau ^2 - 2\tau\Delta_\tau v - \tau ^2S_{\tau}\right) + \frac{n}{2}(\tau-\sigma) \tag{GE\({}_{L_-}\)}\,.
\end{align}
Then for any times $[\sigma,\tau] \in I$ and factor $1 < \alpha < \sup I/\tau$, the following $W_{L_-}$ contraction along the adjoint heat flow for any $\mu,\nu \in \cP(M)$ holds
    \begin{align} \label{eq:WCL--GE-to-WC}
            W_{\tilde L_-^{\alpha\sigma,\alpha\tau}}(\hat P_{\sigma,\alpha\sigma}\mu, \hat P_{\tau,\alpha\tau}\nu) \leq W_{\tilde L_-^{\sigma,\tau}}(\mu,\nu) + (\sqrt{\tau}-\sqrt{\sigma})^2\frac{n}{2}(\alpha-1)\,.\tag{WC\({}_{L_-}\)}  
    \end{align}
\end{prop}

\begin{rmk}\label{rmk:Kuwada-contraction}
This $W_{\tilde L_-}$ contraction recovers the classical Wasserstein contraction at different times on manifolds with nonnegative Ricci curvature, see \cite{kuwadaspace},  which also holds on $\mathsf{RCD}(0,N)$ spaces, \cite{eks}. Indeed, when we are on a static space $(M,g)$, $\tilde L_-((x,\sigma),(y,\tau))=\frac{1}{4}d^2(x,y)$. Moreover the semigroup is time independent: $\hat P_{\sigma,\tau}=\hat P_{\tau-\sigma}$. Denoting by $W_2$ the Wasserstein distance induced by $d_g^2$, and setting $s=\sigma(\alpha-1)$, $t=\tau(\alpha-1)$, \eqref{eq:WCL--GE-to-WC} becomes
    \begin{align}
    W_2( \hat{P}_s\mu, \hat{P}_t\nu) \leq W_2(\mu,\nu) + 2n(\sqrt{t}-\sqrt{s})^2,
\end{align}
which is \cite[Eq.~(1.4)]{kuwadaspace} with $K=0$, see also Remark 2.10 therein. It is interesting to note that the same contraction is derived in duality with two different gradient estimates: in Kuwada it was
\begin{align}\label{eq:staticNGE}
    |\nabla P_t v|^2 \leq P_t|\nabla v|^2 - \frac{2t}{n}(\Delta P_t v)^2\,,
\end{align}
while in our case
\begin{align}
    |\nabla P_{r-s}v|^2-2s \Delta P_{r-s}v \leq P_{r-s}\left(|\nabla v|^2-2r\Delta v \right) + \frac{n}{2}(r-s)\,,
\end{align}
i.e, by commuting the semigroup and the generator, and setting $t=r-s$,
\begin{align}\label{eq:staticNGE2}
    |\nabla P_tv|^2 \leq P_t|\nabla v|^2-2t\Delta P_tv + \frac{nt}{2}\,.
\end{align}
Estimate \eqref{eq:staticNGE} is strictly better than \eqref{eq:staticNGE2} since $\frac{2t}{n}(\Delta P_tv)^2\geq 2t\Delta P_tv -\frac{nt}{2}$ by $\frac{t}{2n}(2\Delta P_t v- n)^2\geq0$. However, one might notice that \eqref{eq:staticNGE2} is not homogeneous in $v$, and up to rescaling one can make (pointwise) $2\Delta P_t v- n=0$ and recover \eqref{eq:staticNGE}. See also Subsection \ref{ss:dimBochnerandL_} for related considerations.
\end{rmk}

\begin{proof}[Proof of Proposition \ref{prop:WL-contraction}]
It is enough to prove the statement for delta measures (see \cite[Lem.~3.3]{kuwadaduality}). Let $\gamma:[\sigma,\tau]\to M$ be a $L_-^{\sigma,\tau}$-geodesic from $x$ to $y$. Then recalling the duality from Remark \ref{rmk:Kantorovich-via-HL}
\begin{align}
    W_{\tilde L_-^{\alpha\sigma,\alpha\tau}}(\hat P_{\sigma,\alpha\sigma}\delta_x, \hat P_{\tau,\alpha\tau}\delta_y) &= \sup_\phi \int Q^{\alpha\sigma,\alpha\tau}\phi d\hat P_{\tau,\alpha\tau} \delta_{\gamma_\tau}- \int \phi d \hat P_{\sigma,\alpha\sigma} \delta_{\gamma_\sigma} \\
    &= \sup_\phi P_{\tau,\alpha\tau}Q^{\alpha\sigma,\alpha\tau}\phi(\gamma_\tau)-P_{\sigma,\alpha\sigma}\phi(\gamma_\sigma) \\
    &=\sup_\phi \int_\sigma ^\tau\partial_\eta P_{\eta,\alpha \eta}Q^{\alpha \sigma,\alpha \eta}\phi(\gamma_\eta) d\eta \\
    &= \sup_\phi \int_\sigma ^\tau -\Delta P_{\eta,\alpha \eta}Q^{\alpha \sigma,\alpha \eta} \phi(\gamma_\eta) + \alpha P_{\eta,\alpha \eta}\Delta Q^{\alpha\sigma,\alpha \eta}\phi(\gamma_\eta) \\
    & \hspace{1cm}+ P_{\eta,\alpha \eta}(\partial_\eta Q^{\alpha\sigma,\alpha \eta}\phi)(\gamma_\eta) + \left\langle \nabla P_{\eta,\alpha \eta}Q^{\alpha\sigma,\alpha \eta}\phi(\gamma_\eta),\dot\gamma_\eta
    \right\rangle_\eta d\eta \\
    &\leq \sup_\phi \int_\sigma ^\tau -\Delta P_{\eta,\alpha \eta}Q^{\alpha \sigma,\alpha \eta} \phi(\gamma_\eta) + \alpha P_{\eta,\alpha \eta}\Delta Q^{\alpha\sigma,\alpha \eta}\phi(\gamma_\eta) \\
    & \hspace{1cm}+ P_{\eta,\alpha \eta}\left(-\frac{1}{(\sqrt{\tau} -\sqrt{\sigma})\sqrt{\eta}}\frac{1}{2}|\nabla Q^{\alpha \sigma,\alpha \eta}\phi(\gamma_\eta)|^2 + \alpha^2(\sqrt{\tau} -\sqrt{\sigma}) \sqrt{\eta}\frac{1}{2}S_{\alpha \eta}\right) \\
    & \hspace{1cm} + \frac{1}{2(\sqrt{\tau}-\sqrt{\sigma})\sqrt{\eta}}|\nabla P_{\eta,\alpha \eta}Q^{\alpha\sigma,\alpha \eta}\phi(\gamma_\eta)|_\eta^2 + \frac{1}{2}(\sqrt{\tau}-\sqrt{\sigma})\sqrt{\eta}|\dot\gamma_\eta|_\eta^2 d\eta
    \end{align}
    using in the last line \eqref{HLforRescaledL-} and Young's inequality. Exploiting the different homogeneity in $\phi$ of the terms, we set up the gradient estimate \eqref{eq:alphaGE} that will be applied to $\frac{\sqrt{\eta}Q^{\alpha \sigma,\alpha \eta} \phi(\gamma_\eta)}{\sqrt{\tau}-\sqrt{\sigma}}$:
    
    \begin{align}
    &...=\sup_\phi \int_\sigma ^\tau \frac{\sqrt{\tau}-\sqrt{\sigma}}{\eta^{3/2}}\Bigg[-\eta\Delta P_{\eta,\alpha \eta}\left(\frac{\sqrt{\eta}Q^{\alpha \sigma,\alpha \eta} \phi(\gamma_\eta)}{\sqrt{\tau}-\sqrt{\sigma}}\right) + \alpha \eta P_{\eta,\alpha \eta}\Delta \left(\frac{\sqrt{\eta}Q^{\alpha\sigma,\alpha \eta}\phi(\gamma_\eta)}{\sqrt{\tau}-\sqrt{\sigma}}\right) \\
    & \hspace{1cm}+ P_{\eta,\alpha \eta}\left(-\frac{1}{2}\left|\frac{\nabla \sqrt{\eta}Q^{\alpha \sigma,\alpha \eta}\phi(\gamma_\eta)}{\sqrt{\tau} -\sqrt{\sigma}}\right|^2 + \alpha^2 \eta^2\frac{1}{2}S_{\alpha \eta}\right) \\
    & \hspace{1cm} + \frac{1}{2}\left|\frac{\nabla P_{\eta,\alpha \eta}\sqrt{\eta}Q^{\alpha\sigma,\alpha \eta}\phi(\gamma_\eta)}{\sqrt{\tau}-\sqrt{\sigma}}\right|^2 \Bigg] + \frac{1}{2}(\sqrt{\tau}-\sqrt{\sigma})\sqrt{\eta}|\dot\gamma_\eta|^2 d\eta \\
    & \overset{\eqref{eq:alphaGE}}{\leq} \sup_\phi \int_\sigma ^\tau  \frac{\sqrt{\tau}-\sqrt{\sigma}}{\eta^{3/2}} \Bigg[ \frac{\eta^2}{2}S_\eta + \frac{n}{4}(\alpha \eta-\eta) \Bigg] + \frac{1}{2}(\sqrt{\tau}-\sqrt{\sigma})\sqrt{\eta}|\dot\gamma_\eta|^2 d\eta \\
    &= (\sqrt{\tau}-\sqrt{\sigma}) \int_\sigma ^\tau \frac{\sqrt{\eta}}{2}|\dot\gamma_\eta|^2 + \frac{\sqrt{\eta}}{2}S_\eta + \frac{\sqrt{\tau}-\sqrt\sigma}{\sqrt{\eta}} \frac{n}{4}(\alpha-1) d\eta  \\
    & = \tilde{L}_-^{\sigma,\tau}(x,y) + (\sqrt{\tau}-\sqrt{\sigma})^2 \frac{n}{2}(\alpha-1)\,.
    \end{align}
\end{proof}

The gradient estimate \eqref{eq:GEL--GE-to-WC} implies the preservation of the $L_-$ Hamilton-Jacobi equation along the heat flow (see Lemma \ref{lem:heat-flow-preserves-HJ} for the $L_0$ analogue and further comments), which was essentially a mid-step in the proof from \eqref{eq:GEL--GE-to-WC} to \eqref{eq:WCL--GE-to-WC}.

\begin{prop} \label{prop:heat-flow-preserves-L_-HJ}
Suppose that \eqref{eq:GEL--GE-to-WC} holds, $f: [\alpha\tau_1,\alpha\tau_2] \times M\to \R{}$, $\alpha> 1$, $[\alpha\tau_1,\alpha\tau_2]\subseteq I$ is Lipschitz and satisfies the Hamilton-Jacobi inequality
\begin{equation}
\partial_\tau f(\tau,x) \leq -\frac{1}{2(\sqrt{\alpha\tau_2}-\sqrt{\alpha\tau_1})\sqrt{ \tau}}|\nabla f(\tau,x)|^2 + (\sqrt{\alpha\tau_2}-\sqrt{\alpha\tau_1}) \frac{\sqrt{\tau}}{2}S_{(\tau,x)}\,,\qquad \text{ a.e. }(\tau,x)\in M\times [\alpha\tau_1,\alpha\tau_2]\,.
\end{equation}

Then the Lipschitz function $P_{\alpha}[f]: [\tau_1, \tau_2] \to \R{}$, defined for $\tau_1\in I$ by $P_{\alpha}[f](\tau,x) := P_{\tau,\alpha\tau}f_{\alpha\tau}(x)$, itself satisfies the Hamilton-Jacobi-type inequality
\begin{align}
\partial_\tau P_{\alpha}[f](\tau,x) \leq -\frac{1}{2(\sqrt{\tau_2}-\sqrt{\tau_1})\sqrt\tau}|\nabla P_{\alpha}[f](\tau,x)|^2 + (\sqrt{\tau_2}-\sqrt{\tau_1}) \frac{\sqrt\tau}{2}S_{(\tau,x)}&+\frac{n}{4}(\alpha-1)\frac{\sqrt{\tau_2}-\sqrt{\tau_1}}{\sqrt\tau} \\ & \text{for a.e. } \tau\in  [\tau_1,\tau_2], \ \text{for all} \ x\in M,
\end{align}

and is therefore dominated by the $\tilde L_-$ distance: for all $x,y \in M$ and $\sigma < \tau\in [\tau_1, \tau_2]$, 
\begin{equation}\label{eq:L--dominated-preserved}
    P_{\alpha}[f](\tau,y) - P_{\alpha}[f](\sigma,x)  \leq (\sqrt \tau_2-\sqrt \tau_1) L^{\sigma,\tau}_-(x,y) + \frac{n}{2}(\alpha-1)(\sqrt\tau_2-\sqrt\tau_1)(\sqrt{\tau}-\sqrt{\sigma})\,.
\end{equation}
\end{prop}
\begin{proof}
    The extra regularity is proved just as in Lemma \ref{lem:heat-flow-preserves-HJ}. Using the Hamilton-Jacobi inequality for $f$ and applying \eqref{eq:GEL--GE-to-WC} to $\frac{\sqrt{\alpha\tau}}{(\sqrt{\alpha\tau_2}-\sqrt{\alpha\tau_1})}f_{\alpha \tau}$ we get
    \begin{align}
        \partial_\tau P_{\tau,\alpha\tau}(f_{\alpha\tau}) 
        &\leq -\Delta P_{\tau,\alpha\tau}f_{\alpha\tau} + \alpha P_{\tau,\alpha\tau}\Delta f_{\alpha\tau} + \alpha \Bigg(-\frac{1}{2(\sqrt{\alpha\tau_2}-\sqrt{\alpha\tau_1})\sqrt{\alpha\tau}}P_{\tau,\alpha\tau}|\nabla f_{\alpha\tau}|^2 \\
        & \hspace{7cm}+(\sqrt{\alpha\tau_2}-\sqrt{\alpha\tau_1}) \frac{\sqrt{\alpha\tau}}{2}P_{\tau,\alpha\tau}S_{\alpha\tau}\Bigg) \\
        &\leq \frac{(\sqrt{\alpha\tau_2}-\sqrt{\alpha\tau_1})}{\sqrt{\alpha\tau}}\left[-\frac{1}{2\tau}\left|\nabla P_{\tau,\alpha\tau}f_{\alpha \tau}\frac{\sqrt{\alpha\tau}}{(\sqrt{\alpha\tau_2}-\sqrt{\alpha\tau_1})}\right|^2 + \frac{\tau}{2}S_{\tau} +\frac{n}{4}(\alpha-1) \right] \\
        &= -\frac{1}{2(\sqrt{\tau_2}-\sqrt{\tau_1})\sqrt\tau}|\nabla P_{\tau,\alpha\tau}f_{\alpha\tau}|^2 + (\sqrt{\tau_2}-\sqrt{\tau_1})\frac{\sqrt{\tau}}{2}S_\tau+\frac{n}{4}\frac{\sqrt{\tau_2}-\sqrt{\tau_1}}{\sqrt\tau}(\alpha-1).
    \end{align}
    Equation \eqref{eq:L--dominated-preserved} follows from integrating the Hamilton-Jacobi inequality along a geodesic.
\end{proof}

We now recover \eqref{eq:GEL--GE-to-WC} from \eqref{eq:WCL--GE-to-WC}.
\begin{prop}
    Suppose that for any times $[\sigma,\tau] \subseteq I$, factor $1 < \alpha < \sup I/\tau$, and measures $\mu,\nu \in \cP(M)$, \eqref{eq:WCL--GE-to-WC} holds. Then for all Lipschitz $v:M \to \R{}$, times $\sigma \in I$, and factors $1 < \alpha < \sup I/\sigma$, we also have \noeqref{eq:GEL--WC-to-GE}
    \begin{align}
    &|\nabla P_{\sigma,\alpha\sigma} v|^2_\sigma + 2\sigma\Delta_\sigma P_{\sigma,\alpha\sigma}v - \sigma^2 S_\sigma(x) \\
    &\leq P_{\sigma,\alpha \sigma}|\nabla v|^2_{\alpha \sigma}(x) + 2\alpha\sigma P_{\sigma,\alpha\sigma}\Delta_{\alpha\sigma} v - \alpha^2\sigma^2 P_{\sigma,\alpha \sigma}S_{\alpha \sigma}(x) + \frac{n}{2}\sigma(\alpha-1)\,.\label{eq:GEL--WC-to-GE}\tag{GE\({}_{L_-}\)}
\end{align}
\end{prop}

\begin{proof}
    We can rewrite \eqref{eq:WCL--GE-to-WC} as
\begin{align}\label{eq:WL-minus-monotoncity}
    \sqrt{\alpha}W_{L_-^{\alpha\sigma,\alpha\tau}}(\hat{P}_{\sigma,\alpha\sigma}\mu,\hat{P}_{\tau,\alpha\tau}\nu)\leq W_{L_-^{\sigma,\tau}}(\mu,\nu)+\frac{n}{2}(\sqrt{\tau}-\sqrt{\sigma})(\alpha-1).
\end{align}
Given $x,y\in M$, $[\sigma,\tau]\subseteq I$, $1<\alpha<\sup I/\tau$, consider a smooth curve $\bar\gamma:[\sigma,\tau]\to M$ from $x$ to $y$. For $\varsigma\in[\sigma,\tau]$ we define 
\begin{align}
    \varsigma\mapsto\mu_{\alpha \varsigma} := \hat{P}_{\varsigma,\alpha \varsigma}\delta_{\bar\gamma_{\varsigma}}
\end{align}
which is a curve of probability measures ``living at time $\alpha \varsigma$'' from $\hat{P}_{\sigma,\alpha \sigma}\delta_{\bar\gamma_\sigma}$ to $\hat{P}_{\tau,\alpha \tau}\delta_{\bar\gamma_\tau}$. By Lemma \ref{lem:dynamiclifting}, we find $\eta\in \cP(C([\alpha\sigma,\alpha\tau]\to M))$ with
\begin{align}
    (\dot{\mu}_{\alpha \varsigma})_{L_-^{\alpha \varsigma}} = \int_{C([\alpha\sigma,\alpha\tau]\to M)}\sqrt{\alpha \varsigma}\frac{1}{2}\left(|\dot{\gamma}|_{\alpha \varsigma}^2+S_{\alpha \varsigma}(\gamma_{\alpha \varsigma})\right)d\eta(\gamma).
\end{align}
Now, with the substitution $\vartheta=(\varrho-\varsigma)>0$
\begin{align}
    (\dot{\mu}_{\alpha \varsigma})_{L_-^{\alpha \varsigma}} &:= \lim_{\vartheta\to0}\frac{W_{L_-^{\alpha \varsigma,\alpha \varsigma+\alpha \vartheta}}(\mu_{\alpha \varsigma},\mu_{\alpha \varsigma+\alpha \vartheta})}{\alpha \vartheta} \\
    &=\lim_{\varrho\to \varsigma}\frac{W_{L_-^{\alpha \varsigma,\alpha \varrho}}(\mu_{\alpha \varsigma},\mu_{\alpha \varrho})}{\alpha(\varrho-\varsigma)},
\end{align}
hence \eqref{eq:WL-minus-monotoncity} implies (when $\tau\to\sigma$)
\begin{align}\label{eq:diffWCL-}
    \alpha^{3/2}(\dot\mu_{\alpha \varsigma})_{L_-^{\alpha \varsigma}} \leq (\dot\delta_{\bar\gamma_\varsigma})_{L_-^\varsigma} + \frac{n}{2}(\alpha-1)\frac{1}{2\sqrt{\varsigma}}.
\end{align}

We now have
\begin{align}
    P_{\tau, \alpha\tau} v(y) -P_{\sigma, \alpha\sigma} v(x)
    &=\int v d\hat{P}_{\tau,\alpha\tau}\delta_{\bar\gamma_\tau} - \int vd\hat{P}_{\sigma,\alpha\sigma}\delta_{\bar\gamma_\sigma} \\
    &= \int vd(\mu_{\alpha\tau}-\mu_{\alpha\sigma}) \\
    &= \int_{C([\alpha\sigma,\alpha\tau])\to M}\int_{\alpha\sigma}^{\alpha\tau} \langle dv(\gamma_\varsigma),\dot\gamma_\varsigma\rangle d\varsigma d\eta(\gamma) \\
    &\leq \int_{C([\alpha\sigma,\alpha\tau])\to M}\int_{\alpha\sigma}^{\alpha\tau} \left(\frac{1}{2\lambda}|\nabla v|_{\varsigma}^2(\gamma_\varsigma) + \frac{\lambda}{2}|\dot\gamma_\varsigma|_{\varsigma}^2 \right) d\varsigma d\eta(\gamma)
    \end{align}
    by Young's inequality for any $\lambda>0$. Setting up to recover the $W_{L^-}$ speed and using the (differentiated) Wasserstein contraction \eqref{eq:diffWCL-}
    \begin{align} 
    ...&= \int_{\alpha\sigma}^{\alpha\tau}\int_M\frac{1}{2\lambda} |\nabla v|^2_{\varsigma}d\mu_\varsigma d\varsigma + \int_C\int_{\alpha\sigma}^{\alpha\tau}\frac{\lambda}{2}(|\dot\gamma_\varsigma|^2_{\varsigma} + S_{\varsigma}(\gamma_\varsigma))d\varsigma d\eta(\gamma) - \int_C\int_{\alpha\sigma}^{\alpha\tau}\frac{\lambda}{2}S_{\varsigma}(\gamma_\varsigma)d\varsigma d\eta(\gamma) \\
    &= \int_{\sigma}^{\tau}\int_M\frac{1}{2\lambda} |\nabla v|^2_{\alpha \varsigma}d\mu_{\alpha \varsigma} \alpha d\varsigma + \int_C\int_{\sigma}^{\tau}\frac{\lambda}{2}(|\dot\gamma_{\alpha \varsigma}|^2_{\alpha \varsigma} + S_{\alpha \varsigma}(\gamma_{\alpha \varsigma}))\alpha d\varsigma d\eta(\gamma) - \int_C\int_{\sigma}^{\tau}\frac{\lambda}{2}S_{\alpha \varsigma}(\gamma_{\alpha \varsigma})\alpha d\varsigma d\eta(\gamma) \\
    &= \int_{\sigma}^{\tau}\frac{\alpha}{2\lambda}P_{\varsigma,\alpha \varsigma}|\nabla v|^2_{\alpha \varsigma}(\bar\gamma_\varsigma)d\varsigma + \int_{\sigma}^{\tau}\frac{\lambda\alpha}{\sqrt{\alpha \varsigma}}(\dot\mu_{\alpha \varsigma})_{L_-^{\alpha \varsigma}} d\varsigma - \int_{\sigma}^{\tau}\frac{\lambda\alpha}{2}P_{\varsigma,\alpha \varsigma}S_{\alpha \varsigma}(\bar\gamma_\varsigma)d\varsigma \\
    &\overset {\eqref{eq:diffWCL-}} {\leq} \int_{\sigma}^{\tau}\frac{\alpha}{2\lambda}P_{\varsigma,\alpha \varsigma}|\nabla v|^2_{\alpha \varsigma}(\bar\gamma_\varsigma)d\varsigma - \int_{\sigma}^{\tau}\frac{\lambda\alpha}{2}P_{\varsigma,\alpha \varsigma}S_{\alpha \varsigma}(\bar\gamma_\varsigma)d\varsigma \\
    &\qquad +\int_{\sigma}^{\tau}\left[\frac{\lambda\alpha}{\sqrt{\alpha \varsigma}}\frac{\sqrt{\varsigma}}{\alpha^{3/2}}\frac{1}{2}(|\dot{\bar\gamma}|^2_\varsigma+S_\varsigma(\bar\gamma_\varsigma)) + \frac{\lambda\alpha}{\sqrt{\alpha \varsigma}}\frac{1}{\alpha^{3/2}}\frac{n}{2}(\alpha-1)\frac{1}{2\sqrt{\varsigma}}\right]d\varsigma \label{eq:L--Lisini-application}
\end{align}

Now divide by $(\tau-\sigma)$ and send $\tau\to\sigma$: the first line becomes
\begin{align}
    \lim_{\tau\to\sigma} \frac{P_{\tau,\alpha\tau}v(\bar\gamma_\tau)-P_{\sigma,\alpha\sigma}v(\bar\gamma_\sigma)}{\tau-\sigma} &= -\Delta P_{\sigma,\alpha\sigma}v(x) + \alpha P_{\sigma,\alpha\sigma}\Delta v(x) + \langle dP_{\sigma,\alpha\sigma} v(x),\dot{\bar\gamma}_\sigma\rangle \\
    &= -\Delta P_{\sigma,\alpha\sigma}v(x) + \alpha P_{\sigma,\alpha\sigma}\Delta v(x) + |\nabla P_{\sigma,\alpha\sigma} v(x)|^2
\end{align}
by choosing $\dot{\bar\gamma}_\sigma=\nabla P_{\sigma,\alpha\sigma}v(x)$. The last line instead, with $\lambda=\alpha$, becomes
\begin{align}
    &\frac{1}{2}P_{\sigma,\alpha \sigma}|\nabla v|^2_{\alpha \sigma}(\bar\gamma_\sigma) - \frac{\alpha^2}{2}P_{\sigma,\alpha \sigma}S_{\alpha \sigma}(\bar\gamma_\sigma) +\frac{1}{2}(|\dot{\bar\gamma}_\sigma|^2_\sigma+S_\sigma(\bar\gamma_\sigma)) + \frac{1}{\sigma}\frac{n}{4}(\alpha-1) \\
    &=\frac{1}{2}P_{\sigma,\alpha \sigma}|\nabla v|^2_{\alpha \sigma}(\bar\gamma_\sigma) - \frac{\alpha^2}{2}P_{\sigma,\alpha \sigma}S_{\alpha \sigma}(\bar\gamma_\sigma) +\frac{1}{2}(|\nabla P_{\sigma,\alpha\sigma}v(x)|^2_\sigma+S_\sigma(\bar\gamma_\sigma)) + \frac{1}{\sigma}\frac{n}{4}(\alpha-1)
\end{align}
Altogether,
\begin{align}
    &-\Delta P_{\sigma,\alpha\sigma}v(x) + \alpha P_{\sigma,\alpha\sigma}\Delta v(x) + \frac{1}{2}|\nabla P_{\sigma,\alpha\sigma} v(x)|^2 \\
    &\leq \frac{1}{2}P_{\sigma,\alpha \sigma}|\nabla v|^2_{\alpha \sigma}(x) - \frac{\alpha^2}{2}P_{\sigma,\alpha \sigma}S_{\alpha \sigma}(x) +\frac{1}{2}S_\sigma(x) + \frac{1}{\sigma}\frac{n}{4}(\alpha-1)
\end{align}
i.e. multiplying by $2\sigma^2$
\begin{align}
    &|\nabla P_{\sigma,\alpha\sigma} \sigma v(x)|^2 - 2\sigma\Delta P_{\sigma,\alpha\sigma}\sigma v(x) - \sigma^2 S_\sigma(x) \\
    &\leq P_{\sigma,\alpha \sigma}|\nabla\sigma v|^2_{\alpha \sigma}(x) - 2\alpha\sigma P_{\sigma,\alpha\sigma}\Delta \sigma v(x) - \alpha^2\sigma^2 P_{\sigma,\alpha \sigma}S_{\alpha \sigma}(x) + \frac{n}{2}\sigma(\alpha-1)\,,
\end{align}
which is the desired gradient estimate for the function $\tilde v = -\sigma v$.
\end{proof}

We conclude the subsection with some modifications of the Wasserstein contraction estimates that can be useful in other contexts. The next corollary shows that, under $\cD\geq0$, we can get a Wasserstein estimate for any four times satisfying the compatibility assumption \eqref{eq:compatibilityoftimes}.
\begin{cor}\label{cor:WC-with-time-translations}
    Consider $(M^n,g_\tau)_{\tau \in I}$ a smooth family of closed Riemannian manifolds with $\dim M = n$, but now parametrized by an arbitrary (not necessarily positive) $I \subseteq \R{}$. Suppose that $\mathcal{D}\geq0$ holds, and let $\sigma_1,\sigma_2,\tau_1,\tau_2 \in I$ with $\sigma_1<\sigma_2$, $\tau_1<\tau_2$, $\sigma_1<\tau_1$ (without loss of generality, up to exchanging $\sigma_i$ and $\tau_i$), $\sigma_2 < \tau_2$, and \begin{align}\label{eq:compatibilityoftimes}
        \tau_2-\tau_1 > \sigma_2-\sigma_1.
    \end{align} Then there exists $T_0$ for which $\frac{\sigma_2+T_0}{\sigma_1+T_0}=\frac{\tau_2+T_0}{\tau_1+T_0}=:\alpha$ and $\sigma_1+T_0>0$, and hence
    \begin{align}
    W_{\tilde{L}_{T_0,-}^{\sigma_2,\tau_2}}(\hat{P}_{\sigma_1,\sigma_2}\mu,\hat{P}_{\tau_1,\tau_2}\nu)\leq W_{\tilde{L}_{T_0,-}^{\sigma_1,\tau_1}}(\mu,\nu)+\frac{n}{2}(\sqrt{\tau_2 - \tau_1}-\sqrt{\sigma_2 -\sigma_1})^2\,.
    \end{align}
    Here $\tilde L_{T_0,-}^{\alpha,\beta}$ is the normalized cost $\tilde L_{T_0,-}^{\alpha,\beta} = (\sqrt{\alpha + T_0} - \sqrt{\beta + T_0})L_{T_0,-}^{\alpha,\beta}$, with $L_{T_0,-}^{\alpha,\beta}(x,y)$ induced by the Lagrangian $L_{T_0,-}(v,x,\tau) :=\sqrt{\tau + T_0}(|v|^2_{g_\tau}+S_{\tau})(x)$.
\end{cor}

\begin{proof}
    Equation $\frac{\sigma_1+T_0}{\sigma_2+T_0}=\frac{\tau_1+T_0}{\tau_2+T_0}$ becomes $T_0=\frac{\sigma_2\tau_1-\sigma_1\tau_2}{(\tau_2-\tau_1)-(\sigma_2-\sigma_1)}$. Moreover, we have $\sigma_1+T_0=\frac{(\tau_1-\sigma_1)(\sigma_2-\sigma_1)}{(\tau_2-\tau_1)-(\sigma_2-\sigma_1)}>0$, thanks to \eqref{eq:compatibilityoftimes}. Now we apply the Wasserstein contraction to the shifted flow $(M,g_{\tau -T_0})_{\tau \in I + T_0}$: the translated flow still satisfies $\mathcal{D}\geq0$, as it is translation-invariant.
\end{proof}

The following allows us to compare two conjugate heat flows at any times using the $L_-$ distance, with their initial distribution at time $\tau=0$. Choosing the same measure we get a continuity in time statement for the conjugate heat flow.
\begin{cor}
    Consider $(M^n,g_\tau)_{\tau\in[0,T]}$ a smooth family of Riemannian manifolds with $\dim M = n$ and $\mathcal{D}\geq0$, then we have the estimate for any $\mu,\nu \in \cP(M)$, times $0<\sigma<\tau < T$,
    \begin{align}
        4W_{\tilde{L}_-^{\sigma,\tau}}(\hat{P}_{0,\sigma}\mu,\hat{P}_{0,\tau}\nu)\leq W_{d_0^2}(\mu,\nu)+2n(\sqrt{\tau}-\sqrt{\sigma})^2.
    \end{align}
    In particular, for $\mu=\nu$
    \begin{align}
        W_{\tilde{L}_-^{\sigma,\tau}}(\hat{P}_{0,\sigma}\mu,\hat{P}_{0,\tau}\mu)\leq \frac{n}{2}(\sqrt{\tau}-\sqrt{\sigma})^2.
    \end{align}
\end{cor}

\begin{proof}
    Fix $0<\bar \sigma<\bar\tau$, let $\bar\tau=\lambda\bar\sigma$. Consider a sequence $\sigma_i\to0$ and let $\tau_i=\lambda\sigma_i$, $\alpha_i=\frac{\bar\sigma}{\sigma_i}$. Apply equation \eqref{eq:WCL--GE-to-WC} with $\sigma=\sigma_i$, $\tau=\tau_i$, $\alpha=\alpha_i$, sending $i\to\infty$ we have (using \cite[Thm.~5.20]{Vil09} for the convergence)
    \begin{align}
        W_{\tilde{L}_-^{\bar\sigma,\bar\tau}}(\hat{P}_{0,\bar\sigma}\mu,\hat{P}_{0,\bar\tau}\nu) \leq W_{d_0^2/4}(\mu,\nu)+ \frac{n}{2}(\sqrt{\bar\tau}-\sqrt{\bar\sigma})^2.
    \end{align}
\end{proof}

\subsection{Entropy \texorpdfstring{$\tau^{-1/2}$}{t12}-convexity and \texorpdfstring{$\mathrm{EVI}$}{EVI}}\label{ss:entropyEVIL-}

\subsubsection{From entropy \texorpdfstring{$\tau^{-1/2}$}{t12}-convexity to $W_{L_-}$ contraction}\label{sss1:entropyEVIL-}
This is the analog of Subsection \ref{ss:entropyEVIL0} for the $L_{-}$ distance. Many aspects will be unchanged from the $L_0$ case, and we will omit the details that have already appeared in that setting. We begin with the implication where the formal computations are correct up to an error with a favorable sign. We consider times $\sigma < \tau$, a.c. measures $\mu_{\sigma} = \rho_\sigma\,dV_\sigma,\,\mu_{\tau} = \rho_\tau\,dV_\tau \in \cP(M)$, and Kantorovitch potentials $(\phi_\eta)_{\eta \in [\sigma,\tau]}: M\to \R{}$ provided by Corollary \ref{cor:dynamic-potentials} with respect to the $L_-$ cost. In particular, $\phi_\eta$ is Lipschitz and solves the following Hamilton-Jacobi equation

\begin{equation}\label{eq:L-minus-cvx-implies-contraction-I}
\partial_\eta \phi_\eta(x) = - \frac{1}{2\sqrt{\eta}}|\nabla \phi_\eta|^2_{g_\eta}(x) + \frac{\sqrt{\eta}}{2}S_\eta(x)\,,\qquad \forall \eta \in [\sigma,\tau]\,,\text{ a.e. }x\in M\,.
\end{equation}

One can also estimate the (one-sided) derivatives of entropy along $L_-
$ Wasserstein geodesics using the semi-concavity properties of the potentials as in Lemma \ref{lem:Top-1-sided-est}, or more directly \cite[l.~(3.18),~(3.24-25)]{Top09}

\begin{align}\label{eq:L-minus-cvx-implies-contraction-II}
\sqrt{\sigma}\left.\frac{d}{d\eta}\right|_{\eta = \sigma^+} \cE(\mu_\eta) &\geq \int_M (\La \nabla \ln(\rho_{\sigma}),\nabla \phi_{\sigma}\Ra - \sqrt{\sigma}S_{\sigma})d\mu_{\sigma}\,, \\
\sqrt{\tau}\left.\frac{d}{d\eta}\right|_{\eta = \tau^-} \cE(\mu_\eta) &\leq \int_M (\La \nabla \ln(\rho_{\tau}),\nabla \phi_{\tau}\Ra - \sqrt{\tau}S_{\tau})d\mu_{\tau}\,.
\end{align}

As in Proposition \ref{prop:L0-cvxty-implies-contr/EVI}, this is sufficient to obtain $W_{L_-}$ contraction for multiplicatively related times if we assume an appropriate entropy convexity along $W_{L_-}$ geodesics. These computations are related to those of \cite[Prop.~16]{Lot09}, but with \eqref{eq:L-minus-cvx-implies-contraction-II} in hand there is no need to perform any (further) computations along $L_-$ geodesics.

\begin{prop}\label{prop:L-minus-convexity-implies-contraction}
Pick times $[\sigma,\tau]\subseteq I$ and suppose also that for any measures $\mu_{\sigma},\,\mu_\tau \in \cP(M)$ there exists a $W_{L_-}$-geodesic $(\mu_{\eta})_{\eta \in [\sigma,\tau]}$ between $\mu_{\sigma}, \mu_{\tau}$ such that

\begin{equation}
\cE(\mu_\eta) - \cE(\mu_{\sigma}) + \frac{1}{\sqrt{\eta}}W_{L_-^{\sigma,\eta}}(\mu_{\sigma},\mu_{\eta}) + \frac{n}{2}\ln(\eta/\sigma)\,,\label{eq:ECL--EC-to-WC}\tag{EC\({}_{L_-}\)}
\end{equation}
is convex in the variable $\eta^{-1/2}$. Then for any $\mu_{\sigma}, \mu_{\tau} \in \cP(M)$ and $1 < \alpha \leq \sup I/\tau$, it holds \noeqref{eq:WCL--EC-to-WC}

\begin{align}
(\sqrt{\alpha \tau} - \sqrt{\alpha\sigma})&\cdot W_{L_-^{\alpha\sigma, \alpha\tau}}(\hat{P}_{\sigma,\alpha\sigma}\mu_{\sigma}, \hat{P}_{\tau, \alpha\tau}\mu_{\tau}) - \frac{n}{2}(\sqrt{\alpha\tau} - \sqrt{\alpha\sigma})^2 \\
&\leq (\sqrt{\tau} - \sqrt{\sigma})\cdot W_{L_-^{\sigma,\tau}}(\mu_{\sigma}, \mu_{\tau}) - \frac{n}{2}(\sqrt{\tau} - \sqrt{\sigma})^2\,.\label{eq:WCL--EC-to-WC}\tag{WC\({}_{L_-}\)}
\end{align}
\end{prop}

\begin{proof}
The technical details--in particular the smoothing of the measures by flowing for some small time $\eps > 0$--are identical to those of the proof of Proposition \ref{prop:L0-cvxty-implies-contr/EVI}. We therefore restrict our attention to the formal computations, which all hold rigorously for endpoint measures that are smooth with full support.

Writing $d/d(\eta^{-1/2}) = -2\eta^{3/2}d/d\eta$, by the convexity assumption we have

\begin{equation}
-\frac{n}{2}(\sqrt{\tau} - \sqrt{\sigma}) \leq \left. \eta^{3/2}\frac{d}{d\eta}\left[\cE(\mu_\eta) + \frac{1}{\sqrt{\eta}}W_{L_-^{\sigma,\eta}}(\mu_{\sigma},\mu_\eta)\right]\right|_{\sigma^+}^{\tau^{-}}\,.
\end{equation}

Note that the existence of the one-sided derivatives of the first term of the RHS is provided by \eqref{eq:L-minus-cvx-implies-contraction-II}, while the derivative for the second term is computed using Proposition \ref{prop:Lisini-for-geodesics} applied to the $L_-$ Lagrangian, yielding
\begin{align}
\cdots &\leq \left.\eta\int_M (\La \nabla \ln( \rho_\eta),\nabla \phi_\eta \Ra - \sqrt{\eta}S_\eta)d\mu_\eta -\frac{1}{2}\int_{M}\phi_\eta\,d\mu_{\eta} + \eta\int_M \left(\frac{1}{2\sqrt{\eta}}|\nabla \phi_\eta|^2 + \frac{\sqrt{\eta}}{2}S_\eta\right)\, d\mu_\eta\right|_{\sigma}^{\tau}\\
&= \left.-\eta\int_M \phi_\eta\Delta\rho_\eta \,dV_\eta -\frac{1}{2}\int_{M}\phi_\eta\,d\mu_{\eta} - \eta\int_M \left(-\frac{1}{2\sqrt{\eta}}|\nabla \phi_\eta|^2 + \frac{\sqrt{\eta}}{2}S_\eta\right)\, d\mu_\eta\right|_{\sigma}^{\tau}\,.
\end{align}

Since $\phi_\eta$ (extended past $\tau$ using the Hopf-Lax semigroup) is Lipschitz and \eqref{eq:L-minus-cvx-implies-contraction-I} holds a.e. in $M$ for $\eta = \sigma,\tau$, one computes 

\begin{equation}
\cdots =-\left.\partial_{\alpha}\right|_{\alpha = 1^+} \sqrt{\alpha}\cdot\left[\int_{M} \phi_{\alpha\tau} d\hat{P}_{\tau, \alpha\tau}[\nu_{\tau}] - \int_{M} \phi_{\alpha\sigma} d\hat{P}_{\sigma, \alpha\sigma}[\nu_{\sigma}]\right]\,.
\end{equation}

Now, as in Proposition \ref{prop:L0-cvxty-implies-contr/EVI}, and modulo the technical details therein, this shows that 

\begin{equation}
\left.\partial_\alpha^+\right|_{\alpha = 1^+} \sqrt{\alpha}\cdot W_{L_-^{\alpha\sigma, \alpha\tau}}(\hat{P}_{\sigma,\alpha\sigma}[\nu_{\sigma}], \hat{P}_{\tau,\alpha\tau}[\nu_{\tau}]) \leq \frac{n}{2}(\sqrt{\tau} - \sqrt{\sigma})\,,
\end{equation}
and thus that the quantity $(\sqrt{\alpha \tau} - \sqrt{\alpha\sigma})\cdot W_{L_-^{\alpha\sigma, \alpha\tau}}(\hat{P}_{\sigma, \alpha\sigma}[\nu_{\sigma}], \hat{P}_{\tau,\alpha\tau}[\nu_{\tau}]) - (n/2)(\sqrt{\alpha\tau} - \sqrt{\alpha\sigma})^2$ is non-increasing in $\alpha > 1$.
\end{proof}

\subsubsection{From gradient estimates to entropy \texorpdfstring{$\tau^{-1/2}$}{t12}-convexity (via $\mathrm{EVI}$).}\label{sss2:entropyEVIL-}
We now turn to the other direction: beginning  from \eqref{eq:GEL--GE-to-EVI} we recover \eqref{eq:ECL--EC-to-WC}, passing through \eqref{eq:EVIL--GE-to-EVI}. The general strategy is the same as in the $L_0$ case, but there are a few additional technical challenges.

\begin{lem}\label{lem:L-minus-action-entropy-est}
Let $(\nu_\eta)_{\eta \in [\sigma,\tau]} = (\rho_\eta dV_\eta)_{\eta \in [\sigma,\tau]}$, $\rho \in C^\infty\big([\sigma,\tau]\times M;(0,\infty)\big)$, $[\sigma,\tau] \subseteq I$, a curve of smooth measures with positive densities. For $\sup I/\tau,\tau/\sigma > \alpha > 1$, write $\lambda_I := (\sqrt{\alpha} - 1)/((\sigma/\alpha)^{-1/2} - \tau^{-1/2})$, $\lambda_{II} := (\sqrt{\alpha} - 1)/(\sigma^{-1/2} - (\tau/\alpha)^{-1/2})$. Then if the $L_-$-type gradient estimate \eqref{eq:GEL--GE-to-EVI} holds, we have the action-entropy estimates
\begin{align}
\lambda_I\cE(\hat{P}_{\tau, \alpha\tau}[\nu_{\tau}]) + \frac{1}{\sqrt{\alpha}}W_{L_-^{\sigma,\alpha\tau}}(\nu_{\sigma},\hat{P}_{\tau, \alpha\tau}[\nu_{\tau}]) &\leq \lambda_I\cE(\nu_{\sigma})  +\frac{1 - \lambda_I/\sqrt{\sigma}}{\sqrt{\alpha}}\int_{\sigma}^{\tau}(\dot\nu_{\tilde\eta})_{L_-^{\tilde \eta}}\,d\tilde\eta  \\
&\qquad + \frac{\lambda_I}{2}\int_{\sigma}^{\tau}\frac{1}{\eta^{3/2}}\left(\int_{\sigma}^{\eta}(\dot{\nu}_{\tilde\eta})_{L_-^{\tilde \eta}}\,d\tilde\eta  - W_{L_-^{\sigma,\eta}}(\nu_{\sigma},\nu_{\eta})\right)\,d\eta\\
&\qquad +n (\sqrt{\alpha} - 1)\sqrt{\tau} - \frac{n}{2}\lambda_I \ln(\tau/\sigma) + O((\sqrt{\alpha} - 1)^2)\,, \label{eq:L-minus-preEVI-I}
\end{align}
and
\begin{align}
-\lambda_{II} \cE(\nu_{\tau}) + W_{L_0^{\alpha\sigma,\tau}}(\hat{P}_{\sigma, \alpha\sigma}[\nu_{\sigma}], \nu_{\tau}) &\leq -\lambda_{II} \cE(\hat{P}_{\sigma, \alpha\sigma}[\nu_{\sigma}]) + (1 + \lambda_{II}/\sqrt{\tau})\int_{\sigma}^{\tau}(\dot{\nu}_{\tilde\eta})_{L_0^{\tilde\eta}}\,d\tilde\eta \\
& \qquad - \frac{\lambda_{II}}{2}\int_{\sigma}^{\tau}\frac{1}{\eta^{3/2}}\left(\int^{\eta}_{\sigma}(\dot{\nu}_{\tilde\eta})_{L_-^{\tilde \eta}}\,d\tilde\eta  - W_{L_-^{\sigma,\eta}}(\nu_{\sigma},\nu_{\eta})\right)\,d\eta \\
&\qquad -n(\sqrt{\alpha} - 1)\sqrt{\sigma} + \frac{n}{2}\lambda_{II}\ln(\tau/\sigma) + O((\sqrt{\alpha} - 1)^2)\,. \label{eq:L-minus-preEVI-II}\
\end{align}
The constant implicit in the $O((\sqrt\alpha - 1)^2)$ error term depends on quantities that are independent of the curve $(\nu_\eta)_{\eta \in [\sigma,\tau]}$ (namely $\text{diam}_{L_-}((M,g_\eta)_{\eta \in [\sigma,\tau]}) $, $-\inf_{(\eta,x)\in I \times M} \min(0, S_\eta(x))$, $n$, $\tau$, $\sigma$), and on a uniform bound on the actions $\sup_{[\tilde\sigma,\tilde{\tau}]\subseteq [\sigma,\tau]} \left|\int_{\tilde\sigma}^{\tilde\tau}(\dot\nu_{\eta})_{L_-^\eta}\,d\eta\right|$.
\end{lem}

\begin{proof}
We proceed as in the proof of Lemma \ref{lem:L0-action-entropy-est}, but with fewer details when the two proofs align. In this case, we consider the $\sqrt{\cdot}$-linear, bijective map of intervals $\xi_\alpha: [\sigma,\alpha\tau] \to [\sigma,\tau
]$ given by

\begin{equation}\label{eq:L-paramdefn}
\xi_\alpha:\begin{cases}\xi_\alpha(\sigma) = \sigma\,,\,\xi_\alpha(\alpha\tau) = \tau\,,& \\ \sqrt{\xi_\alpha(\eta)} = (1 - \lambda) \cdot\sqrt{\eta} + b&\end{cases}\, \implies \lambda = \frac{\sqrt{\tau}\cdot(\sqrt{\alpha} - 1)}{\sqrt{\alpha\tau} - \sqrt{\sigma}}\,,\qquad b = \frac{\sqrt{\sigma\tau}\cdot(\sqrt{\alpha} - 1)}{\sqrt{\alpha\tau} - \sqrt{\sigma}} \,.
\end{equation}

The two properties of $\xi_\alpha$ that will interest us are its values at the boundary of the interval, along with the ODE $\xi_\alpha'(\eta) = (1 - \lambda)\cdot \sqrt{\xi_\alpha(\eta)/\eta}$ it solves. As in Lemma \ref{lem:L0-action-entropy-est}, we now define the ``diagonal'' curve of measures $\nu_{\eta,\alpha} := \rho_{\eta,\alpha}\,dV_\eta := \hat{P}_{\xi_{\alpha}(\eta),\eta}[\nu_{\xi_\alpha(\eta)}]$ for $\eta \in [\sigma,\alpha\tau]$, along with a family of Kantorovich potentials $(\phi_\eta)_{\eta \in [\sigma,\alpha\tau]}: M \to \R{}$ from $\nu_{\sigma,\alpha}$ to $\nu_{\alpha\tau,\alpha}$ given by Corollary \ref{cor:dynamic-potentials}, which are Lipschitz and solve the Hamilton-Jacobi equation \eqref{eq:L-minus-cvx-implies-contraction-I} for every $\eta \in [\sigma,\alpha\tau]$ and a.e. $x \in M$. It will be convenient to label the (non-constant) quantity $\kappa := 1 - \eta\xi_\alpha'(\eta)/\xi_{\alpha}(\eta) = 1 - \sqrt{\eta/\xi_\alpha(\eta)}(1 - \lambda)$. We compute (using $1=(1-\kappa)(\frac{\kappa}{1-\kappa}+1)$)

\begin{align}
\eta^{2}\partial_\eta \cE(\nu_{\eta,\alpha}) &= (1 - \kappa)\int_M P_{\xi_\alpha(\eta),\eta}\left[\eta\Delta_\eta \big(\eta\ln \rho_{\eta,\alpha}\big) - \frac{\kappa}{1 - \kappa} |\nabla\big(\eta\ln \rho_{\eta,\alpha}\big)|^2\right] \,d\nu_{\xi_\alpha(\eta)} \\
& \qquad - \eta^{2}\int_M S_\eta\,d\nu_{\eta,\alpha} - (1 - \kappa)\int_M \Big(\xi_\alpha(\eta)\Delta_{\xi_\alpha(\eta)}P_{ \xi_\alpha(\eta),\eta}[\eta\ln \rho_{\eta,\alpha}]\Big)\,d\nu_{\xi_\alpha(\eta)}\\
&\qquad + (1 - \lambda)\int_M \sqrt{\eta\xi_\alpha(\eta)} P_{\xi_\alpha(\eta),\eta}[\eta\big(1 + \ln \rho_{\eta,\alpha}\big)]\,\left.\partial_\upsilon\right|_{\upsilon = \xi_\alpha(\eta)}d\nu_\upsilon\,,
\end{align}

and

\begin{align}
&\phantom{=}\eta^{3/2}\partial_\eta \int_M \phi_\eta \,d\nu_{\eta,\alpha} \\
&= (1 - \kappa)\int_M P_{ \xi_\alpha(\eta),\eta}\left[\eta\Delta_\eta (\sqrt{\eta}\phi_\eta) - \frac{\kappa}{1 - \kappa}\La\nabla (\sqrt{\eta}\phi_\eta), \nabla (\eta\ln \rho_{\eta,\alpha})\Ra +\frac{\eta^{2}}{2} S_{\eta} -\frac{1}{2}|\nabla(\sqrt{\eta} \phi_\eta)|^2\right]\,d\nu_{\xi_\alpha(\eta)} \\
& \qquad- (1 - \kappa)\int_M\left(\xi_\alpha(\eta)\Delta_{\xi_\alpha(\eta)} P_{\xi_\alpha(\eta),\eta}[\sqrt{\eta}\phi_\eta] \right)\,d\nu_{\xi_\alpha(\eta)} + \kappa\int_M \frac{\eta^2}{2}S_\eta - \frac{1}{2}|\sqrt{\eta}\nabla \phi_\eta|^2\,d\nu_{\eta,\alpha} \\
&\qquad + (1 - \lambda) \int_M \sqrt{\eta\xi_\alpha(\eta)}P_{\xi_\alpha(\eta),\eta}[\sqrt{\eta}\phi_\eta]\,\left.\partial_\upsilon\right|_{\upsilon = \xi_\alpha(\eta)} d\nu_\upsilon\,.
\end{align}

We now verify the remaining steps from the proof of Lemma \ref{lem:L0-action-entropy-est}. In particular, we consider a linear combination $\eta^{3/2}\partial_\eta \int_M \phi_\eta\,d\nu_{\eta,\alpha} + [\kappa/(1 - \kappa)]\eta^2\partial_\eta\cE(\nu_{\eta,\alpha})$ of the LHS's above and attempt to estimate the terms appearing on the RHS's. It is convenient to define the corresponding auxiliary function $f_\eta: = \phi_\eta + \kappa\sqrt{\eta}\ln(\rho_{\eta,\alpha})/(1 - \kappa)$. 

First, using the computation of \eqref{eq:L--Lisini-application} yields 

\begin{align}
&(1 - \lambda)\int_M \sqrt{\eta\xi_\alpha(\eta)}P_{\xi_\alpha(\eta),\eta}[\sqrt{\eta} f_\eta]\,\left.\partial_\upsilon\right|_{\upsilon = \xi_\alpha(\eta)}d\nu_\upsilon \\
&\leq \frac{1 - \lambda}{2}\int_M \sqrt{\frac{\eta}{\xi_\alpha(\eta)}}\Big(|\nabla P_{\xi_\alpha(\eta),\eta}[\sqrt{\eta} f_\eta]|^2 - \xi_\alpha^2(\eta) S_{\xi_\alpha(\eta)}\Big)\,d\nu_{\xi_\alpha(\eta)} + (1 - \lambda)\sqrt{\eta}\xi_\alpha(\eta)(\dot{\nu}_{\xi_\alpha(\eta)})_{L_-^{\xi_\alpha(\eta)}}\\
&= \frac{1 - \kappa}{2}\int_M \Big(|\nabla P_{\xi_\alpha(\eta),\eta}[\sqrt{\eta} f_\eta]|^2 - \xi_\alpha^2(\eta) S_{\xi_\alpha(\eta)}\Big)\,d\nu_{\xi_\alpha(\eta)} + \eta\sqrt{\xi_\alpha(\eta)}\xi_\alpha'(\eta)(\dot{\nu}_{\xi_\alpha(\eta)})_{L_-^{\xi_\alpha(\eta)}}\,. 
\end{align}

Second, using the assumed gradient estimate \eqref{eq:GEL--GE-to-EVI} with initial data $\sqrt{\eta} f_\eta$ between the times $\xi_\alpha(\eta) < \eta$ yields

\begin{align}
&\phantom{-}(1 - \kappa)\int_M P_{\xi_\alpha(\eta),\eta}[\eta\Delta_\eta (\sqrt{\eta}f_\eta) - \frac{1}{2}|\nabla (\sqrt{\eta}f_\eta)|^2 + \frac{\eta^2}{2}S_\eta]\,d\nu_{\xi_\alpha(\eta)} \\
&\qquad - (1 - \kappa)\int_M\left(\xi_\alpha(\eta)\Delta_{\xi_\alpha(\eta)} P_{\xi_\alpha(\eta),\eta}[\sqrt{\eta}f_\eta] -\frac{1}{2} |\nabla P_{\xi_\alpha(\eta),\eta}[\sqrt{\eta}f_\eta]|^2 + \frac{\xi_\alpha^2(\eta)}{2}S_{\xi_\alpha(\eta)} \right)\,d\nu_{\xi_\alpha(\eta)} \\
&\leq \frac{n}{4}(1 - \kappa)(\eta - \xi_\alpha(\eta))\,.
\end{align}

Assembling these two estimates as in Lemma \ref{lem:L0-action-entropy-est}, and multiplying by $(1 - \kappa)$ we obtain

\begin{align}
(1 - \kappa^2)\eta^{3/2}\partial_\eta\int_M \phi_\eta \,d\nu_{\eta,\alpha} + \kappa\eta^2\partial_\eta\cE(\nu_{\eta,\alpha}) &\leq (1 - \kappa)\eta\sqrt{\xi_\alpha(\eta)}(\dot{\nu}_{\xi_\alpha(\eta)})_{L_-^{\xi_\alpha(\eta)}}\xi_\alpha'(\eta) \\
&\qquad - \kappa^2\eta^2\int_M S_\eta \,d\nu_{\eta,\alpha} +\frac{n}{4}(1 - \kappa)^2(\eta - \xi_\alpha(\eta))\,.
\end{align}

Before integrating in $\eta$, we would like to remove the $\eta$ dependence of the $\cE(\nu_{\eta,\alpha})$ coefficient. We compute:

\begin{equation}\label{eq:L-param-rel}
\kappa = 1 - \sqrt{\frac{\eta}{\xi_\alpha(\eta)}}(1 - \lambda) = \frac{b}{\sqrt{\xi_\alpha(\eta)}}
\end{equation}

We should therefore divide the last inequality through by $\eta^2/\sqrt{\xi_\alpha(\eta)}$. Doing so and integrating in time yields

\begin{align}\label{eq:halfwayL-EVI}
\int_{\sigma}^{\alpha\tau}\left((1 - \kappa^2)\sqrt{\frac{\xi_\alpha(\eta)}{\eta}}\partial_\eta\int_M \phi_\eta \,d\nu_{\eta,\alpha}\right)\,d\eta + b\left.\cE(\nu_{\eta,\alpha})\right|_{\sigma}^{\alpha\tau} &\leq (1 - \lambda)\int_{\sigma}^{\alpha\tau}\left(\sqrt{\frac{\xi_\alpha(\eta)}{\eta}}(\dot{\nu}_{\xi_\alpha(\eta)})_{L_-^{\xi_\alpha(\eta)}}\xi_\alpha'(\eta)\right)\,d\eta \\
&\qquad - \int_{\sigma}^{\alpha\tau}\left(\kappa^2\sqrt{\xi_\alpha(\eta)}\int_M S_\eta \,d\nu_{\eta,\alpha}\right)\,d\eta \\
&\qquad +\frac{n}{4}\int_{\sigma}^{\alpha\tau}\left((1 - \kappa)^2\sqrt{\xi_\alpha(\eta)}\frac{(\eta - \xi_\alpha(\eta))}{\eta^2}\right)\,d\eta\,.
\end{align}

To avoid dragging along higher-order errors, we will begin consolidating terms that are uniformly $O\big((\sqrt{\alpha} - 1)^2\big)$ whenever uniform upper bounds $ -S, \,\tau - \sigma, \,W_{L_-^{\sigma,\tau}}(\nu_{\sigma}, \,\nu_{\tau}) \leq C$ hold. For instance, this will include all monomials of degree $\geq 2$ in $b,\lambda,\kappa$, by \eqref{eq:L-paramdefn}, \eqref{eq:L-param-rel}. This leaves us with three integrals to compute. Integrating by parts:

\begin{align}
\int_{\sigma}^{\alpha\tau}\left((1 - \kappa^2)\sqrt{\frac{\xi_\alpha(\eta)}{\eta}}\partial_\eta\int_M \phi_\eta \,d\nu_{\eta,\alpha}\right)\,d\eta &= \left. \sqrt{\frac{\xi_\alpha(\eta)}{\eta}}\left(\int_M \phi_\eta\,d\nu_{\eta_\alpha} - \int_M \phi_{\sigma}\,d\nu_{\sigma,\alpha}\right)\right|_{\eta = \sigma}^{\alpha\tau} + O((\sqrt{\alpha} - 1)^2)\\
&\qquad -\int_{\sigma}^{\alpha\tau}\partial_\eta\left((1 - \lambda) + \frac{b}{\sqrt{\eta}}\right)\left(\int_M \phi_\eta\,d\nu_{\eta_\alpha} - \int_M \phi_{\sigma}\,d\nu_{\sigma,\alpha}\right)\,d\eta  \\
&= \frac{1}{\sqrt{\alpha}}W_{L_-^{\sigma,\alpha\tau}}(\nu_{\sigma,\alpha},\nu_{\alpha\tau,\alpha}) +\frac{b}{2}\int_{\sigma}^{\alpha\tau}\frac{1}{\eta^{3/2}}W_{L_-^{\sigma,\eta}}(\nu_{\sigma,\alpha},\nu_{\eta,\alpha})\,d\eta \\
&\qquad +O((\sqrt{\alpha} - 1)^2)\,,
\end{align}
and by substitution, denoting by $\xi_\alpha^{\circ - 1}:[\sigma,\tau]\to[\sigma,\alpha\tau]$ the inverse of $\xi_\alpha$,
\begin{align}
(1 - \lambda)\int_{\sigma}^{\alpha\tau}\left(\sqrt{\frac{\xi_\alpha(\eta)}{\eta}}(\dot{\nu}_{\xi_\alpha(\eta)})_{L_-^{\xi_\alpha(\eta)}}\xi_\alpha'(\eta)\right)\,d\eta &= (1 - \lambda)\int_{\xi_\alpha(\sigma)}^{\xi_\alpha(\alpha\tau)}\sqrt{\frac{\bar\eta}{\xi_\alpha^{\circ - 1}(\bar\eta)}}(\dot{\nu}_{\bar\eta})_{L_-^{\bar \eta}}\,d\bar \eta \\
&=(1 - \lambda)\left.\sqrt{\frac{\bar\eta}{\xi_\alpha^{\circ - 1}(\bar\eta)}}\int_{\xi_\alpha(\sigma)}^{\bar\eta}(\dot{\nu}_{\tilde\eta})_{L_-^{\tilde \eta}}\,d\tilde\eta\right|_{\bar\eta = \xi_\alpha(\sigma)}^{\xi_\alpha(\alpha\tau)} \\
&\qquad - (1 - \lambda)\int_{\xi_\alpha(\sigma)}^{\xi_\alpha(\alpha\tau)}\partial_{\bar\eta} \left(\sqrt{\frac{\bar\eta}{\xi_\alpha^{\circ - 1}(\bar\eta)}}\right)\int_{\xi_\alpha(\sigma)}^{\bar\eta}(\dot{\nu}_{\tilde\eta})_{L_-^{\tilde \eta}}\,d\tilde\eta d\bar\eta \\
&= \frac{1 - \lambda}{\sqrt{\alpha}}\int_{\sigma}^{\tau}(\dot\nu_{\tilde\eta})_{L_-^{\tilde \eta}}\,d\tilde\eta + \frac{b(1-\lambda)}{2}\int_{\sigma}^{\alpha\tau}\frac{1}{\eta^{3/2}}\int_{\xi_\alpha(\sigma)}^{\xi_\alpha(\eta)}(\dot{\nu}_{\tilde\eta})_{L_-^{\tilde \eta}}\,d\tilde\eta d\eta\,.
\end{align}
Note that to simplify the last integral we used $\partial_{\bar\eta} \left(\sqrt{\frac{\bar\eta}{\xi_\alpha^{\circ - 1}(\bar\eta)}}\right) = -\frac{b}{2}\frac{1}{\eta^{3/2}\xi_\alpha'(\eta)}$ for $\eta=\xi_\alpha^{\circ - 1}(\bar\eta)$.

The expressions appearing in the RHS's of the two computations above are, up to $O((\sqrt{\alpha} - 1)^2)$ error with allowable implicit constant dependence, the $W_{L_-}$ distance and action terms that appear in \eqref{eq:L-minus-preEVI-I}. We finally turn to the last term in \eqref{eq:halfwayL-EVI}, which is of order $O(\sqrt{\alpha} - 1)$, and compute its coefficient
\begin{align}
\lim_{\alpha \to 1^+}\frac{1}{\sqrt{\alpha} - 1}\int_{\sigma}^{\alpha\tau}\frac{n}{4}\left((1 - \kappa)^2\sqrt{\xi_\alpha(\eta)}\frac{(\eta - \xi_\alpha(\eta))}{\eta^2}\right)\,d\eta &= \lim_{\alpha \to 1^+}\frac{n(1 - \lambda)^2}{4}\int_{\sigma}^{\alpha\tau}\frac{(\eta - \xi_\alpha(\eta))}{(\alpha - 1)}\frac{\sqrt{\alpha} + 1}{\sqrt{\xi_\alpha(\eta)}\eta}\,d\eta \\
&= \frac{n}{2}\int_{\sigma}^{\tau}\frac{-\left.\partial_\alpha\right|_{\alpha = 1^+}\xi_\alpha(\eta)}{\eta^{3/2}}\,d\eta \\
&= n \sqrt{\tau} - \frac{n}{2}\frac{\sqrt{\tau\sigma}}{\sqrt{\tau} - \sqrt{\sigma}}\ln(\tau/\sigma)\,,
\end{align}
where we used 
\begin{align}
\left.\partial_\alpha\right|_{\alpha = 1^+}\xi_\alpha(\eta) &= \left.\partial_\alpha\right|_{\alpha = 1^+}\left(\frac{\sqrt{\tau} - \sqrt{\sigma}}{\sqrt{\alpha\tau} - \sqrt{\sigma}}\sqrt{\eta} + \frac{\sqrt{\tau\sigma}\cdot(\sqrt{\alpha} - 1)}{\sqrt{\alpha\tau} - \sqrt{\sigma}}\right)^2 \\
&= -\frac{\sqrt{\tau}}{\sqrt{\tau} - \sqrt{\sigma}}\eta + \frac{\sqrt{\tau\sigma}}{\sqrt{\tau} - \sqrt{\sigma}}\sqrt{\eta}\,.
\end{align}
Gathering $O((\sqrt{\alpha} - 1)^{2})$ errors and unpacking notation gives the desired final expression.

The second inequality follows from running the same argument with $\nu_\eta^\alpha := \hat{P}_{\tilde{\xi}_\alpha(\eta),\eta}[\nu_{\tilde{\xi}_\alpha(\eta)}]$ for the range of times $[\alpha\sigma,\tau]$, where $\tilde{\xi}_\alpha: [\alpha\sigma,\tau] \to [\sigma,\tau]$ solves

\begin{equation}
\tilde{\xi}_\alpha(\eta): \begin{cases}\tilde{\xi}_\alpha(\alpha\sigma) = \sigma\,,\,\tilde{\xi}_\alpha(\tau) = \tau\,,&\\
\sqrt{\tilde{\xi}_\alpha(\eta)} = (1 + \tilde{\lambda})\cdot \sqrt{\eta} + \tilde{b}\end{cases}\,\implies \tilde{\lambda} = \frac{\sqrt{\sigma}(\sqrt{\alpha} - 1)}{\sqrt{\tau} - \sqrt{\alpha\sigma}}\,,\qquad \tilde{b} = -\frac{\sqrt{\tau\sigma}(\sqrt{\alpha} - 1)}{\sqrt{\tau} - \sqrt{\alpha\sigma}}\,.
\end{equation}
\end{proof}

As in the $L_0$ case, these action-entropy estimates can be used to obtain the respective \eqref{eq:EVIL--GE-to-EVI}s.

\begin{prop}\label{prop:L-GE-to-EVI}
Assume that for any Lipschitz $v$ and times $[\sigma,\tau] \subseteq I$, the gradient estimate
\begin{equation}
 |\nabla P_{\sigma,\tau}v|_\sigma^2-2\sigma \Delta_\sigma P_{\sigma,\tau}v-\sigma^2S_{\sigma} \leq P_{\sigma,\tau}\left(|\nabla v|_\tau ^2 - 2\tau\Delta_\tau v - \tau ^2S_{\tau}\right) + \frac{n}{2}(\tau-\sigma)\label{eq:GEL--GE-to-EVI}\tag{GE\({}_{L_-}\)}
\end{equation}
holds. Then, for any Borel probability measures $\mu,\nu \in \cP(M)$ and times $[\sigma,\tau] \subseteq I$, the backward heat flow $\hat P$ satisfies the following $\mathrm{EVI}$s
\begin{align}
\eta\cdot\partial_{\eta^+}^+ W_{L_-^{\eta,\tau}}(\hat{P}_{\sigma,\eta}[\nu], \mu) &\leq \frac{1}{2(\eta^{-1/2} - \tau^{-1/2})} \left[\cE(\mu) -\cE(\hat{P}_{\sigma,\eta}[\nu]) + \frac{W_{L_-^{\eta,\tau}}(\hat{P}_{\sigma,\eta}[\nu], \mu)}{\sqrt{\tau}} \right] \\
&\qquad +\frac{n\ln(\tau/\eta)}{4(\eta^{-1/2} - \tau^{-1/2})} - \frac{n\sqrt{\eta}}{2}\,, \qquad \forall\eta \in [\sigma,\tau)\,, \\
\varsigma\cdot\partial_{\varsigma^+}^+ W_{L_-^{\sigma,\varsigma}}(\nu,\hat{P}_{\tau,\varsigma}[\mu]) &\leq \frac{1}{2(\sigma^{-1/2} - \varsigma^{-1/2})} \left[\cE(\nu) - \cE(\hat{P}_{\tau,\varsigma}[\mu])  - \frac{W_{L_-^{\sigma,\varsigma}}(\nu,\hat{P}_{\tau,\varsigma}[\mu])}{\sqrt{\sigma}}\right]\\
&\qquad + \frac{n\ln(\sigma/\varsigma)}{4(\sigma^{-1/2} - \varsigma^{-1/2})} + \frac{n\sqrt{\varsigma}}{2}\,, \qquad \forall \varsigma \in [\tau,\sup I)\,. \label{eq:EVIL--GE-to-EVI}\tag{EVI\({}_{L_-}\)}
\end{align}
\end{prop}

\begin{proof}
The argument follows the same lines as the proof of the corresponding $L_0$ case Proposition \ref{prop:L0GE-to-EVI}, and we therefore omit the details. In a few words, one selects a $W_{L_-}$ geodesic connecting the measures appearing in the $\mathrm{EVI}$s, uses Lemma \ref{lem:wass-geodesic-smoothing} to smooth this geodesic so that \ref{lem:L-minus-action-entropy-est} can be applied, and then takes the limit in the smoothing parameter $\eps \to 0^+$ followed by the limit $\alpha\to 1^+$. It is perhaps worth noting that the action-$W_{L_-}$ difference integrands in the second lines of \eqref{eq:L-minus-preEVI-I}, \eqref{eq:L-minus-preEVI-II} vanish for $W_{L_-}$ geodesics, and both the action and $W_{L_-}$ distance between any two points along the smoothed curves are uniformly bounded and converge as $\eps \to 0^+$ by Lemma \ref{lem:wass-geodesic-smoothing} and \cite[Thm.~5.20]{Vil09}.
\end{proof}

\begin{prop}\label{prop:L-minus-entropy-cvx}
Assume that for any Borel probability measures $\mu,\nu \in \cP(M)$ and times $[\sigma,\tau] \subseteq I$, the backward heat flow $\hat P$ satisfies the $\mathrm{EVI}$s \eqref{eq:EVIL--GE-to-EVI}. Then for times $\sigma \leq \bar\eta \leq \tau$ with $[\sigma,\tau] \subseteq I$ and a $W_{L_-}$ geodesic $(\mu_{\eta})_{\eta \in [\sigma,\tau]}\subseteq \cP(M)$, the following entropy convexity statement holds:
\begin{align}\label{eq:L-minus-entropy-convexity}
&\phantom{-}\frac{\bar\eta^{-1/2} - \tau^{-1/2}}{\sigma^{-1/2} - \tau^{-1/2}}(\cE(\mu_{\sigma}) + (n/2)\ln\sigma) + \frac{ \sigma^{-1/2} - \bar\eta^{-1/2}}{\sigma^{-1/2} - \tau^{-1/2}}(\cE(\mu_{\tau}) + (n/2)\ln\tau) - (\cE(\mu_{\bar\eta }) + (n/2)\ln\bar\eta) \\
&- \frac{\bar\eta^{-1/2} - \tau^{-1/2}}{\sigma^{-1/2} - \tau^{-1/2}}\Big(\sigma^{-1/2}W_{L_-^{\sigma,\bar\eta}}(\mu_{\sigma}, \mu_{\bar\eta})\Big) + \frac{ \sigma^{-1/2} - \bar\eta^{-1/2}}{\sigma^{-1/2} - \tau^{-1/2}}\Big(\tau^{-1/2}W_{L_-^{\bar\eta,\tau}}(\mu_{\bar\eta}, \mu_{\tau})\Big) \geq 0\,.
\end{align}
\end{prop}

\begin{rmk}
As in Remark \ref{rmk:L0-entropy-alt-form}, the second line of \eqref{eq:L-minus-entropy-convexity} can be rewritten so that the inequality reads as an $(\eta^{-1/2})$-convexity statement for a single quantity. Indeed, if one writes $(\phi_\eta)_{\eta \in [\sigma,\tau]}$ the Kantorovich potentials for the pair of measures $(\mu_{\sigma},\mu_{\tau})$ given by Corollary \ref{cor:dynamic-potentials}, then

\begin{equation}
W_{L_-^{\bar{\sigma},\bar{\tau}}}(\mu_{\bar{\sigma}},\mu_{\bar{\tau}}) = \int_{M}\phi_{\bar{\tau}}d\mu_{\bar{\tau}} - \int_M\phi_{\bar{\sigma}}d\mu_{\bar{\sigma}}\,,\qquad \sigma \leq \bar{\sigma} \leq \bar{\tau} \leq \tau\,,
\end{equation}
and plugging in the above identity and regrouping terms gives

\begin{align}
&- \frac{\bar\eta^{-1/2} - \tau^{-1/2}}{\sigma^{-1/2} - \tau^{-1/2}}\Big(\sigma^{-1/2}W_{L_-^{\sigma,\bar\eta}}^2(\mu_{\sigma}, \mu_{\bar\eta})\Big) + \frac{ \sigma^{-1/2} - \bar\eta^{-1/2}}{\sigma^{-1/2} - \tau^{-1/2}}\Big(\tau^{-1/2}W_{L_-^{\bar\eta,\tau}}^2(\mu_{\bar\eta}, \mu_{\tau})\Big) \\
&= \frac{\bar\eta^{-1/2} - \tau^{-1/2}}{\sigma^{-1/2} - \tau^{-1/2}}\sigma^{-1/2}\int_M \phi_{\sigma}\,d\mu_{\sigma} + \frac{ \sigma^{-1/2} - \bar\eta^{-1/2}}{\sigma^{-1/2} - \tau^{-1/2}}\tau^{-1/2}\int_M \phi_{\tau}\,d\mu_{\tau} - \bar\eta^{-1/2}\int_M \phi_{\bar\eta} \,d\mu_{\bar\eta}\,.
\end{align}
\end{rmk}

\begin{proof}
As in the proof of the $L_0$ version (Proposition \ref{prop:L0-entropy-cvx}), the argument consists of taking the normalized sum of the \eqref{eq:EVIL--GE-to-EVI}s applied respectively to the curves of measures $\left.(\mu_{\eta})_\eta\right|_{[\bar\eta,\tau]}$ resp. $\left.(\mu_{\eta})_\eta\right|_{[\sigma,\bar\eta]}$, then using the $W_{L_-}$ triangle inequality on the sum of Dini derivatives to obtain a lower bound of $0$. To avoid senseless redundancy we omit the details, but we do point out the cancellation of the $\sqrt{\bar\eta}$-dimensional terms with opposite signs in the respective \eqref{eq:EVIL--GE-to-EVI}s.
\end{proof}

\subsection{Characterizations for the \texorpdfstring{$L_+$}{L+} distance}\label{ss:L+}

There is a collection of equivalent conditions that is parallel to those of Theorem \ref{thm:L-charact}, but adapted to the $L_+$ distance--see Definition \ref{df:L+}. We do not include as many details as for the corresponding $L_-$ statements. This is for reasons of redundancy--the arguments are the same as for $L_-$ but with a handful of symbol and sign changes--and of external interest--$L_-$ is more useful in the study of partial regularity for the Ricci flow since it is adapted to blowing up on singularities that appear as one advances in (forward) time, and therefore receives much more attention in the literature. Nevertheless, let $t = T - \tau \in T - I =: \tilde I$ be a choice of forward time variable, then

\begin{thm}\label{thm:L+charact}
     Given $(M^n, g_t)_{t \in \tilde I}$ a smooth time-dependent Riemannian manifold of dimension $M = n$ and $\tilde I \subseteq (0,\infty)$, the following are equivalent:
    \begin{enumerate}
        \item $\mathcal{D}\geq 0$ ;
        \item Bochner formula: for every $v: M \to \R{}$ smooth and times $[s , t] \subseteq \tilde I$
        \begin{align}\label{eq:BochnerL+}
            (\partial_t-\Delta)\left( |\nabla P_{t,s}v|_t^2-2t\Delta_t P_{t,s}v - t^2S_t -\frac{n}{2}t \right) &\leq 0\,;\tag{B\({}_{L_+}\)}
        \end{align}
        \item Gradient estimate: for every $v$ smooth and times $[s , t] \subseteq \tilde I$
        \begin{align}\label{eq:GEL+}
            |\nabla P_{t,s}v|^2_t -2t\Delta P_{t,s}v - t^2S_t \leq P_{t,s}\left(|\nabla v|^2_s-2s\Delta v-s^2 S_s\right) + \frac{n}{2}(t-s)\,;\tag{GE\({}_{L_+}\)}
        \end{align}
        \item Wasserstein contraction: for every probability measures $\mu,\nu \in \cP(M)$, times $[s , t] \subseteq \tilde I$, and factor $1 < a <\sup \tilde I/t$,
        \begin{align}\label{eq:WCL+}
            W_{\Tilde{L}_+^{s,t}}(\hat P_{a s,s}\mu,\hat P_{a t,t}\nu) \leq W_{\Tilde{L}_+^{as,at}}(\mu,\nu) + \frac{n}{2}(a-1)(\sqrt{t}-\sqrt{s})^2\,;\tag{WC\({}_{L_+}\)}
        \end{align}
        \item Entropy convexity: for every $\nu_s,\nu_t \in \cP(M)$ and times $[s,t] \subseteq \tilde I$, there is a $W_{L_{+}^{s,t}}$-geodesic $(\nu_r)_{r \in [s,t]}$ (equivalently, for all such curves) such that
        \begin{align}
            r\mapsto \cE_r(\nu_r) - \cE_s(\nu_s) - \frac{1}{\sqrt{r}}W_{L_{+}^{s,r}}(\mu_{s},\mu_{r}) + \frac{n}{2}\ln(r/s)\ \ \text{is convex in the variable $r^{-1/2}$;}\tag{EC\({}_{L_+}\)}\label{eq:ECL+}
        \end{align}
        \item $\mathrm{EVI}$ characterization: for every $\mu,\nu \in \cP(M)$ and times $[s,t]\subseteq \tilde I$, the heat flow $\hat P_{t,s}\mu$ satisfies the $\mathrm{EVI}$s
        \begin{align}
            -a\cdot\partial_{a^-}^- W_{L_+^{a,t}}(\hat{P}_{s,a}\nu, \mu) &\leq \frac{1}{2(a^{-1/2} - t^{-1/2})} \left[\cE_t(\mu) -\cE_a(\hat{P}_{s,a}\nu) - \frac{W_{L_{+}^{a,t}}(\hat{P}_{s,a}\nu, \mu)}{\sqrt{t}} \right] \\
            &+\frac{n\ln(t/a)}{4(a^{-1/2} - t^{-1/2})} - \frac{n\sqrt{a}}{2}\,, \qquad \forall a \in (\inf \tilde I,s]\,.\\
            -b\cdot\partial_{b^-}^- W_{L_+^{s,b}}(\nu,\hat{P}_{t,b}\mu) &\leq \frac{1}{2(s^{-1/2} - b^{-1/2})} \left[\cE_s(\nu) - \cE_b(\hat{P}_{t,b})  + \frac{W_{L_+^{s,b}}(\nu,\hat{P}_{t,b}\mu)}{\sqrt{s}}\right]\\
            &+ \frac{n\ln(s/b)}{4(s^{-1/2} - b^{-1/2})} + \frac{n\sqrt{b}}{2}\,, \qquad \forall b \in (s,t]\,.\label{eq:EVIL+}\,\tag{EVI\({}_{L_+}\)}\\
    \end{align}
    \end{enumerate}
\end{thm}

\begin{rmk}
The entropy convexity statement \eqref{eq:ECL+} was introduced in \cite[Prop.~20]{Lot09} in the context of Ricci flows, where it was used to obtain a version of reduced volume monotonicity adapted to $L_+$. 
\end{rmk}

\section{Monotonicity of \texorpdfstring{$\cF$}{F} and \texorpdfstring{$\cW$}{W}}\label{s:FandW}

As a palate cleanser for the reader, we now provide a short application for the gradient estimates \eqref{eq:GEL0} and \eqref{eq:GEL-}. We first recall the $\cF$ functional and its scale-invariant sibling $\cW$ defined on a manifold $M$, given a smooth metric $g$ and function $S$ on $M$, a smooth measure $\mu = \rho \,dV$, and a (backwards) time $\tau$:

\begin{equation}
\cF(\mu,g,S) := \int_M (|\nabla f|^2 + S)\,d\mu \,,\qquad \cW(\mu,g,S,\tau) := \int_M [\tau(|\nabla f|^2 + S) + f - n]\,d\mu\,,
\end{equation}
where $f$ is defined via the reverse Gaussian formula $\rho = (4\pi\tau)^{-n/2}e^{-f}$. Recall that $\cF$, known as Fisher information in the static setting, is (up to sign) the derivative of the Boltzmann entropy along the adjoint heat flow. We will mainly be interested in the case when $g = g_\tau$ is time dependent, $S = \partial_\tau dV_g/dV_g$, and $\mu = \hat{P}_{\tau_1,\tau}[\mu_{\tau_1}]$. The main result is the following:

\begin{thm}\label{thm:F-W-monotonicity}
Suppose that $(M,g_\tau)_{\tau \in I}$ is a smooth family of Riemannian manifolds, $\tau_1 \in I$ is any time, and $\mu_{\tau_1} =\rho_{\tau_1\,dV_{\tau_1}} \in \cP(M)$ with $\rho_{\tau_1} \in C^\infty(M;(0,\infty))$ an initial measure with smooth positive density.

\begin{enumerate}
\item If $g$ satisfies the $L_0$ gradient estimate \eqref{eq:GEL0}, then 
\begin{equation}
[\tau_1,\infty)\cap I\ni\tau\qquad\mapsto \qquad\cF(\hat{P}_{\tau_1,\tau}[\mu_{\tau_1}], g_\tau, S_\tau)\,\qquad \text{is non-increasing. }
\end{equation}
\item If $g$ satisfies the $L_{-}$ gradient estimate \eqref{eq:GEL-} and $I\subseteq (0,\infty)$, then 
\begin{equation}
[\tau_1,\infty) \cap I\ni\tau \qquad\mapsto \qquad \cW(\hat{P}_{\tau_1,\tau}[\mu_{\tau_1}], g_\tau,S_\tau,\tau) \,\qquad \text{is non-increasing. }
\end{equation}
\end{enumerate}
\end{thm}

\begin{rmk}
One may view this argument as a version of \cite[Sec.~1.2]{Top09}, but without needing to pass through the intermediate \eqref{eq:WCL0}, \eqref{eq:WCL-} statements.
\end{rmk}

\begin{rmk}
We observe that $S$ need not be equal to $\partial_\tau\ln dV_\tau$ for the first conclusion to hold. It is enough to assume that \eqref{eq:GEL0} holds for the particular choice of function $S : I\times M \to \R{}$. Meanwhile the second conclusion uses $S \geq \partial_\tau dV_\tau/dV_\tau$ to ensure that $\tau \partial_\tau \cW \leq \partial_\tau (\tau^2 \cF - \tau n/2)$.)
\end{rmk}

\begin{proof}
\begin{enumerate}
\item Let us write the smooth measures $\hat{P}_{\tau_1,\tau}[\mu_{\tau_1}] =: \rho_\tau \,dV_\tau$. We then compute by using integration by parts and duality of the forward and adjoint heat propagators
\begin{align}
\cF(\hat{P}_{\tau_1,\tau}[\mu_{\tau_1}], g_\tau, S_\tau) &= \int_M \left(|\nabla \ln\rho_\tau|^2 + S_\tau\right)\,d\hat{P}_{\tau_1,\tau}[\mu_{\tau_1}] \\
&= \int_M \left(-|\nabla \ln\rho_\tau|^2 + S_\tau\right)\,d\hat{P}_{\tau_1,\tau}[\mu_{\tau_1}] + \int_M 2|\nabla \ln\rho_\tau|^2\,\rho_\tau dV_\tau\\
&= \int_M \left(-|\nabla \ln\rho_\tau|^2 + S_\tau\right)\,d\hat{P}_{\tau_1,\tau}[\mu_{\tau_1}] + \int_M 2\La\nabla \ln\rho_\tau,\nabla \rho_\tau\Ra\,dV_\tau\\
&= \int_M P_{\tau_1,\tau}\left[-2\Delta\ln\rho_\tau-|\nabla \ln\rho_\tau|^2 + S_\tau\right]\,d\mu_{\tau_1} \\
\end{align}
Using now \eqref{eq:GEL0}, again integration by parts, and then Young's inequality, we can continue with 
\begin{align}
\cdots&\leq \int_M \left[-2\Delta P_{\tau_1,\tau}[\ln \rho_\tau] - |\nabla P_{\tau_1,\tau}[\ln \rho_\tau]|^2 + S_{\tau_1}\right]\,d\mu_{\tau_1} \\
&=\int_M \left[ - |\nabla P_{\tau_1,\tau}[\ln \rho_\tau]|^2 + S_{\tau_1}\right]\,d\mu_{\tau_1} + 2\int_M \La \nabla P_{\tau_1,\tau}[\ln \rho_{\tau}], \nabla \rho_{\tau_1}\Ra\,dV_{\tau_1} \\
&= \int_M \left[ - |\nabla P_{\tau_1,\tau}[\ln \rho_\tau]|^2 + S_{\tau_1}\right]\,d\mu_{\tau_1} + 2\int_M \La \nabla P_{\tau_1,\tau}[\ln \rho_{\tau}], \nabla \ln\rho_{\tau_1}\Ra\,d\mu_{\tau_1} \\
&\leq \int_M \left[  |\nabla \ln \rho_{\tau_1}|^2 + S_{\tau_1}\right]\,d\mu_{\tau_1} = \cF(\mu_{\tau_1},g_{\tau_1}, S_{\tau_1})\,. \\
\end{align}
\item The second item proceeds similarly, so we do not highlight every step:
\begin{equation}
\tau^2\cF(\hat{P}_{\tau_1,\tau}[\mu_{\tau_1}], g_\tau, S_\tau) - \frac{n}{2}\tau = \int_M P_{\tau_1,\tau}\left[-2\tau\Delta (\tau\ln \rho_\tau) - |\nabla (\tau\ln \rho_{\tau})|^2 + \tau^2S_\tau - \frac{n}{2}\tau\right]\,d\mu_{\tau_1}
\end{equation}
We are set up to apply \eqref{eq:GEL-} to $\tau\ln\rho_\tau$, and we then proceed as before:
\begin{align}
\cdots &\leq \int_M \left[-2\tau_1\Delta P_{\tau_1,\tau}[\tau\ln \rho_\tau] - |\nabla P_{\tau_1,\tau}[\tau\ln \rho_\tau]|^2 + \tau_1^2S_{\tau_1} - \frac{n}{2}\tau_1\right]\,d\mu_{\tau_1} \\
&= \int_M \left[ - |\nabla P_{\tau_1,\tau}[\tau\ln \rho_\tau]|^2 + \tau_1^2S_{\tau_1} - \frac{n}{2}\tau_1\right]\,d\mu_{\tau_1} + 2\int_M \left\La \nabla P_{\tau_1,\tau}[\tau\ln \rho_{\tau}], \nabla (\tau_1\ln\rho_{\tau_1})\right\Ra\,d\mu_{\tau_1} \\
&\leq \int_M \left[  |\nabla (\tau_1\ln \rho_{\tau_1})|^2 + \tau_1^2S_{\tau_1} - \frac{n}{2}\tau_1\right]\,d\mu_{\tau_1} = \tau_1^2\cF(\mu_{\tau_1},g_{\tau_1}, S_{\tau_1}) - \frac{n}{2}\tau_1\,.
\end{align}
\end{enumerate}
We conclude that $\tau \mapsto \tau^2\cF(\hat{P}_{\tau_1,\tau}[\mu_{\tau_1}], g_\tau,S_\tau) -(n/2)\tau$ is non-increasing. To see that $\tau \mapsto \cW(\hat{P}_{\tau_1,\tau}[\mu_{\tau_1}], g_\tau,S_\tau,\tau)$ is also non-increasing, we use the just-concluded monotonicity and note (or rather Topping notes, see \cite[Sec.~1.2]{Top09}) that for any sufficiently smooth family of Riemannian manifolds, that need not (sub)solve any equation, 
\begin{equation}
\partial_\tau \cW(\hat{P}_{\tau_1,\tau}[\mu_{\tau_1}], g_\tau,S_\tau,\tau) = \frac{1}{\tau}\partial_\tau\left(\tau^2\cF(\hat{P}_{\tau_1,\tau}[\mu_{\tau_1}], g_\tau, S_\tau) -\frac{n}{2}\tau\right)\,.
\end{equation}
\end{proof}

\begin{rmk}
    With the same proof as in the $L_-$ case, we also have that under the $L_+$ gradient estimate \eqref{eq:GEL+} we recover the montonicity of the $\mathcal W_+$ entropy, which was first considered in \cite{Feldman-Ilmanen-Ni}.
\end{rmk}

\section{Dimensional inequalities for the \texorpdfstring{$L_0$}{L0} distance}\label{s:dimL0}
In this section we prove the following theorem, which characterizes flows satisfying $\cD\geq0$ and with an upper bound on the dimension. It is formulated in terms of backwards time, but this is simply a matter of preference.

\begin{thm}\label{thm:dimL0charact}
Given $(M,g_\tau)_{\tau \in I}$ a smooth, closed, time-dependent Riemannian manifold of arbitrary dimension, a real number $N \in [0,\infty)$, and $t \in \tilde I$ a choice of forward time variable, the following are equivalent:

\begin{enumerate}
\item $\cD \geq 0$ and $\dim M \leq N$;
\item Any one of conditions 2.-6. of both Theorem \ref{thm:L-charact} and Theorem \ref{thm:L+charact} holds for the shifted flows $(M,g_{\tau - T_0})_{\tau \in (I + T_0) \cap (0,\infty)}$ resp. $(M,g_{t - T_0})_{t \in (\tilde I + T_0) \cap (0,\infty)}$ for all $T_0 \in \R{}$, with dimensional parameter $n := N$; 
\item Bochner inequality: for every $v:M\to\mathbb{\R{}}$ smooth, for every $[\sigma,\tau] \subseteq I$,
\begin{align} \label{eq:BochnerdimL0}
    (-\partial_\sigma - \Delta) \left(|\nabla P_{\sigma,\tau}v|_\sigma^2 + 2\Delta_\sigma P_{\sigma,\tau}v  -S_\sigma\right) \leq -\frac{2}{N}(S_\sigma - \Delta_\sigma P_{\sigma,\tau}v)^2\,;\tag{B\({}_{L_0,N}\)}
\end{align}
\item Gradient estimate: for every $v: M \to \R{}$ smooth, for every $[\sigma,\tau] \subseteq I$,
\begin{align} 
    |\nabla P_{\sigma,\tau}v|_\sigma^2 + 2\Delta_\sigma P_{\sigma,\tau}v - S_\sigma \leq P_{\sigma,\tau}(|\nabla v|_\tau^2 + 2\Delta_\tau v - S_\tau) - \frac{2}{N}\int_\sigma^\tau (|P_{\sigma,\eta}S_\eta - P_{\sigma,\eta}\Delta_\eta P_{\eta,\tau}v|^2)\,d\eta\,;\\\label{eq:GEdimL0}\tag{GE\({}_{L_0,N}\)}
\end{align}
\item Wasserstein contraction: for any probability measures $\mu,\nu \in \cP(M)$, for all times $\sigma,\tau,S,T \in I$ with $\sigma < \tau$, $S < T$, $\tau < T$, and $\sigma< S$, \noeqref{eq:WCdimL0}
\begin{align} \label{eq:WCdimL0}
    W_{\tilde L_0^{S,T}}(\hat{P}_{\sigma,S}\mu,\hat{P}_{\tau,T}\nu) \leq W_{\tilde L_0^{\sigma,\tau}}(\mu,\nu) + \frac{N}{4}\ln\left(\frac{T - \tau}{S - \sigma}\right)\big((T - S) - (\tau - \sigma)\big)\,, \tag{WC\({}_{L_0,N}\)}
\end{align}
where $W_{\tilde L_0^{\sigma,\tau}}$ is the Wasserstein distance with respect to the cost induced by $(\tau - \sigma)\mathcal{L}_0$. 
\end{enumerate}
\end{thm}

In view of Theorems \ref{thm:L0charact} and \ref{thm:L-charact}, there is a conspicuous absence of entropy convexity and $\mathrm{EVI}$-type conditions in Theorem \ref{thm:dimL0charact}. We partially fill this void with Theorem \ref{thm:dimL0EC/EVI}, which at least shows that the corresponding characterizations of Theorem \ref{thm:L-charact} have analogues phrased only in terms of the $L_0$ distance. However the question remains unanswered:

\begin{qu}
    What are (reasonable) time-shifting invariant entropy convexity and $\mathrm{EVI}$ statements that are equivalent to $\cD \geq 0$ and $\dim M \leq N$? Possibly one should make use of the R\'enyi entropy.
\end{qu}

\begin{thm}\label{thm:dimL0EC/EVI}
Given $(M^n,g_\tau)_{\tau\in I}$ smooth time-dependent Riemannian manifold of dimension $\dim M = n$ and $I \subseteq (0,\infty)$, the following are equivalent:
\begin{enumerate}
\item $L_-$ characterizations: Any of the (equivalent) conditions of Theorem \ref{thm:L-charact}.
\item Dimensional $L_0$ entropy convexity: for every $\mu_\sigma,\mu_\tau \in \cP(M)$, $[\sigma,\tau]\subseteq I$, there is a $W_{L_0^{\sigma,\tau}}$ geodesic $(\mu_\eta)_{\eta\in[\sigma,\tau]}$ (equivalently, for all such curves) such that \noeqref{eq:ECdimL0}
\begin{align}
    \eta \mapsto \cE(\nu_\eta) - \cE(\nu_{\sigma}) + W_{L_0^{\sigma,\eta}}(\nu_{\sigma},\nu_\eta) + \frac{n}{8}(\ln(\eta/\sigma))^2 &- \int_{\sigma}^\eta \frac{1}{\tilde\eta}W_{L_0^{\sigma,\tilde\eta}}(\nu_{\sigma},\nu_{\tilde\eta})\,d\tilde\eta\\
    &\text{is convex in the variable $\ln\eta$;}\label{eq:ECdimL0}\tag{EC\({}_{L_0,n}\)}
\end{align}

\item Dimensional $L_0$ $\mathrm{EVI}$: for every $\mu,\nu \in \cP(M)$ and times $[\sigma,\tau] \subseteq I$, the heat flow $\hat P_{\sigma,\tau}\mu$ satisfies the EVIs: \noeqref{eq:EVIdimL0}
\begin{align}
 \eta\cdot\partial_{\eta^+}^+ W_{L_0^{\eta,\tau}}(\hat{P}_{\sigma,\eta}[\nu], \mu) &\leq \frac{1}{\ln(\tau/\eta)} \left[\cE(\mu) -\cE(\hat{P}_{ \sigma,\eta}[\nu]) +W_{L_0^{\eta,\tau}}(\hat{P}_{\sigma,\eta}[\nu], d\mu) -\int_{\eta}^{\tau} \frac{1}{\tilde\tau}  W_{L_0^{\eta,\tilde\tau}}(\hat{P}_{\sigma,\eta}[\nu],\tilde{\nu}_{\tilde\tau})\,d\tilde\tau\right]\\
&\qquad + \frac{n\ln(\tau/\eta)}{8}\,,\qquad \forall\eta \in [\sigma,\tau)\,,\\
\varsigma\cdot \partial_{\varsigma^+}^+ W_{L_0^{\sigma,\varsigma}}(\nu,\hat{P}_{\tau,\varsigma}[\mu]) &\leq \frac{1}{\ln(\varsigma/\sigma)} \left[\cE(\nu) - \cE(\hat{P}_{\tau,\varsigma}[\mu]) - W_{L_0^{\sigma,\varsigma}}(\nu,\hat{P}_{\tau,\varsigma}[\mu]) -\int_{\sigma}^{\varsigma} \frac{1}{\tilde\tau}  W_{L_0^{\tilde{\tau},\varsigma}}(\tilde{\mu}_{\tilde\tau}, \hat{P}_{\tau,\varsigma}[\mu])\,d\tilde\tau\right]\\
&\qquad + \frac{n\ln(\varsigma/\sigma)}{8}\,,\qquad\forall\varsigma \in [\tau,\sup I)\,,\label{eq:EVIdimL0}\tag{EVI\({}_{L_0,n}\)}\\
\end{align}
where $(\tilde{\nu}_{\tau})_{\tau \in [\eta,\tau]}$ is any $W_{L_0^{\eta,\tau}}$-geodesic with endpoints $\hat{P}_{\sigma,\eta}[\nu]$ and $\mu$, while $(\tilde{\mu}_{\tau})_{\tau \in [\sigma,\varsigma]}$ is any $W_{L_0^{\sigma,\varsigma}}$-geodesic with endpoints $\nu$ and $\hat P_{\tau,\varsigma}[\mu]$.
\end{enumerate}
\end{thm}

\begin{rmk}
A similar addendum could be made to the $L_+$ characterizations of Theorem \ref{thm:L+charact}. However, since it is not even clear how useful the characterizations of Theorem \ref{thm:dimL0EC/EVI} are themselves--see Remark \ref{rmk:no-L0-reduced-vol}--we have not written it out.
\end{rmk}

For the remainder of the section, we assume that we have fixed $(M,g_\tau)_{\tau \in I}$ a smooth, closed, time-dependent Riemannian manifold (of arbitrary dimension, unless an assumption on the dimension is explicitly stated) and a choice of forward time variable $t: = T - \tau \in T - I =: \tilde I$.

\subsection{Dimensional \texorpdfstring{$L_0$}{L0} Bochner and dimensional \texorpdfstring{$L_0$}{L0} gradient estimate}\label{ss:nL0Bochner}
The following is a consequence of the Bochner identity of Proposition \ref{prop:DcondiffBochner}.
\begin{prop}
     The condition $\left\{\text{$\mathcal{D}\geq 0$ and $\dim M\leq N$}\right\}$ is equivalent to \noeqref{eq:BochnerdimL0-id-to-ineq}
    \begin{align}
        -(\partial_\tau + \Delta)(|\nabla v|^2 + 2\Delta v - S) \leq -\frac{2}{N}(S-\Delta v)^2 \,. \label{eq:BochnerdimL0-id-to-ineq}\tag{B\({}_{L_0,N}\)}
    \end{align}
    for any time $\tau\in I$, $v$ smooth solution of the heat flow at $\tau$.
\end{prop}
\begin{proof}
    The forward direction follows directly from \ref{prop:DcondiffBochner} and the Cauchy-Schwarz inequality for matrices. Vice versa, assume the above Bochner inequality holds, then using the Bochner identity from Proposition \ref{prop:DcondiffBochner} we get
    \begin{align}
        2|\mathcal{S}-\Hess v|^2 + \cD(-\nabla v) \geq \frac{2}{N}|S-\Delta v|^2.
    \end{align}
    Fix $(\sigma,p,V)\in I\times TM$, choose $v$ with $\nabla v(\sigma,p)=-V$, $(\Hess v -\cS)(\sigma,p) =\lambda g_\sigma$, then at $(\sigma,p)$ we have
    \begin{align}
        2(\dim M)\lambda^2 + \cD(V) \geq \frac{2(\dim M)^2}{N}\lambda^2.
    \end{align}
    For $\lambda=0$ we recover $\cD\geq0$, while for $\lambda\to\infty$ we get $(\dim M)/N\leq 1$.
\end{proof}

\begin{prop}
    The dimensional Bochner inequality for any $\tau \in I$, any smooth solutions $v$ to the heat equation at $\tau$
    \begin{align}
        -(\partial_\tau + \Delta)(|\nabla v|^2 + 2\Delta v - S) \leq -\frac{2}{N}(S-\Delta v)^2 \label{eq:BEL0N2} \tag{B\({}_{L_0,N}\)}
    \end{align}
    is equivalent to the dimensional gradient estimate
    \begin{align} 
    |\nabla P_{\sigma,\tau}v|_\sigma^2 + 2\Delta_\sigma P_{\sigma,\tau}v - S_\sigma \leq P_{\sigma,\tau}(|\nabla v|_\tau^2 + 2\Delta_\tau v - S_\tau) - \frac{2}{N}\int_\sigma^\tau (|P_{\sigma,\eta}S_\eta - P_{\sigma,\eta}\Delta_\eta P_{\eta,\tau}v|^2)\,d\eta\, \\ \label{eq:GEL0N2}\tag{GE\({}_{L_0,N}\)}
\end{align}
    for any smooth $v:M\to\R{}$.
\end{prop}
\begin{proof}
    $\implies$) Following the proof of Prop \ref{prop:L0gradestimate}, $\phi_\eta := |\nabla P_{\eta,\tau}v|_\eta^2 + 2\Delta_\eta P_{\eta,\tau}v - S_\eta$, we get
    \begin{align}
        \partial_\eta P_{\sigma,\eta}(\phi_\eta) = P_{\sigma,\eta}((\partial_\eta+\Delta)\phi_\eta) \geq \frac{2}{N}P_{\sigma,\eta}((S_\eta-\Delta P_{\eta,\tau}v)^2).
    \end{align}
    Integrating in $\eta\in[\sigma,\tau]$, and then using Jensen's inequality (with respect to the heat kernel measure)
    \begin{align} 
    P_{\sigma,\tau}(|\nabla v|_\tau^2 + 2\Delta_\tau v - S_\tau) 
    &\geq |\nabla P_{\sigma,\tau}v|_\sigma^2 + 2\Delta_\sigma P_{\sigma,\tau}v - S_\sigma + \frac{2}{N}\int_\sigma^\tau P_{\sigma,\eta}((S_\eta-\Delta P_{\eta,\tau}v)^2)\,d\eta \\
    &\geq |\nabla P_{\sigma,\tau}v|_\sigma^2 + 2\Delta_\sigma P_{\sigma,\tau}v - S_\sigma + \frac{2}{N}\int_\sigma^\tau (|P_{\sigma,\eta}S_\eta - P_{\sigma,\eta}\Delta_\eta P_{\eta,\tau}v|^2)\,d\eta \,.
    \end{align}

    $\impliedby$) Again, differentiating \eqref{eq:GEL0N2} at $\tau=\sigma$ we get back \eqref{eq:BEL0N2}.
\end{proof}

\subsection{Equivalence of dimensional Bochner inequalities}\label{ss:dimBochnerandL_}

We now show that \eqref{eq:BochnerdimL0-equiv-of-Bochners} is equivalent to the validity of \eqref{eq:BochnerL-} and \eqref{eq:BochnerL+} for all shifted flows $(M,g_{\tau - T_0})_{\tau \in (I + T_0) \cap (0,\infty)}$ resp. $(M,g_{t - T_0})_{t \in (\tilde I + T_0) \cap (0,\infty)}$. This is a consequence of the algebraic properties of these inequalities only, showing that all distinct dimensionally sensitive Bochner inequalities we might consider are actually equivalent, without any additional geometric/analytic assumptions. It also says that, if we consider all shifts, the $L_-$ Bochner inequality improves to a ``dimensionally improved'' $L_-$ $N$-Bochner inequality.

\begin{prop}\label{prop:updated-dim-Bochner-equiv}
    The following are equivalent:

\begin{enumerate}
    \item \label{item:dim-Boch-1} For any time $\tau_0 \in I$ and smooth solution to the heat equation $v$ at $\tau_0$, the $L_0$ $N$-Bochner inequality holds:
    \begin{equation}
        \left.(-\ptau - \Delta)\right|_{\tau = \tau_0}\Big(|\nabla v|^2 - 2\Delta v - S\Big) \leq -\frac{2}{N}(S + \Delta v)^2. \label{eq:BochnerdimL0-equiv-of-Bochners}\tag{B\({}_{L_0,N}\)}
        \end{equation}
    \item \label{item:dim-Boch-2} For any time $\tau_0 = T - t_0 \in I$, smooth solution to the heat equation $v$ at $\tau_0$, and shift $-T_- < \tau_0$, we have the $L_-$ Bochner inequality
    \begin{equation}
    \left.(-\ptau - \Delta)\right|_{\tau = \tau_0}\left(|\nabla v|^2 - 2(T_-+\tau)\Delta v - (T_-+\tau)^2 S + \frac{N}{2}(T_-+\tau)\right) \leq 0\,,
    \end{equation}
    and for any shift $T_+> - t_0$, the $L_+$ Bochner inequality
    \begin{equation}
    \left.(\partial_t - \Delta)\right|_{t = t_0} \left(|\nabla v|^2 - 2(T_+ + t)\Delta v - (T_+ +t)^2 S - \frac{N}{2}(T_+ + t)\right) \leq 0\,.
    \end{equation}
    \item \label{item:dim-Boch-3} For any time $\tau_0 = T - t_0 \in I$, smooth solution to the heat equation $v$ at $\tau_0$, and shift $-T_- < \tau_0$, we have the $L_-$ $N$-Bochner inequality
    \begin{equation}
    \left.(-\ptau - \Delta)\right|_{\tau = \tau_0}\left(|\nabla v|^2 - 2(T_-+\tau)\Delta v - (T_-+\tau)^2 S + \frac{N}{2}(T_-+\tau)\right) \leq -\frac{2}{N} \left( (T_-+\tau)S + \Delta v - \frac{N}{2} \right)^2\,,
    \end{equation}
    and for any shift $T_+ >  - t_0$, the $L_+$ $N$-Bochner inequality
    \begin{equation}
    \left.(\partial_t - \Delta)\right|_{t = t_0} \left(|\nabla v|^2 - 2(T_++t)\Delta v - (T_++t)^2 S - \frac{N}{2}(T_++t)\right) \leq -\frac{2}{N} \left( (T_++t)S + \Delta v + \frac{N}{2} \right)^2\,.
    \end{equation}
\end{enumerate}
\end{prop}

\begin{rmk}
In Proposition \ref{prop:updated-dim-Bochner-equiv}, $T_-,T_+$ are not required to satisfy any additional constraints with respect to $I$ resp. $\tilde I$. They correspond to a choice of time origin $\tau = -T_-$ resp. $t = -T_+$ which serves as the center for parabolic rescaling, but there is no need for the family $(M,g_\tau)$ to actually be defined there. 

In the manipulations below, one can also obtain the inequalities (\ref{item:dim-Boch-2},\ref{item:dim-Boch-3}) from \eqref{item:dim-Boch-1} when $t + T_+<0$ (resp. $\tau + T_-<0$), but theses cases are not needed to recover \eqref{item:dim-Boch-1}. Note also that $t + T_+ > 0$ (resp. $\tau + T_->0$) correspond to comparison with expanding (resp. shrinking) solitons within their intervals of definition and hence are the more geometrically relevant cases. Relatedly, $\cL_+$, $\cL_-$ are only well-defined in the cases of positive (shifted, forward resp. backwards) times.
\end{rmk}

\begin{proof}
\fbox{(\ref{item:dim-Boch-1}) $\implies$ (\ref{item:dim-Boch-3})}: We prove the $L_-$ $N$-Bochner inequality, the $L_+$ is analogous. For any $-T_- < \tau_0\in I$ we have

\begin{align}
    &(-\ptau-\Delta)|_{\tau=\tau_0} \left(|\nabla v|^2 - 2(T_-+\tau)\Delta v - (T_-+\tau)^2 S + \frac{N}{2}(T_-+\tau)\right) \\
    &= (-\ptau-\Delta)|_{\tau=\tau_0} \left(|\nabla v|^2 - 2(T_-+\tau_0)\Delta v - (T_-+\tau_0)^2 S \right) + \left(2\Delta v_{\tau_0} + 2(T_-+\tau_0) S_{\tau_0}\right) - \frac{N}{2} \\
    &= (-\ptau-\Delta)|_{\tau=\tau_0} \left(\left|\nabla \frac{v}{T_-+\tau_0}\right|^2 - 2\Delta \frac{v}{T_-+\tau_0} - S \right)(T_-+\tau_0)^2 + \left(2\Delta v_{\tau_0} + 2(T_-+\tau_0) S_{\tau_0}\right) - \frac{N}{2} \\
    &\leq -\frac{2}{N}\left( S_{\tau_0}+\Delta \frac{v_{\tau_0}}{T_-+\tau_0} \right)^2(T_-+\tau_0)^2 + 2\left(\Delta v_{\tau_0} + (T_-+\tau_0) S_{\tau_0}\right) - \frac{N}{2} \\ 
    &= -\frac{2}{N}\left( (T_-+\tau_0)S_{\tau_0}+\Delta v_{\tau_0} -\frac{N}{2} \right)^2
\end{align}
where we used the $L_0$ $N$-Bochner inequality applied to $\frac{v}{T_-+\tau_0}$. Since $-T_- < \tau_0$ were arbitrary we get it for any $-T_- < \tau_0\in I$.

\fbox{(\ref{item:dim-Boch-2}) $\implies$ (\ref{item:dim-Boch-1})}: Now, using the shifted $L_-$ Bochner:

\begin{align}
    &(-\ptau - \Delta)|_{\tau = \tau_0}\left(|\nabla v|^2 - 2\Delta v - S\right) \\
    &= \frac{1}{(T_-+\tau_0)^2}(-\ptau - \Delta)|_{\tau = \tau_0}\left(|\nabla v(T_-+\tau_0)|^2 - 2(T_-+\tau_0)\Delta \big((T_-+\tau_0)v\big) - (T_-+\tau_0)^2S\right) \\
    &= \frac{1}{(T_-+\tau_0)^2}(-\ptau - \Delta)|_{\tau = \tau_0}\left(|\nabla v(T_-+\tau_0)|^2 - 2(T_-+\tau)\Delta \big((T_-+\tau_0)v\big) - (T_-+\tau)^2S +\frac{N}{2}(T_-+\tau)\right) \\
    &\qquad -\frac{1}{(T_-+\tau_0)^2} \left( 2\Delta \big((T_-+\tau_0)v_{\tau_0}\big) + 2(T_-+\tau_0)S_{\tau_0} \right) + \frac{N}{2(T_-+\tau_0)^2} \\
    &\leq -\frac{2}{T_-+\tau_0} \left( \Delta v_{\tau_0} + S_{\tau_0} \right) + \frac{N}{2(T_-+\tau_0)^2} \\
    &= -\frac{2}{N} \left( \Delta v_{\tau_0} + S_{\tau_0}\right)^2 + \frac{2}{N}\left( \Delta v_{\tau_0} + S_{\tau_0} - \frac{N}{2(T_-+\tau_0)} \right)^2.
\end{align}
We get the $L_0$ $N$-Bochner inequality but with an additional error given by the (nonnegative) second addendum in the last line. In the case $\Delta v_{\tau_0} + S_{\tau_0}>0$, we can choose $T_- > -\tau_0$ such that the term vanishes and we recover the estimate.

Analogously, using the shifted $L_+$ Bochner inequality we get
\begin{align}
    (-\partial_\tau - \Delta)|_{\tau = \tau_0}\left(|\nabla v|^2 - 2\Delta v - S\right)&=(\partial_t - \Delta)|_{t = t_0}\left(|\nabla v|^2 - 2\Delta v - S\right) \\
    &\leq -\frac{2}{N} \left( \Delta v_{t_0} + S_{t_0}\right)^2 + \frac{2}{N}\left( \Delta v_{t_0} + S_{t_0} + \frac{N}{2(T_++t_0)} \right)^2.
\end{align}
Now we are able to get rid of the extra term when $\Delta v_{\tau_0} + S_{\tau_0} = \Delta v_{t_0} + S_{t_0}<0$. Finally, for the case $\Delta v_{\tau_0} + S_{\tau_0}=0$ we can use either of the two estimates and send $T_{\pm}\to\infty$. 
\end{proof}

\subsection{Duality of shifted gradient estimates and \texorpdfstring{$L_0$}{L0}-Wasserstein contraction for unrelated times}\label{ss:GEWCdualitydimL0}

Recall that \eqref{eq:WCL0} relates pairs of times $s < t$ and $s + h < t + h$ for $h > 0$, and was shown to be equivalent to \eqref{eq:GEL0} in Subsection \ref{ss:GEWCdualityL0}. In this subsection we again consider a Wasserstein contraction statement with respect to the $L_0$ cost, but now for pairs of times $(\sigma,\tau)$, $(S,T)$ that need not be additively related. While it would be natural to consider \eqref{eq:GEdimL0} as the gradient estimate dual to \eqref{eq:WCdimL0-GE-to-WC}, it will turn out to be more convenient to use (shifted) \eqref{eq:GEL-} and \eqref{eq:GEL+}. In view of Proposition \ref{prop:updated-dim-Bochner-equiv}, this choice has no logical impact. Let us give a name to the affine map relating the pairs of times: 

\begin{equation}
\xi:[\sigma,\tau] \to [S,T]\,,\qquad \xi(\eta) = \frac{\tau - \eta}{\tau - \sigma}S + \frac{\eta - \sigma}{\tau - \sigma}T\,.
\end{equation}

Here we will work with the normalized Lagrangian $\tilde{L}_0$ from \eqref{eq:L0action}. Computing the Hamiltonian, by Proposition \ref{prop:HLsemigroup}, for any Lipschitz $\phi: M \to \R{}$ we have
\begin{align}
    \partial_\eta Q^{S,\xi(\eta)}\phi 
    &= -\xi'(\eta) \widetilde{H}^{S,\xi(\eta)}(d Q^{S,\xi(\eta)}\phi,Q^{S,\xi(\eta)}\phi,\xi(\eta)) \\
    &= -\xi'\cdot\left(\frac{1}{2(T - S)}|\nabla Q^{S,\xi(\eta)} \phi|^2 - \frac{T - S}{2}S_{\xi(\eta)} \right)\,.
\end{align}

\begin{prop}\label{prop:WL0dim-contraction}
Suppose that the gradient estimates \eqref{eq:GEL-} and \eqref{eq:GEL+} with dimensional parameter $n := N$ hold for all shifts $(M,g_{\tau- T_0})_{\tau \in (I + T_0)\cap(0,\infty)}$ resp. $(M,g_{t- T_0})_{t \in (\tilde I + T_0)\cap(0,\infty)}$. Then for any times $\sigma,\tau,S,T \in I$ with $\sigma < \tau$, $S < T$, $\sigma < S$, and $\tau < T$, the following (dimensional) $W_{L_0}$ contraction along the adjoint heat flow for $\mu,\nu \in \cP(M)$ holds
    \begin{align}\label{eq:WCdimL0-GE-to-WC}\tag{WC\({}_{L_0,N}\)}
         W_{\tilde L_0^{S,T}}(\hat{P}_{\sigma,S}\mu,\hat{P}_{\tau,T}\nu) \leq W_{\tilde L_0^{\sigma,\tau}}(\mu,\nu) + \frac{N}{4}\ln\left(\frac{T - \tau}{S - \sigma}\right)\big((T - S) - (\tau - \sigma)\big)\,.
    \end{align}
\end{prop}

\begin{rmk}
    In the case of $S=\alpha\sigma$, $T=\alpha\tau$, $\alpha\geq1$, the statement says that
    \begin{align}\label{eq:WCdimL0-mult}
        \alpha\mapsto W_{\tilde L_0^{\alpha \sigma,\alpha \tau}}(\hat P_{\sigma,\alpha\sigma}\mu, \hat{P}_{\tau,\alpha\tau}\nu) - \frac{N}{4}\ln\left(\frac{\alpha\tau}{\alpha\sigma}\right)(\alpha\tau - \alpha \sigma) \ \ \ \text{is monotone non-increasing,}
    \end{align}
    which is reminiscent of \eqref{eq:WCL-}. Moreover, if we send $\sigma,\tau\to \infty$ keeping $\tau-\sigma$ and $S - \sigma = (\alpha - 1) \cdot \sigma$ fixed, the dimensional term vanishes. This may be thought of as \eqref{eq:WCL0} being a limiting version of \eqref{eq:WCdimL0-mult} when the (backward) time-origin gets sent to $\pm\infty$.
\end{rmk}

\begin{proof}
We will divide the argument based on a trichotomy for the four times under consideration. If $\tau - \sigma = T - S$, then we are in the case of additively related times, and one first obtains \eqref{eq:GEL0} by applying Proposition \ref{prop:updated-dim-Bochner-equiv} to the assumed gradient estimates, then recalls Proposition \ref{prop:GEtoWCL0}. We now consider the case $\tau - \sigma < T - S$, i.e. $\xi' = (T - S)/(\tau - \sigma) > 1$. By the assumed shifted \eqref{eq:GEL-} gradient estimates and the fact that $0 < (\xi')^{-1}(\xi(\eta)/\eta) < \xi(\eta) /\eta$, we may apply Corollary \ref{cor:L_GradEst_param} with parameter $\lambda := (\xi')^{-1}(\xi(\eta)/\eta)$ to obtain

\begin{align}\label{eq:etaGE}
&(|\nabla P_{\eta,\xi(\eta)}[v]|^2 - 2(\xi(\eta)/\xi')\Delta P_{\eta,\xi(\eta)}[v] - (\xi(\eta)/\xi')^2S_{\eta}) \\
&- (P_{\eta,\xi(\eta)}[|\nabla v|^2] - 2\xi(\eta)P_{\eta,\xi(\eta)}[\Delta v] - \xi^2(\eta)P_{\eta,\xi(\eta)}[S_{\xi(\eta)}
]) \leq \frac{N}{2}\frac{\xi^2(\eta)(1 - (\xi')^{-1})^2}{\xi(\eta) - \eta}\,.
\end{align}

As usual, it is enough to prove the statement for delta measures. Let $\gamma:[\sigma,\tau]\to M$ be the $L_0^{\sigma,\tau}$-geodesic from $x$ to $y$. In the spirit of the proof of Proposition \ref{prop:WL-contraction} we have
{\allowdisplaybreaks
\begin{align}
    W_{\tilde L_0^{S,T}}(\hat P_{\sigma,S}\delta_x, \hat P_{\tau,T}\delta_y) 
    &= \sup_\phi P_{\tau,T}Q^{S,T}\phi(\gamma_{\tau})-P_{\sigma,S}\phi(\gamma_{\sigma}) \\
    &=\sup_\phi \int_{\sigma}^{\tau}\partial_\eta P_{\eta,\xi(\eta)}Q^{S,\xi(\eta)}\phi(\gamma_\eta) d\eta \\
    &= \sup_\phi \int_{\sigma} ^{\tau} -\Delta P_{\eta,\xi(\eta)}Q^{S,\xi(\eta)} \phi(\gamma_\eta) + \xi' P_{\eta,\xi(\eta)}\Delta Q^{S,\xi(\eta)}\phi(\gamma_\eta) \\
    & \hspace{1cm}+ P_{\eta,\xi(\eta)}(\partial_\eta Q^{S,\xi(\eta)}\phi)(\gamma_\eta) + \left\langle \nabla P_{\eta,\xi(\eta)}Q^{S,\xi(\eta)}\phi(\gamma_\eta),\dot\gamma_\eta
    \right\rangle d\eta \\
    &\leq \sup_\phi \int_{\sigma} ^{\tau} -\Delta P_{\eta,\xi(\eta)}Q^{S,\xi(\eta)} \phi(\gamma_\eta) + \xi' P_{\eta,\xi(\eta)}\Delta Q^{S,\xi(\eta)}\phi(\gamma_\eta) \\
    & \hspace{1cm}+ \xi'P_{\eta,\xi(\eta)}\left(-\frac{1}{2(T - S)}|\nabla Q^{S,\xi(\eta)}\phi(\gamma_\eta)|^2 + \frac{T - S}{2}S_{\xi(\eta)}\right) \\
    & \hspace{1cm} + \frac{1}{2(\tau - \sigma)}|\nabla P_{\eta,\xi(\eta)}Q^{S,\xi(\eta)}\phi(\gamma_\eta)|^2 + \frac{\tau - \sigma}{2}|\dot\gamma_\eta|^2 d\eta \\
    &=\sup_\phi \int_{\sigma} ^{\tau} \frac{(T - S)\xi'}{\xi^2(\eta)}\Bigg[-(\xi(\eta)/\xi')\Delta P_{\eta,\xi(\eta)}\left(\frac{\xi(\eta)Q^{S,\xi(\eta)} \phi(\gamma_\eta)}{T - S}\right) \\
    & \hspace{1cm}+  \xi(\eta) P_{\eta,\xi(\eta)}\Delta \left(\frac{\xi(\eta)Q^{S,\xi( \eta)}\phi(\gamma_\eta)}{T-S}\right) \\
    & \hspace{1cm}+ P_{\eta,\xi(\eta)}\left(-\frac{1}{2}\left|\frac{\nabla \xi(\eta)Q^{S,\xi(\eta)}\phi(\gamma_\eta)}{T -S}\right|^2 + \frac{\xi^2(\eta)}{2}S_{\xi(\eta)}\right) \\
    & \hspace{1cm} + \frac{1}{2}\left|\frac{\nabla\xi(\eta) P_{\eta,\alpha \eta}Q^{S,\xi(\eta)}\phi(\gamma_\eta)}{T - S}\right|^2 \Bigg] + \frac{\tau - \sigma}{2}|\dot\gamma_\eta|^2 d\eta \\
    & \overset{\eqref{eq:etaGE}}{\leq} \sup_\phi \int_{\sigma}^{\tau}  \frac{(T - S)\xi'}{\xi^2(\eta)} \Bigg[ \frac{(\xi(\eta)/\xi')^2}{2}S_\eta + \frac{N}{4}\frac{\xi^2(\eta)(1 - (\xi')^{-1})^2}{\xi(\eta) - \eta} \Bigg] + \frac{\tau - \sigma}{2}|\dot\gamma_\eta|^2 d\eta \\
    &= (\tau - \sigma) \int_{\sigma}^{\tau} \frac{1}{2}|\dot\gamma_\eta|^2 + \frac{1}{2}S_\eta + (\tau -\sigma) \frac{N}{4}\frac{(\xi' - 1)^2}{\xi(\eta) - \eta} \,d\eta  \\
    & = \tilde{L}_-^{\sigma,\tau}(x,y) +  \frac{N}{4}((T - S) - (\tau - \sigma))\int_{\sigma}^{\tau}\frac{\xi' - 1}{\xi(\eta) - \eta}\,d\eta.
    \end{align}
}

Note that we occasionally made use of the identity $\xi' \equiv (T - S)/(\tau - \sigma)$ in the above computations. It now remains to rewrite the dimensional error term on the RHS of the above inequality in the desired form. Indeed, since $\eta\mapsto\xi(\eta) - \eta$ is affine and $\xi(\sigma) = S$, $\xi(\tau) = T$, we have

\begin{equation}
\int_{\sigma}^{\tau}\frac{\xi' - 1}{\xi(\eta) - \eta}\,d\eta =  \ln\left(\frac{\xi(\tau) - \tau}{\xi(\sigma) - \sigma}\right) = \ln\left(\frac{T - \tau}{S - \sigma}\right)\,.
\end{equation}

This gives \eqref{eq:WCdimL0-GE-to-WC} in the case where $T - S > \tau - \sigma$. The third case, where the strict inequality is flipped, follows from an analogous argument, but using instead an appropriately shifted \eqref{eq:GEL+}.
\end{proof}

We now turn to the reverse implication: that the dimensional Wasserstein contraction \eqref{eq:WCdimL0-GE-to-WC} implies any of the (equivalent by Proposition \ref{prop:updated-dim-Bochner-equiv}) shifted gradient estimates \eqref{eq:GEL-} and \eqref{eq:GEL+}. This essentially follows the pattern of the non-dimensional version (Proposition \ref{prop:L0-contraction-to-GE}), but using now \eqref{eq:WCdimL0-GE-to-WC} for multiplicatively related times.

\begin{prop}\label{prop:WCdimL0-implies-shifted-GE}
    Suppose that for every $\mu,\nu \in \cP(M)$ and times $\sigma,\tau,S,T \in I$ with $\sigma < \tau$, $S < T$, $\sigma < S$, and $\tau < T$, the Wasserstein contraction \eqref{eq:WCdimL0-GE-to-WC} holds. Then we have the shifted $L_-$-type gradient estimate for times $[\sigma,\tau] \subseteq I$ with $-T_-< \sigma$,
        \begin{align}
            |\nabla P_{\sigma,\tau}v|^2_\sigma-2(\sigma + T_-)\Delta_\sigma P_{\sigma,\tau}v-(\sigma + T_-)^2S_{\sigma} \leq P_{\sigma,\tau}\left(|\nabla v|^2_\tau-2(\tau + T_-)\Delta_\tau v - (\tau + T_-)^2S_{\tau}\right) + \frac{N}{2}(\tau-\sigma)\,,
        \end{align}
    and the shifted $L_+$-type gradient estimate for times $[s,t] \in \tilde I$ with $- T_+ < s$,
        \begin{align}
            |\nabla P_{t,s}v|^2_t -2(t + T_+)\Delta_t P_{t,s}v - (t + T_+)^2S_t \leq P_{t,s}\left(|\nabla v|^2_s-2(s + T_+)\Delta_s v-(s + T_+)^2 S_s\right) + \frac{N}{2}(t-s)\,,
        \end{align}
        for any Lipschitz function $v: M \to \R{}$.
\end{prop}

\begin{proof}
    We show that the shifted $L_-$-type gradient estimates follow from the hypothesis; the shifted $L_+$-type gradient estimates follow in the same way (up to minor changes in notation and a few flipped signs). 
    
    We first notice that the assumed $W_{\tilde{L}_0}$ contraction is invariant under shifting the family $(M,g_{\tau - T_-})_{\tau \in I +T_-}$, so there is no loss in generality in assuming that $T_- = 0$. For any given $x,y \in M$, $[\bar{\sigma}, \bar{\tau}] \subseteq I$, $\alpha > 1$, let $\bar\gamma:[\bar{\sigma},\bar{\tau}]\to M$ be a curve from $x$ to $y$, and define $[\bar{\sigma},\bar{\tau}]\ni \varsigma \mapsto \hat{P}_{\varsigma,\alpha \varsigma}\delta_{\bar\gamma_\varsigma} =: \mu_\varsigma$ which is a curve of probability measures from $\hat{P}_{\bar{\sigma},\alpha \bar{\sigma}}\delta_{\bar\gamma_{\bar{\sigma}}}$ to $\hat{P}_{\bar{\tau},\alpha\bar{\tau}}\delta_{\bar\gamma_{\bar{\tau}}}$. 
    Using Lemma \ref{lem:dynamiclifting} we can find $\eta \in \mathcal{P}(C([\alpha\bar{\sigma},\alpha\bar{\tau}]\to M))$ with $(e_\varsigma)_\#\eta = \mu_\varsigma$ and
    \begin{align}\label{eq:liftingwithsameKE-again}
        \frac{1}{2}\int_{C([\alpha\bar{\sigma},\alpha\bar{\tau}]\to M)} |\dot{\gamma}_\varsigma|_{\varsigma}^2 + S_{\varsigma}(\gamma_\varsigma) d\eta(\gamma) = (\dot{\mu}_\varsigma)_{L_0^{\varsigma}}.
    \end{align}
    Hence
    {\allowdisplaybreaks
    \begin{align}
        P_{\bar{\tau}, \alpha\bar{\tau}}f(y)-P_{\bar{\sigma}, \alpha \bar{\sigma}}f(x) &= 
        \int f d\hat{P}_{\bar{\tau},\varsigma\bar{\tau}}\delta_{\bar\gamma_{\bar{\tau}}} - \int f d\hat{P}_{\bar{\sigma},\alpha\bar{\sigma}}\delta_{\bar\gamma_{\bar{\sigma}}} \\
        &= \int_{C([\alpha\bar{\sigma},\alpha\bar{\tau}]\to M)} \int_{\alpha\bar{\sigma}}^{\alpha\bar{\tau}} \langle df(\gamma_\varsigma), \dot{\gamma}_\varsigma\rangle d\varsigma d\eta(\gamma) \\
        & \leq \int_{C([\alpha\bar{\sigma},\alpha\bar{\tau}]\to M)} \int_{\alpha\bar{\sigma}}^{\alpha\bar{\tau}} \left(\frac{1}{2\alpha} |\nabla f(\gamma_\varsigma)|_{\varsigma}^2 + \frac{\alpha}{2}|\dot{\gamma}_\varsigma|_{\varsigma}^2\right) d\varsigma d\eta(\gamma) \\
        &= \frac{1}{2\alpha} \int_{C} \int_{\alpha\bar{\sigma}}^{\alpha\bar{\tau}} |\nabla f(\gamma_\varsigma)|_{\varsigma}^2 d\varsigma d\eta(\gamma) + \frac{\alpha}{2}\int_C\int_{\alpha\bar{\sigma}}^{\alpha\bar{\tau}}|\dot{\gamma}_\varsigma|_{\varsigma}^2 d\varsigma d\eta(\gamma) \\
        &= \frac{1}{2\alpha} \int_{\alpha\bar{\sigma}}^{\alpha\bar{\tau}} \int_M |\nabla f|_{\varsigma}^2 d\mu_\varsigma d\varsigma + \frac{\alpha}{2}\int_C\int_{\alpha\bar{\sigma}}^{\alpha\bar{\tau}}\left(|\dot{\gamma}_\varsigma|_{\varsigma}^2 + S_{\varsigma}(\gamma_\varsigma)\right) d\varsigma d\eta(\gamma) \\
        & \ \ \ \ - \frac{\alpha}{2}\int_{\alpha\bar{\sigma}}^{\alpha\bar{\tau}}\int_M S_{\varsigma}d\mu_\varsigma d\varsigma \\
        &= \frac{1}{2}\int_{\bar{\sigma}}^{\bar{\tau}} P_{\varsigma,\alpha\varsigma}[|\nabla f|_{\alpha\varsigma}^2](\bar\gamma_\varsigma) d\varsigma + \alpha^2\int_{\bar{\sigma}}^{\bar{\tau}}(\dot{\mu}_{\alpha\varsigma})_{L_0^{\alpha\varsigma}} d\varsigma \\
        & \ \ \ \ - \frac{\alpha^2}{2}\int_{\bar{\sigma}}^{\bar{\tau}}P_{\varsigma,\alpha\varsigma}[S_{\alpha\varsigma}](\bar\gamma_\varsigma)d\varsigma.
    \end{align}}
    We used Young's inequality, repeatedly the fact that $\mu_\varsigma =(e_\varsigma)_\#\eta$, and \eqref{eq:liftingwithsameKE-again} in the last equality.

    Now, taking $\lim_{\vartheta \to \varsigma}\frac{1}{(\vartheta - \varsigma)^2}$ of \eqref{eq:WCdimL0-GE-to-WC} for times $\varsigma, \vartheta,\alpha\varsigma, \alpha\vartheta$ gives
    \begin{align}
        \alpha^2(\dot{\mu}_{\alpha\varsigma})_{L_0^{\alpha\varsigma}} \leq (\dot{\delta}_{\bar\gamma_{\varsigma}})_{L_0^{\varsigma}} + \frac{N}{4\varsigma}(\alpha - 1) = \frac{1}{2}\left(|\dot{\bar\gamma}_{\varsigma}|_{\varsigma}^2 + S_{\varsigma}(\bar\gamma_{\varsigma}) \right)+ \frac{N}{4\varsigma}(\alpha - 1)\,.
    \end{align}
    Substituting in the previous estimate,
    \begin{align}
        P_{\bar{\tau},\alpha\bar{\tau}}f(y)-P_{\bar{\sigma},\alpha\bar{\sigma}}f(x)
        &\leq \frac{1}{2}\int_{\bar{\sigma}}^{\bar{\tau}} P_{\varsigma,\alpha\varsigma}(|\nabla f|_{\alpha\varsigma}^2)(\bar\gamma_\varsigma) d\varsigma + \frac{1}{2}\int_{\bar{\sigma}}^{\bar{\tau}}\left(|\dot{\bar\gamma}_\varsigma|_{\varsigma}^2 + S_{\varsigma}(\bar\gamma_\varsigma)\right) d\varsigma \\
        & \ \ \ \ - \frac{\alpha^2}{2}\int_{\bar{\sigma}}^{\bar{\tau}}P_{\varsigma,\alpha\varsigma}S_{\alpha\varsigma}(\bar\gamma_\varsigma) d\varsigma + \frac{N}{4}(\alpha - 1)\int_{\bar{\sigma}}^{\bar{\tau}}\frac{1}{\varsigma}\,d\varsigma\,.
    \end{align}
    In the estimate we have the freedom to choose $y$ and $\bar\gamma$ accordingly, and hence we pick $\bar\gamma$ to have $\dot{\bar\gamma}_{\bar{\sigma}}=\nabla P_{\bar{\sigma},\alpha\bar{\sigma}} [f](x)$. We now divide everything by $(\bar{\tau} - \bar{\sigma})$ and take the limit as $\bar{\tau}\downarrow \bar{\sigma}$. The left-hand side becomes 
    \begin{align}
        \partial_\varsigma|_{\varsigma=\bar{\sigma}} P_{\varsigma,\alpha\varsigma}[f](\bar\gamma_\varsigma) &= -\Delta P_{\bar{\sigma},\alpha\bar{\sigma}}[f](x) + \alpha P_{\bar{\sigma},\alpha\bar{\sigma}}[\Delta f](x) + \langle \nabla P_{\bar{\sigma},\alpha\bar{\sigma}}f(x),\dot{\bar\gamma}(\bar{\sigma}) \rangle \\
        &= -\Delta P_{\bar{\sigma},\alpha\bar{\sigma}}[f](x) + \alpha P_{\bar{\sigma},\alpha\bar{\sigma}}[\Delta f](x) + |\nabla P_{\bar{\sigma},\alpha \bar{\sigma}} f|^2_{\bar{\sigma}}(x)\,.
    \end{align}
    And comparing this with the right-hand side gives
    \begin{align}
        &\phantom{\leq}-\Delta P_{\bar{\sigma},\alpha\bar{\sigma}}[f](x) + \alpha P_{\bar{\sigma},\alpha\bar{\sigma}}[\Delta f](x) + |\nabla P_{\bar{\sigma},\alpha \bar{\sigma}} f|^2_{\bar{\sigma}}(x) \\ &\leq 
        \frac{1}{2} P_{\bar{\sigma},\alpha\bar{\sigma}}[|\nabla f|_{\alpha\bar{\sigma}}^2](x) + \frac{1}{2}\left(|\nabla P_{\bar{\sigma},\alpha\bar{\sigma}} [f]|_{\bar{\sigma}}^2(x) + S_{\bar{\sigma}}(x)\right) - \frac{\alpha^2}{2}P_{\bar{\sigma},\alpha\bar{\sigma}}[S_{\alpha\bar{\sigma}}](x)  + \frac{N(\alpha - 1)}{4\bar{\sigma}}\,.
    \end{align}
    We conclude
    \begin{equation}
        -\Delta P_{\bar{\sigma},\alpha\bar{\sigma}}[f] +\alpha P_{\bar{\sigma},\alpha\bar{\sigma}}[\Delta f] + \frac{1}{2} |\nabla P_{\bar{\sigma},\alpha\bar{\sigma}} [f]|^2_{\bar{\sigma}} \leq 
        \frac{1}{2} P_{\bar{\sigma},\alpha\bar{\sigma}}[|\nabla f|_{\alpha\bar{\sigma}}^2] + \frac{1}{2}S_{\bar{\sigma}} - \frac{\alpha^2}{2}P_{\bar{\sigma},\alpha\bar{\sigma}}[S_{\alpha\bar{\sigma}}]+ \frac{N(\alpha - 1)}{4\bar{\sigma}}\,.
    \end{equation}

    Picking now $\bar{\sigma} := \sigma$, $f := v/\sigma$, $\alpha := \tau/\sigma$, and multiplying the above line by $2\sigma^2$, we obtain the desired inequality.
\end{proof}

\subsection{Entropy \texorpdfstring{$\ln\tau$}{lnt}-convexity and \texorpdfstring{$\mathrm{EVI}$}{EVI}}\label{ss:nL0EC/EVI}

This subsection houses the proof of Theorem \ref{thm:dimL0EC/EVI}, showing that \eqref{eq:ECdimL0-EC-to-WC} and \eqref{eq:EVIdimL0-GE-to-EVI} are equivalent to \eqref{eq:GEL--GE-to-dimL0EVI}. Many features of this subsection have already been encountered in the $L_0$ and $L_\pm$ cases, so outside of Lemma and Proposition statements we provide details only when they differ from these previous cases. 

\begin{prop}\label{prop:dim-L0-convexity-implies-contraction}
Assume that $\dim M = n$ and $I \subseteq (0,\infty)$, and pick times $[\sigma,\tau] \subseteq I$. Suppose that for any measures $\mu_{\sigma},\,\mu_{\tau}\in \cP(M)$, there exists a $W_{L_0}$ geodesic $(\mu_{\eta})_{\eta \in [\sigma,\tau]}$ between $\mu_{\sigma},\mu_{\tau}$ such that

\begin{equation}\label{eq:ECdimL0-EC-to-WC}\tag{EC\({}_{L_0,n}\)}
\cE(\mu_\eta) - \cE(\mu_{\sigma}) + W_{L_0^{\sigma,\eta}}(\mu_{\sigma},\mu_\eta) + \frac{n}{8}(\ln(\eta/\sigma))^2 - \int_{\sigma}^\eta \frac{1}{\tilde\eta}W_{L_0^{\sigma,\tilde\eta}}(\mu_{\sigma},\mu_{\tilde\eta})\,d\tilde\eta
\end{equation}

is convex in the variable $\ln\eta$. Then for any $\mu_{\sigma}, \mu_{\tau} \in \cP(M)$ and $1 < \alpha < \sup I/\tau$, it holds that

\begin{align}\label{eq:WCdimL0-mult-times}
 \alpha W_{L_0^{\alpha\sigma, \alpha\tau}}(\hat{P}_{\alpha\sigma,\sigma}[\mu_{\sigma}], \hat{P}_{\alpha\tau,\tau}[\mu_{\tau}]) - \frac{n}{4}\alpha\ln\left(\frac{\tau}{\sigma}\right) \leq W_{L_0^{\sigma,\tau}}(\mu_{\sigma}, \mu_{\tau}) - \frac{n}{4}\ln\left(\frac{\tau}{\sigma}\right)\,.
\end{align}
\end{prop}

\begin{rmk}
The above Wasserstein contraction estimate \eqref{eq:WCdimL0-mult-times}, which is \eqref{eq:WCdimL0-GE-to-WC} in the special case of multiplicatively related times, implies \eqref{eq:GEL--GE-to-dimL0EVI}. This is precisely the content of the proof of Proposition \ref{prop:WCdimL0-implies-shifted-GE}.
\end{rmk}

\begin{proof}
The technical details--in particular the smoothing of the measures by flowing for some small time $\eps > 0$--are identical to those of the proof of Proposition \ref{prop:L0-cvxty-implies-contr/EVI}. We therefore restrict our attention to the formal computations, which all hold rigorously for endpoint measures that are smooth with full support.

Writing $d/d(\ln\eta) = \eta d/d\eta$ and letting $(\phi_\sigma,\phi_\tau)$ be Kantorovich potentials with respect to $L_0$ from $\mu_\sigma$ to $\mu_\tau$, by the convexity assumption we have
\begin{align}
-\frac{n}{4}\ln(\tau/\sigma)&\leq \left. \eta\frac{d}{d\eta}\left[\cE(\mu_\eta) + \int_M \phi_\eta \,d\mu_\eta\right]\right|_{\sigma^+}^{\tau^{-}} - W_{L_0^{\sigma,\tau}}(\mu_{\sigma}, \mu_{\tau}) \\
&= \left. \eta\frac{d}{d\eta}\left[\cE(\mu_\eta) + \int_M \phi_\eta \,d\mu_\eta\right]\right|_{\sigma^+}^{\tau^{-}} - \int_M\phi_{\tau}\,d\mu_{\tau} + \int_M \phi_{\sigma}\,d\mu_{\sigma}\,\,.
\end{align}

Note that the existence of these one-sided derivatives is provided by Lemma \ref{lem:Top-1-sided-est} and Proposition \ref{prop:Lisini-for-geodesics}, which as in Proposition \ref{prop:L0-cvxty-implies-contr/EVI} yields the further inequality (with $\mu_\eta=\rho_\eta dV_\eta$ for $\eta = \sigma,\tau$)
\begin{align}
\cdots &\leq  \left[-\eta\int_M \phi_\eta\Delta\rho_\eta \,dV_\eta -\int_{M}\phi_\eta\,d\mu_{\eta} - \eta\int_M \left(-\frac{|\nabla \phi_\eta|^2}{2} + \frac{S_\eta}{2}\right)\, d\mu_\eta\right]_{\sigma}^{\tau}\,.
\end{align}

Since $\phi_\eta$ (extended past $\tau$ using the Hopf-Lax semigroup) is Lipschitz and solves the Hamilton-Jacobi equation \eqref{eq:L0-HJ} a.e. in $M$ for $\eta = \sigma,\tau$, one computes

\begin{equation}
\cdots =-\left.\frac{d}{d\alpha}\right|_{\alpha = 1^+} \alpha\cdot\left[\int_{M} \phi_{\alpha\tau} d\hat{P}_{\alpha\tau,\tau}[\mu_{\tau}] - \int_{M} \phi_{\alpha\sigma} d\hat{P}_{\alpha\sigma,\sigma}[\mu_{\sigma}]\right]\,.
\end{equation}

Now, as in Proposition \ref{prop:L0-cvxty-implies-contr/EVI}, and modulo the technical details therein, this shows that 

\begin{equation}
\left.\partial^+\right|_{\alpha = 1^+} \left[\alpha\cdot W_{L_0^{\alpha\sigma, \alpha\tau}}(\hat{P}_{\alpha\sigma,\sigma}[\mu_{\sigma}], \hat{P}_{\alpha\tau,\tau}[\mu_{\tau}]) - \frac{n}{4}\ln(\tau/\sigma)\alpha\right] \leq 0\,,
\end{equation}
and thus that the quantity $\alpha\cdot W_{L_0^{\alpha\sigma, \alpha\tau}}(\hat{P}_{\alpha\sigma,\sigma}[\mu_{\sigma}], \hat{P}_{\alpha\tau,\tau}[\mu_{\tau}]) - (n/4)\alpha\ln(\tau/\sigma)$ is non-increasing in $\alpha > 1$.
\end{proof}

To pass now from gradient estimates \eqref{eq:GEL--GE-to-dimL0EVI} to the entropy convexity statement \eqref{eq:ECdimL0-EC-to-WC}, we also proceed as in the $L_0$ and $L_{-}$ cases.

\begin{lem}\label{lem:dim-L0-action-entropy-est}
Assume that $\dim M = n$ and $I\subseteq (0,\infty)$, and let $(\nu_\eta)_{\eta \in [\sigma,\tau]} = (\rho_\eta dV_\eta)_{\eta \in [\sigma,\tau]}$, $\rho \in C^\infty\big([\sigma,\tau]\times M; (0,\infty)\big)$, $[\sigma,\tau] \subseteq I$, be a curve of smooth measures with positive densities. For $\sup I/\tau,\tau/\sigma > \alpha > 1$, write $\lambda_I := \ln (\alpha)/\ln(\alpha\tau/\sigma)$, $\lambda_{II} := \ln(\alpha)/\ln\big(\tau/(\alpha\sigma)\big)$. Then if the $L_-$-type gradient estimate \eqref{eq:GEL--GE-to-dimL0EVI} holds, we have the action-entropy estimates
\begin{align}
\lambda_I\cE(\hat{P}_{\tau, \alpha\tau}[\nu_{\tau}]) + W_{L_0^{\sigma,\alpha\tau}}(\nu_{\sigma},\hat{P}_{\tau, \alpha\tau}[\nu_{\tau}]) &\leq \lambda_I\cE(\nu_{\sigma}) + (1 - \lambda_I) \int_{\sigma}^{\tau}(\dot{\nu}_{\tilde\eta})_{L_0^{\tilde\eta}}\,d\tilde\eta \\
&\qquad -\lambda_I\int_{\sigma}^{\tau}\frac{1}{\eta}\left(\int_{\eta}^{\tau}(\dot{\nu}_{\tilde\eta})_{L_0^{\tilde\eta}}\,d\tilde\eta\right)\,d\eta\\
&\qquad +\lambda_I\frac{n}{8}\big(\ln(\tau/\sigma)\big)^2  + O((\ln\alpha)^2)\,,\label{eq:dim-L0-preEVI-I}
\end{align}
and
\begin{align}
-\lambda_{II} \cE(\nu_{\tau}) + W_{L_0^{\alpha\sigma,\tau}}(\hat{P}_{\sigma,\alpha\sigma}[\nu_{\sigma}],\nu_{\tau}) &\leq -\lambda_{II} \cE(\hat{P}_{\sigma,\alpha\sigma}[\nu_{\sigma}]) + (1 + \lambda_{II})\int_{\sigma}^{\tau}(\dot{\nu}_{\tilde\eta})_{L_0^{\tilde\eta}}\,d\tilde\eta \\
&\qquad - \lambda_{II}\int_{\sigma}^{\tau}\frac{1}{\eta}\left(\int_{\sigma}^{\eta}(\dot{\nu}_{\tilde\eta})_{L_0^{\tilde \eta}}\,d\tilde\eta\right)\,d\eta \\
&\qquad -\lambda_{II}\frac{n}{8}\big(\ln(\tau/\sigma)\big)^2 + O((\ln\alpha)^2)\,. \label{eq:dim-L0-preEVI-II}
\end{align}
The constant implicit in the $O((\ln\alpha)^2)$ error term depends on quantities that are independent of the curve $(\nu_\eta)_{\eta \in [\sigma,\tau]}$ (namely $\text{diam}_{L_0}((M,g_\eta)_{\eta \in [\sigma,\tau]}) $, $-\inf_{(\eta,x)\in I \times M} \min(0, S_\eta(x))$, $n$, $\tau$, $\sigma$), and on a uniform bound on the actions $\sup_{[\tilde\sigma,\tilde{\tau}]\subseteq [\sigma,\tau]} \left|\int_{\tilde\sigma}^{\tilde\tau}(\dot\nu_{\eta})_{L_0^\eta}\,d\eta\right|$.
\end{lem}
\begin{proof}
We proceed as in the proof of Lemma \ref{lem:L-minus-action-entropy-est}, but with even fewer details when the two proofs align. In this case, we consider the $\ln$-linear, bijective map of intervals $\xi_\alpha: [\sigma,\alpha\tau] \to [\sigma,\tau]$ given by

\begin{equation}
\xi_\alpha:\begin{cases}\xi_\alpha(\sigma) = \sigma\,,\,\xi_\alpha(\alpha\tau) = \tau\,,& \\ \ln(\xi_\alpha(\eta)) = (1 - \lambda) \cdot\ln(\eta) + b&\end{cases}\, \implies \lambda = \frac{\ln \alpha}{\ln(\alpha\tau/\sigma)}\,,\qquad b= \frac{\ln\sigma\cdot \ln\alpha}{\ln(\alpha\tau/\sigma)}\,.
\end{equation}

The two properties of $\xi_\alpha$ that will interest us are its values at the boundary of the interval, along with the ODE $\xi_\alpha'(\eta) = (1 - \lambda)\cdot \xi_\alpha(\eta)/\eta$ it solves. As in Lemma \ref{lem:L0-action-entropy-est}, we now define the ``diagonal'' curve of measures $\rho_{\eta,\alpha}\,dV_\eta := \nu_{\eta,\alpha} := \hat{P}_{\xi_{\alpha}(\eta),\eta}[\nu_{\xi_\alpha(\eta)}]$ for $\eta \in [\sigma,\alpha\tau]$, along with a family of Kantorovich potentials $(\phi_\eta)_{\eta \in [\sigma,\alpha\tau]}: M \to \R{}$ from Corollary \ref{cor:dynamic-potentials}, which are Lipschitz and solve the Hamilton-Jacobi equation \eqref{eq:L0-HJ} for every $\eta \in [\sigma,\alpha\tau]$ and a.e. $x \in M$. We compute

\begin{align}
\eta^{2}\partial_\eta \cE(\nu_{\eta,\alpha}) &= (1 - \lambda)\int_M P_{\xi_\alpha(\eta),\eta}[\eta\Delta_\eta \big(\eta\ln \rho_{\eta,\alpha}\big) - \frac{\lambda}{1 - \lambda} |\nabla\big(\eta\ln \rho_{\eta,\alpha}\big)|^2] \,d\nu_{\xi_\alpha(\eta)} \\
&\qquad - (1 - \lambda) \int_M \Big(\xi_\alpha(\eta)\Delta_{\xi_\alpha(\eta)}P_{\xi_\alpha(\eta),\eta}[\eta\ln \rho_{\eta,\alpha}]\Big)\,d\nu_{\xi_\alpha(\eta)} - \eta^{2}\int_M S_\eta\,d\nu_{\eta,\alpha}\\
&\qquad + (1 - \lambda)\int_M \xi_\alpha(\eta) P_{\xi_\alpha(\eta),\eta}[\eta\big(1 + \ln \rho_{\eta,\alpha}\big)]\left.\partial_\upsilon\right|_{\upsilon = \xi_\alpha(\eta)} \,d\nu_\upsilon\,,
\end{align}

and

\begin{align}
\eta^{2}\partial_\eta \int_M \phi_\eta \,d\nu_{\eta,\alpha} &= (1 - \lambda)\int_M P_{\xi_\alpha(\eta),\eta}[\eta\Delta_\eta (\eta\phi_\eta) - \frac{\lambda}{1 - \lambda}\La\nabla (\eta\phi_\eta), \nabla (\eta\ln \rho_{\eta,\alpha})\Ra +\frac{\eta^{2}}{2} S_{\eta} -\frac{1}{2}|\nabla (\eta\phi_\eta)|^2]\,d\nu_{\xi_\alpha(\eta)} \\
&\qquad - (1 - \lambda)\int_M\left(\xi_\alpha(\eta)\Delta_{\xi_\alpha(\eta)} P_{ \xi_\alpha(\eta),\eta}[\eta\phi_\eta] \right)\,d\nu_{\xi_\alpha(\eta)} + \lambda\int_M \frac{\eta^2}{2}S_\eta - \frac{1}{2}|\eta\nabla \phi_\eta|^2\,d\nu_{\eta,\alpha} \\
&\qquad + (1 - \lambda) \int_M \xi_\alpha(\eta)P_{\xi_\alpha(\eta),\eta}[\eta\phi_\eta]\left.\partial_\upsilon\right|_{\upsilon = \xi_\alpha(\eta)}\,d\nu_\upsilon\,.
\end{align}

We now verify the remaining steps from the proof of Lemma \ref{lem:L0-action-entropy-est}. In particular, we consider the linear combination $\eta^2\partial_\eta(\int_M \phi_\eta\,d\nu_{\eta,\alpha} + [\lambda/(1 -\lambda)]\cE(\nu_{\eta,\alpha}))$ of the LHS's above and attempt to estimate the terms of the resulting RHS. As in the dimensionless case, we define the auxiliary quantity $f_\eta: = \phi_\eta + \lambda\ln(\rho_{\eta,\alpha})/(1 - \lambda)$. 

First, similarly to the proof of Lemma \ref{lem:L0-action-entropy-est} we estimate

\begin{align}
&\phantom{\leq}(1 - \lambda)\int_M \xi_\alpha(\eta)P_{\xi_\alpha(\eta),\eta}[\eta f_\eta]\,\left.\partial_\upsilon\right|_{\upsilon = \xi_\alpha(\eta)}d\nu_\upsilon \\
&\leq \frac{1 - \lambda}{2}\int_M \Big(|\nabla P_{\xi_\alpha(\eta),\eta}[\eta f_\eta]|^2 - \xi_\alpha^2(\eta) S_{\xi_\alpha(\eta)}\Big)\,d\nu_{\xi_\alpha(\eta)} + \xi_\alpha(\eta)\eta\xi_\alpha'(\eta)(\dot{\nu}_{\xi_\alpha(\eta)})_{L_0^{\xi_\alpha(\eta)}}\,. 
\end{align}

Second, we apply the gradient estimate \eqref{eq:GEL--GE-to-dimL0EVI} with initial data $\eta f_\eta$ between the times $\xi_\alpha(\eta) < \eta$ to bound

\begin{align}
&(1 - \lambda)\int_M P_{\xi_\alpha(\eta),\eta}[\eta\Delta_\eta (\eta f_\eta) - \frac{1}{2}|\nabla (\eta f_\eta)|^2 + \frac{\eta^2}{2}S_\eta]\,d\nu_{\xi_\alpha(\eta)} \\
&\qquad - (1 - \lambda)\int_M\left(\xi_\alpha(\eta)\Delta_{\xi_\alpha(\eta)} P_{\xi_\alpha(\eta),\eta}[\eta f_\eta] -\frac{1}{2} |\nabla P_{\xi_\alpha(\eta),\eta}[\eta f_\eta]|^2 + \frac{\xi_\alpha^2(\eta)}{2}S_{\xi_\alpha(\eta)} \right)\,d\nu_{\xi_\alpha(\eta)} \\
&\leq \frac{n}{4}(1 - \lambda)( \eta - \xi_\alpha(\eta))\,.
\end{align}

We now assemble these two estimates as in Lemma \ref{lem:L0-action-entropy-est}, divide by $\eta^2$, and then integrate between $\sigma$ and $\alpha\tau$ to obtain

\begin{align}
(1 + \lambda)\int_M \phi_\eta \,d\nu_{\eta,\alpha} + \left.\frac{\lambda}{1 - \lambda}\cE(\nu_{\eta,\alpha}) \right|_{\sigma}^{\alpha\tau} &\leq \int_{\sigma}^{\alpha\tau}\frac{\xi_\alpha(\eta)}{\eta}(\dot{\nu}_{\xi_\alpha(\eta)})_{L_0^{\xi_\alpha(\eta)}}\xi_\alpha'(\eta)\,d\eta \\
&\qquad - \frac{\lambda^2}{1 - \lambda}\int_{\sigma}^{\alpha\tau}\int_M S_\eta \,d\nu_{\eta,\alpha}d\eta \\
&\qquad +\frac{n}{4}(1 - \lambda)\int_{\sigma}^{\alpha\tau}\left(\frac{1}{\eta} - \frac{\xi_\alpha(\eta)}{\eta^2}\right)\,d\eta \,.
\end{align}

Up to $O((\ln \alpha)^2)$ error, it remains to compute the first and third integrals of the RHS above. For the first, we integrate by parts to make the action of $\nu_\eta$ appear:

\begin{align}
\int_{\xi_\alpha(\sigma)}^{\xi_\alpha(\alpha\tau)}\frac{\bar\eta}{\xi_\alpha^{\circ - 1}(\bar\eta)}(\dot{\nu}_{\bar\eta})_{L_0^{\bar \eta}}\,d\bar \eta &= \frac{\xi_\alpha(\alpha\tau)}{\alpha\tau} \int_{\xi_\alpha(\sigma)}^{\xi_\alpha(\alpha\tau)}(\dot{\nu}_{\tilde\eta})_{L_0^{\tilde\eta}}\,d\tilde\eta+\frac{\lambda}{1 - \lambda}\int_{\xi_\alpha(\sigma)}^{\xi_\alpha(\alpha \tau)}\frac{1}{\xi_\alpha^{\circ-1}(\bar\eta)}\left(\int_{\xi_\alpha(\sigma)}^{\bar\eta}(\dot{\nu}_{\tilde\eta})_{L_0^{\tilde\eta}}\,d\tilde\eta\right)\,d\bar\eta \\
&= \frac{1}{\alpha} \int_{\sigma}^{\tau}(\dot{\nu}_{\tilde\eta})_{L_0^{\tilde\eta}}\,d\tilde\eta+\lambda\int_{\sigma}^{\alpha\tau}\frac{\xi_\alpha(\eta)}{\eta^2}\left(\int_{\xi_\alpha(\sigma)}^{\xi_\alpha(\eta)}(\dot{\nu}_{\tilde\eta})_{L_0^{\tilde\eta}}\,d\tilde\eta\right)\,d\eta\,.\\
&=  \int_{\sigma}^{\tau}(\dot{\nu}_{\tilde\eta})_{L_0^{\tilde\eta}}\,d\tilde\eta +\frac{1 - \alpha}{\alpha\ln(\tau/\sigma)}\int_{\sigma}^{\tau}\frac{1}{\eta}\left(\int_{\sigma}^{\tau}(\dot\nu_{\tilde\eta})_{L_0^{\tilde\eta}}\,d\tilde\eta\right) d\eta \\
&\qquad +\lambda\int_{\sigma} ^{\alpha\tau}\frac{\xi_\alpha(\eta)}{\eta^2}\left(\int_{\xi_\alpha(\sigma)}^{\xi_\alpha(\eta)}(\dot{\nu}_{\tilde\eta})_{L_0^{\tilde\eta}}\,d\tilde\eta\right)\,d\eta\,.\\
\end{align}
The final two terms of the RHS above can be combined, up to allowable error, to form the desired weighted action integral:
\begin{align}
&\phantom{=}\frac{1 - \alpha}{\alpha\ln(\tau/\sigma)}\int_{\sigma}^{\tau}\frac{1}{\eta}\left(\int_{\sigma}^{\tau}(\dot\nu_{\tilde\eta})_{L_0^{\tilde\eta}}\,d\tilde\eta\right) d\eta +\lambda\int_{\sigma}^{\alpha\tau}\frac{\xi_\alpha(\eta)}{\eta^2}\left(\int_{\xi_\alpha(\sigma)}^{\xi_\alpha(\eta)}(\dot{\nu}_{\tilde\eta})_{L_0^{\tilde\eta}}\,d\tilde\eta\right)\,d\eta \\
&= \frac{\ln \alpha}{\ln(\tau/\sigma)}\int_{\sigma}^{\tau}\frac{1}{\eta}\left(-\int_{\sigma}^{\tau}(\dot\nu_{\tilde\eta})_{L_0^{\tilde\eta}}\,d\tilde\eta + \int_{\sigma}^{\eta}(\dot\nu_{\tilde\eta})_{L_0^{\tilde\eta}}\,d\tilde\eta\right)\,d\eta+ O((\ln\alpha)^2) \\
&= -\frac{\ln \alpha}{\ln(\tau/\sigma)}\int_{\sigma}^{\tau}\frac{1}{\eta}\left(\int_{\eta}^{\tau}(\dot\nu_{\tilde\eta})_{L_0^{\tilde\eta}}\,d\tilde\eta\right)\,d\eta+ O((\ln\alpha)^2)\,.
\end{align}

Finally, we compute the $O(\ln\alpha)$ dimensional integral, whose $\ln\alpha$ coefficient is

\begin{align}
\lim_{\alpha \to 1^+}\frac{1}{\ln\alpha}\left(\frac{n}{4}(1 - \lambda)\int_{\sigma}^{\alpha\tau}\left(\frac{1}{\eta} - \frac{\xi_\alpha(\eta)}{\eta^2}\right)\,d\eta\right) &= \lim_{\alpha \to 1^+}\frac{n}{4}(1 - \lambda)\int_{\sigma}^{\tau}\frac{\alpha - 1}{\eta^2\ln\alpha}\frac{\eta - \xi_\alpha(\eta)}{\alpha - 1}\,d\eta \\
&= \frac{n}{4}\int_{\sigma}^{\tau}\frac{-\left.\partial_\alpha\right|_{\alpha = 1^+}\xi_\alpha(\eta)}{\eta^2}\,d\eta \\
&= \frac{n}{8}\ln(\tau/\sigma)\,,
\end{align}
since
\begin{equation}
\left.\partial_\alpha\right|_{\alpha = 1^+} \xi_\alpha(\eta) = \left.\partial_\alpha\right|_{\alpha = 1^+}\exp\left(\frac{\ln(\tau/\sigma)}{\ln(\alpha\tau/\sigma)}\ln \eta + \frac{\ln\sigma\cdot \ln \alpha}{\ln(\alpha\tau/\sigma)}\right) = \frac{\eta\ln(\sigma/\eta)}{\ln(\tau/\sigma)}\,.
\end{equation}

After multiplying through by $(1 - \lambda)$, this gives \eqref{eq:dim-L0-preEVI-I}. \eqref{eq:dim-L0-preEVI-II} follows from running the same argument (but with fewer steps for the action term) with $\nu_\eta^\alpha := \hat{P}_{\tilde{\xi}_\alpha(\eta),\eta}[\nu_{\tilde{\xi}_\alpha(\eta)}]$ for the range of times $[\alpha\sigma,\tau]$, where $\tilde{\xi}_\alpha: [\alpha\sigma,\tau] \to [\sigma,\tau]$ solves

\begin{equation}
\tilde{\xi}_\alpha(\tau): \begin{cases}\tilde{\xi}_\alpha(\alpha\sigma) = \sigma\,,\,\tilde{\xi}_\alpha(\tau) = \tau\,,&\\
\ln(\tilde{\xi}_\alpha(\eta)) = (1 + \tilde{\lambda})\cdot \ln\eta + \tilde{b}\end{cases}\,\implies \tilde{\lambda} = \frac{\ln \alpha}{\ln(\tau/(\alpha\sigma))}\,,\qquad \tilde b = -\frac{\ln\tau\cdot\ln \alpha}{\ln(\tau/(\alpha\sigma))}\,.
\end{equation}
\end{proof}

As with the dimensionless $L_0$ and $L_{\pm}$ cases, we can take a normalized sum of inequalities \eqref{eq:dim-L0-preEVI-I} and \eqref{eq:dim-L0-preEVI-II} applied along a $W_{L_0}$ geodesic, then send $\alpha \to 1^+$.

\begin{prop}\label{prop:dimL0GE-to-EVI}
Assume that $\dim M = n$, $I \subseteq (0,\infty)$, and that for any Lipschitz $v$, any $[\sigma,\tau] \subseteq I$ the gradient estimate

\begin{equation}
 |\nabla P_{\sigma,\tau}v|_\sigma^2-2\sigma \Delta_\sigma P_{\sigma,\tau}v-\sigma^2S_{\sigma} \leq P_{\sigma,\tau}\left(|\nabla v|_\tau ^2 - 2\tau\Delta_\tau v - \tau ^2S_{\tau}\right) + \frac{n}{2}(\tau-\sigma)\,.\label{eq:GEL--GE-to-dimL0EVI}\tag{GE\({}_{L_-}\)}
\end{equation}
holds. Then, for any Borel probability measures $\mu,\nu \in \cP(M)$ and times $[\sigma,\tau] \subseteq I$, the backward heat flow $\hat P$ satisfies the following $\mathrm{EVI}$s

\begin{align}
 \eta\cdot\partial_{\eta^+}^+ W_{L_0^{\eta,\tau}}(\hat{P}_{\sigma,\eta}[\nu], \mu) &\leq \frac{1}{\ln(\tau/\eta)} \left[\cE(\mu) -\cE(\hat{P}_{ \sigma,\eta}[\nu]) +W_{L_0^{\eta,\tau}}(\hat{P}_{\sigma,\eta}[\nu], d\mu) -\int_{\eta}^{\tau} \frac{1}{\tilde\tau}  W_{L_0^{\eta,\tilde\tau}}(\hat{P}_{\sigma,\eta}[\nu],\tilde{\nu}_{\tilde\tau})\,d\tilde\tau\right]\\
&\qquad + \frac{n\ln(\tau/\eta)}{8}\,,\qquad \forall\eta \in [\sigma,\tau)\,.\label{eq:EVIdimL0(1)}\\
\varsigma\cdot \partial_{\varsigma^+}^+ W_{L_0^{\sigma,\varsigma}}(\nu,\hat{P}_{\tau,\varsigma}[\mu]) &\leq \frac{1}{\ln(\varsigma/\sigma)} \left[\cE(\nu) - \cE(\hat{P}_{\tau,\varsigma}[\mu]) - W_{L_0^{\sigma,\varsigma}}(\nu,\hat{P}_{\tau,\varsigma}[\mu]) -\int_{\sigma}^{\varsigma} \frac{1}{\tilde\tau}  W_{L_0^{\tilde{\tau},\varsigma}}(\tilde{\mu}_{\tilde\tau}, \hat{P}_{\tau,\varsigma}[\mu])\,d\tilde\tau\right]\\
&\qquad + \frac{n\ln(\varsigma/\sigma)}{8}\,,\qquad\forall\varsigma \in [\tau,\sup I)\,.\label{eq:EVIdimL0-GE-to-EVI}\tag{EVI\({}_{L_0,n}\)}
\end{align}

Here, $(\tilde{\nu}_{\tau})_{\tau \in [\eta,\tau]}$ is any $W_{L_0^{\eta,\tau}}$-geodesic with endpoints $\hat{P}_{\sigma,\eta}[\nu]$ and $\mu$, while $(\tilde{\mu}_{\tau})_{\tau \in [\sigma,\varsigma]}$ is any $W_{L_0^{\sigma,\varsigma}}$-geodesic with endpoints $\nu$ and $\hat P_{\tau,\varsigma}[\mu]$.
\end{prop}

With Lemma \ref{lem:dim-L0-action-entropy-est}, the proof of Proposition \ref{prop:dimL0GE-to-EVI} proceeds very similarly to the $L_0$ case of Proposition \ref{prop:L0GE-to-EVI} and the $L_-$ case of Proposition \ref{prop:L-GE-to-EVI}, so we omit it. Similarly, the same method used to prove Propositions \ref{prop:L0-entropy-cvx} and \ref{prop:L-minus-entropy-cvx} in the $L_0$ resp. $L_-$ cases gives a corresponding dimensionally sensitive $L_0$ version.

\begin{prop}\label{prop:dim-L0-entropy-cvx}
Assume that $ \dim M = n$, $I \subseteq (0,\infty)$, and suppose that for any Borel probability measures $\mu,\nu \in \cP(M)$ and times $[\sigma,\tau] \subseteq I$, the backward heat flow $\hat P$ satisfies the $\mathrm{EVI}$s \eqref{eq:EVIdimL0-GE-to-EVI}. Then for times $ \sigma \leq \bar\eta \leq \tau$ with $[\sigma,\tau] \subseteq I$ and a $W_{L_0}$ geodesic $(\mu_{\eta})_{\eta \in [\sigma,\tau]}\subseteq \cP(M)$, the following entropy convexity statement holds:

\begin{align}\label{eq:dim-L0-entropy-convexity}
&\frac{\ln(\tau/\bar\eta)}{\ln(\tau/\sigma)}(\cE(\mu_{\sigma}) + (n/8)(\ln\sigma)^2) + \frac{ \ln(\bar\eta/\sigma)}{\ln(\tau/\sigma)}(\cE(\mu_{\tau}) + (n/8)(\ln\tau)^2) - (\cE(\mu_{\bar\eta }) + (n/8)(\ln\bar\eta)^2) \\
&\qquad - \frac{\ln(\tau/\bar\eta)}{\ln(\tau/\sigma)}W_{L_0^{\sigma,\bar\eta}}(\mu_{\sigma}, \mu_{\bar\eta}) + \frac{ \ln(\bar\eta/\sigma)}{\ln(\tau/\sigma)}W_{L_0^{\bar\eta,\tau}}(\mu_{\bar\eta}, \mu_{\tau}) \\
&\geq \frac{\ln (\bar\eta/\sigma)}{\ln(\tau/\sigma)}\int_{\bar\eta}^{\tau}\frac{W_{L_0^{\sigma,\eta}}(\mu_{\sigma},\mu_{\eta})}{\eta} \,d\eta +\frac{\ln (\tau/\bar\eta)}{\ln(\tau/\sigma)}\int_{\sigma}^{\bar\eta}\frac{W_{L_0^{\sigma,\eta}}(\mu_{\sigma},\mu_{\eta})}{\eta} \,d\eta\,,
\end{align}
and we recover \eqref{eq:ECdimL0}.
\end{prop}

\begin{rmk}\label{rmk:no-L0-reduced-vol}
    This convexity statement seems to be new even for Ricci flows. One might hope that by following the strategy of \cite[Cor.~8]{Lot09}, Theorem \ref{prop:dim-L0-entropy-cvx} would give rise to a new version of reduced volume monotonicity but now for the $L_0$ cost. However, the cited argument does not work here, as the convex quantity in Theorem \ref{thm:dimL0EC/EVI} does not tend to a constant as $\ln \eta \to -\infty$. This can be seen by noticing that all terms apart from $(n/8)(\ln\bar\eta)^2$ are $O(\ln \eta)$ as $\eta \to 0$. This consideration also shows that one cannot obtain the desired convergence as $\eta\to 0$ by simply adding an $\ln\eta$-convex function of $\eta$.
\end{rmk}

\begin{proof}
    Given a geodesic $(\mu_\eta)_{\eta\in[\sigma,\tau]}$, applying both \eqref{eq:EVIdimL0-GE-to-EVI}s and using the triangle inequality as in the proof of Proposition \ref{prop:L0-entropy-cvx} gives
    \begin{align}
        &\frac{1}{\ln(\tau/\bar\eta)}\left[ \cE(\mu_\tau)-\cE(\mu_{\bar\eta})+W_{L_0^{\bar\eta,\tau}}(\mu_{\bar\eta},\mu_\tau) - \int_{\bar\eta}^\tau \frac{1}{\eta}W_{L_0^{\bar\eta,\eta}}(\mu_{\bar\eta},\mu_\eta)d\eta \right] +\frac{n}{8}\ln(\tau/\bar\eta) \\
        &+ \frac{1}{\ln(\bar\eta/\sigma)}\left[  \cE(\mu_\sigma)-\cE(\mu_{\bar\eta})-W_{L_0^{\sigma,\bar\eta}}(\mu_\sigma,\mu_{\bar\eta}) - \int_\sigma^{\bar\eta} \frac{1}{\eta}W_{L_0^{\eta,\bar\eta}}(\mu_\eta,\mu_{\bar\eta})d\eta \right] + \frac{n}{8}\ln(\bar{\eta}/\sigma) \geq 0.
    \end{align}
    Multiplying by $\frac{\ln(\tau/\bar\eta)\ln(\bar\eta/\sigma)}{\ln(\tau/\sigma)}>0$, and using the relation
    \begin{align}
        &\frac{\ln(\tau/\bar\eta)\ln(\bar\eta/\sigma)}{\ln(\tau/\sigma)}\frac{n}{8}\left(\ln(\tau/\bar\eta)+\ln(\bar\eta/\sigma)\right) = \frac{n}{8\ln(\tau/\sigma)}\left( \ln^2\tau\ln(\bar\eta/\sigma)+\ln^2\sigma\ln(\tau/\bar\eta)-\ln^2\bar\eta\ln(\tau/\sigma) \right) \,,
    \end{align}
    yields
    \begin{align}
        &\frac{\ln(\tau/\bar\eta)}{\ln(\tau/\sigma)}(\cE(\mu_\sigma)+\frac{n}{8}\ln^2(\sigma)) + \frac{\ln(\bar\eta/\sigma)}{\ln(\tau/\sigma)}(\cE(\mu_\tau)+\frac{n}{8}\ln^2(\tau)+W_{L_0^{\sigma,\tau}}(\mu_\sigma,\mu_\tau)) - (\cE(\mu_{\bar\eta})+\frac{n}{8}\ln^2(\bar\eta)+W_{L_0^{\sigma,\bar\eta}}(\mu_\sigma,\mu_{\bar\eta})) \\
        &\geq \frac{\ln(\bar\eta/\sigma)}{\ln(\tau/\sigma)}\int_{\bar\eta} ^\tau \frac{1}{\eta}W_{L_0^{\bar\eta,\eta}}(\mu_{\bar\eta},\mu_\eta)d\eta + \frac{\ln(\tau/\bar\eta)}{\ln(\tau/\sigma)} \int_\sigma^{\bar\eta} \frac{1}{\eta}W_{L_0^{\eta,\bar\eta}}(\mu_\eta,\mu_{\bar\eta})d\eta
    \end{align}
    which is the claimed inequality. Now we note that for any $\sigma_0 \leq \sigma$ within the maximal interval of definition for $\mu_\eta$, we can rewrite the RHS as

\begin{align}
    &\phantom{=} \frac{\ln(\bar\eta/\sigma)}{\ln(\tau/\sigma)}\int_{\bar\eta}^{\tau}\frac{1}{\eta}W_{L_0^{\bar\eta,\eta}}(\mu_{\bar\eta},\mu_\eta)\,d\eta + \frac{\ln(\tau/\bar\eta)}{\ln(\tau/\sigma)}\int_{\sigma}^{\bar\eta}\frac{1}{\eta}W_{L_0^{\eta,\bar\eta}}(\mu_\eta,\mu_{\bar\eta})\,d\eta +(1 - 1)\frac{\ln(\bar\eta/\sigma)\ln(\tau/\bar\eta)}{\ln(\tau/\sigma)}W_{L_0^{\sigma_0,\bar\eta}}(\mu_{\sigma_0},\mu_{\bar\eta}) \\
    &= \frac{\ln(\bar\eta/\sigma)}{\ln(\tau/\sigma)}\int_{\bar\eta}^{\tau}\frac{1}{\eta}W_{L_0^{\sigma_0,\eta}}(\mu_{\sigma_0},\mu_\eta)\,d\eta - \frac{\ln(\tau/\bar\eta)}{\ln(\tau/\sigma)}\int_{\sigma}^{\bar\eta}\frac{1}{\eta}W_{L_0^{\sigma_0,\eta}}(\mu_{\sigma_0},\mu_{\eta})\,d\eta \\
    &\qquad+ \left(\frac{\ln(\bar\eta/\sigma)}{\ln(\tau/\sigma)} + \frac{\ln(\tau/\bar\eta)}{\ln(\tau/\sigma)} - 1\right)\int_{\sigma_0}^{\bar\eta}\frac{1}{\eta}W_{L_0^{\sigma_0,\eta}}(\mu_{\sigma_0},\mu_\eta)\,d\eta \\
    &= \frac{\ln(\bar\eta/\sigma)}{\ln(\tau/\sigma)}\int_{\sigma_0}^{\tau}\frac{1}{\eta}W_{L_0^{\sigma_0,\eta}}(\mu_{\sigma_0},\mu_\eta)\,d\eta + \frac{\ln(\tau/\bar\eta)}{\ln(\tau/\sigma)}\int_{\sigma_0}^\sigma \frac{1}{\eta}W_{L_0^{\sigma_0,\eta}}(\mu_{\sigma_0},\mu_\eta)\,d\eta - \int_{\sigma_0}^{\bar\eta}\frac{1}{\eta}W_{L_0^{\sigma_0,\eta}}(\mu_{\sigma_0},\mu_\eta)\,d\eta\,.
\end{align}

Applying the resulting convex combination inequality for any triple of times in $[\sigma,\tau]$ with $\sigma_0 := \sigma$ shows that the function $\cE(\nu_\eta) - \cE(\nu_{\sigma}) + W_{L_0^{\sigma,\eta}}(\nu_{\sigma},\nu_\eta) + \frac{n}{8}(\ln\eta)^2 - \int_{\sigma}^\eta \frac{1}{\tilde\eta}W_{L_0^{\sigma,\tilde\eta}}(\nu_{\sigma},\nu_{\tilde\eta})\,d\tilde\eta$ is convex in $\ln\eta$, and thus so is the function in \eqref{eq:ECdimL0}, as they differ just by the term $-\frac{n}{4}\ln\eta\ln\sigma + \frac{n}{8}\ln^2\sigma$, which is $(\ln\eta)$-affine.

\end{proof}

\section{Extensions to smooth metric flows with a weight}\label{s:weight}

A natural generalization of a smooth metric flow $(M^m,g_t)_{t \in I}$ (i.e. a time-dependent Riemannian manifold) is a smooth metric-{\it measure} flow, where one additionally includes the data of a time-dependent weight $e^{-U_t}$ for a smooth family of smooth functions $U_t : M \to \R{}$. In this Section we write down a (``dimensionless'') $\cD_{U,\infty}$ condition adapted to this new setting, along with a ``dimensional'' $\cD_{U,N}$ condition. The latter includes a parameter $N \in [m,\infty]$, which may be thought of as the effective dimension of $(M^m,g_t, e^{-U_t}\,dV_t)_{t \in I}$--see Remark \ref{rmk:weighted-warped-prod-equiv} for a possible precise interpretation of this last statement. In terms of $(\cD \geq 0)$-type conditions, the quantities $\cD_{U,\infty}$ and $\cD_{U,N}$ are the time-dependent analogs of the Bakry-\'Emery-Ricci tensors $\Ric_{U} = \Ric_{U,\infty}$ resp. $\Ric_{U,N}$ in the static setting (see \cite[pg.~395]{Vil09}). We note here that a comparable extension to smooth metric-measure spaces was treated by Lott \cite[Sec.~8]{Lot09} in the setting of solutions to the Ricci flow.

In the remainder of the section, we assume that we have in hand a smooth time-dependent weighted Riemannian manifold $(M^m,g_t,\mathfrak{m}_t = e^{-U_t}\,dV_t)_{t \in I}$ with topological dimension $\dim M = m$. We also assume that we have picked a $T \in \R{}$ and made the corresponding choice of backwards time variable $\tau := T - t$ for $\tau \in T - I =: \tilde I$.

\subsection{The dimensionless case \texorpdfstring{$\cD_{U,\infty} \geq 0$}{DU0}.}\label{subsec:infinity-weighted}

As before, we write $\partial_tg_t=-2\mathcal{S}_t$, but now introduce $\partial_t \mathfrak{m}_t=-S_U\mathfrak{m}_t$, i.e. $S_U=\Tr_g{\mathcal{S}}+\partial_tU$. Let $\Delta_U=\Delta_g-\langle\nabla U,\nabla\cdot\rangle_g$ be the weighted Laplacian, $P_{t,s}^U$ the semigroup induced by it, with $ (P_{t,s}^U)^*$ its adjoint and $\hat{P}_{t,s}^U$ the backward semigroup for measures (see Subsection \ref{subsec:t-dep-mflds} for details in the unweighted case). In particular,
\begin{align}
    \partial_tP_{t,s}^Uv_t=\Delta_{U} P_{t,s}^Uv_s\,, \\
    \partial_sP_{t,s}^Uv_t = -P_{t,s}\Delta_{U}v_s\,.
\end{align}

As in the unweighted case, $\cD_{U,\infty}$ appears as the ambiguously signed term in Bochner formula for the (forward) heat equation associated to $\Delta_U$:

\begin{prop}\label{prop:weighted-BE-equiv}
    For any time $t \in I$ and smooth solution to the weighted heat flow $v_t=P_{t,s}^U v_s$ at $t$ we have the Bochner formula
    \begin{align}\label{eq:BochnerWeighted}
        (\partial_t-\Delta_U)(|\nabla v|^2 - 2\Delta_Uv - S_U)=-2|\Hess v+\mathcal{S}|^2  - \cD_{U,\infty}(\nabla v)\,,
    \end{align}
    where
    \begin{align}
        \cD_{U,\infty}(X)&:=2(\Ric_U-\mathcal{S})(X,X)+2\left\langle 2\div_U\mathcal{S} - \nabla S_U,X \right\rangle + \partial_t S_U- \Delta_U S_U-2|\mathcal{S}|^2 \,,\\
        \Ric_U&=\Ric+\Hess U\,,\qquad \div_U T :=\div T - T(\nabla U)\,.
    \end{align}
\end{prop}

\begin{rmk} Let us take a moment to record some basic considerations of the condition $\cD_{U,\infty} \geq 0$:
\begin{itemize}
    \item If the space is static, we recover the notion of nonnegative weighted Ricci curvature, i.e. $\Ric_{U} \geq 0$.
    \item If $d\mu_t \equiv d\mu_0$, then $S_U\equiv0$ and choosing $X=0$, the $\cD_{U,\infty}$ condition implies $\mathcal{S}=0$, hence the space must be static.
    \item If $U_t\equiv U_0$, one can replace $S_U=\Tr\mathcal{S}$ in the $\cD_{U,\infty}$ condition.
    \item If $g_t\equiv g_0$, it becomes
    \begin{align}
        \Ric_U(X,X)-\nabla_X\partial_tU + \frac{1}{2}(\partial_t - \Delta_U)\partial_tU \geq 0\,.
    \end{align}
    In particular, it must have nonnegative weighted Ricci curvature at every time and $\partial_t U$ must be a supersolution of the weighted heat flow. 
    \item Consider a new flow defined by $g_t =: (\Phi_t^{V})^*\tilde{g}_t$, where $\Phi_t^V$ is the time-dependent family of diffeomorphisms generated by $V := -\nabla U_t$. The diffeomorphism invariance of $\Ric$ allows us to rewrite the homogeneity $2$ term of $\cD_{U,\infty}$ as $\Ric_U - \cS = (\Phi_t^V)^*\big(\Ric_{\tilde{g}_t} - \tilde{\cS}\big)$, with the functorial notation $\tilde{\cS} = -\frac{1}{2}\partial_t \tilde{g}_t$. We note however that this absorption trick does not hold for the homogeneity $0$ and $1$ terms, therefore the $\cD_{U,\infty}$ condition for $g$ need not imply the (unweighted) $\cD$ condition for $\tilde{g}$. This is in contrast with the weighted Ricci flow setting--see \cite[Sec.~8]{Lot09}
\end{itemize}
\end{rmk}

\begin{ex}\label{ex:wRF}
    The weighted Ricci flow $(M,g_t,e^{-U_t}dV_{g_t})$ considered by Lott in \cite[Sec.~8]{Lot09}:
    \begin{align}
        \begin{cases}
            \partial_tg_t=-2(\Ric+\Hess U_t) \\
            \partial_t U = \Delta U-|\nabla U|^2
        \end{cases}
    \end{align}
    satisfies $\cD_{U,\infty}\equiv 0$. 
\end{ex}
\begin{ex}\label{ex:gRF}    
    In fact, the previous one is a specific case (setting $H\equiv 0$) of the generalized Ricci flow gauge-changed with a dilation flow which we now recall, see \cite{Streets23, KopferStreetsOTGRF}. Consider ($(M, g_t, H_t, U_t)_{t \in I}$, where $H_t = \sum_{k = 1}^n H_{t,k} \in \Gamma(\bigwedge^*T^*M)$ is a (possibly inhomogeneous) linear combination of $k$-forms, solving
    \begin{align}
        \begin{cases}
            \partial_t g = -2\Ric_U + \frac{1}{2}H^2 \\
            \partial_t H = \Delta_{d} H - d(\iota_{\nabla U} H) \\
            \partial_tU = \Delta U - |\nabla U|^2 + \frac{1}{4}\sum_{k = 1}^n \frac{k - 1}{k}|H_k|^2\,.
        \end{cases}
    \end{align}
    Some computations lead to
    \begin{align}
        \cD_{U,\infty}(X) &= 2(\Ric_U - \cS)(X,X) + 2\La 2\div_U \cS - \nabla S_U,X\Ra + (\partial_t - \Delta_U)S_U - 2|\cS|^2 \\
        &= \frac{1}{2}H^2(X,X) + \La d^*H + \iota_{\nabla U}H, \iota_X H\Ra + \frac{1}{2} |d^*H + \iota_{\nabla U}H|^2 \\
        &= \frac{1}{2}|\iota_{X + \nabla U}H + d^*H|^2\,.
    \end{align}
    Therefore $(M,g_t,e^{-U_t}\,dV_{g_t})_{t\in I}$ satisfies the $\cD_{U,\infty}$ condition, and the properties from Theorem \ref{thm:DUinfty} hold ((2), (4), (5) already appeared in \cite{Streets23, KopferStreetsOTGRF}).
\end{ex}

Also as in the unweighted case, the $\cD_{U,\infty} \geq 0$ condition admits a familiar list of equivalent formulations. 

\begin{thm}\label{thm:DUinfty}
    For $(M,g_t,\mathfrak{m}_{t} = e^{-U_t}\,dV_t)_{t \in I}$ a smooth family of weighted Riemannian manifolds, the following conditions are equivalent:
    \begin{enumerate}
        \item $\cD_{U,\infty}\geq0$;    
        \item Bochner inequality: for every $v: M\to \R{}$ smooth, every $[s,t] \subseteq I$,
        \begin{align}
        (\partial_t-\Delta_{U,t})(|\nabla P_{t,s}^U v|^2_t - 2\Delta_{U,t} P_{t,s}^U v -S_{U,t}) \leq 0\,;
        \end{align}    
        \item Gradient estimate: For every $v: M \to \R{}$ smooth, every $[s, t] \subseteq I$,
        \begin{align}
               |\nabla P_{t,s}^Uv|_t^2 - 2\Delta_{U,t} P_{t,s}^Uv - S_{U,t}\leq P_{t,s}^U( |\nabla v|^2_s - 2\Delta_{U,s} v - S_{U,s})\,;
        \end{align}
        \item Wasserstein contraction: for every probability measure $\mu,\nu \in \cP(M)$, times $[s,t] \subseteq I$, and flow time $0 < h < \sup I - t$,
        \[W_{L_{U,0}^{s + h,t + h}}(\hat{P}_{s+h,s}^U\mu,\hat{P}_{t+h,t}^U\nu) \leq W_{L_{U,0}^{s,t}}(\mu,\nu).\]
        where $L_{U,0}^{a,b}$ is the cost induced by the Lagrangian $L_{U,0}(v,x,t):=(|v|_{g_t}^2+S_{U,t})(x)$ for $(v,x) \in TM$;
         \item Entropy convexity: for every $\mu_s,\mu_t\in P(M)$, $[s,t] \subseteq I$, there is a $W_{L_{U,0}^{s,t}}$-geodesic $(\mu_r)_{r\in[s,t]}$ (equivalently, for all such curves) s.t.
        \begin{align}
            r\mapsto \mathcal{E}(\mu_r\mid\mathfrak{m}_r) - \mathcal{E} (\mu_s\mid \mathfrak{m}_s) - W_{L_{U,0}^{s,r}}(\mu_s,\mu_r) \ \ \text{is convex,}
        \end{align}
        where $\cE(\mu\mid \mathfrak{m})$ is the relative entropy, i.e. $\cE(\mu\mid \mathfrak{m}) := \int_M \ln(d\mu/d\mathfrak{m})\,d\mu$ if $\mu \ll \mathfrak{m}$ and $+\infty$ otherwise;
        \item $\mathrm{EVI}$ characterization: for every $\mu,\nu \in \cP(M)$ and times $[s,t]\subseteq I$, the heat flow $\hat P_{t,s}^U\mu$ satisfies the $\mathrm{EVI}$s:
        \begin{align}
            -\partial_{a^-}^- W_{L_{U,0}^{s,a}}(\mu,\hat{P}_{t,a}^U\nu) & \leq \frac{1}{a-s}\left(\cE(\mu\mid\mathfrak m_s) - \cE(\hat{P}_{t,a}^U\nu\mid \mathfrak m_a) + W_{L_{U,0}^{s,a}}(\mu, \hat{P}_{t,a}^U\nu)\right) \,, \qquad \forall a \in (s, t]\,.\\
            -\partial_{b^-}^-W_{L_{U,0}^{b,t}}( \hat{P}_{s,b}^U\mu, \nu) & \leq \frac{1}{t - b}\left(\cE(\nu\mid\mathfrak m_t) - \cE(\hat{P}_{s,b}^U\mu \mid \mathfrak m_b) - W_{L_{U,0}^{b, t}}( \hat{P}_{s,b}^U\mu, \nu)\right) \,, \qquad\forall b \in (\inf I,s]\,.
        \end{align}
    \end{enumerate}
\end{thm}

With Proposition \ref{prop:weighted-BE-equiv} in hand, the proof of the above is, up to $U$ sub/superscripts, identical to that of the unweighted $L_0$ version Theorem \ref{thm:L0charact}. We therefore omit it.

\subsection{The dimensional case \texorpdfstring{$\cD_{U,N} \geq 0$}{DUN0}.}\label{subsec:N-weighted}

Here consider a version of Subsection \ref{subsec:infinity-weighted} that incorporates a dimensional parameter $N \in [m, \infty]$ (recall $m:= \dim M$). We first show the equivalence of a dimensional variant of the $\cD_{U,\infty}$ condition with a Bochner inequality for the (forward) $\Delta_U$ heat equation with dimensional improvement, which will be the analogue of \eqref{eq:BochnerdimL0-equiv-of-Bochners}, but now for weighted spaces. 

\begin{prop}\label{prop:weighted-n-BE-equiv}
The following conditions are equivalent:
\begin{enumerate}
\item For any $X \in T_pM$, it holds
\begin{equation}
\cD_{U,N}(X) := \cD_{U,\infty}(X) - \frac{2}{N - m}(\La\nabla U ,X\Ra - \partial_t U)^2 \geq 0\,.
\end{equation}
\item For any times $[s,t] \subseteq I$, one has the Bochner inequality
\begin{equation}
(\partial_t - \Delta_U)(|\nabla P_{t,s}^U v|^2 - 2\Delta_U P_{t,s}^U v - S_U) \leq - \frac{2}{N} (S_U + \Delta_U P_{t,s}^U v)^2\,.
\end{equation}
\end{enumerate}
\end{prop}

\begin{proof}
$(\underline{\implies})$: Letting $v_t=P_{t,s}^U v$, by Proposition \ref{prop:weighted-BE-equiv} and the assumption, we have the Bochner formula, then inequality
\begin{align}
(\partial_t - \Delta_U)(|\nabla v_t|^2 - 2\Delta_U  v_t - S_U) &= -2|\Hess v_t + \cS|^2 - \cD_{U,\infty}(\nabla v_t)\\
&\leq -2|\Hess v_t + \cS|^2  - \frac{2}{N - m}(\La\nabla U , \nabla v_t\Ra - \partial_t U)^2\,.
\end{align}
We estimate the desired RHS applying the Peter-Paul inequality
\begin{align}
\frac{2}{N}(S_U + \Delta_U v_t)^2 &= \frac{2}{N}(\Tr \cS + \partial_t U + \Delta v_t - \La \nabla U,\nabla v_t\Ra)^2\\
&= \frac{2}{N}\left[(\Tr \cS + \Delta v_t)^2 - 2(\Tr \cS + \Delta v_t)(\La \nabla U,\nabla v_t\Ra - \partial_t U) + (\La \nabla U,\nabla v_t\Ra - \partial_t U)^2\right] \\
&\leq \frac{2}{N}\left[(\Tr \cS + \Delta v_t)^2 + \frac{N - m}{m}(\Tr \cS + \Delta v_t)^2 + \frac{m}{N - m}(\La \nabla U,\nabla v_t\Ra - \partial_t U)^2 + (\La \nabla U,\nabla v_t\Ra - \partial_t U)^2\right] \\
&= \frac{2}{N}\left[\frac{N}{m}(\Tr \cS + \Delta v_t)^2 + \frac{N}{N - m}(\La \nabla U,\nabla v_t\Ra - \partial_t U)^2\right] \\
&= \frac{2}{m}(\Tr \cS + \Delta v_t)^2 + \frac{2}{N - m}(\La \nabla U,\nabla v_t\Ra - \partial_t U)^2 \\
&\leq 2|\cS + \Hess v_t|^2 + \frac{2}{N - m}(\La \nabla U,\nabla v_t\Ra - \partial_t U)^2
\end{align}
This, appended to the previous intermediate Bochner inequality, gives the desired result. \vspace{0.5cm}

$(\underline{\impliedby})$: For the point $(p,t) \in M\times I$ and a tangent vector $X \in T_pM$ of interest, select a function $v: M \to \R{}$ such that 
\begin{align}
(\Hess v)_p +\cS_p = \frac{ \partial_t U(p)-\La (\nabla U)_p, X\Ra}{N - m}g\,, \qquad(\nabla u)_p = X \,.
\end{align}
The assumption then implies

\begin{align}
-(\cD_{U,\infty})_p(\nabla v) - 2|\Hess v + \cS|^2(p) &=  \left.(\partial_t - \Delta_U)\right|_{t = s}(|\nabla P_{t,s}^Uv|^2 - 2\Delta_U P_{t,s}v - S_U)(p) \\
&\leq - \frac{2}{N} (S_U + \Delta_U v)^2(p) \\
&= -\frac{2}{N}(\Tr(\cS +  \Hess v) + \partial_t U - \La\nabla U,\nabla v\Ra )^2(p)\,.
\end{align}
Now rearranging and then remembering our choice of $v$, we obtain
\begin{align}
-(\cD_{U,\infty})_p(X) &\leq  2|\Hess v + \cS|^2(p) -\frac{2}{N}(\Tr(\cS + \Hess v) + \partial_t U - \La\nabla U,\nabla v\Ra )^2(p) \\
&= \left(2\frac{m}{(N - m)^2} - \frac{2}{N}\left(\frac{m}{N -m} + 1\right)^2\right)\left(\La (\nabla U)_p, X\Ra - \partial_t U(p)\right)^2 \\
&= -\frac{2}{N - m}(\La (\nabla U)_p, X\Ra - \partial_t U(p))^2\,.
\end{align}
\end{proof}

\begin{rmk}\label{rmk:weighted-warped-prod-equiv}
In the static setting the $(n+k)$ weighted Ricci tensor can be interpreted in terms of a warped product with a given $k$ manifold \cite{LottBakryEmery}, and a similar approach was used to define the $N$-weighted Ricci flow in \cite{Lot09}. Here we study $\cD_{U,N}$ from this point of view. Fix a compact time-independent $k$-dimensional ``fiber'' manifold $(F^{k}, h)$, and consider the time-dependent warped product
\begin{equation}
(P, \tilde{g}_t) := (M,g_t) \times_{f_t} (F,h) = (M\times F, g_t + f_t^2\cdot h)\,,\qquad f_t := e^{-\frac{U_t}{k}}\,.  
\end{equation}
One can see that there is a bijection between $F$-invariant solutions to the heat equation on $(P,\tilde{g}_t)_{t \in I}$ and solutions to the weighted heat flow $(\partial_t - \Delta_U) u = 0$ on $(M, g_t, d\mu_t)_{t \in I}$, and in particular
    \begin{equation}
        \cD_{\tilde{g}}(X) =\cD_{g,U,m + k}(X) \qquad \text{for $X\in T_pM$}
    \end{equation} 
i.e. for horizontal tangent vectors $X \in T_pM \leq T_p M + T_q F \cong T_{(p,q)}P$. Indeed the $N$-dimensional $L_0$ Bochner inequality of Proposition \ref{prop:updated-dim-Bochner-equiv} for $F$-invariant functions on $(P,\tilde{g}_t)$ corresponds to the $(U,N)$-Bochner inequality of Proposition \ref{prop:weighted-n-BE-equiv} for $(M, g_t, d\mu_t)$.

Moreover, we have 
    \begin{equation}
        \cD_{\tilde{g}}(Z) \geq 0 \ \ \ \text{for all $Z \in T_{(p,q)}P$} \ \ \ \iff \ \ \ \begin{cases}
            \cD_{\tilde{g}}(X)=\cD_{g,U,m + k}(X)\geq0 \ \ \ &\text{for all $X\in T_pM$} \\
            \cD_{\tilde{g}}(Y) \geq 0 &\text{for all $Y\in T_qF$} \,.
        \end{cases}
    \end{equation} 
This follows from the fact that $\Ric_{\tilde{g}} - \cS_{\tilde{g}}$ splits orthogonally in the fiber and base directions and that $2\div_{\tilde g}(\cS_{\tilde g}) - \nabla^{\tilde g} S_{\tilde g}$ is everywhere tangent to $M$ (in fact this is true for both terms individually). Finally we compute

\begin{align}
\cD_{\tilde{g}}(Y) &= 2(\Ric_{\tilde{g}} - \cS_{\tilde{g}})(Y, Y) + (\partial_t S_{\tilde g} - \Delta_{\tilde g} S_{\tilde g} - 2|\cS_{\tilde g}|^2)\\
&= 2\left(\frac{\Ric_h(Y,Y)}{e^{-2U/k}} - \frac{1}{k}\left((\partial_t - \Delta_g) U + |\nabla^g U|^2\right)|Y|^2_{\tilde{g}}\right) \\
&\qquad + \left((\partial_t - \Delta_U)(\Tr_g\cS_g + \partial_t U) - 2|\cS_g|_g^2 - \frac{2}{k}|\partial_t U|^2\right)\,, \qquad Y\in T_qF\,.
\end{align}
Since there is no homogeneity $1$ term, $\cD_{\tilde{g}}(Y) \geq 0$ is seen to be equivalent to the non-negativity of the final two lines above. The non-negativity of the second term is already required for the $\cD_{\tilde{g}}$ condition in the horizontal direction, while the non-negativity of the first term requires $U$ to be a subsolution to a semilinear parabolic PDE on $(M,g_t)_{t \in I}$. The latter is guaranteed for example if $\Ric_h \geq \lambda > 0$, by replacing $U$ with $U + C$ for some constant $C = C(\lambda,k, \|U\|_{C^{2,1}}) \gg 1$.
\end{rmk}

\begin{ex}\label{ex:NwRF}
    The $N$-weighted Ricci flow $(M,g_t,e^{-U_t}dV_{g_t})$ considered by Lott in \cite[Sec.~8]{Lot09}:
    \begin{align}
        \begin{cases}
            \partial_tg_t=-2(\Ric+\Hess U_t - \frac{1}{N-n}dU\otimes dU) \\
            \partial_t U = \Delta U-|\nabla U|^2 
        \end{cases}
    \end{align}
    satisfies $\cD_{U,N}\equiv 0$.
\end{ex}

The Bochner inequality of Proposition \ref{prop:weighted-n-BE-equiv} is the same (up to $U$ super/subscript notation) as the dimensional $L_0$ Bochner inequality \eqref{eq:BochnerdimL0}. Correspondingly, all statements equivalent to \eqref{eq:BochnerdimL0} by Theorem \ref{thm:dimL0charact} have weighted analogs that are equivalent to the conditions in Proposition \ref{prop:weighted-n-BE-equiv}. We write down this list of equivalent statements but omit the proof since it may be obtained directly from the proof of Theorem \ref{thm:dimL0charact}, which essentially spans multiple sections of this paper, by simply adding $U$ super/subscripts where appropriate. We also simply reference the relevant unweighted conditions and indicate the required modifications instead of rewriting all the conditions with the weighted quantities substituted, since the resulting statement would span multiple pages. 

\begin{thm}
Let $(M,g_t,\mathfrak m_t = e^{-U_t}\,dV_t)_{t\in I}$ be a smooth time-dependent weighted Riemannian manifold, and $N \in [\dim M, \infty]$ a real number. Fix a backward time variable $\tau := T - t \in T - I =: \tilde I$. Then the following conditions are equivalent.

\begin{enumerate}
\item $\cD_{U,N} \geq 0$.
\item Bochner inequality: The Bochner inequality $\eqref{eq:BochnerdimL0}$ holds, or equivalently the Bochner inequalities $\eqref{eq:BochnerL-}$ and $\eqref{eq:BochnerL+}$ hold on $(M,g_{\tau - T_0})_{\tau \in (\tilde I + T_0) \cap (0,\infty)}$ resp. $(M,g_{t - T_0})_{t \in (I + T_0) \cap (0,\infty)}$ for all shifts $T_0 \in \R{}$. 
\item Gradient estimate: The gradient estimate \eqref{eq:GEdimL0} holds, or equivalently the gradient estimates \eqref{eq:GEL-} and \eqref{eq:GEL+} hold on $(M,g_{\tau - T_0})_{\tau \in (\tilde I + T_0) \cap (0,\infty)}$ resp. $(M,g_{t - T_0})_{t \in (I + T_0) \cap (0,\infty)}$ for all shifts $T_0 \in \R{}$. 
\item Wasserstein contraction:  The Wasserstein contraction \eqref{eq:WCL0} holds, or equivalently the Wasserstein contractions \eqref{eq:WCL-} and \eqref{eq:WCL+} hold on $(M,g_{\tau - T_0})_{\tau \in (\tilde I + T_0) \cap (0,\infty)}$ resp. $(M,g_{t - T_0})_{t \in (I + T_0) \cap (0,\infty)}$ for all shifts $T_0 \in \R{}$. 
\item Entropy Convexity: The convexity of entropies \eqref{eq:ECL-} and \eqref{eq:ECL+} hold on $(M,g_{\tau - T_0})_{\tau \in (\tilde I + T_0) \cap (0,\infty)}$ resp. $(M,g_{t - T_0})_{t \in (I + T_0) \cap (0,\infty)}$ for all shifts $T_0 \in \R{}$. 
\item $\mathrm{EVI}$ characterization: The (pairs of) $\mathrm{EVI}$s \eqref{eq:EVIL-} and \eqref{eq:EVIL+} hold on $(M,g_{\tau - T_0})_{\tau \in (\tilde I + T_0) \cap (0,\infty)}$ resp. $(M,g_{t - T_0})_{t \in (I + T_0) \cap (0,\infty)}$ for all shifts $T_0 \in \R{}$.
\end{enumerate}

In all conditions referenced above, one must interpret $n := N$, $S := S_U$, $\Delta := \Delta_U$, $P := P^U$, $\hat P := \hat P^U$, $W_{L_0} := W_{L_{U,0}}$, $\cE_\tau(\mu) := \cE(\mu\mid \mathfrak{m}_\tau)$, $W_{L_-} := W_{L_{U,-}}$ with $L_{U,-}$ the cost induced by the Lagrangian $L_{U,-}(v,x,\tau) := \sqrt{\tau}(|v|^2_{g_\tau} + S_{U,\tau})(x)$, and $W_{L_+} := W_{L_{U,+}}$ with $L_{U,+}$ the cost induced by the Lagrangian $L_{U,+}(v,x,t) := \sqrt{t}(|v|^2_{g_t} + S_{U,t})(x)$.
\end{thm}

\begin{appendices}
\section{Properties of \texorpdfstring{$L^{s,t}$}{Lst}}\label{Appe:Lagrangian}

In this appendix, we always assume that we are working on a smooth, closed manifold $M$ equipped with a time dependent Lagrangian $L: TM \times [s,t]\to \R{}$ satisfying the assumptions of Section \ref{subsec:LagrangianProperties}.

\begin{lem}\label{lem:Eofgeod}
    Given $(x,s)$, $(y,t)$ there always exists (at least) a $C^1$ length minimizing curve $\gamma$ with $\cL(\gamma)=L^{s,t}(x,y)$.
    \end{lem}
\begin{proof}
    The existence of a minimizer can be proved by the direct method, see \cite[Thm.~6.1.2]{CanSin04}. A refined argument implies that $\dot\gamma_t$ is bounded and that $\gamma_t\in C^1$, \cite[Thm.~6.2.5]{CanSin04}.
\end{proof}

\begin{rmk}
    Note that in the case $L=L_m$, $m=0,-,+$, $L_m\in C^\infty(TM\times I)$, $\nabla_{v,v}L_m>0$, hence by \cite[Thm.~6.2.8]{CanSin04}, any minimizer is in fact smooth.
\end{rmk}

\begin{lem}\label{lem:costfromlagrangian}
    As $t\to s$, $L^{s,t}(x,\cdot)$ converges to $+\infty$ uniformly on any $B_r(x)^c$, $r>0$, and $L^{s,t}(x,x)$ converges to $0$.
\end{lem}

\begin{proof}
    Fix $r>0$, $d(x,y)>r$, let $\gamma$ be the minimizer for $L^{s,t}(x,y)$. By a change of variables $r(u)=u(t-s)+s$ and for $\tilde\gamma_u=\gamma_{r(u)}$, we have
    \begin{align}
        L^{s,t}(x,y) &= \int_s^tL(\dot\gamma_r,\gamma_r,r)dr \\
        &=\int_0^1 L(\dot\gamma_{r(u)},\gamma_{r(u)},r(u))(t-s)du \\
        &= \int_0^1 L(\dot{\tilde \gamma}_u/(t-s),\tilde\gamma_u,r(u))(t-s)du \\
        &\geq \int_0^1 \tilde\theta(|\dot{\tilde \gamma}_u|/(t-s))(t-s)du \\
        &\geq \int_0^1 \tilde\theta(|\dot{\eta}_u|/(t-s))(t-s)du \,.
    \end{align}
    In the first inequality we are used that there exists $\tilde\theta$ convex and superlinear, with $L(v,x,t)\geq \theta(|v|)-c_0\geq \tilde\theta(|v|)$. Convexity is used in the line below to apply Jensen's inequality, and $\eta$ is the constant speed reparametrization of $\tilde\gamma$. By superlinearity of $\tilde\theta$, the integral diverges as $(t-s)\to0$.
    
    For the case of $x=y$, we can choose the constant curve $\gamma\equiv x$ as competitor, and using continuity of $L$ in the second two variables
    \begin{align}\label{eq:Lxxupperbound}
        L^{s,t}(x,x)\leq \int_s^t L(0,x,r)dr \leq \int_s^t C \to 0.
    \end{align}
    On the other hand, superlinearity also implies a lower bound on $L$
    \begin{align}
        L^{s,t}(x,x)=\inf_\gamma \int_s^t L(\dot\gamma_r,\gamma_r,r)dr \geq \inf_\gamma \int_s^t -C \to 0.
    \end{align}
\end{proof}

\begin{lem}\label{lem:cost-is-Lip}
    The induced cost $(t,x,y)\mapsto L^{s,t}(x,y)$ is locally Lipschitz in the space-time $(s,b]\times M\times M$ with respect to the background distance $d=d_g$ and the Euclidean distance in time. More precisely, for any compact interval $[s+\delta,b]$, there is $C>0$ such that for any $s+\delta \leq t_1\leq t_2\leq b$, $x,y_1,y_2\in M$, 
\begin{align}
    |L^{s,t_2}(x,y_2)-L^{s,t_1}(x,y_1)| \leq C(|t_2-t_1| + d(y_1,y_2)) \end{align}
\end{lem}
\begin{proof}
    Note that $|L^{s,t_2}(x,y_2)-L^{s,t_1}(x,y_1)| \leq |L^{s,t_2}(x,y_2)-L^{s,t_1}(x,y_2)| + |L^{s,t_1}(x,y_2)-L^{s,t_1}(x,y_1)|$, hence it is enough to prove the local Lipschitzianity for space and time separately, uniformly.

1) We show $L^{s,t_2}(x,y) \leq L^{s,t_1}(x,y) + C|t_s-t_1|$. Let $\gamma$ be a minimizer for $L^{s,t_1}(x,y)$, consider 
\begin{align}
    \eta(r)=\begin{cases}
    \gamma_r \ \ &s\leq r \leq t_1 \\
    y \ \ & t_1\leq r\leq t_2 \,.
\end{cases}
\end{align}
Then
\begin{align}
    L^{s,t_2}(x,y) &\leq \int_s^{t_2} L(\dot\eta,\eta,r)dr \\
    &= \int_s^{t_1} L(\dot\gamma,\gamma,r)dr + \int_{t_1}^{t_2}L(0,y,r) dr \\
    &\leq L^{s,t_1(x,y)} + C|t_2-t_1|.
\end{align}

2) We show $L^{s,t_1}(x,y) \leq L^{s,t_2}(x,y) + C|t_2-t_1|$. Let $\gamma$ minimize $L^{s,t_2}(x,y)$, and consider $\eta(r)=\gamma_{\theta(r)}=\gamma(s+(r-s)\frac{(t_2-s)}{(t_1-s)})$. Then, setting $\frac{t_2-s}{t_1-s}=1+\epsilon$
\begin{align}
    L^{s,t_1}(x,y) &\leq \int_s^{t_1}L(\dot\eta,\eta,r)dr \\
    &= \int_s^{t_2} L\left((1+\epsilon)\dot\gamma _{\theta} ,\gamma_\theta, s+(\theta-s)\frac{1}{1+\epsilon}\right)\frac{1}{1+\epsilon}d\theta \\
    &\leq \int_s^{t_2} L\left(\dot\gamma _{\theta},\gamma_\theta, \theta \right)+ |d_v L(\dot\gamma _{\theta},\gamma_\theta, \theta)| |\dot\gamma_\theta|\epsilon + |\partial_r L(\dot\gamma _{\theta},\gamma_\theta, \theta)| \epsilon(\theta-s) \\
    &\qquad +|L(\dot\gamma _{\theta},\gamma_\theta, \theta)|\epsilon + o(\epsilon) \\
    &\leq L^{s,t_2}(x,y) + C\epsilon \\
    &\leq L^{s,t_2}(x,y) + \frac{C}{t_1-s}(t_2-t_1).
\end{align}
Here we used that $|\dot\gamma_\theta|$ is bounded, see \ref{lem:Eofgeod}.

3) We show $L^{s,t}(x,y_1) \leq L^{s,t}(x,y_2) + Cd(y_1,y_2)$. Let $\gamma$ be a minimizer for $L^{s,t}(x,y_2)$, set $d=d(y_1,y_2)$ and assume $d \leq (b - s)/4$. First consider the case $t \leq (s + b)/2$. Then define $\eta:[s,t + d]\to M$ by 
\begin{align}
    \eta(r)= \begin{cases}
        \gamma_r \ \ \ &s\leq r \leq t \\
        \gamma^d_{r - t} & t\leq r\leq t + d
    \end{cases}
\end{align}
where $\gamma^d:[0, d]\to M$ is a geodesic for the background distance $d$ from $y_2$ to $y_1$ of constant unit speed. Then
\begin{align}
    L^{s,t+d}(x,y_1)\leq L^{s,t}(x,y_2)+\int_t^{t+ d}L(\dot\gamma,\gamma,r)dr \leq L^{s,t}(x,y_2) + C d.
\end{align}
Moreover, by step 2) $L^{s,t}(x,y_1) \leq L^{s,t+ d}(x,y_1) + C d\leq L^{s,t}(x,y_2) + 2Cd$ and we conclude. In the case $t \geq (s + b)/2$, passing from $y_1$ to $y_2$ should happen before time $t$; let $\tilde\gamma$ be instead a minimizer for $L^{s,t -d}(x,y_2)$ (note that $t - d \geq s + (b - s)/4$ is uniformly far from $s$) and define $\tilde\eta: [s,t] \to M$ by

\begin{align}
    \tilde\eta(r) = 
        \begin{cases}
            \tilde\gamma_r \ \ \ &s \leq r \leq t - d\,, \\
            \gamma^d_{r + d - t} &t - d\leq r \leq t\,,
        \end{cases}
\end{align}

with $\gamma^d$ as before. Then, similarly to the previous case, we combine the estimate $L^{s,t - d}(x,y_2) \leq L^{s,t}(x,y_2) + Cd$ from time-Lipschitzianity with the estimate $L^{s,t}(x,y_1) \leq L^{s,t - d}(x,y_2) + Cd$ from testing the definition of $L$ with $\tilde\eta$, to obtain $L^{s,t}(x,y_1) \leq L^{s,t}(x,y_2) + 2Cd$. If $d \geq (b - s)/4$, then a Lipschitz estimate amounts to a uniform bound on $|L^{s,t}(x,z)|$, which follows from compactness and the assumptions on $L$. 

\end{proof}

\begin{rmk}
    Analogous regularity holds when varying the initial space-time point $(x,s)$, but it would not hold for the $L_\pm$ case at $s=0$ (see also Remark \ref{rmk:LmaregoodLagr}). Nevertheless, we need only regularity in $(t,y)$ in the study of the (forward) Hopf--Lax semigroup.
\end{rmk}

\section{Regularity of the Hamilton-Jacobi equation}\label{Appe:HJ}
We first prove a statement that holds for general Lagrangians satisfying the assumptions of Section \ref{subsec:LagrangianProperties}., adapting the approach of Evans \cite[Ch.~3,~Thm.~6]{EvansPDE}.
\begin{prop}\label{prop:HLsemigroup}
    For any initial $\phi\in \mathrm{Lip}(M)$, the function $Q^{s,t}\phi(x):[s,b]\times M\to \R{}$ is Lipschitz in space-time. Moreover, $Q^{s,t}\phi$ solves the Hamilton-Jacobi equation at all differentiability points of $(t,x)\mapsto Q^{s,t}\phi(x)$
    \begin{align}
        \begin{cases}
            &\partial_tQ^{s,t}\phi(x) + H(d Q^{s,t}\phi(x),x,t) = 0 \,,\\
            &Q^{s,t}\phi \to \phi\qquad  \text{uniformly},
        \end{cases}
    \end{align}
    where $H:T^*M\times [a,b]\to\mathbb{R}$ is the Hamiltonian. 
    
    Moreover, for all points $(t,x)$ at which $Q^{s,t}\phi(x)$ is differentiable in $x \in M$, it is automatically also differentiable from the left in $t \in (s,b]$ and has Dini derivatives satisfying the Hamilton-Jacobi inequalities 

    \begin{equation}
        \partial_{t^\pm}^+Q^{s,t}\phi(x) \leq -H(dQ^{s,t}\phi(x),x,t) \leq \partial_{t^-}^-Q^{s,t}\phi(x)\,.
    \end{equation}
\end{prop}

\begin{proof}
    Consider $t_1,t_2> s$, $x,y\in M$, and let $z\in M$ be such that $Q^{s,t_1}\phi(x)=L^{s,t_1}(z,x)+\phi(z)$ (it exists by compactness of $M$ and continuity of $L^{s,t_1}(\cdot,x)$). Then
    \begin{align}
        Q^{s,t_2}\phi(y) - Q^{s,t_1}\phi(x) \leq L^{s,t_2}(z,y)-L^{s,t_1}(z,x) \leq C(|t_2-t_1|+d(x,y)), 
    \end{align}
    where the last inequality follows from Lipschitzianity of $L$ in space-time (by Lemma \ref{lem:cost-is-Lip}). 

    This shows Lipschitzianity for $t_1,t_2>s$ (which in fact holds for any continuous $\phi$ given), we now consider the case $t_1=s$ (and rename $t_2=t$), with the convention $Q^{s,s}\phi=\phi$. By
    \begin{align}
        |Q^{s,t}\phi(x)-\phi(y)| \leq |Q^{s,t}\phi(x)-\phi(x)|+|\phi(x)-\phi(y)|
    \end{align}
    and using $\phi\in\Lip$ (say $\Lip\phi \leq C$), we only need to bound the first term on the LHS. On the one hand
    \begin{align}
        Q^{s,t}\phi(x)-\phi(x) \leq L^{s,t}(x,x)\leq C(t-s)
    \end{align}
    using Equation \eqref{eq:Lxxupperbound}. On the other hand, let $z$ be optimal in the definition of $Q^{s,t}\phi(x)$, let $\gamma$ be a geodesic (wrt the fixed metric $g$) from $z$ to $x$ parametrized on $[s,t]$, then
    \begin{align}
        Q^{s,t}\phi(x)-\phi(x) &= \phi(z)+L^{s,t}(z,x) -\phi(x) \\
        &\geq -C d(z,x) + L^{s,t}(z,x) \\
        &\geq \left[ \int_s^t -C|\dot\gamma|_g + L(\dot\gamma,\gamma,r) dr \right] \\
        &= -\left[ \int_s^t g( \dot\gamma, C\frac{\dot \gamma}{|\dot\gamma|_g}) - L(\dot\gamma,\gamma,r) dr \right] \\
        &\geq - \int_s^t H(C\frac{\dot \gamma}{|\dot\gamma|_g}, \gamma, r) dr \\
        &\geq -C(t-s)
    \end{align}
    since $H$ is continuous and hence bounded on $SM\times [s,t]$. This in particular implies $Q^{s,t}\phi\to \phi$ as $t\to s$.

    We now show that if $x$ is a differentiability point of $x\mapsto u_t(x) := Q^{s,t}\phi(x)$, then Hamilton-Jacobi inequalities for the Dini derivatives are satisfied. The Hamilton-Jacobi {\it equality} (for the two-sided or left-sided derivative, depending on whether $Q^{s,t}\phi(x)$ is differentiable at $(t,x)$ or not) will then hold by combining the upper bound on $\partial_{t^\pm}^+u_t(x)$ and lower bound on $\partial_{t^-}^-u_t(x)$.

    $(\leq)$ For any $v\in T_xM$, consider for some $h>0$ a smooth curve $\gamma:[t - h,t]\to M$ with $\gamma_t = x$, $\dot\gamma_t=v$. Since $Q^{s,t}\phi=Q^{t - h,t}u_{t - h}$,
    \begin{align}
       u_t(x) - u_{t-h}(\gamma_{t-h}) &= \inf_y \left\{ L^{t - h,t}(y,x) + u_{t - h}(y) \right\} - u_{t - h}(\gamma_{t - h}) \\
       &\leq L^{t - h,t}(\gamma_{t-h},x).
    \end{align}
    Dividing by $h$ and as $h\to 0$, we get
    \begin{align}
        \partial_{t^-}^+ u_t(x) + du_t(x) (v)  \leq \lim_{h\downarrow 0}\frac{1}{h}L^{t - h,t}(\gamma_{t-h}, x) \leq \lim_{h\downarrow 0}\frac{1}{h}\mathcal L(\gamma) = L(v,x,t).
    \end{align}
    Hence
    \begin{align}
        \partial_{t^-}^+ u_t(x) + \sup_v \left\{ du_t(x) (v) - L(v,x,t) \right\} \leq 0
    \end{align}
    and the second addendum is exactly $L^*(du_t(x),x,t)=H(du_t(x),x,t)$. The inequality $\partial^+_{t^+}u_t \leq -H(du_t(x),x,t)$ follows similarly, but instead considering $\tilde\gamma:[t,t + h]\to M$ with $\tilde\gamma_t = x$, $\dot{\tilde\gamma}_t = v$.

    $(\geq)$ Let $z\in M$ s.t. $u_t(x)= L^{s,t}(z,x)+u_s(z)$. Let $\gamma:[s,t]\to M$ be the minimizing geodesic from $z$ to $x$. Then
    \begin{align}
        u_t(x)-u_{t-h}(\gamma_{t-h}) \geq L^{s,t}(z,x) - L^{s,t-h}(z,\gamma_{t-h}).
    \end{align}
    Divide by $h$ and send $h\downarrow 0$ to obtain
    \begin{align}
        \partial_{t^-}^- u_t(x) + du_t(x) (\dot\gamma_t) \geq L\left(\dot\gamma_t,x,t\right).
    \end{align}
    Hence
    \begin{align}
        0 &\leq \partial_{t^-}^- u_t(x) + \left( du_t(x) (\dot\gamma_t) - L\left(\dot\gamma_t,x,t\right) \right) \\
        &\leq \partial_{t^-}^- u_t(x) + \sup_v\left( du_t(x) (v) - L(v,x,t) \right) \\
        &= \partial_{t^-}^- u_t(x) +  H(du_t(x),x,t).
    \end{align}
\end{proof}

\begin{rmk}
    Following the computations of \cite[Sec.~7.4]{rf1}, one can see that the cost induced by $L$ itself satisfies the HJ equality. For the $L_-$ case, this equation can be found in \cite[pg.~15]{perelman2002entropy}.
\end{rmk}

In the case of centered, quadratic dependence of the Lagrangian in the velocity, one can obtain without too much fuss an a.e. equation for {\it every} time-slice. The following holds in much larger generality and is probably well-known to viscosity solution practitioners, but it was difficult to find a precise reference so we state and prove what we need. It will (only) be used for the implications of type $(\mathrm{EC}) \implies (\mathrm{WC})$.

\begin{prop}\label{prop:fiber-quadratic-HJ}
Given the hypotheses of Proposition \ref{prop:HLsemigroup}, assume further that $ L(\cdot,p,t) = A_{p,t}(\cdot, \cdot) + C_{p,t}$ is $C^2$ in all variables with $A_{p,t} \in T_p^*M\cdot T_p^*M$ and $C_{p,t}\in \R{}$ for each $(t,p) \in [s,b]\times M$. Suppose also that $\phi$ is semiconcave or semiconvex with quadratic modulus (see \cite[Def.~10.10]{Vil09}). Then for all points $(t,x) \in [s,b]\times M$ at which $Q^{s,t}\phi(x)$ is differentiable in $x \in M$, it is automatically also differentiable in $t$ and solves the Hamilton-Jacobi equation. In particular, the following holds (with derivatives taken from one side at the endpoints):

\begin{equation}
\partial_tQ^{s,t}\phi(x) + H(dQ^{s,t}(x),x,t) = 0\,,\qquad \forall t \in [s,b]\,,\text{ a.e. }x\in M\,.
\end{equation}
\end{prop}

\begin{proof}
In view of Proposition \ref{prop:HLsemigroup}, it remains to provide the lower bound $\partial_{t^+}^- Q^{s,t}(x) \geq -H(dQ^{s,t}\phi(x), x,t)$ for each $t \in [s,b)$, a.e. $x \in M$. The potential difficulty is in controlling the minimizers in the definition of the Hopf-Lax semigroup for $Q^{t, t+h}(Q^{s,t}\phi)(y_h)$ with $y_h \to x$ as $h \downarrow 0$, but the additional assumptions will provide local uniqueness of the minimizer assuming that $y_h$ approaches $x$ along a geodesic and $Q^{s,t}\phi(x)$ is regular enough at $x$.

Writing $u_t:= Q^{s,t}\phi: M \to \R{}$, by \cite[Thm.~6.4.3]{CanSin04} we have that $u_t$ is semiconcave (with quadratic modulus) for $t \in (s,b]$. Using this along with the assumed semiconcavity or semiconvexity of $u_s = \phi$, we have that for any given $t\in [s,b]$, $u_t$ is differentiable and admits a Hessian in the sense of Alexandrov at almost every $x\in M$ (see \cite[Thm.~14.1]{Vil09}) and we fix such a point. In particular, we have a symmetric linear operator $S: T_xM\to T_xM$ with 

\begin{equation}
u_t(\exp^{g_t}_x(v)) = u_t(x) + \La v,\nabla^{g_t}u_t(x)\Ra_{g_t} + \frac{\La S(v), v\Ra_{g_t}}{2} + o(|v|^2_{g_t}) \qquad \text{as }v\to 0\,.
\end{equation}

This would hold true for any fixed smooth metric $g_t$ near $x$, but the most accurate choice in this setting is $(g_t)_z := 2A_{z,t}$, with $A$ as in the hypothesis of the Proposition. On the other hand, we have 

\begin{align}
h\cdot L^{t,t + h}(y,z) &= \inf_{\stackrel{\gamma: [t,t+h]\to M}{\gamma_t = y,\,\gamma_{t + h} = z}} h\int_t^{t+h}L(\dot\gamma_s, \gamma_s, s)\,ds \\
&= \inf_{\gamma} h^2\int_0^1 L(\dot\gamma_{t + hs},\gamma_{t + hs}, t+ hs)\,ds \\
&= \inf_{\stackrel{\tilde{\gamma}:[0,1]\to M}{ \tilde\gamma_0 =y,\,\tilde\gamma_1 = z}}  \int_0^1 h^2L(\dot{\tilde{\gamma}}(s)/h, \tilde{\gamma}(s),t + hs)\,ds = \inf_{\tilde \gamma}  \int_0^1 L^h(\dot{\tilde{\gamma}}(s), \tilde{\gamma}(s),s)\,ds\,,
\end{align}

where we have written $L^h(v,p,s) := h^2L(v/h, p, t + hs)$. Note that by the additional assumption on the form of $L$, $L^h$ $C^2$-converges on $TM \times [0,1]$ to $L^0(v,p,s) := (1/2)\cdot |v|^2_{g_t(p)}$, so we also have $C^2$ convergence of the cost $h\cdot L^{t,t + h}(\cdot,z) \to (1/2)\cdot d^2_{g_t}(\cdot,z)$ on $B^{g_t}_{\mathrm{injrad}_{g_t}/2}(z)$. This last statement follows in the case under consideration because for $C^2$-converging Lagrangians $L(x,v,t)$ that additionally have $C^1$ convergence of $(\nabla_{v,v}^2L)$ and uniform convexity $(\nabla_{v,v}^2L) \geq \lambda > 0$, the geodesic equation involves only $C^1$-converging quantities and the Jacobi field equation involves only uniformly bounded quantities (in any fixed local coordinate system). This allows us to transfer the convexity of $d^2_{g_t}(\cdot, z)$ near $z$ to the $O(h^{-1})$ convexity of $L^{t,t + h}(\cdot, z)$,
\begin{align}
L^{t,t + h}(\exp^{g_t}_y(v),z) &- L^{t,t + h}(y,z)  - \La v,(\nabla^{g_t}L^{t,t +h}(\cdot, z))_y \Ra_{g_t} \\
&= \frac{\big((\nabla^2 d^2_{g_t}(\cdot, z))_y/2 + \Psi(h)\big)(v,v) + o(|v|_{g_t}^2)}{2h}\geq \frac{|v|_{g_t}^2}{4h}\,,
\end{align}
for any $z,y \in M$, $v \in T_yM$, whenever $h$, $\,d_{g_t}(y,z)$, $|v|_{g_t}$ are small enough. We now let $\gamma:[t,t + h] \to M$ be a (uniquely minimizing for $h$ small enough, by adapting \cite[pg.~2631]{KleLot08} to quadratic Lagrangians) $L^{t,t+h}$-geodesic with $\gamma_t = x$, $\dot\gamma_t = \nabla^{g_t} u_t(x)$. By \cite[Thm.~10.15]{Vil09}, and then the specific form of the Lagrangians under consideration, we also have $(\nabla^{g_t}L^{t,t +h}(\cdot, \gamma_{t + h}))_{\gamma_t} = -(\nabla_v^{g_{t}} L)(\gamma_t,\dot\gamma_t, t) = -\dot\gamma_t$. In particular, $(\nabla^{g_t}L^{t,t +h}(\cdot, \gamma_{t + h}))_x = -\nabla^{g_t}u_t(x)$.

Finally we can conclude by noting that for $h$, $d_{g_t}(x,\gamma_{t + h})$, $|v|_{g_t}$, small enough and $v \neq 0$, the second-order expansions for $u_t$ and $L^{t,t + h}(\cdot, \gamma_{t + h})$ near $x$ give
\begin{equation}
\big(L^{t,t+h}(\cdot,\gamma_{t + h}) + u_t\big)(\exp_x^{g_t}(v)) - \big(L^{t,t + h}(x,\gamma_{t + h}) + u_t(x)\big)\geq \frac{\La (S + I/2h)(v),v \Ra_{g_t}}{2} + o(|v_t|^2_{g_t}) > |v|_{g_t}^2 > 0\,.
\end{equation}
In particular, $L^{t,t+h}(\cdot,\gamma_{t + h}) + u_t$ achieves a unique minimum in a neighborhood of $x$, which is at $x$. By Lemma \ref{lem:costfromlagrangian} and boundedness of $u_t$, we conclude that the minimizer in the definition of Hopf-Lax lies in this neighborhood for $h$ small enough, and thus

\begin{equation}
u_{t + h}(\gamma_{t + h}) -  u_t(x) = \inf_{y \in M} \{L^{t,t+h}(y,\gamma_{t + h}) + u_t(y)\} - u_t(x) = L^{t,t + h}(x,\gamma_{t + h})\,.
\end{equation}
Taking now $\liminf_{h\downarrow 0}\frac{1}{h}$ of the RHS, we obtain

\begin{equation}
\partial_{t^+}^-u_t(x) + du_t(x)(\dot{\gamma}_t) = L(\dot{\gamma}_t,x,t) \qquad\implies \qquad \partial_{t^+}^-u_t(x) = L(\dot{\gamma}_t,x,t) - du_t(x)(\dot{\gamma}_t) = -H(du_t(x), x,t)\,,
\end{equation}
where for the last equality, we used the choice $\dot\gamma_t := \nabla^{g_t}u_t(x)$ and that the optimizer in the fiberwise Legendre transform of $L(\cdot,x,t)$ is given by the $g_t$-dual covector. 
\end{proof}

\section{Dynamic transport plans}\label{Appe:Dynplans}
\begin{proof}[Proof of Proposition \ref{prop:Lisini-for-geodesics}]
    For the first claim and the existence of $\eta$ optimal, see \cite[Thm.~7.21]{Vil09}. Moreover we can compute the speed of the curve for $r\in[s,t]$:
    \begin{align}\label{eq:energyplan2}
        (\dot{\mu}_r)_{L^r} &= \lim_{h\to 0}\frac{W_{L^{r,r+h}}(\mu_r,\mu_{r+h})}{h} \\
        &\overset{\cite{Vil09}}{=} \lim_{h\to0} \frac{1}{h} \int_C\int_r^{r+h} L(\dot\gamma_u,\gamma_u,u)du d\eta \\
        &= \int_CL(\dot\gamma_r,\gamma_r,r) d\eta.
    \end{align}

Finally, we want to relate the velocity $\dot\gamma_r$ to $d\phi_r$. Because $\pi=(e_r,e_t)_\#\eta$ is an optimal plan between $\mu_r$ and $\mu_t$, it holds (\cite[Thm.~5.10-(ii)-(c)]{Vil09}) that
\begin{align}
    \phi_{t}(y)-\phi_{r}(x)= L^{r,t}(x,y) \ \ \ \text{for $\pi$-a.e. $(x,y)$.}
\end{align}
On the other hand
\begin{align}
    \phi_t(y)-\phi_r(x) \leq L^{r,t}(x',y)+ \phi_r(x') - \phi_r(x)
\end{align}
hence (temporarily working with a support metric $g$) in a neighborhood of $0 \in T_xM$
\begin{align}
    \phi_r(\exp_x(v))-\phi_r(x) \geq L^{r,t}(x,y)-L^{r,t}(\exp_x(v),y) \geq \langle p,v \rangle + o(|v|)
\end{align}
if $p\in T_x^*M$ is a superdifferential of $L^{r,t}(\cdot,y)$ at $x$. By \cite[Prop.~10.15-(i)]{Vil09}, we can take $p=-d_v L(x,\dot\gamma_r,r)$, where $\gamma$ is an $L^{r,t}$ geodesic from $x$ to $y$. But then if $x$ is a differentiability point of $\phi_r$ ($\mu_r$-a.e. because $\mu_s$ is assumed absolutely continuous, and hence $\mu_r$ is as well by \cite[Thm.~8.7]{Vil09}) we get
\begin{align}
    d\phi_r(x)=-d_v L(x,\dot\gamma_r,r)
\end{align}
for $\pi$-a.e. $(x,y)$.
Because $L$ is strictly convex and superlinear the Legendre transform is a homeomorphism, $d_v L(x,\dot\gamma_r,r)=(\dot\gamma_r)^*$ and hence $\dot\gamma_r=-(d\phi_r)^*$. Since the latter identity holds on the $\eta$-full measure set
\begin{align}
    e_r^{-1}(\{x\in M\mid \phi_r \text{ differentiable at }x\}) \cap(e_r,e_t)^{-1}(\{(x,y)\in M\times M\mid \phi_t(y) - \phi_r(x) = L^{r,t}(x,y)\}) \subseteq C
\end{align}
we can conclude
\begin{align}
    (\dot{\mu}_r)_{L^r} = \int_C L(-(d\phi_r(\gamma_r))^*,\gamma_r,r)d\eta = \int_M L(-(d\phi_r(x))^*,x,r)d\mu_r\,.
\end{align}
\end{proof}

Now we prove Lemma \ref{lem:dynamiclifting}. In \cite{Lisini2007} the same conclusion holds for all a.c. curves in the Wasserstein space (induced by $d^p$): while we expect something similar might still hold for the Lagrangian case, we follow the more practical approach from \cite{Lott-somecalculations}.
\begin{proof}[Proof of Lemma \ref{lem:dynamiclifting}]
    Let $\mu_r=\rho_rdV_g$, $dV_g$ being the background volume measure, then we can find a unique (up to translation) $\phi_r$ such that its gradient (taken in the metric $2A_{x,r}$) solves the continuity equation:
    \begin{align}
        \partial_r\rho_r = -\div(\rho_r\nabla\phi_r)\,.
    \end{align}
    We then consider $\Phi_r$ to be the flow generated by $(\nabla\phi_r)_{r\in[s,t]}$, let $\Phi_{[s,t]}:M\mapsto C^1([s,t],M)$, $\Phi_{[s,t]}(x)=(\Phi_r(x))_{r\in[s,t]}$. We then define $\eta=(\Phi_{[s,t]})_\#\mu_s$, by construction we have $(e_r)_\#\eta=\mu_r$.

    $\leq)$ This direction holds in fact for any $\eta$ with $(e_r)_\#\eta=\mu_r$:
    \begin{align}
        (\dot\mu_r)_{L^r} &= \lim_{h\to0} \frac{W_{L^{r,r+h}}(\mu_r,\mu_{r+h})}{h} \\
        &\leq \lim_{h\to0} \frac{1}{h}\int_{M\times M} L^{r,r+h}(x,y) d(e_r,e_{r+h})_\#\eta \\
        &= \lim_{h\to0} \frac{1}{h} \int_{C} L^{r,r+h}(\gamma_r,\gamma_{r+h}) d\eta(\gamma) \\
        &\leq \lim_{h\to0} \frac{1}{h} \int_{C} \int_r^{r+h} L(\dot \gamma_u,\gamma_u,u) dud\eta(\gamma) \\
        &= \int_{C} L(\dot \gamma_r,\gamma_r,r) d\eta(\gamma).
    \end{align}

    $\geq)$ Fix $\mu_r$, $\mu_{r+h}$. Let $(\nu_u^h)_{u\in[r,r+h]}$ be a $W_{L^{r,r+h}}$ geodesic between them: $\nu_r^h=\mu_r$, $\nu_{r+h}^h=\mu_{r+h}$, and let $\theta^h\in\cP(C)$ be the lifting of $\nu^h$. Then for any $f\in C^1$:
    \begin{align}
        \int fd\mu_r- \int fd\mu_{r+h} &=  \int_C f(\gamma_{r+h})-f(\gamma_r) d\theta^h(\gamma) \\
        &= \int_C \int_r^{r+h} \langle df(\gamma_u),\dot\gamma_u\rangle du d\theta^h(\gamma) \\
        &\leq \int_C \int_r^{r+h} H(df(\gamma_u),\gamma_u,u) du d\theta^h + \int_C\int_r^{r+h}L(\dot\gamma_u,\gamma_u,u)dud\theta^h \\
        &= \int_r^{r+h}\int_M H(df(x),x,u) d\nu_u^h du + W_{L^{r,r+h}}(\mu_r,\mu_{r+h})
    \end{align}
    On the other hand, using the continuity equation
    \begin{align}
        \int fd\mu_r- \int fd\mu_{r+h} &= \int_r^{r+h} \frac{d}{du}\int fd\mu_u du \\
        &=\int_r^{r+h} \int \langle df,\nabla\phi_u\rangle d\mu_udu\,.
    \end{align}
    Hence 
    \begin{align}
        W_{L^{r,r+h}}(\mu_r,\mu_{r+h}) \geq \int_r^{r+h} \int \langle df,\nabla\phi_u\rangle d\mu_udu - \int_r^{r+h}\int_M H(df(x),x,u) d\nu_u^h du\,.
    \end{align}

Since $H$ is continuous in all inputs and $\nabla \phi$ is smooth, we have $H(df(\cdot), \cdot,u) \to H(df(\cdot), \cdot,r)$ and $\La df,\nabla \phi_u\Ra \to \La df,\nabla \phi_r\Ra \in C^0(M)$ uniformly as $u \to r$. We also have smooth (and thus also weak) convergence $\mu_u \weak \mu_r$, and uniformly weak convergence $\nu_{u_h}^h \weak \nu_r \in \cP(M)$ as $h\to0$ for any sequence $u_h\in[r,r+h]$. This last weak convergence holds because Proposition \ref{prop:W-implies-weak-conv} can be applied, as 
\begin{align}
\lim_{h \to 0}W_{L^{r,u_h}}(\nu_r^h,\nu_{u_h}^h) &= \lim_{h \to 0}W_{L^{r,r+h}}(\nu_r^h,\nu_{r+h}^h) - W_{L^{u_h,r+h}}(\nu_{u_h}^h,\nu_{r+h}^h) \\
&\leq \lim_{h\to0} \int_r^{r+h}\int_CL(\dot\gamma_u,\gamma_u,u)d\eta du - \int_{u_h}^{r+h}\int_CL(\dot\gamma_u,\gamma_u,u)d\theta^h du \\
&\leq \lim_{h\to0} \int_r^{r+h}\int_CL(\dot\gamma_u,\gamma_u,u)d\eta du + Ch\,.
\end{align}
For the first integral we used Proposition \ref{prop:Lisini-for-geodesics} along with the fact that $(e_{r})_\#\eta=\mu_r$, $(e_{r + h})_\#\eta=\mu_{r + h}$ (hence it is admissible), while for the second integral the bound on the Lagrangian $L(v,x,r)\geq -C$.

The pairings--of a uniformly converging sequence of functions with a weakly converging sequence of probability measures--themselves converge, i.e.
\begin{equation}
\lim_{u \to r} \int \langle df,\nabla\phi_u\rangle d\mu_u =  \int \langle df,\nabla\phi_r\rangle d\mu_r\,,\qquad 
\lim_{h \to 0}\int_M H(df(x),x,u_h) d\nu_{u_h} = \int_M H(df(x),x,r) d\nu_r \ \ \ \text{uniformly in $u_h$}.
\end{equation}

Taking $\lim_{h \to 0}\frac{1}{h}$ then gives
\begin{align}
        (\dot\mu_r)_{L^r} \geq \int_M \langle df(x),\nabla\phi_r(x)\rangle d\mu_r(x) - \int_M H(df(x),x,r) d\mu_r(x)\,.
\end{align}
Because $L$ is quadratic, $L(v,x,r)=A_{x,r}(v,v)+C_{x,r}$, the optimizer in the Legendre transform is just the dual through $2A_{x,r}$, hence choosing $f=\phi_r$ we conclude
\begin{align}
        (\dot\mu_r)_{L^r} &\geq  \int_M L(\nabla \phi_r,x,r) d\mu_r \\
        &=\int_C L(\dot\gamma_r,\gamma_r,r) d\eta \,.
\end{align}
\end{proof}

\section{Curve-smoothing lemma}

The following lemma will allow us to smooth out the densities of measures along $W_{L_{0/\pm}}$ geodesics in a way that is not too violent from the point of view of $L_{0/\pm}$ action or entropy. There is no significant difference between the arguments for the costs $L_{0}$, $L_-$, and $L_+$, so we state and prove the result for all three simultaneously.

\begin{lem}\label{lem:wass-geodesic-smoothing}
Let $(M,g_\tau)_{\tau \in I}$ be a family of smooth Riemannian manifolds parametrized by backwards time $\tau = T - t$, and $(\mu_\tau)_{\tau \in [\tau_1,\tau_2]}\subseteq \cP(M)$ a $W_{L_{0/\pm}^{\tau_1,\tau_2}}$ Wasserstein geodesic with a.c. endpoints $\mu_{\tau_i} \in \cP^{ac}(M,dV_{\tau_i})$, $i = 1,2$. Choose a cost $L_m$, $m = 0,+,-$. The times $[\tau_1 ,\tau_2] \subseteq I$ are required to satisfy $\tau_2 < \sup I$, and also $0 < \tau_1$ in the case $m = -$ and $0 < T - \tau_2$ in the case $m = +$. Assume that \eqref{eq:WCL0} ($m = 0$), \eqref{eq:WCL-} ($m =  -$), or \eqref{eq:WCL+} ($m = +$) hold, and for $\eps > 0$, define the $\eps$-related time operation for all $\tau \in [\tau_1,\tau_2]$

\begin{equation}
\tau^\eps := \begin{cases}
    \tau + \eps,&m = 0\,,\qquad\eps < \sup I - \tau_2 \\
    e^\eps\tau,& m = -\,,\qquad\eps < \ln (\sup I/\tau_2) \\
    T - e^{-\eps}(T - \tau),& m = +\,,\qquad \eps < \ln\big((T - \tau_2)/(T - \sup I)\big)
\end{cases}\,.
\end{equation}

Then there are curves $(\mu_\tau^\eps)_{\tau \in [\tau_1^\eps, \tau_2^\eps]} = (\rho_\tau^\eps\,dV_\tau)_{\tau \in [\tau_1^\eps, \tau_2^\eps]}$ such that:
\begin{itemize}
\item (regularity) We have the space-time regularity of densities $\rho^\eps \in C^\infty\big(M\times [\tau_1^\eps,\tau_2^\eps] ; (0,\infty)\big)$.
\item (action convergence) For any times $\tau_1 \leq \tilde\tau_1 \leq \tilde\tau_2 \leq \tau_2$ we have the bounded convergences
\begin{align}\label{eq:curve-conv}
\int_{\tilde\tau_1}^{\tilde\tau_2} (\dot{\mu}_{\sigma})_{L_{0/\pm}^\sigma}\,d\sigma &= \lim_{\eps \to 0} \int_{\tilde\tau_1^\eps}^{\tilde\tau_2^\eps} (\dot{\mu}_{\sigma}^\eps)_{L_{0/\pm}^\sigma}\,d\sigma\,,\\
\left|\int_{\tilde\tau_1^\eps}^{\tilde\tau_2^\eps} (\dot{\mu}_{\sigma}^\eps)_{L_{0/\pm}^\sigma}\,d\sigma\right| &\leq C = C(\tau_2,\,\tau_1,\,\mu_{\tau_1},\mu_{\tau_2},\, (M,g_\tau)_{\tau \in I})\,. 
\end{align}

(For the definition and existence of the metric derivative for geodesics, see Proposition \ref{prop:Lisini-for-geodesics}).
\item (entropy convergence) We have the convergence for each time (possibly to $+\infty$)

\begin{equation}\label{eq:entropy-conv}
\lim_{\eps \to 0} \cE(\mu_{\tau^\eps}^\eps) = \cE(\mu_\tau)\,,\qquad \forall \tau \in [\tau_1,\tau_2]\,.
\end{equation}
\end{itemize}
\end{lem}

\begin{rmk}
The assumption that \eqref{eq:WCL0}, \eqref{eq:WCL-}, resp. \eqref{eq:WCL+} hold is not actually needed. It is in fact sufficient that $\cD(X) \geq -C(|X|^2 + 1) > -\infty$ for some $C \in \R{}$, which is automatic since $M$ is assumed smooth and closed. However, to carefully check this one would need to keep track of the corresponding error throughout the proof of $\cD \geq 0 \implies$ \eqref{eq:WCL0}, \eqref{eq:WCL-}, \eqref{eq:WCL+}. We therefore leave the details to the interested reader.
\end{rmk}

\begin{proof}
Let us begin by remarking that since $\mu_{\tau_1} = \rho_{\tau_1}\,dV_{\tau_1}$ and $\mu_{\tau_2} = \rho_{\tau_2}\,dV_{\tau_2}$ are assumed absolutely continuous, each measure along the geodesic $(\rho_\tau\,dV_\tau)_{\tau \in [\tau_1,\tau_2]}$ is as well by \cite[Thm.~8.7]{Vil09}.
We define $(\mu_{\tau^\eps}^\eps)_{\tau^\eps \in [\tau_1^\eps,\tau_2^\eps]}$ using the ``sliding window'' formula

\begin{equation}
\mu^\eps_{\tau^\eps} := \hat{P}_{\tau,\tau^\eps}[\mu_{\tau}]\,,\qquad\rho_{\tau^\eps}^\eps = P_{\tau,\tau^\eps}^*[\rho_{\tau}]\,,\qquad \forall \tau \in [\tau_1,\tau_2]\,. 
\end{equation}

Let us note that $\rho^\eps_{\tau^\eps}$ is nowhere $0$ for $\eps > 0$ by the strong maximum principle, and is a smooth function on $M$ for each $\tau^\eps \in [\tau_1^\eps,\tau_2^\eps]$. To see that it has the desired regularity in time, we can make use of Proposition \ref{prop:Lisini-for-geodesics} to write $\mu_\tau = (e_\tau)_\# \eta$ for some $\eta \in \cP(C([\tau_1,\tau_2],M))$ that is supported on $L_{0/\pm}$ geodesics. These curves are in particular $C^\infty$ in their parameter, so writing $p_{\sigma,\tau}^*(x,y)$ the adjoint heat kernel,

\begin{align}
\lim_{\eta \to 0}\frac{\rho_{(\tau + \eta)^\eps}^\eps(x) - \rho_{\tau^\eps}^\eps(x)}{\eta} &= \lim_{\eta \to 0} \frac{1}{\eta}\left(\int_M p_{\tau + \eta, (\tau + \eta)^\eps}^*(x,y)\,d\mu_{\tau + \eta}(y) - \int_M p_{\tau,\tau^\eps}^*(x,y)\,d\mu_{\tau}(y)\right) \\
&= \lim_{\eta \to 0} \frac{1}{\eta}\left(\int_C p_{\tau + \eta,(\tau + \eta)^\eps}^*(x,\gamma_{\tau + \eta})\,d\nu(\gamma) - \int_C p_{\tau,\tau^\eps}^*(x,\gamma_{\tau})\,d\nu(\gamma)\right) \\
&= \int_C \La d_yp_{\tau,\tau^\eps}^*(x,\gamma_{\tau}),\dot{\gamma}_{\tau}\Ra + \left.\partial_\eta\right|_{\eta = 0}(p^*_{\tau+ \eta,\tau^\eps } + p^*_{\tau,(\tau + \eta)^\eps})(x,\gamma_{\tau})\,d\nu(\gamma)\,,\qquad x \in M\,,
\end{align}
where we passed the limiting difference quotient into the integral by dominated (even bounded) convergence. The RHS is a smooth function of $x$ by similarly differentiating under the integral sign, and arbitrary derivatives in $x$ are continuous in time because $\nu$ is concentrated on uniformly $C^1$ curves. One may take arbitrarily many $\tau$ derivatives in this way, showing that $\rho^\eps_{\tau^\eps}(x)$ has all partial derivatives in the $x,\tau$ variables, which are moreover continuous. This shows that $\rho^\eps \in C^\infty\big([\tau_1^\eps,\tau_2^\eps]\times M;(0,\infty)\big)$ as desired. \vspace{0.5cm}

We now turn to the convergence of $W_{L_{0/\pm}}$ actions. The heat flow for Borel probability measures is weakly continuous up to the initial time, i.e. $\mu^\eps_{\tau_\eps} \weak \mu_\tau$ as $\eps \to 0$ for all $\tau \in [\tau_1,\tau_2]$. Applying this to the endpoints of any interval $[\tilde\tau_1, \tilde\tau_2] \subseteq [\tau_1,\tau_2]$ yields the convergence of transportation costs

\begin{equation}
W_{L_{0/\pm}^{\tilde\tau_1 ^\eps,\tilde\tau_2 ^\eps}}(\mu_{\tilde\tau_1^\eps}^\eps, \mu_{\tilde\tau_2^\eps}^\eps) \to W_{L_{0/\pm}^{\tilde\tau_1,\tilde\tau_2}}(\mu_{\tilde\tau_1},\mu_{\tilde\tau_2}) \qquad \text{as }\eps \to 0\,.
\end{equation}
as a result of \cite[Thm.~5.20]{Vil09}. Now let $(\phi_\tau^\eps)_{\tau \in [\tilde\tau_1^\eps,\tilde\tau_2^\eps]}$ be the Kantorovich potentials for the pair $(\mu_{\tilde\tau_1^\eps}^\eps, \mu_{\tilde\tau_2^\eps}^\eps)$ given by Corollary \ref{cor:dynamic-potentials}. We can use $\phi_\tau^\eps$ as a candidate for the dual problem w.r.t any other pair of measures, followed by the assumed \eqref{eq:WCL0}, \eqref{eq:WCL-}, or \eqref{eq:WCL+}, in
\begin{align}
\frac{W_{L_{0/\pm}^{\sigma_1^\eps,\sigma_2^\eps}}(\mu_{\sigma_1^\eps}^\eps,\mu_{\sigma_2^\eps}^\eps)}{\sigma_2^\eps - \sigma_1^\eps} &= \frac{1}{\sigma_2^\eps - \sigma_1^\eps}\left.\int_M \phi_{\sigma}^\eps \,d\mu_{\sigma}^\eps \right|_{ \sigma_1^\eps}^{\sigma_2^\eps} \\
&= \frac{1}{\sigma_2^\eps - \sigma_1^\eps}\left.\left(\int_M\phi_\sigma^\eps \,d\hat{P}_{\sigma,\sigma^\eps}\mu_{\sigma}\right)\right|_{ \sigma_1^\eps}^{\sigma_2^\eps} \\
&\leq \frac{W_{L_{0/\pm}^{\sigma_1^\eps, \sigma_2^\eps}}(\hat{P}_{\sigma_1,\sigma_1^\eps}\mu_{\sigma_1}, \hat{P}_{\sigma_2,\sigma_2^\eps}\mu_{\sigma_2})}{\sigma_2^\eps - \sigma_1^\eps} \\
&\leq \frac{(1 + \Psi(\eps))\cdot W_{L_{0/\pm}^{\sigma_1, \sigma_2}}(\mu_{\sigma_1}, \mu_{\sigma_2}) + \Psi(\eps)\cdot (\sigma_2 - \sigma_1)}{(1 + \Psi(\eps))\cdot (\sigma_2 - \sigma_1)}\,,\qquad \tilde\tau_1 ^\eps \leq \sigma_1^\eps < \sigma_2^\eps \leq \tilde\tau_2^\eps\,.
\end{align}
Now, we can use Propositions \ref{prop:Lisini-for-geodesics} and Lemma \ref{lem:dynamiclifting} to take the limit $\sigma_1^\eps, \sigma_2^\eps \to \sigma^\eps$ on both sides, obtaining $(\dot{\mu}_{\sigma ^\eps}^\eps)_{L_{0/\pm}^{\sigma^\eps}} \leq (1 + \Psi(\eps))\cdot (\dot{\mu}_{\sigma})_{L_{0/\pm}^{\sigma}} + \Psi(\eps)$ for any $\tilde\tau_1^\eps \leq \sigma^\eps \leq\tilde\tau_2^\eps$. Integrating in the variable $\sigma^\eps \in [\tilde\tau_1^\eps,\tilde\tau_2^\eps]$ gives
\begin{align}
W_{L_{0/\pm}^{\tilde\tau_1^\eps,\tilde\tau_2^\eps}}(\mu_{\tilde\tau_1^\eps}^\eps,\mu_{\tilde\tau_2^\eps}^\eps) \leq \int_{\tilde\tau_1^\eps}^{\tilde\tau_2^\eps}(\dot{\mu}_\sigma^\eps)_{L_{0/\pm}^{\sigma}}\,d\sigma &\leq \int_{\tilde\tau_1}^{\tilde\tau_2}\Big((1 + \Psi(\eps))\cdot(\dot{\mu}_{\sigma})_{L_{0/\pm}^{\sigma}}+ \Psi(\eps)\Big)(1 + \Psi(\eps))\,d\sigma  \\
&= (1 + \Psi(\eps))\cdot W_{L_{0/\pm}^{\tilde\tau_1,\tilde\tau_2}}(\mu_{\tilde\tau_1}, \mu_{\tilde\tau_2})+ \Psi(\eps)\,,
\end{align}
where we used the $W_{L_{0/\pm}}$ triangle inequality for the first inequality. After taking $\limsup_{\eps \to 0}$ of each expression in the above chain of inequalities, along with the previously established convergence of transportation costs applied to the LHS, implies that all quantities above converge, and to the same limit $\int_{\tilde\tau_1}^{\tilde\tau_2}(\dot\mu_\sigma)_{L_{0/\pm}^{\sigma}}\,d\sigma$. This gives \eqref{eq:curve-conv}. The same chain of inequalities, along with the uniform lower bound $L_{0/\pm}(v,x,\tau) \geq -C$ on the Lagrangians inducing the costs of $W_{L_{0/\pm}}$, gives uniform bounds on the action (with perhaps a larger $C$).
\begin{equation}
-C(\tau_2^\eps - \tau_1^\eps) \leq \int_{\tilde\tau_1 ^\eps}^{\tilde\tau_2^\eps} (\dot\mu_\sigma^\eps)_{L_{0/\pm}^\sigma}\,d\sigma \leq   C( W_{L_{0/\pm}^{\tau_1,\tau_2}}(\mu_{\tau_1},\mu_{\tau_2}) + 1 + \tau_2^\eps - \tau_1^\eps)\,.
\end{equation}

It remains to check the entropy convergence. Since $[0,\infty)\ni x\mapsto x\ln x$ is bounded from below, and $P^*_{\tau, \tau^\eps}[\rho_\tau] \to \rho_\tau$ pointwise a.e. as $\eps \to 0$ by the Lebesgue differentiation theorem, Fatou's Lemma implies that 
\begin{equation}
\cE(\mu_\tau) = \int_M \rho_\tau \ln \rho_\tau\, dV_\tau \leq \liminf_{\eps \to 0}\int_M \rho_{\tau^\eps}^\eps \ln (\rho_{\tau^\eps}^\eps)\,dV_{\tau^\eps} = \liminf_{\eps \to 0}\cE(\mu_{\tau^\eps}^\eps)\,,\qquad \forall \tau \in [\tau_1,\tau_2]\,.
\end{equation}
On the other hand, the convexity of $x \mapsto x\ln x$ along with Jensen's inequality applied in the probability space $(M_x, p^*_{\tau,\tau^\eps}(y,x)\,dV_\tau(x)/P_{\tau,\tau^\eps}^*[1](y))$ for each $y \in M$ yields
\begin{align}
\cE(\mu_{\tau^\eps}^\eps)&=\int_M(P^*_{\tau, \tau^\eps}[\rho_{\tau}])\ln(P^*_{\tau, \tau^\eps}[\rho_{\tau}])\,dV_{\tau^\eps} \\
&= \int_M (P^*_{\tau, \tau^\eps}[1])\left(\frac{P^*_{\tau, \tau^\eps}[\rho_{\tau}]}{P^*_{\tau, \tau^\eps}[1]}\right)\ln\left(\frac{P^*_{\tau, \tau^\eps}[\rho_{\tau}]}{P^*_{\tau, \tau^\eps}[1]}\right)\,dV_{\tau^\eps} + \int_M \ln (P^*_{\tau, \tau^\eps}[1])\,d\hat{P}_{\tau, \tau^\eps}[\mu_{\tau}] \\
&\leq \int_M (P^*_{\tau, \tau^\eps}[1])\frac{P^*_{\tau, \tau^\eps}[\rho_{\tau}\ln(\rho_{\tau })]}{P^*_{\tau, \tau^\eps}[1]}\,dV_{\tau^\eps} + \int_M \ln (P^*_{\tau, \tau^\eps}[1])\,d\hat{P}_{\tau, \tau^\eps}[\mu_{\tau}]\\
&=\int_M P^*_{\tau, \tau^\eps}[\rho_{\tau}\ln(\rho_{\tau })]\,dV_{\tau^\eps} + \int_M \ln (P^*_{\tau, \tau^\eps}[1])\,d\hat{P}_{\tau, \tau^\eps}[\mu_{\tau}] \\
&= \cE(\mu_{\tau}) + \int_M P_{\tau,\tau^\eps}[\ln (P_{\tau, \tau^\eps}^*[1])]\,d\mu_{\tau}\,,\qquad \forall \tau \in [\tau_1,\tau_2]\,. 
\end{align}
Note that we used that the adjoint heat flow preserves the integral over $M$ to obtain the first term in the final equality. Moreover, the integrand $P_{\tau,\tau^\eps}[\ln (P_{\tau, \tau^\eps}^*[1])]$ of the second term tends uniformly to $0$ as $\eps \to 0$. Taking $\limsup_{\eps\to 0}$ on both sides shows $\limsup_{\eps \to 0}\cE(\mu_{\tau^\eps}^\eps) \leq\cE(\mu_\tau)$. Combining the upper and lower bounds on $\cE(\mu_\tau)$ gives the desired convergence.
\end{proof}
\end{appendices}

\bibliographystyle{alpha}
\bibliography{Refs}

\newcommand{\etalchar}[1]{$^{#1}$}
\begin{thebibliography}{CCG{\etalchar{+}}07}

\bibitem[ABS21]{AmbBruSem21}
Luigi Ambrosio, Elia Bru\'e, and Daniele Semola.
\newblock {\em Lectures on optimal transport}, volume 130 of {\em Unitext}.
\newblock Springer, Cham, [2021] \copyright 2021.
\newblock La Matematica per il 3+2.

\bibitem[Agm82]{Agmonlectures}
Shmuel Agmon.
\newblock {\em Lectures on exponential decay of solutions of second-order elliptic equations: bounds on eigenfunctions of {$N$}-body {S}chr\"odinger operators}, volume~29 of {\em Mathematical Notes}.
\newblock Princeton University Press, Princeton, NJ; University of Tokyo Press, Tokyo, 1982.

\bibitem[AGS14a]{AGS1}
Luigi Ambrosio, Nicola Gigli, and Giuseppe Savar\'e.
\newblock Calculus and heat flow in metric measure spaces and applications to spaces with {R}icci bounds from below.
\newblock {\em Invent. Math.}, 195(2):289--391, 2014.

\bibitem[AGS14b]{AGS2}
Luigi Ambrosio, Nicola Gigli, and Giuseppe Savar\'{e}.
\newblock Metric measure spaces with {R}iemannian {R}icci curvature bounded from below.
\newblock {\em Duke Math. J.}, 163(7):1405--1490, 2014.

\bibitem[AGS15]{AmbGigSav15}
Luigi Ambrosio, Nicola Gigli, and Giuseppe Savar\'{e}.
\newblock Bakry-\'{E}mery curvature-dimension condition and {R}iemannian {R}icci curvature bounds.
\newblock {\em Ann. Probab.}, 43(1):339--404, 2015.

\bibitem[Bam20a]{BamlerEntropy}
Richard~H Bamler.
\newblock Entropy and heat kernel bounds on a {R}icci flow background.
\newblock {\em arXiv preprint arXiv:2008.07093}, 2020.

\bibitem[Bam20b]{BamlerStructure}
Richard~H Bamler.
\newblock Structure theory of non-collapsed limits of {R}icci flows.
\newblock {\em arXiv preprint arXiv:2009.03243}, 2020.

\bibitem[Bam23]{BamlerCompactness}
Richard~H. Bamler.
\newblock Compactness theory of the space of super {R}icci flows.
\newblock {\em Invent. Math.}, 233(3):1121--1277, 2023.

\bibitem[BR25]{bustamante2025deriving}
Ignacio Bustamante and Martin Reiris.
\newblock Deriving {P}erelman's entropy from {C}olding's monotonic volume.
\newblock {\em arXiv preprint arXiv:2501.12949}, 2025.

\bibitem[Bre25]{brena2025perelman}
Camillo Brena.
\newblock Perelman's entropy and heat kernel bounds on {RCD} spaces.
\newblock {\em arXiv preprint arXiv:2503.03017}, 2025.

\bibitem[CCG{\etalchar{+}}07]{rf1}
Bennett Chow, Sun-Chin Chu, David Glickenstein, Christine Guenther, James Isenberg, Tom Ivey, Dan Knopf, Peng Lu, Feng Luo, and Lei Ni.
\newblock {\em The {R}icci flow: techniques and applications. {P}art {I}}, volume 135 of {\em Mathematical Surveys and Monographs}.
\newblock American Mathematical Society, Providence, RI, 2007.
\newblock Geometric aspects.

\bibitem[Che15]{Cheng15}
Li-Juan Cheng.
\newblock The radial part of {B}rownian motion with respect to $\mathcal{L}$-distance under {R}icci flow.
\newblock {\em J. Theoret. Probab.}, 28(2):449--466, 2015.

\bibitem[Col12]{ColdingMonoton}
Tobias~Holck Colding.
\newblock New monotonicity formulas for {R}icci curvature and applications. {I}.
\newblock {\em Acta Math.}, 209(2):229--263, 2012.

\bibitem[CRH20]{CabRivHasl2020}
Esther Cabezas-Rivas and Robert Haslhofer.
\newblock Brownian motion on {P}erelman's almost {R}icci-flat manifold.
\newblock {\em J. Reine Angew. Math.}, 764:217--239, 2020.

\bibitem[CRT12]{CabezasRivas-Topping}
Esther Cabezas-Rivas and Peter~M. Topping.
\newblock The canonical shrinking soliton associated to a {R}icci flow.
\newblock {\em Calc. Var. Partial Differential Equations}, 43(1-2):173--184, 2012.

\bibitem[CS81]{CarmonaSimon}
R.~Carmona and B.~Simon.
\newblock Pointwise bounds on eigenfunctions and wave packets in {$N$}-body quantum systems. {V}. {L}ower bounds and path integrals.
\newblock {\em Comm. Math. Phys.}, 80(1):59--98, 1981.

\bibitem[CS04]{CanSin04}
Piermarco Cannarsa and Carlo Sinestrari.
\newblock {\em Semiconcave functions, {H}amilton-{J}acobi equations, and optimal control}, volume~58 of {\em Progress in Nonlinear Differential Equations and their Applications}.
\newblock Birkh\"auser Boston, Inc., Boston, MA, 2004.

\bibitem[DS08]{DanSav08}
Sara Daneri and Giuseppe Savar\'e.
\newblock Eulerian calculus for the displacement convexity in the {W}asserstein distance.
\newblock {\em SIAM J. Math. Anal.}, 40(3):1104--1122, 2008.

\bibitem[EKS15]{eks}
Matthias Erbar, Kazumasa Kuwada, and Karl-Theodor Sturm.
\newblock On the equivalence of the entropic curvature-dimension condition and {B}ochner's inequality on metric measure spaces.
\newblock {\em Invent. Math.}, 201(3):993--1071, 2015.

\bibitem[ELS25]{erbar2025synthetic}
Matthias Erbar, Zhenhao Li, and Timo Schultz.
\newblock Synthetic notions of {R}icci flow for metric measure spaces.
\newblock {\em arXiv preprint arXiv:2501.07175}, 2025.

\bibitem[Eva98]{EvansPDE}
Lawrence~C. Evans.
\newblock {\em Partial differential equations}, volume~19 of {\em Graduate Studies in Mathematics}.
\newblock American Mathematical Society, Providence, RI, 1998.

\bibitem[FIN05]{Feldman-Ilmanen-Ni}
Michael Feldman, Tom Ilmanen, and Lei Ni.
\newblock Entropy and reduced distance for {R}icci expanders.
\newblock {\em J. Geom. Anal.}, 15(1):49--62, 2005.

\bibitem[Gig15]{Gig15}
Nicola Gigli.
\newblock On the differential structure of metric measure spaces and applications.
\newblock {\em Mem. Amer. Math. Soc.}, 236(1113):vi+91, 2015.

\bibitem[Ham82]{Hamilton82}
Richard~S. Hamilton.
\newblock Three-manifolds with positive {R}icci curvature.
\newblock {\em J. Differential Geometry}, 17(2):255--306, 1982.

\bibitem[HKN22]{HKN22}
Robert Haslhofer, Eva Kopfer, and Aaron Naber.
\newblock Differential {H}arnack inequalities on path space.
\newblock {\em Adv. Math.}, 410:Paper No. 108714, 47, 2022.

\bibitem[HN18a]{HaslhoferNaber18RF}
Robert Haslhofer and Aaron Naber.
\newblock Characterizations of the {R}icci flow.
\newblock {\em J. Eur. Math. Soc. (JEMS)}, 20(5):1269--1302, 2018.

\bibitem[HN18b]{HaslhoferNaber18Bochner}
Robert Haslhofer and Aaron Naber.
\newblock Ricci curvature and {B}ochner formulas for martingales.
\newblock {\em Comm. Pure Appl. Math.}, 71(6):1074--1108, 2018.

\bibitem[Hua10]{Huang}
Hong Huang.
\newblock Optimal transportation and monotonic quantities on evolving manifolds.
\newblock {\em Pacific J. Math.}, 248(2):305--316, 2010.

\bibitem[JZ16]{JianZhang16}
Renjin Jiang and Huichun Zhang.
\newblock Hamilton's gradient estimates and a monotonicity formula for heat flows on metric measure spaces.
\newblock {\em Nonlinear Anal.}, 131:32--47, 2016.

\bibitem[Ken23]{Kennedy23BochnerRF}
Christopher Kennedy.
\newblock A {B}ochner formula on path space for the {R}icci flow.
\newblock {\em Calc. Var. Partial Differential Equations}, 62(3):Paper No. 83, 25, 2023.

\bibitem[KL08]{KleLot08}
Bruce Kleiner and John Lott.
\newblock Notes on {P}erelman's papers.
\newblock {\em Geom. Topol.}, 12(5):2587--2855, 2008.

\bibitem[KL21]{KuwadaLi}
Kazumasa Kuwada and Xiang-Dong Li.
\newblock Monotonicity and rigidity of the {$\cW$}-entropy on {$\mathsf{RCD}(0,N)$} spaces.
\newblock {\em Manuscripta Math.}, 164(1-2):119--149, 2021.

\bibitem[Kop18]{Kop18}
Eva Kopfer.
\newblock Gradient flow for the {B}oltzmann entropy and {C}heeger's energy on time-dependent metric measure spaces.
\newblock {\em Calc. Var. Partial Differential Equations}, 57(1):Paper No. 20, 40, 2018.

\bibitem[KP11]{kuwada-philipowski}
Kazumasa Kuwada and Robert Philipowski.
\newblock Coupling of {B}rownian motions and {P}erelman's $\mathcal{L}$-functional.
\newblock {\em J. Funct. Anal.}, 260(9):2742--2766, 2011.

\bibitem[KS18]{KopStu18}
Eva Kopfer and Karl-Theodor Sturm.
\newblock Heat flow on time-dependent metric measure spaces and super-{R}icci flows.
\newblock {\em Comm. Pure Appl. Math.}, 71(12):2500--2608, 2018.

\bibitem[KS23]{kunikawa2023almost}
Keita Kunikawa and Yohei Sakurai.
\newblock Almost splitting and quantitative stratification for super {R}icci flow.
\newblock {\em arXiv preprint arXiv:2309.11882}, 2023.

\bibitem[KS24]{KopferStreetsOTGRF}
Eva Kopfer and Jeffrey Streets.
\newblock Optimal transport and generalized {R}icci flow.
\newblock {\em SIGMA Symmetry Integrability Geom. Methods Appl.}, 20:Paper No. 003, 15, 2024.

\bibitem[Kuw10]{kuwadaduality}
Kazumasa Kuwada.
\newblock Duality on gradient estimates and {W}asserstein controls.
\newblock {\em J. Funct. Anal.}, 258(11):3758--3774, 2010.

\bibitem[Kuw15]{kuwadaspace}
Kazumasa Kuwada.
\newblock Space-time {W}asserstein controls and {B}akry-{L}edoux type gradient estimates.
\newblock {\em Calc. Var. Partial Differential Equations}, 54(1):127--161, 2015.

\bibitem[Li12]{XDLi12}
Xiang-Dong Li.
\newblock Perelman's entropy formula for the {W}itten {L}aplacian on {R}iemannian manifolds via {B}akry-{E}mery {R}icci curvature.
\newblock {\em Math. Ann.}, 353(2):403--437, 2012.

\bibitem[Lis07]{Lisini2007}
Stefano Lisini.
\newblock Characterization of absolutely continuous curves in {W}asserstein spaces.
\newblock {\em Calc. Var. Partial Differential Equations}, 28(1):85--120, 2007.

\bibitem[LL15]{LiLi15}
Songzi Li and Xiang-Dong Li.
\newblock The {$W$}-entropy formula for the {W}itten {L}aplacian on manifolds with time dependent metrics and potentials.
\newblock {\em Pacific J. Math.}, 278(1):173--199, 2015.

\bibitem[Lot03]{LottBakryEmery}
John Lott.
\newblock Some geometric properties of the {B}akry-\'emery-{R}icci tensor.
\newblock {\em Comment. Math. Helv.}, 78(4):865--883, 2003.

\bibitem[Lot08]{Lott-somecalculations}
John Lott.
\newblock Some geometric calculations on {W}asserstein space.
\newblock {\em Comm. Math. Phys.}, 277(2):423--437, 2008.

\bibitem[Lot09]{Lot09}
John Lott.
\newblock Optimal transport and {P}erelman's reduced volume.
\newblock {\em Calc. Var. Partial Differential Equations}, 36(1):49--84, 2009.

\bibitem[LV09]{MR2480619}
John Lott and C\'edric Villani.
\newblock Ricci curvature for metric-measure spaces via optimal transport.
\newblock {\em Ann. of Math. (2)}, 169(3):903--991, 2009.

\bibitem[LY86]{LiYau86}
Peter Li and Shing-Tung Yau.
\newblock On the parabolic kernel of the {S}chr\"odinger operator.
\newblock {\em Acta Math.}, 156(3-4):153--201, 1986.

\bibitem[MT10]{mccanntopp}
Robert~J. McCann and Peter~M. Topping.
\newblock Ricci flow, entropy and optimal transportation.
\newblock {\em Amer. J. Math.}, 132(3):711--730, 2010.

\bibitem[Mü10]{mullermonot}
Reto Müller.
\newblock Monotone volume formulas for geometric flows.
\newblock {\em J. Reine Angew. Math.}, 643:39--57, 2010.

\bibitem[Nab13]{naber2013characterizations}
Aaron Naber.
\newblock Characterizations of bounded ricci curvature on smooth and nonsmooth spaces.
\newblock {\em arXiv preprint arXiv:1306.6512}, 2013.

\bibitem[Ni04]{Ni04}
Lei Ni.
\newblock The entropy formula for linear heat equation.
\newblock {\em J. Geom. Anal.}, 14(1):87--100, 2004.

\bibitem[Per02]{perelman2002entropy}
Grisha Perelman.
\newblock The entropy formula for the ricci flow and its geometric applications.
\newblock {\em arXiv preprint math/0211159}, 2002.

\bibitem[Str23]{Streets23}
Jeffrey Streets.
\newblock Scalar curvature, entropy, and generalized {R}icci flow.
\newblock {\em Int. Math. Res. Not. IMRN}, 2023(11):9481--9510, 2023.

\bibitem[Stu06a]{MR2237206}
Karl-Theodor Sturm.
\newblock On the geometry of metric measure spaces. {I}.
\newblock {\em Acta Math.}, 196(1):65--131, 2006.

\bibitem[Stu06b]{MR2237207}
Karl-Theodor Sturm.
\newblock On the geometry of metric measure spaces. {II}.
\newblock {\em Acta Math.}, 196(1):133--177, 2006.

\bibitem[Stu18]{sturmsrf}
Karl-Theodor Sturm.
\newblock Super-{R}icci flows for metric measure spaces.
\newblock {\em J. Funct. Anal.}, 275(12):3504--3569, 2018.

\bibitem[Stu21]{SturmUpper}
Karl-Theodor Sturm.
\newblock Remarks about synthetic upper {R}icci bounds for metric measure spaces.
\newblock {\em Tohoku Math. J. (2)}, 73(4):539--564, 2021.

\bibitem[Top06]{topping2006lectures}
Peter Topping.
\newblock {\em Lectures on the Ricci flow}, volume 325.
\newblock Cambridge University Press, 2006.

\bibitem[Top09]{Top09}
Peter Topping.
\newblock $\mathcal{L}$-optimal transportation for {R}icci flow.
\newblock {\em J. Reine Angew. Math.}, 636:93--122, 2009.

\bibitem[Vil09]{Vil09}
C\'edric Villani.
\newblock {\em Optimal transport}, volume 338 of {\em Grundlehren der mathematischen Wissenschaften [Fundamental Principles of Mathematical Sciences]}.
\newblock Springer-Verlag, Berlin, 2009.
\newblock Old and new.

\bibitem[vRS05]{sturmvonrenesse}
Max-K. von Renesse and Karl-Theodor Sturm.
\newblock Transport inequalities, gradient estimates, entropy, and {R}icci curvature.
\newblock {\em Comm. Pure Appl. Math.}, 58(7):923--940, 2005.

\end{thebibliography}
\end{document}